\documentclass[oneside, 11pt,a4paper]{amsart}

\usepackage{stmaryrd}
\usepackage{subcaption}
\usepackage{graphicx} % Required for inserting images
\usepackage{amssymb} \usepackage{amsfonts} \usepackage{amsmath}
\usepackage{amsthm} \usepackage{epsfig} 
\usepackage{color}
\usepackage[T1]{fontenc}
\usepackage{pinlabel}
\usepackage{amscd}
\usepackage[all]{xy}
\usepackage{pinlabel}
\usepackage{mathrsfs}  
\usepackage[most]{tcolorbox}
\usepackage{hyperref}
\usepackage{float}
\usepackage{tikz}
\usetikzlibrary{decorations.pathreplacing, backgrounds, calc}
\usepackage{geometry}
\usepackage{tikz-cd}
\usepackage[shortlabels]{enumitem}
\usepackage{bm}
\usepackage{xfrac}
\usepackage{cleveref}
\usepackage[]{mdframed}
\usepackage{lmodern}
\usepackage{mathtools}
\DeclarePairedDelimiter{\ceil}{\lceil}{\rceil}
\DeclarePairedDelimiter{\floor}{\lfloor}{\rfloor}
\newtheorem{thm}{Theorem}[section]

\newtheorem{lem}[thm]{Lemma}
\newtheorem{add}[thm]{Addendum}
\newtheorem{corollary}[thm]{Corollary}
\newtheorem{prop}[thm]{Proposition}
\newtheorem{prob}[thm]{Problem}

\theoremstyle{definition}
\newtheorem{defn}[thm]{Definition}

\newtheorem{remark}{Remark}[section]

\newcommand{\bdve}[1]{\boldsymbol{\mathbf{#1}}}
\newcommand{\ve}{\varepsilon}
\newcommand{\Se}{S_\ve}
\newcommand{\bs}{\bm{s}}

\newcommand{\bt}{\bm{t}}

\newcommand{\bk}{\bm{k}}
\newcommand{\br}{\bm{r}}

\newcommand{\bH}{\ensuremath {\mathbb{H}}}
\newcommand{\Z}{\ensuremath {\mathbb{Z}}}
\newcommand{\Q}{\ensuremath {\mathbb{Q}}}
\newcommand{\R}{\ensuremath {\mathbb{R}}}
\newcommand{\F}{\ensuremath {\mathbb{F}}}

\newcommand{\cD}{\mathcal{D}}
\newcommand{\cK}{\mathcal{K}}
\newcommand{\cL}{\mathcal{L}}
\newcommand{\cM}{\mathcal{M}}
\newcommand{\cN}{\mathcal{N}}

\newcommand{\cX}{\mathcal{X}}
\newcommand{\cY}{\mathcal{Y}}

\newcommand{\bE}{\mathbb{E}}
\newcommand{\bI}{\mathbb{I}}

\newcommand{\bdE}{\bdve {E}}
\newcommand{\bdI}{\bdve {I}}
\newcommand{\bdx}{\bdve {x}}
\newcommand{\bdy}{\bdve {y}}
\newcommand{\bdw}{\bdve {w}}

\newtheoremstyle{namedd}{}{}{\itshape}{}{\bfseries}{.}{ }{#1 \thmnote{#3}}
\theoremstyle{namedd}
\newtheorem*{namedtheorem}{Theorem}

\DeclareMathOperator{\id}{{id}}

\DeclareMathOperator{\hboxtimes}{{\hat\boxtimes}}

\newcommand{\bCFm}{\bdve{CF}^{-}}
\newcommand{\bHFm}{\bdve{HF}^{-}}
\newcommand{\cCFL}{\mathcal{C\!F\!L}}
\newcommand{\cCFK}{\mathcal{C\hspace{-.5mm}F\hspace{-.3mm}K}}

\newcommand{\CF}{\mathit{CF}}

\newcommand{\HF}{\mathit{HF}}

\newcommand{\HFL}{\mathit{HFL}}

\newcommand{\Sl}{\mathit{Sl}}

\newcommand{\ds}[1]{{\color{blue}Diego: #1}}

\newtcbtheorem{Chapter}{\bfseries Chapter}{enhanced,drop shadow={black!50!white},
  coltitle=black,
  top=0.3in,
  attach boxed title to top left=
  {xshift=1.5em,yshift=-\tcboxedtitleheight/2},
  boxed title style={size=small,colback=pink}
}{chapter} 

\newcommand{\ie}{i.\,e.~}

\author{Diego Santoro}
\author{Hugo Zhou}

\address{HUN-REN Alfréd Rényi Institute of Mathematics, Reáltanoda utca 13-15, 1053, Budapest, Hungary}
\email{\href{mailto:diego.santoro95@gmail.com}{diego.santoro95@gmail.com}}

\address{Department of Mathematics, University of Michigan,  Ann Arbor, MI, 48109,   USA}
\email{\href{mailto:hugozhou@umich.edu}{hugozhou@umich.edu}}

\title{On L-space Surgeries on two-bridge links}

\begin{document}

\begin{abstract}
We classify the sets of $L$-space surgeries on all two-bridge links, providing the first examples of hyperbolic links for which such sets cannot be described as unions of finitely many rectangles in $\Q^2$.

The proof relies on several different techniques, each of which is applicable in greater generality: we introduce a sufficient diagrammatic condition for links in $S^3$ to be persistently foliar, a property that implies that every non-trivial surgery on such links supports a coorientable taut foliation. We define a simplified model for the Heegaard Floer homology of rational surgeries on two-component $L$-space links, following the work of Manolescu-Ozsv\'ath, Liu, and Zemke, and use it to obtain obstructions to $L$-space surgeries. Finally, we use explicit computations of Turaev torsions to determine $L$-space surgeries in the case of generalised $L$-space links. 

Among the consequences of our results, we obtain an optimal uniform bound on the volume of any hyperbolic $L$-space that is surgery on a two-bridge link, together with a classification of all $L$-space satellite knots whose associated two-component pattern link is a two-bridge link.
\end{abstract}

\maketitle

\tableofcontents

\section{Introduction}\label{section:intro}
An oriented rational homology sphere $M$ is an $L$-space if \(\widehat{\HF}(M)\) has dimension equal to the cardinality \(|H_1(M;\Z)|\) of its first homology group. Since this is the smallest possible rank, $L$-spaces have minimal Heegaard Floer homology. Standard examples include lens spaces and, more generally, elliptic \(3\)-manifolds \cite[Proposition~2.3]{OSlens05}. There are also many hyperbolic examples: for instance, every $(r_1,r_2,r_3)$-surgery on the Borromean rings, with $r_i\geq 1$, is an $L$-space. Although defined purely in terms of Heegaard Floer homology, $L$-spaces have been a useful  in studying several different topics in low-dimensional topology. For example, they can be used to study unknotting numbers (\cite{Gre14,McCoy17}), the cosmetic crossing conjecture (\cite{LidMoo17}), and to prove vanishing results for Seiberg-Witten invariants (\cite{AgoLin20,BFS24}). 
They also do not support coorientable taut foliations (\cite{OSgenus04,Bow16,KR17}), and are conjecturally characterized among irreducible \(3\)-manifolds by this property; see \cite{Juh15}. In a different direction, Boyer, Gordon and Watson conjectured in \cite{BGW13} that an irreducible $3$-manifold is an $L$-spaces if and only if its fundamental group is not left-orderable. The conjectured equivalence of these three notions is now known as the \emph{$L$-space conjecture} and it has been proven for graph manifolds \cite{BC, HRRW, Ras17}, but remains open in general.
\newline

A feature of Heegaard Floer theory is that information about the Heegaard Floer homology of a \(3\)-manifold \(M\) can often be extracted from a presentation of \(M\) as surgery on a link. This is especially effective when determining whether $M$ is an $L$-space. In \cite{GorNem16}, the definition of an \emph{$L$-space link} was introduced.
\begin{defn}
An $n$-component link \(L\subset S^3\) is an \emph{$L$-space link} if all of its positive  large integer surgeries are $L$-spaces, \ie if there exists $(p_1,\dots,p_n)\in \Z^n$ such that the surgery \(S^3_{d_1,\dots, d_n}(L)\) is an $L$-space for every integers $d_1,\dots, d_n$ with \(d_i\geq p_i\), \(\forall i=1,\dots, n\).     
\end{defn}
This definition can be naturally extended to that of \emph{generalised $(\pm,\dots,\pm)\,L$-space links}, obtained by considering the corresponding types of generalised large surgeries. According to this definition, an $L$-space link is a generalised $(+,\dots, +)\,L$-space link.

There is a significant difference between studying $L$-space surgeries on knots and on links with two or more components. For example, if a knot $K$ in $S^3$ has a non-meridional $L$-space surgery, then, up to mirroring, $K$ is an $L$-space knot. More precisely, the $r$-surgery on such a knot $K$ is an $L$-space if and only if \(r\in [2g(K)-1, \infty]\) \cite[Lemma 2.13]{HeddenCabling2}, where $g(K)$ denotes the genus of $K$ and \(r\in \overline{\Q}=\Q \cup \{\infty\}\). In addition, $L$-space knots satisfy strong topological constraints: they are fibered with right-veering monodromy, they are strongly quasipositive, and their $3$-dimensional and $4$-dimensional genus coincide and agree with their $\tau$ invariant \cite{Ghi08,Ni07,Hed10,HonKazMat07}.
For links, many of these properties do not hold. In particular, $L$-space links need not be fibered \cite[Example~3.9]{yajing_lspace}, or strongly quasipositive \cite[Proposition~1.5]{BeiCav16}, and there exist links with positive $L$-space surgeries that are not $L$-space links. Examples of such links will appear also in this paper. 

In general, our understanding of the possible shapes and structural properties of the set of $L$-space surgeries on a link, considered as a subset of \(\overline{\Q}^n\), is very limited.
In \cite{GorNem18}, the authors provide sufficient conditions for this set to be bounded from below for two-component $L$-space links, and they completely characterise when this happens for two-component algebraic links. In \cite{GorLiuMoo20,BeibeiLiu21}, among other results, the possible regions of integer $L$-space surgeries on two-component $L$-space links are studied, and, in some cases, classified.
The first non-trivial examples of complete classifications of \emph{rational} $L$-space surgeries are given in \cite{Ras20} where, using a gluing formula for $L$-spaces together with the classification result of \cite{Ras17}, the author classifies the set $L$-space surgeries on satellites by torus links. %\ds{to check}
\newline

For hyperbolic links, very few examples of classifications are known, and in those cases the set of $L$-space (rational) surgeries is fairly simple, being either empty of union of finitely many rectangles in $\Q^n$; for instance, we have one rectangle for the Whitehead link, and two rectangles for the Borromean rings\footnote{See, for instance, \cite{San22Whit, San24diag}.}. The goal of this paper is to describe the sets of $L$-space surgeries on all hyperbolic links with two components belonging to the family of \emph{two-bridge links}, and, by doing so, to provide more interesting examples of sets of $L$-space surgeries. 

We refer to Section~\ref{section:two-bridge} for definitions and some properties of two-bridge links. For now, recall that two-bridge links can have either one or two components, and in this paper we will always consider two-component links, since for two-bridge knots the situation is well-understood\footnote{The only two-bridge knots with non-meridional $L$-space surgeries are the $T(2,k)$ torus knots.}. Two-bridge links have unknotted components, and hence any surgery where at least one surgery coefficient is $\infty$ is either an $L$-space, indeed a lens space, or $S^2\times S^1$. For this reason, we will focus only on \emph{non-trivial} surgeries, i.e. those where no surgery coefficient is $\infty$.

Our first result is a classification of two-bridge links with non-trivial $L$-space surgeries. Indeed, we prove a stronger statement:

\begin{thm}\label{thm: classification of links with L-space surgeries}
A two-bridge link $L$ has a non-trivial $L$-space surgery if and only if it is isotopic, up to mirror, to $L_{n,k}$ or $L'_{n,k}$, for some $n,k\geq 0$. Moreover, every non-trivial surgery on any other two-bridge link supports a coorientable taut foliation.
\end{thm}

\begin{figure}
    \centering
    \includegraphics[width=0.4\textwidth]{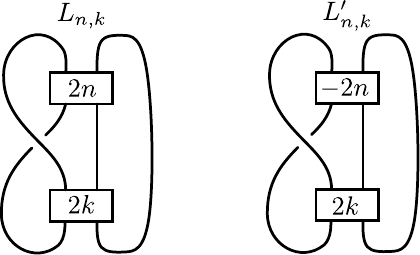}
    \caption{The only two-bridge links, up to mirror, with non-trivial $L$-space surgeries.}
    \label{fig: L and L'}
\end{figure}

The links \(L_{n,k}\) and \(L'_{n,k}\) are depicted in Figure \ref{fig: L and L'}, where we use boxes to indicate the number of crossings between two strands, with the convention that right-handed crossings are positive and left-handed crossings are negative. We observe that the families $L_{n,k}$ and $L'_{n,k}$, are not disjoint, and their intersection is given by the $T(2,2k)$ torus links.
\newline

The bulk of the proof of Theorem~\ref{thm: classification of links with L-space surgeries} is presented in Section~\ref{section:foliations} and exploits diagrammatic properties of two-bridge links. More precisely, we introduce the notion of a \emph{persistently foliar link}\footnote{The definition is a straightforward generalisation of that of persistently foliar knot, given by Delman and Roberts in \cite{DRdiamond}; see Definition \ref{def: pers fol link}.}, and we prove a general theorem which provides sufficient conditions for an $n$-component link to be persistently foliar; we refer to Section \ref{section:foliations}, and in particular to Theorem \ref{thm: pers fol link}, for details. As a direct consequence of the definition, all non-trivial surgeries on a persistently foliar link support coorientable taut foliations, and since $L$-spaces do not, Theorem \ref{thm: pers fol link} allows us to identify the links in the families $L_{n,k}$ and $L'_{n,k}$ as the only two-bridge links that can have $L$-space surgeries. We show that this is indeed the case in Section \ref{sec: classification}. 
\newline
 
The links $L_{n,k}$ already appeared in the literature in \cite{yajing_lspace}, where Y. Liu conjectured that these are exactly the two-bridge $L$-space links. This was proved by Dawra in his thesis \cite{DawraThesis}. As a consequence of our results we obtain a different proof of Liu's conjecture. In addition, with the next result, we classify all $L$-space surgeries on the links $L_{n,k}$. Recall that a two-bridge link is either hyperbolic or a  $T(2,2k)$ torus link, with $k\in\Z$. Dehn surgeries on torus links are Seifert fibered manifolds, or connected sums thereof, and since the $L$-spaces in this family are classified \cite{LiscaStipsicz}, we focus on the hyperbolic two-bridge links.

\begin{thm}\label{thm: L-space links}
A hyperbolic two-bridge link $L$ is an $L$-space link if and only if it is isotopic to $L_{n,k}$ for some $n\geq 1, k\geq 1$. In this case, for rationals $r,s$, the $(r,s)$-surgery on $L_{n,k}$ is an $L$-space if and only if $\operatorname{min}\{r,s\}\geq n+k-1$.
\end{thm}

\begin{figure}
    \centering
    \includegraphics[width=0.5\textwidth]{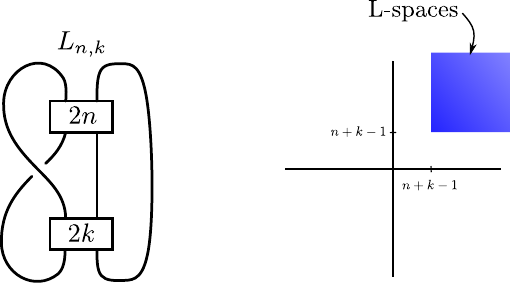}    \caption{The set of $L$-space surgeries on the link $L_{n,k}$, for $n\geq 1, k\geq 1$.}
    \label{fig: L with L-space surgeries}
\end{figure}

We are also able to complement Theorem~\ref{thm: L-space links} by showing that every non-$L$-space surgery on $L_{n,k}$ supports a coorientable taut foliation.

\begin{thm}\label{thm: CTF iff Lspace for Lnk}
Let $n,k\geq 1$, and let $M$ be a Dehn surgery on the link $L_{n,k}$. Then $M$ is not an $L$-space if and only if it supports a coorientable taut foliation.    
\end{thm}
We prove Theorems~\ref{thm: L-space links} and~\ref{thm: CTF iff Lspace for Lnk} in Section~\ref{subsec: class for Lnk}.
When \(k=1\) or \(n=1\), the links \(L_{n,k}\) coincide with the fibered hyperbolic two-bridge links admitting \(L\)-space surgeries and were previously studied in \cite{San22Whit, San23twobridge}. Interestingly, while in \cite{San22Whit, San23twobridge} the first author used taut foliations to show that a certain explicitly described collection of surgeries on \(L_{n,1}\) constituted the complete set of \(L\)-space surgeries on \(L_{n,1}\), here the direction of the argument is partially reversed. Indeed, some of the taut foliations provided by Theorem~\ref{thm: CTF iff Lspace for Lnk} are obtained as a consequence of the classification result Theorem~\ref{thm: L-space links}. This is made possible by Lyu's recent study of persistently foliar \((1,1)\) non-\(L\)-space knots~\cite{Lyu25}.
\newline

Focusing on the links $L'_{n,k}$, we show in Section \ref{sec: classification} that they are generalised $(+,-)\,L$-space links and that, except for the torus links, they are not $L$-space links. Since two-bridge links are symmetric, \ie they admit an isotopy exchanging their components, a two-bridge link is a generalised $(+,-)\,L$-space link if and only if it is a generalised $(-,+)\,L$-space link. We classify the $L$-space surgeries on the hyperbolic links in $L'_{n,k}$ with the following theorem.
\begin{thm}\label{thm: gen L-space links}
A  hyperbolic two-bridge link $L$ is a generalised $(+,-)\,L$-space link if and only if it is isotopic to $L'_{n,k}$, with $n\geq 2, k\geq 1$. 
In this case, denote $l=n+k$ and fix $r,s\in \Q^2$. Then the $(r,s)$-surgery on $L'_{n,k}$ is an $L$-space if and only if $r$ and $s$  satisfy one of the following conditions:
\begin{itemize}[label=--, leftmargin=2em, itemsep=1.2ex]

\item If $r < 1-2n$, then
\[
s \geq \frac{1}{\left\lfloor \frac{r + (2n-1)}{l} \right\rfloor}.
\]

\item If $r\in [-m,-m+1)$ for some integer $-m\in [1-2n,-2]$,then  
\[
s\geq -1-2n+m.
\] 
\item If $r \in \big[-\tfrac{1}{m},\, -\tfrac{1}{m+1}\big)$ for some integer $m \geq 1$, then
\[
s \geq 1 - 2n - ml.
\]

\item If $r=0$, then $s\ne 0$.

\item If $r \in \big(\tfrac{1}{m+1},\, \tfrac{1}{m}\big]$ for some integer $m \geq 1$, then
\[
s \leq 2k+1 + ml.
\]

\item If $r\in (m-1,m]$ for some integer $m\in [2,2k+1]$, then
\[s\geq 2k+3-m.
\]

\item If $r > 2k+1$, then
\[
s \leq \frac{1}{\left\lceil \frac{r - (2k+1)}{l} \right\rceil}.
\]

\end{itemize}
\end{thm}
\begin{figure}
    \centering
    \includegraphics[width=0.65\linewidth]{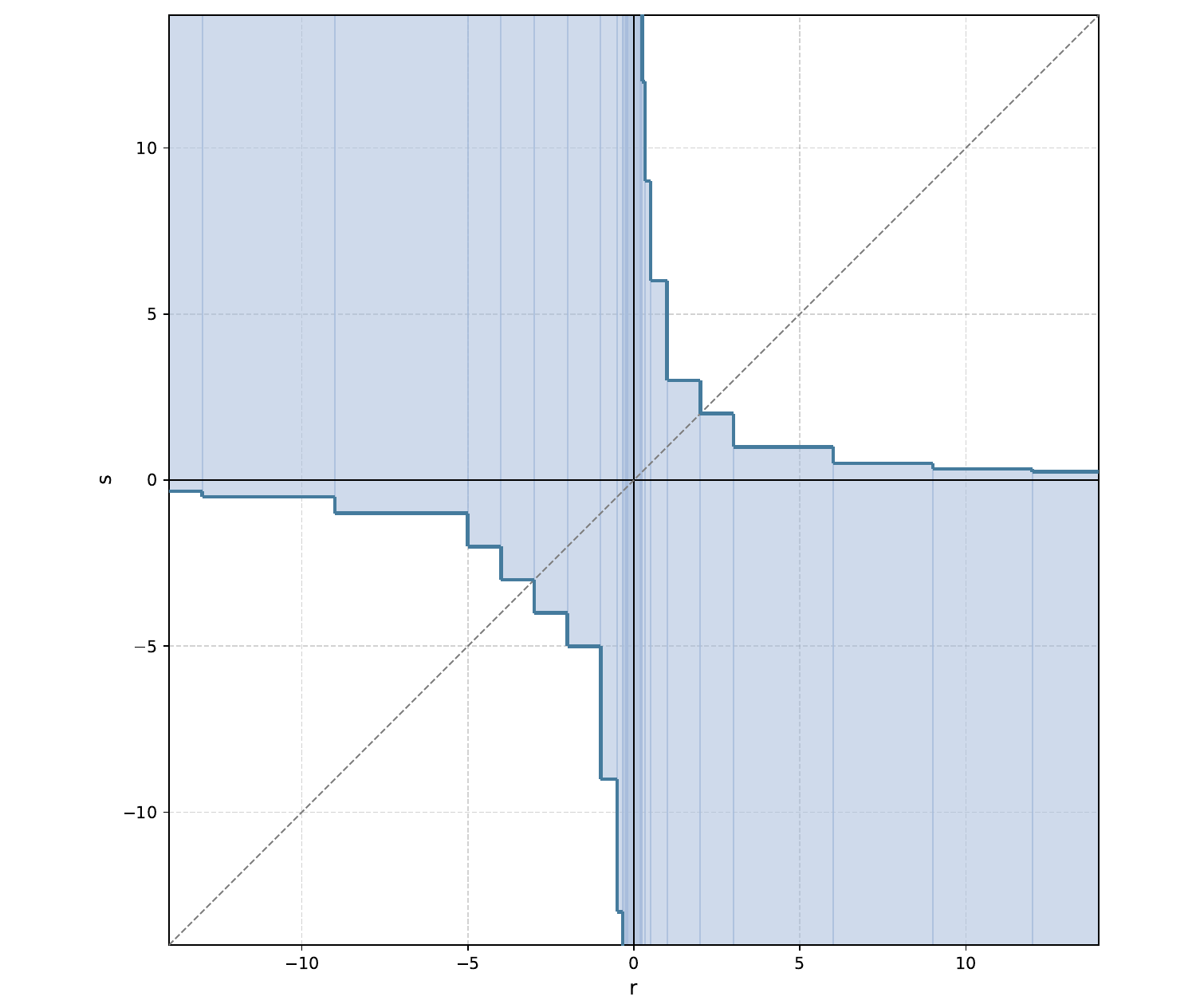}
\caption{The set of $L$-space surgeries on the link $L'_{3,1}$. Since two-bridge links are symmetric, the set is invariant under the reflection across the line $r=s$.}
\label{fig: gen $L$-space surgeries}
\end{figure}
Although not immediately apparent, the statement of Theorem~\ref{thm: gen L-space links} is invariant under exchanging $r$ and $s$, reflecting the symmetry of two-bridge links. We show in Figure~\ref{fig: gen $L$-space surgeries} the set of $L$-space surgeries for the link $L'_{3,1}$, which exhibits the features of the general case. 
\newline

The statement of Theorem \ref{thm: gen L-space links} can be rephrased more concisely in terms of \emph{box-convex} sets. We refer to Section \ref{subsec: box convexity} and specifically to Theorem \ref{thm: reformulation} for this reformulation, and for more comments regarding the structure of the set of $L$-space surgeries on these links. 

The proof of Theorem~\ref{thm: gen L-space links} is presented in Section~\ref{subsec: class for L'nk} and differs substantially from the one of Theorem~\ref{thm: L-space links}. In fact, the links $L_{n,k}$ are $L$-space links, and hence we can take advantage of some simplifications in the rational link surgery formula, that we describe in more detail below. On the other hand, the hyperbolic links in the family \(L'_{n,k}\) are not $L$-space link; in this case we proceed with an explicit and detailed study of the Turaev torsions of the manifolds obtained by filling only one of the boundary components of the exterior of the link. 
\subsection{Rational link surgery formula and the simplified model}
In the proof of Theorem \ref{thm: L-space links}, we obstruct certain rational surgeries on $L_{n,k}$ from yielding $L$-spaces. This is achieved using a simplified model of the rational link surgery formula, following \cite{mo_linksurgery,yajing_lspace,zemke_bordered}.

Given an oriented $n$-component link $L\subset S^3$ and an integer framing $\Lambda = (d_1,\dots,d_n)$,
 Manolescu and Ozsv\'ath constructed in \cite{mo_linksurgery} a \emph{link surgery complex} $C_\Lambda(L)$ from the link Floer complex of $L$ and proved that 
the homology of $C_\Lambda(L)$ is the \emph{completed Heegaard Floer homology} \(\bHFm(S^3_{\Lambda}(L))\). One can similarly recover  \(\widehat{\HF}(S^3_{\Lambda}(L))\) from $C_\Lambda(L)$.
The construction is referred to as the (integral) link surgery formula.

Zemke reinterpreted the link surgery formula  in the language of bordered modules in  \cite{zemke_bordered}. As a direct consequence, a version  of the link surgery formula which allows
for rational surgeries can be defined; see \cite[Remark 1.8]{zemke_bordered}. Since an explicit description of this rational link surgery formula has not appeared in the literature, we provide the details in Section \ref{subsec: rational_link_surgery_formula}
Specifically, we define a \emph{(rational) link surgery complex} by extending the construction of $C_\Lambda(L)$ to allow rational surgery coefficients, and verify that this complex is chain homotopy equivalent to \(\bCFm(S^3_{\Lambda}(L))\); see Theorem~\ref{thm: rational_surgery_chain}.
%H_*(C_\Lambda(L)) \cong \bHFm(S^3_{\Lambda}(L))$ \ds{this statement now is for the complexes, not the homologies. We can also state (or at least refer to) the theorem here? since now we have a thm environment} still holds for rational surgeries.

In the case of integral framings, the link surgery formula admits certain simplifications under favorable conditions. 
By construction,  $C_\Lambda (L)$ is a hypercube of chain complexes. This means that its underlying module is a direct sum of subspaces $C_\ve$ for $\ve \in \{0,1\}^n$ and the differential is the sum of the maps $D_{\ve,\ve'}$ for $\ve \leq \ve'$.
% one can pass to the homology before applying the formula, see \cite[Section 5]{OSinteger} for the knot case. Furthermore, 
Y. Liu  \cite{yajing_lspace} defined a \emph{simplified model} $H_\Lambda(L)$ (referred to by Liu as the perturbed complex)   by
\begin{itemize}
    \item first replacing each $C_\ve$ by $H_*(C_\ve)$; and
    \item then retaining only the maps induced by the maps $D_{\ve,\ve'}$ with $|\ve'-\ve|=1$.
\end{itemize}
It was proved by Liu \cite{yajing_lspace} that for surgeries with integral framings $\Lambda$ on a two-component $L$-space link $L$, the complex $C_\Lambda(L)$ is homotopy equivalent to $H_\Lambda(L)$. The proof relies on the fact that the link Floer homology of an $L$-space link is supported in even homological gradings.
In a different setting, Zemke proved that, for integer surgeries on plumbed trees, $H_\Lambda(L) \simeq C_\Lambda(L)$ in \cite{zemke_lattice}. 
% interpreting the link surgery formula via bordered modules \cite{zemke_bordered}. This approach was formalised in \cite{BLZlattice}; furthermore, the machinery  readily extends \ds{is this true? or we need to add the hypothses on the alexander gradings?} \hz{need the Alexander grading assumption} to show that for integral framings $\Lambda$, $H_\Lambda(L) \simeq C_\Lambda(L)$  for $L$-space links with formal link Floer complexes.
% Here, a complex is called \emph{formal} if it is homotopy equivalent
% % quasi-isomorphic 
% to the free resolution of its homology, following the terminology in \cite{BLZlattice}. It is shown in \cite{BLZlattice} that plumbed $L$-space links have formal link Floer complexes, and in \cite{CZZ_sat} that the same holds for two-component $L$-space links.
% \ds{maybe here we want to say something on this alexander formality? }
We extend  Liu's result on two-component $L$-space links to the setting of rational surgeries.
\begin{thm}\label{thm: intro_formal_link_H_Lambda}
    If $L\subset S^3$ is a two-component $L$-space link, then for any rational surgery coefficients $\Lambda$, we have $H_\Lambda(L) \simeq C_\Lambda(L)$.   
\end{thm}
% The arguments are essentially the same as those in \cite{zemke_lattice, BLZlattice}. Since this statement does not appear explicitly in the literature, we include the details for completeness in Section~\ref{subsec: simplified model}.
% We also note that Theorem~\ref{thm: intro_formal_link_H_Lambda} extends Liu's result on two-component $L$-space links to the setting of rational surgeries.

According to the simplified model, for a two-component $L$-space link $L$,  $\bHFm(S^3_\Lambda(L))$ is determined by the link Floer homology of $L$, which in turn is determined by the multivariable Alexander polynomials of its sublinks  via the so-called \emph{$H$-function} \cite{GorNem15}. Therefore to determine whether a surgery yields an $L$-space it suffices to study the $H$-function of $L.$ This idea was used in \cite{GorNem18,BeibeiLiu21} to obstruct certain integer surgeries on $L$ from being an $L$-space. Combining these arguments with the simplified model for rational link surgery, we are able to produce obstructions to $L$-space surgeries on $L_{n,k}$ in Section \ref{section: very good points and non L space surgeries}.
\subsection{Some consequences}
We conclude the introduction by collecting some consequences of the main results of this paper.
\subsubsection{Hyperbolic volumes}
It follows from the main result of \cite{Lac04} that the set of volumes of hyperbolic two-bridge links is unbounded. This was already observed for fibered two-bridge links in \cite[Proposition~1.7]{San23twobridge}, and it is not difficult to provide, in the same way, families of non-fibered two-bridge links whose volume goes to infinity.
On the other hand, a direct consequence of Theorem~\ref{thm: classification of links with L-space surgeries} is that the volumes of hyperbolic $L$-spaces that are surgeries on two-bridge links are uniformly bounded. Let $v_8\sim3.663862...$ denote the volume of the ideal regular hyperbolic octahedron.
\begin{thm}\label{thm: bound on volume on L-spaces}
If $M$ is a hyperbolic $L$-space and that is surgery on a two-bridge link, then $\operatorname{vol}(M)<2v_8$. Moreover this bound is optimal.
\end{thm}
\begin{proof}
Since $M$ is an $L$-space, by Theorem \ref{thm: classification of links with L-space surgeries} it must be surgery one of the links in the families \(L_{n,k}\) and \(L'_{n,k}\). All these links are surgery on the four-component link \(L8n7\), shown in Figure~\ref{fig: $L8n7$}, whose volume is $2v_8$; see, for example, \cite{Kumar}. 
\begin{figure}
    \centering
    \includegraphics[width=0.15\textwidth]{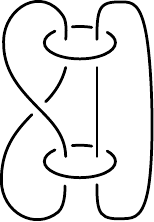}
    \caption{The link $L8n7$. Every $L$-space surgery on a two-bridge link is surgery on it.}
    \label{fig: $L8n7$}
\end{figure}
By \cite{ThuNotes}, volume decreases under hyperbolic Dehn filling and moreover, for parameters $n,k\in \mathbb{Z}$ and $r,s\in \Q$, the volume of the \((\frac{1}{n},\frac{1}{k},r ,s)\)-surgery on $L8n7$ converges to $2v_8$ as all parameters tend to infinity in absolute value. This yields the desired result.
\end{proof}
\subsubsection{Dehn surgery $($non-$)$detection results}
It is now well-known \cite{Ghi08} that if \(K\subset S^3\) is a non-trivial knot and the manifold \(S^3_{1}(K)\), obtained by $1$-surgery on $K$, is an $L$-space, then $K$ is the right-handed trefoil knot. More recently, it was shown that if \(S^3_{2}(K)\) is an $L$-space, then $K$ is either the trefoil or the cinquefoil \cite{FarReiWan24}.

For links with two components, such strong detection results are less common. On the other hand, the property that $(1,1)$-surgery yields an $L$-space characterises the Whitehead link among two-component links with unknotted components and linking number zero \cite{GorLidLiuMoo23}, and among fibered hyperbolic two-bridge links \cite{San23twobridge}. Moreover, by using \cite{BeibeiLiu21,GorLidLiuMoo23} one obtains the following.
\begin{prop}
Let $L$ be an $L$-space link with two unknotted components. If \(S^3_{1,1}(L)\) is an $L$-space, then $L$ is the Whitehead link or the \(T(2,2l)\) torus link. In particular, if $L$ is hyperbolic, then $L$ is the Whitehead link. \end{prop}
\begin{proof}
If $L$ has linking number zero, this follows from \cite{GorLidLiuMoo23}. If $L$ has linking number $l=1$, then \(S^3_{1,1}(L)\) has $b_1>0$, and in particular it is not an $L$-space. If $L$ has linking number $l>1$ and \(S^3_{1,1}(L)\) is an $L$-space, then Lemma \ref{lem: small surgery implies gen L space link} implies that $L$ is a generalised $(\pm,\mp)\,L$-space link. By \cite{BeibeiLiu21}, we can conclude that $L$ is the $T(2,2l)$ torus link.
\end{proof}
In contrast to the previous proposition, Theorem~\ref{thm: gen L-space links} implies that if one does not require $L$ to be an $L$-space link, then there exist infinitely many hyperbolic links whose $(1,1)$-surgery yields an $L$-space.
\begin{thm}
There exist infinitely many non-isotopic hyperbolic links with two unknotted components whose $(1,1)$-surgery is an $L$-space. 
\end{thm}

\begin{proof}
For $n\geq 2$ and $k\geq 1$, the link $L'_{n,k}$ is hyperbolic, and Theorem~\ref{thm: gen L-space links} implies that $(1,1)$-surgery on each of them is an $L$-space. It also follows from the explicit description of their $L$-space surgeries that, for fixed $k$, the links $L'_{n,k}$ are pairwise non-isotopic.
\end{proof}
On the other hand, as consequence of Theorems~\ref{thm: L-space links} and \ref{thm: gen L-space links}, we can deduce the following detection results for the families $L_{n,k}$ and \(L'_{n,k}\).
\begin{prop}
Let $L$ be a two-bridge link, and suppose that the $(r,s)$-surgery on $L$ is an $L$-space, with $rs>\operatorname{lk}(L)^2$. Then $L$ is isotopic to one of the links $L_{n,k}$, for $k,n \geq 0$.
\end{prop}

\begin{prop}
Let $L$ be a two-bridge link, and suppose that the $(r,s)$-surgery on $L$ is an $L$-space, and that at least one of the following conditions holds:
\begin{itemize}
    \item either $|r|$<1 or $|s|<1$, or 
    \item $rs<0$.
\end{itemize} 
Then $L$ is isotopic to one of the links $L'_{n,k}$, for $k,n\geq 0$.
\end{prop}

\subsubsection{Satellite operations}\label{subsubsec: satellite operations}
We briefly recall the satellite operation. Let $K\subset S^3$ be a knot, $P \subset S^1 \times \mathbb{D}^2$ a knot in the solid
torus, and let $m\in \mathbb{Z}$ be an integer. We define a new knot $P(K,m)\subset S^3$ by identifying $S^1\times \mathbb{D}^2$ with a tubular neighbourhood of $K$ in such a way that $S^1 \times \{\theta\}$, where $\theta\in \partial\mathbb{D}^2$, is identified
to the $m$-framed longitude of $K$. The knots $K$ and $P$ are called \emph{companion} and \emph{pattern} knot respectively, and $P(K,m)$ is called a \emph{satellite knot}. 

If we view the solid torus containing $P$ as the tubular neighbourhood of an unknot in $S^3$, and denote by $K_0$ its meridian, then we can naturally associate to the pattern knot the two-component link $L_P=K_0 \cup P\subset S^3$. We call $L_P$ the \emph{pattern link}. 

As a consequence of our main result, together with parts of the proofs presented in this paper, we obtain a classification of all $L$-space satellite knots whose pattern link is a two-bridge link.

\begin{thm}\label{thm: satellite}
Let $P(K,m)$ be the satellite knot of a non-trivial knot $K$, and assume that the pattern link $L_P$ is a hyperbolic two-bridge link. Then $P(K,m)$ is an $L$-space knot if and only if 
\begin{itemize}
\item$K$ is an $L$-space knot, 
\item $L_P$ is isotopic to $L'_{n,k}$ with $n\geq 2, k\geq 1$, and 
\item $m > 2g(K)-1$.
\end{itemize}
Moreover, if $L_P$ is not isotopic to any of the links $L'_{n,k}$, then every non-trivial surgery on $P(K,m)$ supports a coorientable taut foliation.
\end{thm}
Note that if the pattern $L_P$ is a torus two-bridge link, i.e. a $T(2,2l)$ torus link, then the satellite $P(K,m)$ is a cable of $K$, and $L$-space cable knots are completely classified by \cite{HeddenCabling2,Homcable}. For the proof of Theorem \ref{thm: satellite} we refer to Section \ref{sec: satellite}.

\subsection*{Structure of the paper}The structure of the paper is as follows. The main body is divided into three parts: Sections~\ref{section:foliations}--\ref{sec: pers fol two-bridge}, Sections~\ref{section: Heegaard Floer Preliminaries}--\ref{section:rational link surgery}--\ref{section: very good points and non L space surgeries}, and Sections~\ref{section: turaev torsion}--\ref{sec: classification}. We have tried to make these three parts as self-contained as possible, with minimal interdependence, so that the reader may approach each of them independently. 

Sections~\ref{section:foliations}--\ref{sec: pers fol two-bridge} are devoted to the construction of taut foliations. More precisely, in Section~\ref{section:foliations} we prove Theorem~\ref{thm: pers fol link}, which, roughly speaking, states that every non-trivial surgery on an $n$-component link admitting a special diagram supports a coorientable taut foliation. Next, in Section~\ref{sec: pers fol two-bridge}, we apply this theorem, together with some variations of it, to the study of two-bridge links, proving the part of Theorem~\ref{thm: classification of links with L-space surgeries} concerning taut foliations. In the last part of this section, we also construct foliations on some surgeries on the links \(L_{n,k}\) and \(L'_{n,k}\).  

In Sections~\ref{section: Heegaard Floer Preliminaries}--\ref{section:rational link surgery}--\ref{section: very good points and non L space surgeries}, we define the rational link surgery complex and obtain a simplified model for two-component L-space links, which we use to derive obstructions to $L$-space surgeries. In Section~\ref{section: Heegaard Floer Preliminaries}, we review preliminaries on Heegaard Floer homology and the related bordered invariants. In  Section~\ref{section:rational link surgery}, we review the integer link surgery formula by \cite{mo_linksurgery}, and use Zemke's formulation from \cite{zemke_bordered,zemke_general} to prove a link surgery formula that allows rational coefficients; see Theorem \ref{thm: rational_surgery_chain}. In the same section, we show that for two-component  $L$-space links this complex admits a simplified model; see Theorem \ref{thm: intro_formal_link_H_Lambda}. In Section~\ref{section: very good points and non L space surgeries} we use this simplified model to obtain an obstruction to $L$-space surgeries on two-component $L$-space links and apply it to $L_{n,k}.$

In Sections~\ref{section: turaev torsion}--\ref{sec: classification}, we classify the $L$-space surgeries on the links $L_{n,k}$ and \(L'_{n,k}\). We begin by recalling some properties of the Turaev torsion,  and explaining how, by the results of \cite{RR}, it can be used to determine the set of \(L\)-space fillings of rational homology solid tori. We also discuss, in Section~\ref{subsec: box convexity}, the consequences of \cite{RR} for manifolds with multiple boundary components, and introduce the notion of \emph{weak box-convexity}. In Section~\ref{sec: classification}, we focus on the links \(L_{n,k}\) and \(L'_{n,k}\). For \(L_{n,k}\), which we show are \(L\)-space links, we classify their surgeries by combining the arguments developed  in Section~\ref{prop: some ctf Lnk} and Section~\ref{section: very good points and non L space surgeries}. For the links \(L'_{n,k}\), we use a different technique, based essentially on a detailed study of the Turaev torsions of the manifold obtained by filling only one boundary component of their exteriors.

Finally, there are two remaining sections. In Section~\ref{section:two-bridge}, we introduce some basic notions on two-bridge links and prove several technical results used throughout the paper. In particular, we use plumbing calculus to describe different diagrams of two-bridge links, which are used in Section~\ref{sec: pers fol two-bridge}, and we compute the multivariable Alexander polynomials of $L_{n,k}$ and \(L'_{n,k}\), which are needed in Sections~\ref{section: very good points and non L space surgeries} and Section~\ref{sec: classification}. In Section~\ref{sec: satellite}, we prove Theorem~\ref{thm: satellite} on satellite knots.

\subsection*{Acknowledgements}
DS would like to thank Alice Merz for many helpful conversations.  HZ would like to thank his postdoc mentor Linh Truong for helpful conversations and generous support. The authors thank Ian Zemke for pointing out a mistake in an earlier draft of this paper and for clarifying some technical details. The authors thank Qingfeng Lyu for pointing out that his result could be used to prove Theorem~\ref{thm: CTF iff Lspace for Lnk}. DS was
partially supported by FWF project P 34318 and by ERC Advanced Grant KnotSurf4d. HZ is supported by an AMS-Simons travel grant. The authors thank the Max Planck Institute for Mathematics in Bonn for its hospitality, where this project was initiated.

\section{Two-bridge links}\label{section:two-bridge}
In this section, we recall some basic notions on two-bridge links and prove some technical results that will be needed later in the paper. More precisely, in Section \ref{subsec: plumbing calculus} we use plumbing calculus to find suitable diagrams of two-bridge links, which will be crucial to construct in Section \ref{sec: pers fol two-bridge} to construct taut foliations. Next, in Section \ref{subsec: alexander poly}, we compute the multivariable Alexander polynomials of the links $L_{n,k}$ and $L'_{n,k}$, which we will use in Sections \ref{subsec: H_function_L_n_k} and $\ref{subsec: class for L'nk}$ respectively.  We refer to \cite{BZ} for more details on two-bridge links. 
\newline
%\begin{figure}[]
%    \centering   
%    \includegraphics[width=0.9\textwidth]{Figures/two_bridge_crossings.pdf}
%    \caption{The two-bridge link $L(a_1,\dots, a_n)$. We adopt the convention that a negative left-handed crossing is a right-handed crossing, and viceversa.}
%    \label{figure:two bridge}
%\end{figure}

A \emph{two-bridge link} can be described by a rational number $\frac{p}{q}$, where $p$ and $q$ are coprime integers with $p>0$, $q$ is odd and $0<|q|<p$, in the following way. We fix a sequence of integers $(a_1,\dots,a_n)$ such that
\begin{equation}\label{eq:continued fraction}
\frac{p}{q}=a_1-\cfrac{1}{a_2-\cfrac{1}{\ddots-\cfrac{1}{a_n}}}         
\end{equation}
and consider the link defined by the diagram (disregarding the orientations for now) in Figure \ref{fig: two bridge ori}. We denote this link by $L(a_1,\dots, a_n)$\footnote{We warn the reader the difference in notation with \cite{San23twobridge}: the link $L(a_1,a_2,\dots,a_n)$ here corresponds to $L(a_1,-a_2,\dots,(-1)^{n-1}a_n)$ there.}.

We are interested in the case when $L(a_1,\dots, a_n)$ has two components, and this happens exactly when the fraction \(\frac{p}{q}\) has numerator $p$ even, see \cite[Proposition~12.3]{BZ}. 
When $L(a_1,\dots, a_n)$ is a link we orient the components as in Figure \ref{fig: two bridge ori}.

\begin{figure}
    \centering
    \includegraphics[width=0.8\textwidth]{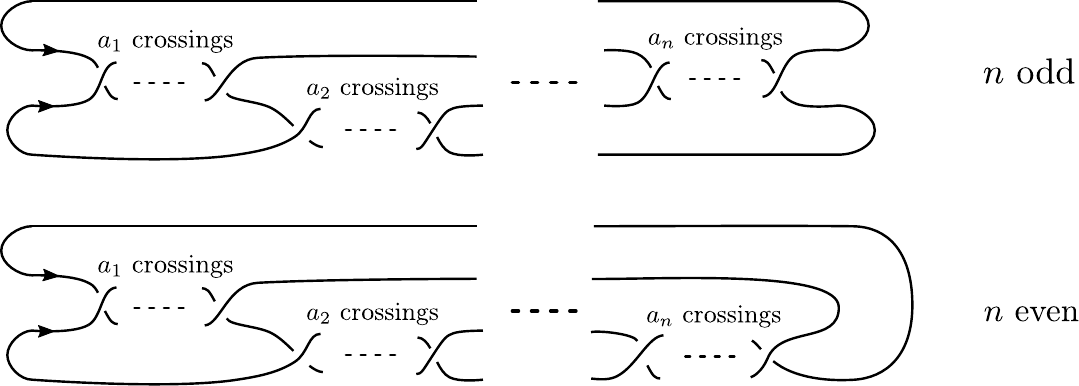}
    \caption{The two-bridge link $L(a_1,\dots, a_n)$. When it has two components, we orient it as illustrated. We adopt the convention that a negative left-handed crossing is a right-handed crossing, and viceversa.}
    \label{fig: two bridge ori}
\end{figure}

A priori it could happen that the isotopy class of the two-bridge link associated to \(\frac{p}{q}\) depends on the choice of the continued fraction representation of $\frac{p}{q}$. This is not the case, by the following theorem.

\begin{thm}[\cite{Schubert}, see also \cite{BZ}]\label{teo: classification two-bridge links}
Let $L=L(a_1\dots,a_n)$ and $L'=L(b_1,\dots,b_m)$ be two oriented two-bridge links and let $\frac{p}{q}$ and $\frac{p'}{q'}$ be the rational numbers defined as in equation~\eqref{eq:continued fraction}. Then the links $L$ and $L'$ are isotopic if and only if $p=p'$ and $q'\equiv q^{\pm 1 } \mod 2p$. If $p=p'$ and $q'\equiv q+p \mod 2p$ or $qq'\equiv 1+p \mod 2p$,
then $L$ and $L'$ are isotopic after reversing the orientation of one of the components.
\end{thm}
We denote by $b(p,q)$ the two-bridge link associated to the rational number $\frac{p}{q}$. We also recall the following fact, that will be used repeatedly in the paper; see \cite[Exercise~2.1.16]{Kaw}. 

\begin{lem}\label{lemma: symmetric}
Two-bridge links are symmetric, \ie there exists an ambient isotopy of $S^3$ interchanging their components.\qed
\end{lem}

\subsection{Plumbing calculus}\label{subsec: plumbing calculus}
In Section \ref{sec: pers fol two-bridge}, we exploit different diagrammatic presentations of two-bridge links to construct taut foliations on their surgeries. For this purpose, we briefly introduce some notions from plumbing calculus, which provides a convenient way to obtain various diagrams of a two-bridge link. We restrict ourselves to recalling only what is strictly necessary for our purposes and refer to \cite{Neu81} for the general theory of plumbing calculus and to \cite{BonSieb} for its connection with links in $S^3$.
\newline

We can encode two-bridge links with linear plumbing graphs, \ie linear graphs whose vertices have an integer weight assigned to them, by associating to a graph
\begin{center}
\begin{tikzpicture}[scale=0.7]

  % Nodo 1
  \fill (0,0) circle (2pt);
  \node at (0,0.4) {  $a_1$};

  % Nodo 2
  \fill (1.5,0) circle (2pt);
  \node at (1.5,0.4) {  $a_2$};

  % Puntini intermedi
  \node at (3,0) {$\cdots$};

  % Nodo n-1
  \fill (4.5,0) circle (2pt);
  \node at (4.5,0.4) {  $a_{n-1}$};

  % Nodo n
  \fill (6,0) circle (2pt);
  \node at (6,0.4) {  $a_n$};

  % Archi
  \draw (0,0) -- (1.5,0);
  \draw (1.5,0) -- (2.4,0); % fino al puntinato
  \draw (3.6,0) -- (4.5,0); % dopo il puntinato
  \draw (4.5,0) -- (6,0);

\end{tikzpicture}
\end{center}

the (unoriented) two-bridge link $L(-a_1,-a_2,\dots, -a_n)$, with diagram described in Figure~\ref{fig: two bridge ori}.

We introduce the following moves on linear plumbing graphs.

\begin{itemize}
\item \emph{blow-down}: a vertex with weight $\epsilon\in \{\pm 1\}$ can be removed in the following ways, according to its valence: 

\begin{center}
\begin{tikzpicture}[scale=0.7]

  % initial dots
  \node at (-5,0) {$\cdots$};

  % Node 1 
  \fill (-4,0) circle (2pt);
  \node at (-4,0.4) {  $a$};

  % Nodo 2
  \fill (-2.5,0) circle (2pt);
  \node at (-2.5,0.4) {  $\epsilon$};

  % Nodo 3
  \fill (-1,0) circle (2pt);
  \node at (-1,0.4) {  $b$};

  % final node
  \node at (0,0) {$\cdots$};

    % lines
 \draw (-4.5,0) -- (-4,0);
  \draw (-4,0) -- (-2.5,0); 
  \draw (-2.5,0) -- (-1,0);
  \draw (-1,0) -- (-0.5,0);

 \node at (1,0) {$\rightarrow$};
 
   % initial dots
  \node at (2,0) {$\cdots$};

  % Node 1 
  \fill (3,0) circle (2pt);
  \node at (3,0.4) { $a-\epsilon$};

  % Nodo 3
  \fill (4.5,0) circle (2pt);
  \node at (4.5,0.4) { $b-\epsilon$};

  % final node
  \node at (5.5,0) {$\cdots$};

    % lines
 \draw (2.5,0) -- (3,0);
  \draw (3,0) -- (4.5,0); 
  \draw (4.5,0) -- (5,0);

  % initial dots
  \node at (-3.5,-1) {$\cdots$};

  % Node 1 
  \fill (-2.5,-1) circle (2pt);
  \node at (-2.5,-0.6) {$a$};

  % Nodo 2
  \fill (-1,-1) circle (2pt);
  \node at (-1,-0.6) {  $\epsilon$};

    % lines
 \draw (-3,-1) -- (-2.5,-1);
  \draw (-2.5,-1) -- (-1,-1);

 \node at (1,-1) {$\rightarrow$};
 
   % initial dots
  \node at (2,-1) {$\cdots$};

  % Node 1 
  \fill (3,-1) circle (2pt);
  \node at (3,-0.6) {  $a-\epsilon$};

    % lines
 \draw (2.5,-1) -- (3,-1);

\end{tikzpicture}
\end{center}
We call \emph{blow-up} the inverse operation.
\item \emph{$0$-chain absorption}: a vertex with weight $0$ can be removed as shown: 
\begin{center}
\begin{tikzpicture}[scale=0.7]

  % initial dots
  \node at (-5,0) {$\cdots$};

  % Node 1 
  \fill (-4,0) circle (2pt);
  \node at (-4,0.4) {  $a$};

  % Nodo 2
  \fill (-2.5,0) circle (2pt);
  \node at (-2.5,0.4) {  $0$};

  % Nodo 3
  \fill (-1,0) circle (2pt);
  \node at (-1,0.4) {  $b$};

  % final node
  \node at (0,0) {$\cdots$};

    % lines
 \draw (-4.5,0) -- (-4,0);
  \draw (-4,0) -- (-2.5,0); 
  \draw (-2.5,0) -- (-1,0);
  \draw (-1,0) -- (-0.5,0);

 \node at (1,0) {$\rightarrow$};
 
   % initial dots
  \node at (2,0) {$\cdots$};

  % Node 1 
  \fill (3,0) circle (2pt);
  \node at (3,0.4) { $a+b$};

  % final node
  \node at (4,0) {$\cdots$};

    % lines
 \draw (2.5,0) -- (3,0);
  \draw (3,0) -- (3.5,0);

  % initial dots
  \node at (-3.5,-1) {$\cdots$};

  % Node 1 
  \fill (-2.5,-1) circle (2pt);
  \node at (-2.5,-0.6) {$a$};

  % Nodo 2
  \fill (-1,-1) circle (2pt);
  \node at (-1,-0.6) {  $0$};

    % lines
 \draw (-3,-1) -- (-2.5,-1);
  \draw (-2.5,-1) -- (-1,-1);

 \node at (1,-1) {$\rightarrow$};
 
   % initial dots
  \node at (2,-1) {$\cdots$};

  % Node 1 
  \fill (3,-1) circle (2pt);
  \node at (3,-0.6) {  $a$};

    % lines
 \draw (2.5,-1) -- (3,-1);

\end{tikzpicture}
\end{center}
\end{itemize}
\begin{defn}
Two linear plumbing graphs $\mathcal{T}$ and $\mathcal{T}'$ are \emph{equivalent} if they are related by a sequence of blow-downs and $0$-chain absorptions, and their inverses.
\end{defn}
For us, the moves just introduced are relevant for the following fact. We refer to \cite[Chapter~12]{BonSieb} for details.
\begin{prop}\label{prop: plumbing graphs moves}
Let $\mathcal{T}$ and $\mathcal{T}'$ be two equivalent linear plumbing graphs. Then the corresponding two-bridge links are isotopic. 
\end{prop}

\begin{remark} \label{rem: linear_plumbing}
We briefly motivate the appearance of the minus signs in the correspondence between plumbing graphs and two-bridge links. We adopt the convention that the oriented lens space $L(p,q)$ is defined as $-\frac{p}{q}$-surgery on the unknot and, in our notations, it is the double branched cover of the two-bridge link $b(p,q)$. A linear\footnote{A similar discussion holds more generally for \emph{embedded} plumbing graphs.} plumbing graph $\mathcal{T}$ defines a $3$-manifold in the following way. First, one considers the link obtained by associating to each vertex of $\mathcal{T}$ an unknotted circle in $S^3$ and by linking two circles as a Hopf link whenever the corresponding vertices are connected by an edge. Next, the weights of the vertices define surgery coefficients on the components of such link, and hence a $3$-manifold. One can show that this $3$-manifold is the double branched cover of the two-bridge link associated to $\mathcal{T}$, and that the $3$-manifolds associated to the linear graph with weights $(a_1,\dots, a_n)$ is the $\frac{p}{q}$-surgery on the unknot, where
$$
\frac{p}{q}=a_1-\cfrac{1}{a_2-\cfrac{1}{\ddots-\cfrac{1}{a_n}}}.    
$$
This, as an oriented manifold, is $L(p,-q)$ and hence the two-bridge link associated to such a graph is $b(p,-q)=L(-a_1,-a_2,\dots, -a_n)$.
\end{remark}

\begin{comment}\begin{defn}
A plumbing graph is a finite tree $\mathcal{T}$ where every vertex has an integer weight assigned to it.
\end{defn}

A plumbing graph can be used to define a $4$-manifold with boundary $W(\mathcal{T}$ by gluing two-handles to $D^4$ along a framed link obtained by associating to each vertex of $\mathcal{T}$ an unknotted circle in $S^3$, framed according to the weight of the vertex, and by linking as the Hopf link two circles when the corresponding vertices are connected an edge.

If we fix an embedding of $\mathcal{T}$ in the plane (and some additional choices, see \ds{referece}) we can also describe a (unoriented) link $L(\mathcal{T})$ in $S^3$ in the following way: each vertex of $\mathcal{T}$ represents
a twisted band, the weight being the number of half-twists (positive for right-handed twists, negative
for left-handed); when two vertices are connected by an edge, the corresponding bands are plumbed
one to the other. In this way one gets a (possibly non-orientable) surface $S(\mathcal{T})$, and $L(\mathcal{T})$ is the
boundary of this surface.

We are interested in case when $\mathcal{T}$ is a linear graph.
\end{comment}

We collect, for the sake of completeness, some elementary lemmas on linear plumbing graphs that we will need later.

\begin{lem}\label{lemma: operation on plumb graphs 1}
For each $k\geq 0$, the linear plumbing graphs
\begin{center}
\begin{tikzpicture}[scale=0.7]

  % Node 1 
  \fill (-6,0) circle (2pt);
  \node at (-6,0.4) {  $-2$};
  
  % initial dots
  \node at (-5,0) {$\cdots$};

  % Node 1 
  \fill (-4,0) circle (2pt);
  \node at (-4,0.4) {  $-2$};

  % Nodo 2
  \fill (-2.5,0) circle (2pt);
  \node at (-2.5,0.4) {  $-1$};
% square bracket
\draw (-6,-0.5) -- (-2.5,-0.5);      % segmento orizzontale
\draw (-6,-0.5) -- (-6,-0.2);      % tratto verticale sinistro
\draw (-2.5,-0.5) -- (-2.5,-0.2);      % tratto verticale destro
\node at (-4.25,-0.8) {$k$ vertices};
  % Nodo 3
  \fill (-1,0) circle (2pt);
  \node at (-1,0.4) {$a$};

  % final node
  \node at (0,0) {$\cdots$};

    % lines
    \draw (-6,0) -- (-5.5,0);
 \draw (-4.5,0) -- (-4,0);
  \draw (-4,0) -- (-2.5,0); 
  \draw (-2.5,0) -- (-1,0);
  \draw (-1,0) -- (-0.5,0);

 \node at (1.75,0) {and};
 
   % node 1
  \fill (4,0) circle (2pt);
  \node at (4,0.4) {$a+k$};

  % final dots 
  \node at (5,0) {$\cdots$};

    % lines
 \draw (4,0) -- (4.5,0);
  \end{tikzpicture}
  \end{center}
  are equivalent.
\end{lem}

\begin{proof}
    The proof is straightforward by induction on $k$. When $k=0$ there is nothing to prove, and when $k=1$ the trees are, by definition, related by a single blow-down. In general, after applying one blow-down we obtain
    \begin{center}
\begin{tikzpicture}[scale=0.7]

  % Node 1 
  \fill (-6,0) circle (2pt);
  \node at (-6,0.4) {  $-2$};
  
  % initial dots
  \node at (-5,0) {$\cdots$};

  % Node 1 
  \fill (-4,0) circle (2pt);
  \node at (-4,0.4) {  $-2$};

  % Nodo 2
  \fill (-2.5,0) circle (2pt);
  \node at (-2.5,0.4) {  $-1$};
% square bracket
\draw (-6,-0.5) -- (-2.5,-0.5);      % segmento orizzontale
\draw (-6,-0.5) -- (-6,-0.2);      % tratto verticale sinistro
\draw (-2.5,-0.5) -- (-2.5,-0.2);      % tratto verticale destro
\node at (-4.25,-0.8) {$k$ vertices};
  % Nodo 3
  \fill (-1,0) circle (2pt);
  \node at (-1,0.4) {$a$};

  % final node
  \node at (0,0) {$\cdots$};

    % lines
    \draw (-6,0) -- (-5.5,0);
 \draw (-4.5,0) -- (-4,0);
  \draw (-4,0) -- (-2.5,0); 
  \draw (-2.5,0) -- (-1,0);
  \draw (-1,0) -- (-0.5,0);

 \node at (1,0) {$\rightarrow$};

\begin{scope}[shift={(1,0)}]    
  % Node 1 
  \fill (1,0) circle (2pt);
  \node at (1,0.4) {  $-2$};
  
  % initial dots
  \node at (2,0) {$\cdots$};

  % Node 1 
  \fill (3,0) circle (2pt);
  \node at (3,0.4) {  $-2$};

  % Nodo 2
  \fill (4.5,0) circle (2pt);
  \node at (4.5,0.4) {  $-1$};
% square bracket
\draw (1,-0.5) -- (4.5,-0.5);      % segmento orizzontale
\draw (1,-0.5) -- (1,-0.2);      % tratto verticale sinistro
\draw (4.5,-0.5) -- (4.5,-0.2);      % tratto verticale destro
\node at (2.75,-0.8) {$k-1$ vertices};
  % Nodo 3
  \fill (6,0) circle (2pt);
  \node at (6,0.4) {$a+1$};

  % final node
  \node at (7,0) {$\cdots$};

    % lines
    \draw (1,0) -- (1.5,0);
 \draw (2.5,0) -- (3,0);
  \draw (3,0) -- (4.5,0); 
  \draw (4.5,0) -- (6,0);
  \draw (6,0) -- (6.5,0);
\end{scope}
  \end{tikzpicture}
  \end{center}
  and the claim follows by induction.
\end{proof}

\begin{lem}\label{lemma: operation on plumb graphs 1.5}
For each $k\geq 0$, the linear plumbing graphs
\begin{center}
\begin{tikzpicture}[scale=0.7]

  % Node 1 
  \fill (-6,0) circle (2pt);
  \node at (-6,0.4) {  $-2$};
  
  % initial dots
  \node at (-5,0) {$\cdots$};

  % Node 1 
  \fill (-4,0) circle (2pt);
  \node at (-4,0.4) {  $-2$};

  % Nodo 2
  \fill (-2.5,0) circle (2pt);
  \node at (-2.5,0.4) {  $a$};
% square bracket
\draw (-6,-0.5) -- (-4,-0.5);      % segmento orizzontale
\draw (-6,-0.5) -- (-6,-0.2);      % tratto verticale sinistro
\draw (-4,-0.5) -- (-4,-0.2);      % tratto verticale destro
\node at (-5,-0.8) {$k$ vertices};

  % final node
  \node at (-1.5,0) {$\cdots$};

    % lines
    \draw (-6,0) -- (-5.5,0);
 \draw (-4.5,0) -- (-4,0);
  \draw (-4,0) -- (-2,0); 

 \node at (0.75,0) {and};
 
   % node 1
  \fill (3,0) circle (2pt);
  \node at (3,0.4) {$k+1$};
  % Nodo 2
  \fill (4.5,0) circle (2pt);
  \node at (4.5,0.4) {  $a+1$};
  
  % final dots 
  \node at (5.5,0) {$\cdots$};

    % lines
 \draw (3,0) -- (5,0);
  \end{tikzpicture}
  \end{center}
  are equivalent.
\end{lem}
\begin{proof}
We first blow-up the edge to the left of the vertex with weight $a$ to obtain   
\begin{center}
\begin{tikzpicture}[scale=0.7]

  % Node 1 
  \fill (-6,0) circle (2pt);
  \node at (-6,0.4) {  $-2$};
  
  % initial dots
  \node at (-5,0) {$\cdots$};

  % Node 1 
  \fill (-4,0) circle (2pt);
  \node at (-4,0.4) {  $-1$};

  % Nodo 2
  \fill (-2.5,0) circle (2pt);
  \node at (-2.5,0.4) {  $1$};
% square bracket

\draw (-6,-0.5) -- (-4,-0.5);      % segmento orizzontale

\draw (-6,-0.5) -- (-6,-0.2);      % tratto verticale sinistro

\draw (-4,-0.5) -- (-4,-0.2);      % tratto verticale destro

\node at (-5,-0.8) {$k$ vertices};
  % Nodo 3
  \fill (-1,0) circle (2pt);
  \node at (-1,0.4) {$a+1$};

  % final node
  \node at (0,0) {$\cdots$};

    % lines
    \draw (-6,0) -- (-5.5,0);
 \draw (-4.5,0) -- (-4,0);
  \draw (-4,0) -- (-2.5,0); 
  \draw (-2.5,0) -- (-1,0);
  \draw (-1,0) -- (-0.5,0);
  \end{tikzpicture}
  \end{center}
and conclude by Lemma \ref{lemma: operation on plumb graphs 1}.  
\end{proof}

\begin{lem}\label{lemma: operation on plumb graphs 2}
For each $n\geq 1$ and $h,k\geq 1$, the linear plumbing graphs
\begin{center}
\begin{tikzpicture}[scale=0.7]

\node at (-6.75,0) {  $\mathcal{T}=$};
  % Node 1 
  \fill (-6,0) circle (2pt);
  \node at (-6,0.4) {  $-2$};
  
  % initial dots
  \node at (-5,0) {$\cdots$};

  % Node 2
  \fill (-4,0) circle (2pt);
  \node at (-4,0.4) {  $-2$};

  % Nodo 3
  \fill (-2.5,0) circle (2pt);
  \node at (-2.5,0.4) {  $2n$};

% square bracket
\draw (-6,-0.5) -- (-4,-0.5);      % segmento orizzontale
\draw (-6,-0.5) -- (-6,-0.2);      % tratto verticale sinistro
\draw (-4,-0.5) -- (-4,-0.2);      % tratto verticale destro
\node at (-5,-0.8) {$k$ vertices};

 % Node 4 
  \fill (-1,0) circle (2pt);
  \node at (-1,0.4) {  $-2$};
  
  % initial dots
  \node at (0,0) {$\cdots$};

  % Node 5 
  \fill (1,0) circle (2pt);
  \node at (1,0.4) {  $-2$};

% square bracket
\draw (-1,-0.5) -- (1,-0.5);      % segmento orizzontale
\draw (-1,-0.5) -- (-1,-0.2);      % tratto verticale sinistro
\draw (1,-0.5) -- (1,-0.2);      % tratto verticale destro
\node at (0,-0.8) {$h$ vertices};

    % lines
    \draw (-6,0) -- (-5.5,0);
 \draw (-4.5,0) -- (-4,0);
  \draw (-4,0) -- (-2.5,0); 
  \draw (-2.5,0) -- (-1,0);
  \draw (-1,0) -- (-0.5,0);
  \draw (0.5,0) -- (1,0);

 \node at (2,0) {and};
 
 \begin{scope}[shift={(0.5,0)}]
 \node at (2.75,0) {$\mathcal{T}'=$};
   % node 1
  \fill (3.5,0) circle (2pt);
  \node at (3.5,0.4) {$k$};
  % node 2
  \fill (5,0) circle (2pt);
  \node at (5,0.4) {$-2$};
  
  % dots 
  \node at (6,0) {$\cdots$};
  
   % node 3
  \fill (7,0) circle (2pt);
  \node at (7,0.4) {$-2$};
  % square bracket
\draw (5,-0.5) -- (7,-0.5);      % segmento orizzontale
\draw (5,-0.5) -- (5,-0.2);      % tratto verticale sinistro
\draw (7,-0.5) -- (7,-0.2);      % tratto verticale destro
\node at (6,-0.8) {$2n+1$ vertices};
   % node 4
  \fill (8.5,0) circle (2pt);
  \node at (8.5,0.4) {$h$};

    % lines
 \draw (3.5,0) -- (5,0);
 \draw (5,0) -- (5.5,0);
  \draw (6.5,0) -- (7,0);
  \draw (7,0) -- (8.5,0);
  \end{scope}
  \end{tikzpicture}
  \end{center}
  are equivalent.
\end{lem}

\begin{proof}
By Lemma \ref{lemma: operation on plumb graphs 1.5} the graph  $\mathcal{T}$ is equivalent to
\begin{center}
\begin{tikzpicture}
% node 1
  \fill (3,0) circle (2pt);
  \node at (3,0.4) {$k+1$};
  % Nodo 2
  \fill (4.5,0) circle (2pt);
  \node at (4.5,0.4) {  $2n+2$};
% node 3
\fill (6,0) circle (2pt);
  \node at (6,0.4) {  $h+1$};
    % lines
 \draw (3,0) -- (6,0);
  \end{tikzpicture}
  \end{center}
and by blowing-up the edges adjacent to the central vertex we have
\begin{center}
\begin{tikzpicture}
% node 1
  \fill (3,0) circle (2pt);
  \node at (3,0.4) {$k$};
  % node 2
  \fill (4.5,0) circle (2pt);
  \node at (4.5,0.4) {$-1$};
  
  % Nodo 3
  \fill (6,0) circle (2pt);
  \node at (6,0.4) {  $2n$};
% node 4
  \fill (7.5,0) circle (2pt);
  \node at (7.5,0.4) {$-1$};
% node 5
\fill (9,0) circle (2pt);
  \node at (9,0.4) {$h$};
    % lines
 \draw (3,0) -- (9,0);
  \end{tikzpicture}.
  \end{center}
We repeat this last operation $n$ times and perform one $0$-chain absorption to conclude.
\end{proof}

\begin{lem}\label{lemma: operation on plumb graphs 3}
For each $n\geq 2$ and $h,k\geq 1$, the linear plumbing graphs
\begin{center}
\begin{tikzpicture}[scale=0.7]

\node at (-6.75,0) {  $\mathcal{T}=$};
  % Node 1 
  \fill (-6,0) circle (2pt);
  \node at (-6,0.4) {  $-2$};
  
  % initial dots
  \node at (-5,0) {$\cdots$};

  % Node 2
  \fill (-4,0) circle (2pt);
  \node at (-4,0.4) {  $-2$};

  % Nodo 3
  \fill (-2.5,0) circle (2pt);
  \node at (-2.5,0.4) {  $-2n$};

% square bracket
\draw (-6,-0.5) -- (-4,-0.5);      % segmento orizzontale
\draw (-6,-0.5) -- (-6,-0.2);      % tratto verticale sinistro
\draw (-4,-0.5) -- (-4,-0.2);      % tratto verticale destro
\node at (-5,-0.8) {$k$ vertices};

 % Node 4 
  \fill (-1,0) circle (2pt);
  \node at (-1,0.4) {  $-2$};
  
  % initial dots
  \node at (0,0) {$\cdots$};

  % Node 5 
  \fill (1,0) circle (2pt);
  \node at (1,0.4) {  $-2$};

% square bracket
\draw (-1,-0.5) -- (1,-0.5);      % segmento orizzontale
\draw (-1,-0.5) -- (-1,-0.2);      % tratto verticale sinistro
\draw (1,-0.5) -- (1,-0.2);      % tratto verticale destro
\node at (0,-0.8) {$h$ vertices};

    % lines
    \draw (-6,0) -- (-5.5,0);
 \draw (-4.5,0) -- (-4,0);
  \draw (-4,0) -- (-2.5,0); 
  \draw (-2.5,0) -- (-1,0);
  \draw (-1,0) -- (-0.5,0);
  \draw (0.5,0) -- (1,0);

 \node at (2,0) {and};
 \begin{scope}[shift={(0.5,0)}]
     
 \node at (2.75,0) {$\mathcal{T}'=$};
 \begin{scope}[shift={(0.5,0)}]   % node 1
  \fill (3.5,0) circle (2pt);
  \node at (3.5,0.4) {$k+2$};
  % node 2
  \fill (5,0) circle (2pt);
  \node at (5,0.4) {$2$};
  
  % dots 
  \node at (6,0) {$\cdots$};
  
   % node 3
  \fill (7,0) circle (2pt);
  \node at (7,0.4) {$2$};
  % square bracket
\draw (5,-0.5) -- (7,-0.5);      % segmento orizzontale
\draw (5,-0.5) -- (5,-0.2);      % tratto verticale sinistro
\draw (7,-0.5) -- (7,-0.2);      % tratto verticale destro
\node at (6,-0.8) {$2n-3$ vertices};
   % node 4
  \fill (8.5,0) circle (2pt);
  \node at (8.5,0.4) {$h+2$};

    % lines
 \draw (3.5,0) -- (5,0);
 \draw (5,0) -- (5.5,0);
  \draw (6.5,0) -- (7,0);
  \draw (7,0) -- (8.5,0);
  \end{scope}
  \end{scope}
  \end{tikzpicture}
  \end{center}
  are equivalent.
\end{lem}
\begin{proof}
The graph $\mathcal{T}$ is equivalent by Lemma \ref{lemma: operation on plumb graphs 1.5} to
\begin{center}
\begin{tikzpicture}[scale=1]

  % Node 1 
  \fill (-4,0) circle (2pt);
  \node at (-4,0.4) {  $k+1$};

  % Nodo 2
  \fill (-2.5,0) circle (2pt);
  \node at (-2.5,0.4) {  $-2n+2$};

  % Nodo 3
  \fill (-1,0) circle (2pt);
  \node at (-1,0.4) {$h+1$};
    % lines 
  \draw (-4,0) -- (-2.5,0); 
  \draw (-2.5,0) -- (-1,0);
\end{tikzpicture}
\end{center}
which in turn is equivalent to $\mathcal{T}'$
 by arguing as in the proof of Lemma \ref{lemma: operation on plumb graphs 2}
\end{proof}

We conclude this section by finding suitable descriptions of two-bridge links. The corresponding diagrams, obtained as in Figure \ref{fig: two bridge ori}, will be used in the next sections to construct taut foliations on their surgeries.

We first address the case of non-fibered two-bridge links with the next two lemmas.
\begin{lem}\label{lemma: nonfib two-bridge links std diagrams} 
Let $L$ be a non-fibered two-bridge link, and suppose that $L$ is not isotopic, up to mirror, to $L(2,2,\dots, 2, 2z)$ for any non-zero $z\in \mathbb{Z}$. Then $L=L(a_1,\dots, a_m)$, where $m\geq 3$ is odd, $|a_i|\geq 2$ for all $i$, and $a_j$ is even whenever $j$ is odd. 

Moreover, either $L=L(2a,2b,2c)$, or $m\geq 5$ and one of the following non-exclusive conditions holds:
\begin{itemize}
    \item[1a$)$] There exist an even index \(i\) and odd indices \(i_1,i_2\) such that \(|a_i|>2\) and \(\operatorname{sign}(a_{i_1})\neq \operatorname{sign}(a_{i_2})\).
    \item[1b$)$] There exist an even index \(i\) and an odd index \(i'\) such that \(|a_i|>2\) and \(|a_{i'}|>2\).
    \item[2a$)$] There exist an odd index \(i\) and indices \(i_1\neq i\), \(i_2\neq i\) such that \(|a_i|>2\) and \(\operatorname{sign}(a_{i_1})\neq \operatorname{sign}(a_{i_2})\).
    \item[2b$)$] There exist distinct odd indices \(i\neq i'\) such that \(|a_i|>2\) and \(|a_{i'}|>2\).
\end{itemize}
\end{lem}
\begin{proof}
It follows from \cite[Exercise~2.1.14]{Kaw} that any two-bridge link $L$ can be written as $L=L(\alpha_1,\dots, \alpha_m)$, where $\alpha_j$ is a non-zero even integer for all $j$. In particular, if $L$ has two components, then $m$ must be odd. Since $L$ is non-fibered by hypothesis, we deduce that $m\geq 3$ and that there exists at least one index $i$ such that $|\alpha_i|>2$. Indeed, if $m=1$, then $L$ is a torus link and hence fibered; moreover, if $|\alpha_j|=2$ for all $j$, then $L$ is fibered by \cite[Proposition~2]{GabKaz90}\footnote{The proof presented there is for knots, but the same proof works also for links.}. 
If $m=3$, then $L$ satisfies the statement of the lemma. Otherwise, we consider two cases:
\newline

\begin{tabular}{|c|}
\hline
\emph{The index $i$ is even:}
\\
\hline
\end{tabular} If the sequence $(\alpha_1,\dots, \alpha_m)$ satisfies condition \emph{1a$)$} or \emph{1b$)$}, there is nothing to prove. Otherwise, for every odd \(j\), all the \(\alpha_j\) have the same sign and satisfy \(|\alpha_j|=2\). After replacing \(L\) by its mirror image if necessary, we may assume that \(\alpha_j=2\) for every odd \(j\). 

Now let \(j\) be the smallest index such that \(\alpha_j\neq 2\). By hypothesis, such an index exists, and it must be even; write \(j=2k\) for some \(k\ge 1\). In terms of plumbing graphs, the link \[
L(2,\dots,2,\alpha_{2k},2,\alpha_{2k+2},\dots)
\]
corresponds to the following graph:
\begin{center}
\begin{tikzpicture}[scale=0.7]

  % Nodo 1
  \fill (-3,0) circle (2pt);
  \node at (-3,0.4) {$-2$};

  \node at (-1.5, 0) {$\cdots$};
  
 % Node 2
  \fill (-0.5,0) circle (2pt);
  \node at (-0.5,0.4) {  $-2$};
  %arrow
 % \draw[->] (0.25,-0.65) -- (0.25,-0.15);
  
  % Node 3 
  \fill (1,0) circle (2pt);
  \node at (1,0.4) {  $-\alpha_{2k}$};

  \node at (2,0) {  $\cdots$ };
  \node at (2.5,0) {,};
% square bracket
\draw (-3,-0.5) -- (-0.5,-0.5);      % segmento orizzontale
\draw (-3,-0.5) -- (-3,-0.2);      % tratto verticale sinistro
\draw (-0.5,-0.5) -- (-0.5,-0.2);      % tratto verticale destro
\node at (-1.8,-0.8) {\footnotesize{$2k-1$ vertices}};

    % lines
    \draw (-3,0) -- (-2,0);
    \draw (-1,0) -- (1.5,0);

 \end{tikzpicture}
 \end{center}
 % By performing a blow-up along the edge indicated by the arrow and applying Lemma~\ref{lemma: operation on plumb graphs 1}, we obtain the graph
 which by Lemma~\ref{lemma: operation on plumb graphs 1.5} is equivalent to
\begin{center}
\begin{tikzpicture}[scale=0.7]

  % Nodo 1
  \fill (-3,0) circle (2pt);
  \node at (-3,0.4) {  $2k$};

 % Node 2
  \fill (-1,0) circle (2pt);
  \node at (-1.3,0.4) {  $-\alpha_{2k}+1$};
  
  \node at (0,0) {  $\cdots$};
  \node at (0.5,0) {.};
  
 % lines
    \draw (-3,0) -- (-0.5,0);    
 \end{tikzpicture}
 \end{center}
This implies by Proposition \ref{prop: plumbing graphs moves} that 
\[
L=L(-2k,\alpha_{2k}-1, 2, \alpha_{2k+2},\dots).
\]
This description of \(L\) satisfies condition \emph{1a$)$}, except when \(2k=m-1\). In that case,  by applying Lemma~\ref{lemma: operation on plumb graphs 1.5} on the right,
\[
L=L(-2k,\alpha_{2k}-1,2)=L(-2k,\alpha_{2k}-2,-2),
\]
and \(L\) is of the form \(L(2a,2b,2c)\), as required.
\newline

\begin{tabular}{|c|}
\hline
\emph{The index $i$ is odd:}
\\
\hline
\end{tabular}
Suppose that the sequence \((\alpha_1,\dots,\alpha_m)\) does not satisfy any of the four desired conditions. Up to taking the mirror image, we may assume that \(\alpha_j=2\) for all \(j\neq i\). Write \(i=2k+1\), where \(k\ge 0\), and let \(h\) be such that \(2k+2h+1=m\). We can assume that \(k,h\ge 1\), since otherwise \(L=L(2,\dots,2,2z)\), contrary to our assumptions on \(L\).

If \(\alpha_i=-2n<0\), then the corresponding plumbing graph is

\begin{center}
\begin{tikzpicture}[scale=0.7]

  % Node 1 
  \fill (-6,0) circle (2pt);
  \node at (-6,0.4) {  $-2$};
  
  % initial dots
  \node at (-5,0) {$\cdots$};

  % Node 2
  \fill (-4,0) circle (2pt);
  \node at (-4,0.4) {  $-2$};

  % Nodo 3
  \fill (-2.5,0) circle (2pt);
  \node at (-2.5,0.4) {  $2n$};

% square bracket
\draw (-6,-0.5) -- (-4,-0.5);      % segmento orizzontale
\draw (-6,-0.5) -- (-6,-0.2);      % tratto verticale sinistro
\draw (-4,-0.5) -- (-4,-0.2);      % tratto verticale destro
\node at (-5,-0.8) {\footnotesize{$2k$ vertices}};

 % Node 4 
  \fill (-1,0) circle (2pt);
  \node at (-1,0.4) {  $-2$};
  
  % Node 5 
  \fill (1,0) circle (2pt);
  \node at (1,0.4) {  $-2$};

% square bracket
\draw (-1,-0.5) -- (1,-0.5);      % segmento orizzontale
\draw (-1,-0.5) -- (-1,-0.2);      % tratto verticale sinistro
\draw (1,-0.5) -- (1,-0.2);      % tratto verticale destro
\node at (0,-0.8) {\footnotesize{$2h$ vertices}};

  % final node
  \node at (0,0) {$\cdots$};

    % lines
    \draw (-6,0) -- (-5.5,0);
 \draw (-4.5,0) -- (-4,0);
  \draw (-4,0) -- (-2.5,0); 
  \draw (-2.5,0) -- (-1,0);
  \draw (-1,0) -- (-0.5,0);
  \draw (0.5,0) -- (1,0);
\end{tikzpicture}
\end{center}
and this graph defines, by Lemma \ref{lemma: operation on plumb graphs 2} and Proposition \ref{prop: plumbing graphs moves}, the same two-bridge link as
\begin{center}
\begin{tikzpicture}[scale=0.7]  
% node 1
  \fill (3.5,0) circle (2pt);
  \node at (3.5,0.4) {$2k$};
  % node 2
  \fill (5,0) circle (2pt);
  \node at (5,0.4) {$-2$};
  
  % dots 
  \node at (6,0) {$\cdots$};
  
   % node 3
  \fill (7,0) circle (2pt);
  \node at (7,0.4) {$-2$};
  % square bracket
\draw (5,-0.5) -- (7,-0.5);      % segmento orizzontale
\draw (5,-0.5) -- (5,-0.2);      % tratto verticale sinistro
\draw (7,-0.5) -- (7,-0.2);      % tratto verticale destro
\node at (6,-0.8) {\footnotesize{$2n+1$ vertices}};
   % node 4
  \fill (8.5,0) circle (2pt);
  \node at (8.5,0.4) {$2h$};
\node at (9,0) {,};

    % lines
 \draw (3.5,0) -- (5,0);
 \draw (5,0) -- (5.5,0);
  \draw (6.5,0) -- (7,0);
  \draw (7,0) -- (8.5,0);
  \end{tikzpicture}
  \end{center}
and hence $L=L(-2k, 2,\dots, 2,-2h)$. Applying again \cite[Proposition~2]{GabKaz90}, together with the assumption that \(L\) is non-fibered, we deduce that at least one between \(2k\) and \(2h\) is greater than \(2\). Therefore, this description of \(L\) satisfies condition \emph{2a$)$}. 

If \(\alpha_i=2n>0\), then the corresponding plumbing graph is
\begin{center}
\begin{tikzpicture}[scale=0.7]
  % Node 1 
  \fill (-6,0) circle (2pt);
  \node at (-6,0.4) {  $-2$};
  
  % initial dots
  \node at (-5,0) {$\cdots$};

  % Node 2
  \fill (-4,0) circle (2pt);
  \node at (-4,0.4) {  $-2$};

  % Nodo 3
  \fill (-2.5,0) circle (2pt);
  \node at (-2.5,0.4) {  $-2n$};

% square bracket
\draw (-6,-0.5) -- (-4,-0.5);      % segmento orizzontale
\draw (-6,-0.5) -- (-6,-0.2);      % tratto verticale sinistro
\draw (-4,-0.5) -- (-4,-0.2);      % tratto verticale destro
\node at (-5,-0.8) {$2k$ vertices};

 % Node 4 
  \fill (-1,0) circle (2pt);
  \node at (-1,0.4) {  $-2$};
  
  % initial dots
  \node at (0,0) {$\cdots$};

  % Node 5 
  \fill (1,0) circle (2pt);
  \node at (1,0.4) {  $-2$};

% square bracket
\draw (-1,-0.5) -- (1,-0.5);      % segmento orizzontale
\draw (-1,-0.5) -- (-1,-0.2);      % tratto verticale sinistro
\draw (1,-0.5) -- (1,-0.2);      % tratto verticale destro
\node at (0,-0.8) {$2h$ vertices};

    % lines
    \draw (-6,0) -- (-5.5,0);
 \draw (-4.5,0) -- (-4,0);
  \draw (-4,0) -- (-2.5,0); 
  \draw (-2.5,0) -- (-1,0);
  \draw (-1,0) -- (-0.5,0);
  \draw (0.5,0) -- (1,0);
\end{tikzpicture}
\end{center}
and by Lemma \ref{lemma: operation on plumb graphs 3}, the link \(L\) is isotopic to \(L(-2k-2,-2,\dots,-2,-2h-2)\), and hence it satisfies condition \emph{2a$)$} of the statement, unless \(2n=4\), in which case \(L=L(-2k-2,-2,-2h-2)\).
\end{proof}

\begin{lem}\label{lemma: std diagram for m=3}
Suppose that \(L=L(2a,2b,2c)\) is non-fibered and not isotopic, up to mirror, to \(L(2,2,\dots, 2, 2z)\) for any non-zero \(z\in \mathbb{Z}\). Then, up to mirror image, \(L\) satisfies one of the following:
\begin{itemize}
\item \(L\) can be written as in Lemma~\ref{lemma: nonfib two-bridge links std diagrams}, with \(m\geq 5\) and satisfying one of the four conditions 1a\()\)-- 2b\(\)) of that statement;
\item \(L=L(-2,2\beta -1, 2)\), with \(|2\beta|>2\);
\item \(L=L(3,2,\dots,2,3)\), where the sequence \((3,2,\dots,2,3)\) has even length;
\item \(L=L(2\alpha-1, -2, 2\gamma -1)\), with \(2\alpha,2\gamma>2\).
\end{itemize}
\end{lem}
\begin{proof}
We assume \(2a>0\) without loss of generality, and we distinguish two cases.

\begin{tabular}{|c|}
\hline
Case $2b\ne 2$. 
\\
\hline
\end{tabular}
If \(2a>2\), then by Lemma \ref{lemma: operation on plumb graphs 1.5} we have
\[
L=L(-2,\dots,-2,2b-1,2c),
\]
where the sequence \((-2,\dots,-2,2b-1,2c)\) has length at least five. Hence, if \(2c>0\) or \(|2c|>2\), then \(L\) satisfies condition \emph{1a$)$} or \emph{1b$)$} of Lemma \ref{lemma: nonfib two-bridge links std diagrams}. Otherwise, \(2c=-2\), and by Lemma \ref{lemma: operation on plumb graphs 1.5} again,
\[
L=L(-2,\dots,-2,2b,2),
\]
with \(|2b|>2\), since \(L\) is non-fibered. In this case as well, \(L\) satisfies condition \emph{1a$)$} of the preceding lemma.

If \(2a=2\) and \(2c>2\), we argue as before, exchanging \(a\) and \(c\). If \(2c<-2\), then
\[
L(2,2b,2c)=L(2,2b+1,2,\dots,2)=L(-2,2b,2,\dots,2),
\]
and we conclude as before that \(L\) satisfies condition \emph{1a)}. Finally, if \(2c=\pm 2\), then we have again \(|2b|>2\), and
\[
L=L(2,2b,2)=L(-2,2b-1,2) \text{ or } L=L(2,2b,-2).
\]
In the latter case, arguing as in the proof of Lemma \ref{lemma: operation on plumb graphs 2}, we see that \(L\) is isotopic, up to mirror image, to \(L(3,2,\dots,2,3)\), where \((3,2,\dots,2,3)\) has even length.

\begin{tabular}{|c|}
\hline
Case $2b= 2$. 
\\
\hline
\end{tabular} 
Note that we must have \(2a>2\), since otherwise \(L=L(2,2,2c)\), contrary to our assumptions. If \(2c<0\), then by applying Lemma \ref{lemma: operation on plumb graphs 1.5} twice, we obtain
\[
L(2a,2,2c)=L(-2,\dots,-2,2,2,\dots,2),
\]
and hence \(L\) is fibered, a contradiction. If \(2c>0\), then
\[
L(2a,2,2c)=L(2a-1,-2,2c-1),
\]
which is what we wanted.
\end{proof}
We now turn, with the next lemma, to the fibered case.
\begin{lem}\label{lemma: fibered case std diagram}
Let $L$ be a fibered two-bridge link, and suppose that $L$ is not isotopic, up to mirror, to $L(2,2,\dots, 2, 2z)$ for any non-zero $z\in \mathbb{Z}$. Then either \(L=L(a_1,\dots,a_m)\), where this description satisfies one of the conditions of Lemmas \ref{lemma: nonfib two-bridge links std diagrams} and \ref{lemma: std diagram for m=3}, or, up to mirror image, one of the following holds:
\begin{itemize}
\item $L=L(\alpha+1,3,\pm 2)$, with $\alpha\geq 3$ odd;
\item $L=L(\alpha+1,4,\beta+1)$, with $\alpha, \beta\geq 2$ even;
\item $L=L(\alpha+1,3,a_{\alpha+2},\dots,a_{m-\beta-1},\pm 3, \pm (\beta+1))$, with $\alpha,\beta\geq 2$ even and $|a_i|=2$ for all $i=\alpha+2,\dots, m-\beta-1$.
\end{itemize}
\end{lem}
\begin{proof}
By \cite[Proposition~2]{GabKaz90}, we may assume that \(L=L(a_1,\dots,a_m)\) with \(|a_i|=2\) for all \(i\), and, up to mirror image, that \(a_1=-2\). Let \(\alpha+1\) be the smallest index \(i\) such that \(a_i=2\). Our hypotheses on \(L\) imply that \(1\le \alpha < m-1\). By Lemma \ref{lemma: operation on plumb graphs 1.5}, we have
\[
L=L(\alpha+1,3,a_{\alpha+2},\dots,a_m).
\]
In particular, if \(\alpha\ge 3\) is odd, then either \(L\) satisfies condition \emph{1b$)$} of Lemma \ref{lemma: nonfib two-bridge links std diagrams}, or \(L=L(\alpha+1,3,\pm 2)\).

Assume that \(\alpha=1\). In this case, there exists \(i>\alpha+1\) such that \(a_i=-2\); otherwise \(L=(-2,2,2,\dots,2)\), contrary to our assumptions. Thus, if \(m\ge 5\), then \(L\) satisfies condition \emph{1a$)$} of Lemma \ref{lemma: nonfib two-bridge links std diagrams}; otherwise, if \(m=3\), then \(L=L(2,3,-2)\), and so \(L\) satisfies one of the conditions of Lemma \ref{lemma: std diagram for m=3}.

We are left with the case in which \(\alpha\) is even. In this case, we apply the same argument starting from \(a_m\). More precisely, let \(\beta\) be the largest integer such that
\[
a_m=\cdots=a_{m-\beta+1}=\mp 2.
\]
If \(\beta\) is odd, then we are in one of the cases considered above. Otherwise, both \(\alpha\) and \(\beta\) are even, and applying Lemma \ref{lemma: operation on plumb graphs 1.5} we deduce that either
\[
L=L(\alpha+1,3,a_{\alpha+2},\dots,a_{m-\beta-1},\pm 3,\pm(\beta+1)),
\]
with \(\alpha+2\le m-\beta-1\), or \(\alpha+1=m-\beta\) and
\[
L=L(\alpha+1,4,\beta+1).
\]
In both cases, \(L\) satisfies one of the conditions in the statement.
\end{proof}

\begin{remark}\label{rem: L_n,k other description}
Observe that, as a consequence of Lemma \ref{lemma: operation on plumb graphs 1.5}, each of the links \(L_{n,k}\) and \(L'_{n,k}\) in Figure~\ref{fig: L and L'} is isotopic to \(L(2,\dots,2z)\) for some non-zero \(z\in\mathbb{Z}\), and viceversa. Hence, in the previous lemmas, we could equivalently have assumed that we were working with two-bridge links not isotopic to any of \(L_{n,k}\), \(L'_{n,k}\), or their mirrors.
\end{remark}

\subsection{Alexander polynomials}\label{subsec: alexander poly}

In this section we compute the multivariable Alexander polynomials of the links \(L_{n,k}\) and \(L'_{n,k}\). These polynomials will be used in Section \ref{subsec: H_function_L_n_k} to compute the $H$-function of $L_{n,k}$, and in Section \ref{subsec: class for L'nk} to determine the Turaev torsions of the manifolds obtained by filling one component of the links $L'_{n,k}$. 
Recall that the multivariable Alexander polynomial $\Delta_L$ of a link $L$ with more than one component depends on the orientation of $L$, so we fix the orientation on the link \(b(p,q)\) as described in Figure \ref{fig: two bridge ori}. Moreover, since the Alexander polynomial is unchanged by taking mirror images, we can assume that $q$ is positive.

In \cite{Hos20}, Hoste gives an explicit formula for $\Delta_{b(p,q)}$ and describes a useful algorithm to compute it, derived from an interpretation of such formula in terms of walks on the $2$-dimensional integer lattice. We note that our conventions for two-bridge links differ from those in \cite{Hos20}: the oriented link $K_{\sfrac{p}{q}}$ there corresponds to the oriented link $b(p,-q)$ in this paper.

The algorithm in \cite{Hos20} proceeds as follows. Let $\mathcal{A}$ be the set of positive multiples of $q$ that are strictly less than $pq$; that is,
\[
\mathcal{A}=\{iq|\, 0<i<p\}.
\]
To each multiple \(iq\), associate the quantity \(h_i=(-1)^{\lfloor iq/p\rfloor}\), and, when \(i\neq 1,p-1\), the quantity \(
v_i=\frac{h_{i-1}+h_{i+1}}{2}.\)
These data determines a path in \(\mathbb{Z}^2\) as follows:
\begin{itemize}\item Start at \((0,0)\in\mathbb{Z}^2\).
\item For each even multiple \(2iq\) of \(q\) in \(\mathcal{A}\), take a step from a  point \(P\) to the point \(P+(h_{2i},v_{2i})\).
\end{itemize}
In other words, the horizontal component of the step associated  to the multiple \(2iq\) is \(h_{2i}\), while the vertical component depends on the values of \(h_{2i-1}\) and \(h_{2i+1}\): if they are equal and positive is \(+1\), if they are equal and negative is $-1$, and is $0$ if they are different. We refer to the vector \((h_{2i}, v_{2i})\) as the \emph{coordinates of the step}.

In \cite{Hos20}, Hoste shows that the multivariable Alexander polynomial of \(b(p,q)\) is represented by the polynomial whose coefficient of \(x^iy^j\) is \((-1)^{i+j}k_{i,j}\), where \(k_{i,j}\) is the number of times the path visits the point \((i,j)\).

\subsubsection{Alexander polynomials of the links \(L_{n,k}\).}\label{subsubsec: Alex poly of link L_n,k}
We start with the links \(L_{n,k}\), with \(n,k\ge 1\), depicted on the left-hand side of Figure~\ref{fig: L and L'}. By direct computation, one checks that
\[
L_{n,k}=b(4nk+2n+2k,-2k-1).
\]
If we set \(p=2n+1\) and \(q=2k+1\), then \(L_{n,k}=b(p',-q)\), where \(p'=pq-1\). As observed above, in order to compute the Alexander polynomial, we can work with the mirror \(b(p',q)\).

We introduce some notation. As before, let \(\mathcal{A}\) denote the set of multiples of \(q\) contained in \((0,p'q)\), and observe that, since \(p'\) and \(q\) are coprime,
\[
\mathcal{A}=\mathcal{A}_1\sqcup\cdots\sqcup \mathcal{A}_q,
\]
where
\[
\mathcal{A}_m=\mathcal{A}\cap [(m-1)p',mp'], \qquad 1\le m\le q.
\]
Moreover, since for each fixed \(m\) there are exactly \(p\) multiples of \(q\) in \([(m-1)p',mp']\), the sets \(\mathcal{A}_1\) and \(\mathcal{A}_q\) have cardinality \(|\mathcal{A}_1|=|\mathcal{A}_q|=p-1\), while \(|\mathcal{A}_m|=p\) for \(1<m<q\). Notice that, by definition, if \(iq\in \mathcal{A}_m\), then \(h_i=(-1)^{m-1}\). 
\begin{figure}[]
    \centering   
    \includegraphics[width=0.8\textwidth]{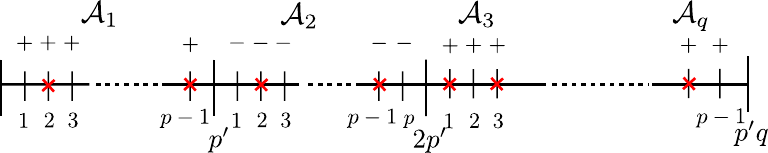}
    \caption{The figure shows the set \(\mathcal{A}\) and its partition into the subsets \(\mathcal{A}_m\). The multiples of \(q\) in each \(\mathcal{A}_m\) are labelled by the integers \(1,\dots,|\mathcal{A}_m|\). Above each multiple \(iq\), we indicate the sign of the corresponding \(h_i\).}
    \label{figure:Hoste}
\end{figure}

In the following proposition, we use the symbol \(\doteq\) to denote equality up to multiplication by units in \(\mathbb{Z}[x^{\pm1},y^{\pm1}]\).

\begin{prop}\label{prop: Alex poly of L_{n,k}}
For \(n,k\ge 1\), the multivariable Alexander polynomial \(\Delta_{n,k}\) of \(L_{n,k}\) is given by
\[
\Delta_{n,k}\doteq
\left(\sum_{i=0}^{k}x^{i}y^{k-i}\right)r_{n-1}
-
\left(\sum_{i=0}^{k-1}x^{i}y^{k-i-1}\right)r_n,
\]
where \(r_n=1+xy+\cdots+x^ny^n\).
\end{prop}
\begin{proof}
We want to locate the even multiples of \(q\) in the various sets \(\mathcal{A}_m\), for \(m=1,\dots,q\). We label the elements of each \(\mathcal{A}_m\) by the integers \(1,\dots,|\mathcal{A}_m|\), according to their natural order. For instance, the even multiples of \(q\) in \(\mathcal{A}_1\) are the elements with labels \(2,4,\dots,p-1\); the same holds for \(\mathcal{A}_2\), and so on. 
Since \(p=2n+1\) is odd, it follows by induction that, for \(m\neq 1\), the even multiples of \(q\) in \(\mathcal{A}_m\) are those with odd label when \(m\) is odd, and those with even label when \(m\) is even. 
See Figure \ref{figure:Hoste} for a pictorial description.

From the definitions, the coordinates of the step associated to an even multiple in \(\mathcal{A}_m\) are:
\begin{itemize}
    \item \(\bigl((-1)^{m-1},0\bigr)\) if the multiple has label \(1\) or \(|\mathcal{A}_m|\);
    \item \(\bigl((-1)^{m-1},(-1)^{m-1}\bigr)\) otherwise.
\end{itemize}
Applying Hoste's algorithm to the multiples of \(q\) in \(\mathcal{A}_1\cup\mathcal{A}_2\), we obtain a path that starts at \((0,0)\) and ends at \((0,-1)\). This path consists of \(n-1\) steps with coordinates \((1,1)\) and one step with coordinates \((1,0)\), corresponding to the even multiples of \(q\) in \(\mathcal{A}_1\), followed by \(n\) steps with coordinates \((-1,-1)\), corresponding to the even multiples of \(q\) in \(\mathcal{A}_2\). This gives the summand
\[
P_1=r_{n-1}-x^n y^{n-1}\bigl(1+x^{-1}y^{-1}+\cdots+x^{-n}y^{-n}\bigr)
= r_{n-1}-y^{-1}r_n.
\]
Next, the algorithm applied to the multiples in \(\mathcal{A}_3\cup\mathcal{A}_4\) first prescribes a step with coordinates \((1,0)\), and then exactly the same pattern of steps as above. Thus we obtain the same path as before, shifted so that it starts at \((1,-1)\), and the corresponding summand is
\(
P_2=xy^{-1}P_1.
\)
Iterating this process, and recalling that \(q=2k+1\), we obtain summands \(P_j\) associated to \(\mathcal{A}_{2j-1}\cup\mathcal{A}_{2j}\), for \(j=1,\dots,k\), satisfying
\(
P_j=xy^{-1}P_{j-1},
\)
and hence, by induction,
\[
P_j=x^{j-1}y^{1-j}P_1.
\]
This also shows that the path associated to \(\mathcal{A}_1\sqcup\cdots\sqcup\mathcal{A}_{2k}\) ends at the point \((k-1,-k)\).

Finally, the multiples of \(q\) in \(\mathcal{A}_q\) correspond to one step with coordinates \((1,0)\), followed by \(n-1\) steps with coordinates \((1,1)\), and so give the summand
\[
P'=x^k y^{-k}r_{n-1}.
\]
In summary, we have
\begin{equation*}
\begin{split}
\Delta_{n,k}
&\doteq \sum_{j=1}^k P_j + P' \\
&= \left(\sum_{j=1}^k x^{j-1}y^{1-j}\right)P_1 + x^k y^{-k}r_{n-1} \\
&= \left(\sum_{i=0}^{k-1} x^i y^{-i}\right)P_1 + x^k y^{-k}r_{n-1} \\
&= \left(\sum_{i=0}^{k-1} x^i y^{-i}\right)\bigl(r_{n-1}-y^{-1}r_n\bigr) + x^k y^{-k}r_{n-1} \\
&= \left(\sum_{i=0}^k x^i y^{-i}\right)r_{n-1}
   - \left(\sum_{i=0}^{k-1} x^i y^{-i-1}\right)r_n.
\end{split}
\end{equation*}
Multiplying by \(y^k\), we obtain the desired formula.
\begin{comment}\item The even multiples of $q$ in $\mathcal{A}_m$, for odd $m\ne 1,q$, correspond to one step with coefficient $(1,0)$, then $(n-1)$ steps with with coefficient $(1,1)$, and finally one step with coefficient $(1,0)$, giving a summand equal to
\[
(-1)^{a_m+b_m}x^{a_m}y^{b_m}(-x+xr_{n-1}-x^ny^{n-1}),
\]
with $(a_m, b_m)$ to be determined.
\item Finally, the even multiples of $q$ in $
\mathcal{A}_q$ correspond to one step with direction $(1,0)$, followed to $(n-1)$ steps with direction $(1,1)$, and a summand equal to
\[
(-1)^{a_q+b_q}x^{a_q}y^{b_q}(-x+xr_{n-1}),
\]
with $(a_q, b_q)$ to be determined.
\end{comment}
\end{proof}
We refer the reader to Figure \ref{fig: H_function_L_n_k} for (a translate of) the set of lattice points visited by the path in the preceding proof.
\subsubsection{Alexander polynomials of the links \(L'_{n,k}\).}\label{subsubsec: Alex poly of link L'_n,k} We focus now on the links \(L'_{n,k}\), with \(n,k\ge 1\), described on the right-hand side of Figure~\ref{fig: L and L'}. They are isotopic to the mirrors of \(b(4nk+2n-2k,2n-1)=b(pq+1,q)\), where \(p=2k+1\) and \(q=2n-1\). For this reason, we work with $b(pq+1,q)$, and we orient it according to the convention set at the beginning of Section \ref{section:two-bridge}, i.e. as in Figure \ref{figure: L' ori}. We set \(p'=pq+1\), so that \(L'_{n,k}=b(p',q)\).
\begin{figure}[]
    \centering   
    \includegraphics[width=0.16\textwidth]{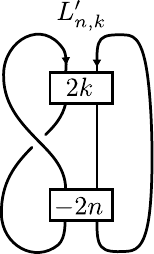}
    \caption{The orientation on the link $L'_{n,k}$ used to compute its Alexander polynomial.}
    \label{figure: L' ori}
\end{figure}

As before, we have
\[
\mathcal{A}=\mathcal{A}_1\sqcup\cdots\sqcup\mathcal{A}_q,
\]
where \(\mathcal{A}\) is the set of multiples of \(q\) in \((0,p'q)\), and
\[
\mathcal{A}_m=\mathcal{A}\cap[(m-1)p',mp'], \qquad m=1,\dots,q.
\]
Each \(\mathcal{A}_m\) has cardinality \(p\), and since \(p\) is odd, if we label the multiples of $q$ in each \(\mathcal{A}_m\) by \(1,\dots,p\), then the even multiples of \(q\) in \(\mathcal{A}_m\) are those whose parity agrees with that of \(m-1\).
Recall that, for \(k\ge 0\), we use the notation \(r_k=1+xy+\cdots+x^k y^k\).
\begin{prop}\label{prop: Alex poly of L'_{n,k}}
For \(n,k\ge 1\), the multivariable Alexander polynomial \(\Delta'_{n,k}\) of \(L'_{n,k}\) is
\[
\Delta'_{n,k}\doteq
\left(\sum_{i=0}^{n-1}x^{-i}y^{i}\right)r_k
-
\left(\sum_{i=0}^{n-2}x^{-i}y^{i+1}\right)r_{k-1}.
\]
\end{prop}
\begin{proof}The argument is similar to the one in Proposition \ref{prop: Alex poly of L_{n,k}}. We begin by assuming \(n\ge 2\), so that \(q\ge 2\), and apply Hoste's algorithm to the multiples of \(q\) in \(\mathcal{A}_1\cup\mathcal{A}_2\). This yields \(k=\frac{p-1}{2}\) steps with coordinates \((1,1)\), corresponding to \(\mathcal{A}_1\), followed by one step with coordinates \((-1,0)\), then \(k-1\) steps with coordinates \((-1,-1)\), and finally one step with coordinates \((-1,0)\) again. The path obtained before this last step starts at \((0,0)\) and ends at \((0,1)\), and it gives the summand
\[
P_1=r_k-yr_{k-1}.
\]
Next, we arrive at the point \((-1,1)\) with the last step, and we apply the same reasoning to \(\mathcal{A}_3\cup\mathcal{A}_4\). The path obtained by carrying out all the steps associated to \(\mathcal{A}_3\cup\mathcal{A}_4\) except the last one is exactly the same as before, shifted so that the starting point is \((-1,1)\). Thus we obtain the summand 
\( 
P_2=x^{-1}yP_1.
\)
Iterating, and recalling that \(q=2n-1\), we obtain summands \(P_j\) associated to \(\mathcal{A}_{2j-1}\cup\mathcal{A}_{2j}\), for \(j=1,\dots,n-1\), satisfying
\(
P_j=x^{-1}yP_{j-1},
\)
which implies by induction that
\[
P_j=x^{1-j}y^{j-1}P_1.
\]This also shows that the path associated to \(\mathcal{A}_1\sqcup\cdots\sqcup\mathcal{A}_{2n-2}\) ends at the point \((1-n,n-1)\). Finally, the multiples of \(q\) in \(\mathcal{A}_q\) correspond to \(k\) steps with coordinates \((1,1)\), and so give the summand
\[
P'=x^{1-n}y^{n-1}r_k.
\]
Therefore,
\begin{equation*}
\begin{split}
\Delta'_{n,k}
&\doteq \sum_{j=1}^{n-1}P_j+P' \\
&=\left(\sum_{j=1}^{n-1}x^{1-j}y^{j-1}\right)P_1+x^{1-n}y^{n-1}r_k \\
&=\left(\sum_{i=0}^{n-2}x^{-i}y^i\right)P_1+x^{1-n}y^{n-1}r_k \\
&=\left(\sum_{i=0}^{n-2}x^{-i}y^i\right)(r_k-yr_{k-1})+x^{1-n}y^{n-1}r_k \\
&=\left(\sum_{i=0}^{n-1}x^{-i}y^i\right)r_k-\left(\sum_{i=0}^{n-2}x^{-i}y^{i+1}\right)r_{k-1},
\end{split}
\end{equation*}
which is the desired result.

When \(n=1\), the link \(L'_{1,k}\) is a torus link. In this case, $\mathcal{A}=\mathcal{A}_1$ and we have
\[
\Delta'_{1,k}\doteq \sum_{i=0}^{k}x^{i}y^i=r_k,
\]
and this concludes the proof.
\end{proof}

\section{Diagrams and persistently foliar links}\label{section:foliations}
The goal of this section is to prove Theorem \ref{thm: pers fol link}, which is a general theorem regarding the existence of taut foliations on all non-trivial surgeries on $n$-component links. This theorem will be applied in Section \ref{sec: pers fol two-bridge} to two-bridge links.

The section is roughly divided in three parts. We first introduce the notation needed to state Theorem \ref{thm: pers fol link}. We then recall some notions regarding the theory of branched surfaces, and finally, in Section \ref{subsec:proof of pers fol thm}, we prove Theorem \ref{thm: pers fol link}.
\newline

Recall that a slope on a torus $T$ is an element of the projective space $\mathbb{P}(H_1(T; \R))$ and that rational slopes, i.e. the elements of $\mathbb{P}(H_1(T;\Q))$, are in bijection with the set of isotopy classes of (unoriented) essential simple closed curves on $T$.  There is an identification, canonical up to isotopy, between \(T\) and \(\sfrac{H_1(T;\mathbb{R})}{H_1(T;\mathbb{Z})}\). A \emph{linear foliation of slope \(s\)} on \(T\) is a foliation isotopic to the image in \(T\) of the foliation by parallel lines of slope \(s\) on \(H_1(T;\mathbb{R})\). 

If \(M\) is a compact orientable \(3\)-manifold whose boundary is a union of tori \(T_1,\dots,T_n\), a \emph{multislope} is an \(n\)-tuple \((s_1,\dots,s_n)\), where each \(s_i\) is a slope on the torus \(T_i\). If \(M\) is the exterior of a link \(L\subset S^3\), we say that a multislope \((s_1,\dots,s_n)\) is \emph{non-trivial} if none of the \(s_i\) is the meridian of a component of \(L\).

\begin{defn}\label{def: pers fol link}
A link \(L\subset S^3\) is \emph{persistently foliar} if, for every non-trivial multislope on \(L\), there exists a coorientable taut foliation in the exterior of \(L\) intersecting the boundary of the link exterior transversely in a linear foliation of that multislope.
\end{defn}

This definition generalises that of a persistently foliar knot, introduced by Delman and Roberts in \cite{DRdiamond}. Observe that every non-trivial surgery on a persistently foliar link supports a coorientable taut foliation, obtained by capping off the leaves of the taut foliation in the exterior of the link with the meridional discs of the solid tori. 
\newline

\begin{figure}[]
    \centering
    \includegraphics[width=0.55\textwidth]{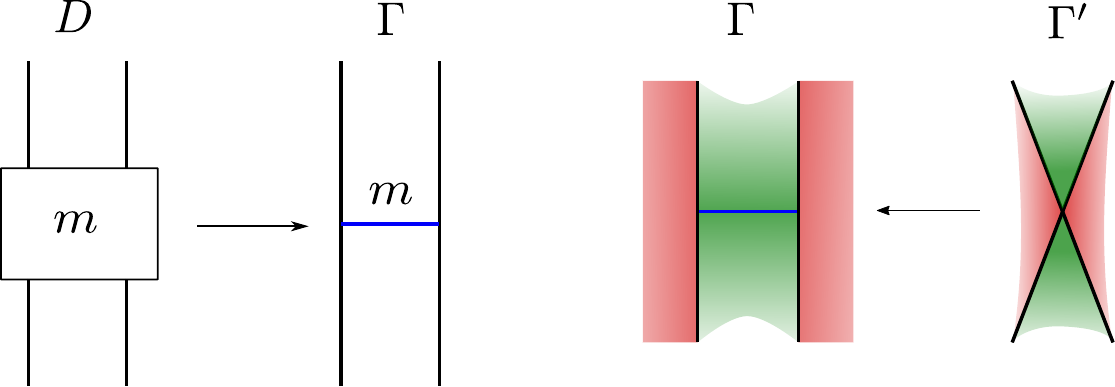}
    \caption{Left: Construction of $\Gamma$ from $D$. Right: Colouring the complementary regions of $\Gamma$ via a checkerboard colouring of the complementary regions of $\Gamma'$.}
    \label{fig: Gamma' intro}
\end{figure}
The main theorem of this section states, roughly speaking, that links with special diagrams are persistently foliar. The diagrammatic condition appearing in the statement comes from adapting and extending the main construction of \cite{San24diag} to the case of links with multiple components. We begin by introducing some notation and definitions.

\begin{defn}
Let $D$ be a diagram of a link in $S^3$. A \emph{bigon} in $D$ is a complementary region of the link projection in $S^2$ having two vertices, corresponding to two crossings of the diagram, in its boundary. A \emph{twist region} is either a maximal connected chain of bigons arranged in a row — that is, one that is not contained in a longer such chain — or a single crossing not adjacent to any bigons. 
\end{defn}

We will always suppose that the diagram is alternating in a twist region. Otherwise, the diagram can be replaced by another one with fewer crossings by using Reidemeister moves. 

\begin{defn} Let $D$ be a link diagram, and let $R$ be a twist region containing at least two crossings. The \emph{weight} $w(R)$ of $R$ is the integer whose absolute value equals the number of crossings in $R$, with sign determined by the crossing type: positive if the crossings in $R$ are right-handed and negative if they are left-handed. \end{defn}

\begin{defn}
Let \(K\) be a component of $L$, and let $R$ be a twist region in a diagram of $L$. We say that $K$ \emph{passes through} $R$ if at least one of the two strands of $L$ in $R$ belongs to $K$. 
\end{defn}
Suppose that \(D\) is a link diagram whose twist regions all contain at least two crossings. We associate to \(D\) two graphs with \(\mathbb{Z}\)-weighted edges, \(\mathcal{G}_g\) and \(\mathcal{G}_r\), as follows:

\begin{itemize}[leftmargin=*]
\item \textbf{Step 1.} We replace each twist region in \(D\) by a blue arc with weight \(w(R)\), as Figure~\ref{fig: Gamma' intro}--left shows. We denote the resulting graph in \(S^2\) by \(\Gamma\).

\item \textbf{Step 2.} We colour the complementary regions of \(\Gamma\) as follows. First, we collapse each blue arc of \(\Gamma\) to a point, obtaining a \(4\)-valent graph \(\Gamma' \subset S^2\), whose complementary regions correspond naturally to those of \(\Gamma\). We then checkerboard-colour the complementary regions of \(\Gamma'\) and transfer this colouring to \(\Gamma\) via the correspondence (see Figure~\ref{fig: Gamma' intro}--right). For convenience, we colour the regions green and red.

\item \textbf{Step 3.} The vertices of the graph \(\mathcal{G}_g\) correspond to the green complementary regions of \(\Gamma\). Each blue arc in \(\Gamma\) that meets some green regions \emph{only} at its endpoints determines a weighted edge between the corresponding vertices, with weight equal to that of the blue arc. The graph \(\mathcal{G}_r\) is defined analogously, using the red complementary regions and the blue arcs that meet them only at their endpoints.
\end{itemize}
We refer to Figure \ref{fig: example theorem} for a concrete example of the construction of the graphs.
\newline

We denote by $\mathcal{G}$ the disjoint union $\mathcal{G}_r\sqcup \mathcal{G}_g$. By construction, the edges of $\mathcal{G}$ are canonically identified with the twist regions of $D$. Using this identification, we say that a component of $L=K_1\sqcup \cdots \sqcup K_n$ \emph{passes through} an edge of $\mathcal{G}$ if it passes through the corresponding twist region. We denote by $\mathcal{R}_i$ the set of edges of $\mathcal{G}$ the component $K_i$ passes through.

\begin{figure}[]
    \centering
    \includegraphics[width=0.65\textwidth]{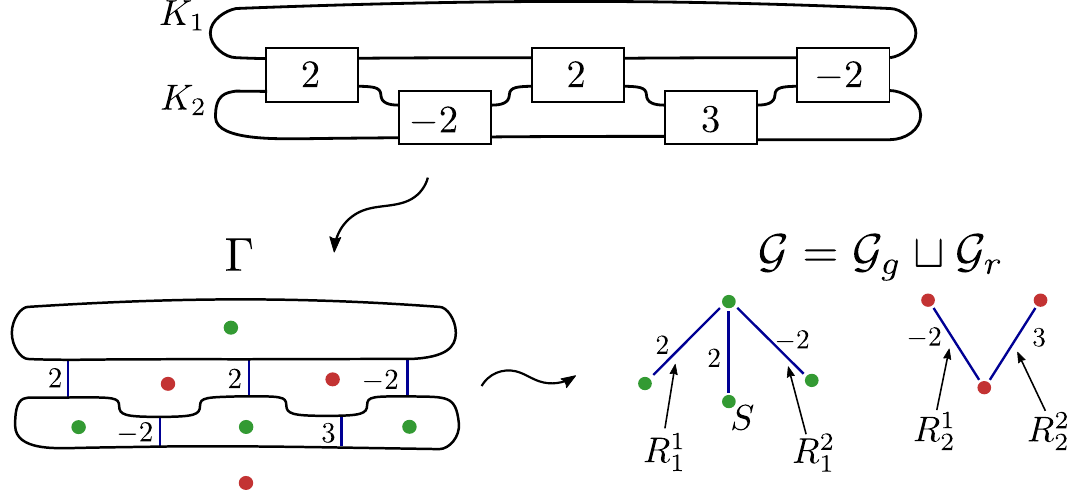}
    \caption{The link \(K_1 \sqcup K_2\) is persistently foliar by Theorem \ref{thm: pers fol link}. The picture shows the graphs \(\Gamma\) and \(\mathcal{G}=\mathcal{G}_g \sqcup \mathcal{G}_r\) obtained from the diagram. For \(i=2\), the edges \(R_i^1\) and \(R_i^2\) satisfy the first condition of the theorem. For \(i=1\), \(K_1\) passes only through twist regions with an even number of crossings, and since \(R_1^1\) and \(R_1^2\) have opposite weights, the second condition of the theorem is satisfied.}
    \label{fig: example theorem}
\end{figure}
\begin{thm}\label{thm: pers fol link}
Let \(D\) be a diagram of a link
\(L = K_1 \sqcup \cdots \sqcup K_n\) in \(S^3\).
Assume that every twist region of \(D\) contains at least two crossings, and that the graphs \(\mathcal{G}_r\) and \(\mathcal{G}_g\) are connected. Suppose further that there exist a vertex \(S\) of \(\mathcal{G}\) and a collection of pairwise distinct edges
\[
R_i^j \in \mathcal{R}_i \quad \text{for } i=1,\dots,n \text{ and } j=1,2,
\]
none of which is adjacent to \(S\), such that for each \(i\) one of the following holds:
\begin{itemize}
\item \(\max\{|w(R_i^1)|, |w(R_i^2)|\} \geq 3\);
\item every edge in \(\mathcal{R}_i\) has even weight and \(w(R_i^1) = -w(R_i^2)\).
\end{itemize}
Then \(L\) is persistently foliar. In particular, every non-trivial surgery on \(L\) supports a coorientable taut foliation.
\end{thm}
Figure~\ref{fig: example theorem} shows a sample application of the theorem to a link with two components.

\begin{remark}
If the diagram $D$ has sufficiently many twist regions compared to the number of components of $L$, then the hypothesis on the existence of the vertex $S$ in the statement Theorem \ref{thm: pers fol link} is always satisfied. Indeed, an Euler characteristic argument as in \cite[Lemma~4.3]{San24diag} shows that, given any collection of $2n$ distinct edges of \(\mathcal{G}\), there exists a vertex $S$ of $\mathcal{G}$ not adjacent to any of them whenever D has more than $4n-2$ twist regions.
\end{remark}

\begin{remark}
In the case of knots, a very similar theorem appears in \cite[Theorem~4.1]{San24diag}. Nevertheless, even in this case, Theorem \ref{thm: pers fol link} is more general, as it also applies to knot diagrams in which all twist regions have two crossings, provided that two of them contain crossings of opposite type.
\end{remark}

The proof of Figure~\ref{thm: pers fol link} is based on the construction presented in \cite{San24diag}. For this reason, we will often refer the reader to that paper for proofs and more details. At the same time, in order to explain the modifications required to establish the more general statement given here and to clarify as much as possible the key ideas underlying the proof, we need to recall some notions. We begin with branched surfaces

\subsection{Branched surfaces}
Let $M$ be a compact orientable manifold, whose boundary is a possibly empty union of tori.

\begin{figure}[]
    \centering
    \includegraphics[width=0.7\textwidth]{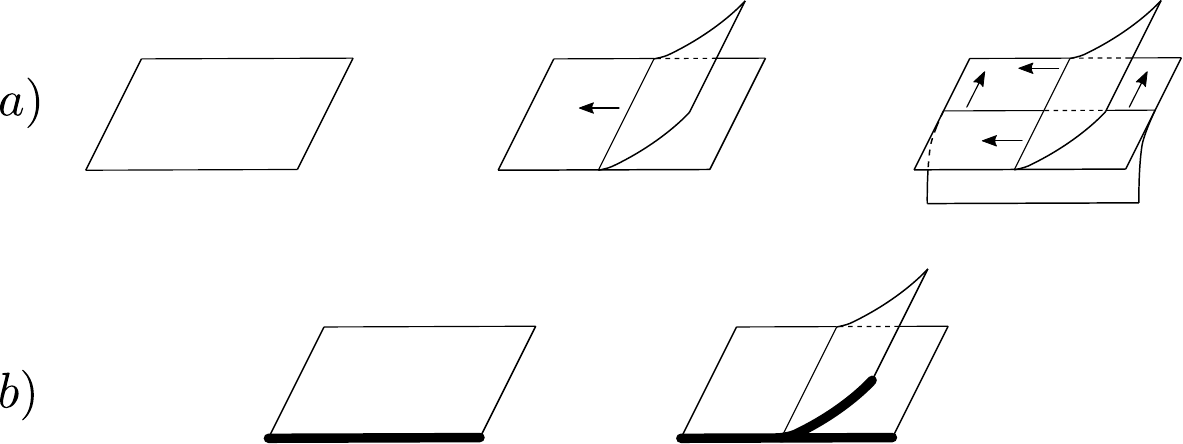}
    \caption{Local models for a branched surface, with cusp directions.}
    \label{branched surface}
\end{figure}
\begin{defn}
A \emph{branched surface with boundary} in  $M$ is a compact subset $B\subset M$ locally diffeomorphic to one of the models shown of Figure~\ref{branched surface} $a)$ (in $\R^3$) or Figure~\ref{branched surface} $b)$ (in the closed half-space), where $\partial B:= B\cap \partial M$ is represented with a bold line.
\end{defn}

We define the \emph{branch locus} of $B$ as the set of points where $B$ is not locally homeomorphic to a surface, and the \emph{triple points of $B$} as those where the branch locus is not locally homeomorphic to an arc.

The complement of the branch locus in $B$ is a union of connected surfaces. The abstract closures of these surfaces -- taken under any path metric on $M$ -- are called the \emph{sectors} of $B$. Similarly, the complement of the set of triple points within the branch locus is a union of $1$-dimensional connected manifolds. To each of these manifolds we associate
arrows in $B$, called \emph{cusp directions}, indicating the direction of the smoothing, as illustrated in Figure \ref{branched surface}.

We denote by $N_B$ a fibered regular neighbourhood of $B$, well-defined up to isotopy, whose local structure is depicted in Figure~\ref{regular neighbourhood}.
The boundary of $N_B$ naturally decomposes into three compact subsurfaces: the \emph{horizontal boundary} $\partial_h N_B$, the \emph{vertical boundary} $\partial_v N_B$, and $N_B\cap \partial M$. The horizontal boundary is transverse to the interval fibers of $N_B$, while the vertical boundary intersects them, if at all, in one or two proper closed subintervals contained in their interiors.

\begin{figure}[]
    \centering
    \includegraphics[width=0.55\textwidth]{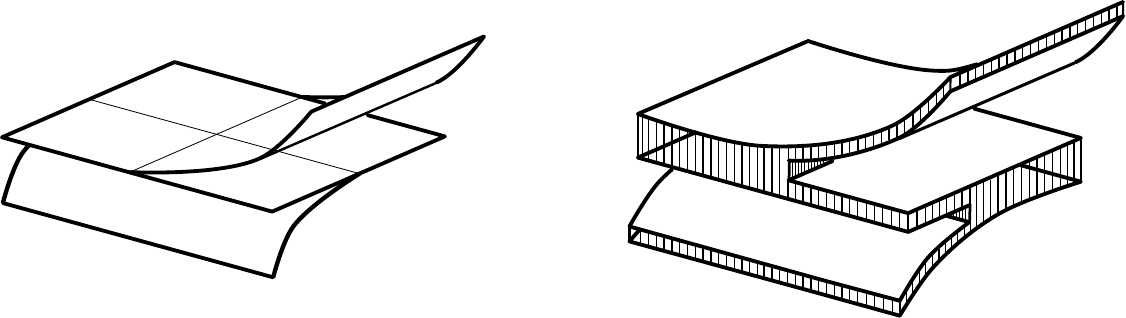}
    \caption{Regular neighbourhood of a branched surface.}
    \label{regular neighbourhood}
\end{figure}

\subsubsection{From branched surfaces to laminations.}
A common strategy to construct foliations on a $3$-manifold is to find intermediate objects, called \emph{laminations}. 
\begin{defn}\label{lamination abstr}
A lamination $\mathcal{L}$ in $M$ is a decomposition of a closed subset of $M$ into a union of injectively immersed connected surfaces, called \emph{leaves} of $\mathcal{L}$, such that $(M, \mathcal{L})$ is locally homeomorphic to $(\R^3, \R^2\times C)$ or to $(\mathbb{H}^3, \mathbb{H}^2\times C)$. Here, $C$ is a closed subset of $\R$ depending on the chart, $\mathbb{H}^2$ is the closed half-plane, and $\mathbb{H}^3=\mathbb{H}^2\times \mathbb{R}$ is the closed half-space.
\end{defn}
Branched surfaces provide an useful combinatorial tool that can be used to define laminations.
\begin{defn}\label{lamination}
Let $B$ be a branched surface in a $3$-manifold $M$. A lamination $\mathcal{L}$ is \emph{carried by B} if $\mathcal{L}\subset N_B$ and $\mathcal{L}$ intersects the fibers of $N_B$ transversely. We say that $\mathcal{L}$ is \emph{fully carried} by $B$ if it is carried by $B$ and intersects every fiber of $N_B$.
\end{defn}

Analogous definitions apply to train tracks on surfaces, and, for later use, we introduce the following notion.
\begin{defn}
Let $\tau$ be a train track on a torus $T$. A rational slope $s$ on $T$ is \emph{realised} by $\tau$ if $\tau$ fully carries a union of finitely many curves of slope $s$.
\end{defn}
Not all branched surfaces fully carry laminations (see, for example, \cite[Proposition~3.1]{GabOer}); however, in \cite{Li_laminar}, Li introduces the notion of \emph{laminar branched surface}, and shows that they always fully carry laminations. We refer to \cite{Li_laminar} for details and definitions, and note only that a crucial requirement for a branched surface to be laminar is that it contains no \emph{$($half$)$ sink discs}.

\begin{defn}[\cite{Li_laminar}]
Let $B$ be a branched surface in $M$ and let $S$ be a sector in $B$.
We say that $S$ is a \emph{sink disc} if $S$ is homeomorphic to a disc, $S\cap \partial M=\emptyset$, and the cusp direction of any smooth curve or arc in its boundary points into $S$.
We say that $S$ is a \emph{half sink disc} if $S$ is a disc, $S  \cap \partial M\neq \emptyset$, and the cusp direction of any smooth arc in $\partial S\setminus \partial M$ points into $S$.
\end{defn}

We fix the following notation. Let $M$ be a manifold whose boundary is union of tori $T_1, \dots, T_k$, and let $h\leq k$ be a positive integer. Given a rational multislope $(s_1,\dots, s_h)$, we denote by $M(s_1,\dots, s_h)$ the manifold obtained by filling the boundary component $T_i$ along the slope $s_i$, for $i=1, \dots, h$. When $h < k$, this manifold has non-empty boundary. 

In the case of our interest, which concern fillings of $3$-manifolds with torus boundary, the following theorem plays a key role. We state it in a slightly more general and detailed form than the one appearing in \cite{Li_boundary}; the additional details follow from its proof. We have not defined the notion of \emph{essential} lamination, appearing in the statement, since we will not need it; we refer to \cite{GabOer} for details.

\begin{thm}[{\cite[Theorem~2.5]{Li_boundary}}]\label{boundary train tracks}
Let $M$ be a compact $3$-manifold whose boundary is union of tori $T_1,\dots, T_k$. Let $B$ be a laminar branched surface in $M$ with $B\cap \partial M\subset T_1\cup \cdots \cup T_h$, for some $h\leq k$, and assume that, for each $i=1, \dots, h$, the complement $T_i \setminus \partial B$ is a union of bigons. Let $(s_1,\dots, s_h)$ be any rational multislope realised by the train track $\partial B$, and suppose that $B$ does not carry a torus that bounds a solid torus in $M(s_1,\dots,s_h)$. Then, there exists an essential lamination $\mathcal{L}$ in $M$ fully carried by $B$ that intersects $T_i$ in parallel simple closed curves of slope $s_i$, for $i=1, \dots, h$. Moreover, this lamination extends to an essential lamination of the filled manifold $M(s_1,\dots, s_h)$.
\end{thm}

\subsubsection{From laminations to foliations.}
In general, the presence of a lamination \(\mathcal{L}\) does not guarantee the existence of a foliation extending it. Here, we present one situation in which this does happen, following \cite{DRdiamond}. We assume some familiarity with sutured manifold theory, and refer the reader to \cite{Gabfol} for definitions.

\begin{defn}[\cite{DRdiamond}]
A cooriented branched surface $B$ is \emph{taut} if the exterior of $B$ is a taut sutured manifold and if through every sector of $B$ there is a closed oriented curve that is
positively transverse to $B$.
\end{defn}

If $\mathcal{L}$ is a lamination in $M$ carried by a branched surface $B$, we will suppose that $\partial_h N_B$ is contained in $\mathcal{L}$. Moreover, we denote by $M_{|\mathcal{L}}$ the metric completion of $M\setminus \mathcal{L}$ under any path metric on  $M\setminus \mathcal{L}$.

\begin{defn}[\cite{DRdiamond}]
Suppose that $B$ fully carries a lamination $\mathcal{L}$. A component $A$ of $\partial_v N_B$ satisfies the \emph{noncompact extension property} relative to $(B, \mathcal{L})$ if there exists a proper embedding of $[0,1]\times [0, \infty)$  in $M_{|\mathcal{L}} \cap N_B$ satisfying the following conditions:
\begin{itemize}
    \item the image of $[0,1] \times \{0\}$ is contained in a fiber of $N_B$ and contains a fiber of $A$;
    \item the image of $\{0,1\}\times [0,\infty)$ is contained in leaves of $\mathcal{L}$.
\end{itemize}  
\end{defn}
As the name suggests, the noncompact extension property can be used to extend laminations to foliations, as the following result shows. 
\begin{thm}[{\cite[Proposition~3.10]{DRdiamond}}]\label{cusps implies persistently foliar}
Suppose that \(\partial_1 M,\dots,\partial_m M\) are torus boundary components of \(M\), and let \(B\) be a taut, cooriented branched surface in \(M\) that is disjoint from them and fully carries a lamination \(\mathcal{L}\). For each \(i\), let \((Y_i,\partial_v Y_i)\) denote the complementary region of \(N_B\) containing \(\partial_i M\). Assume that
\[
(Y_i,\partial_v Y_i)\cong (\partial_i M\times [0,1],\, \mathcal{A}_1\cup\cdots\cup\mathcal{A}_{2n_i}),
\]
where $\partial_i M \times \{0\} = \partial_i M$ and $\mathcal{A}_1, \dots, \mathcal{A}_{2n_i}$ are disjoint essential annuli in $ \partial_i M \times \{1\}$ satisfying
the noncompact extension property relative to $(B, \mathcal{L})$. Then, for any multislope $(s_1,\dots, s_m)$ on $\partial_1 M,\dots, \partial_m M$ with the property that no $s_i$ is isotopic in $Y_i$ to the
common core of the annuli $\mathcal{A}_k$ there exists a cooriented taut foliation
that extends $\mathcal{L}$ and intersects $\partial_1 M,\dots, \partial_m M$ transversely in a linear foliation of such multislope.
\end{thm}
The statement above differs from that of Delman and Roberts in \cite{DRdiamond}, in that it guarantees the existence of foliations realising multislopes rather than slopes. However, their proof adapts readily to this more general setting, as communicated privately by the authors.

\subsection{Proof of Theorem \ref{thm: pers fol link}}\label{subsec:proof of pers fol thm}
In this section we prove Theorem \ref{thm: pers fol link}.
Let \(L\subset S^3\) with components \(K_1,\dots, K_n\), and fix a diagram \(D\) of it. Starting from \(D\), we define a new link \(\mathcal{L}\subset S^3\)  as follows. For each twist region, we first remove all crossings in it if it contains an even number of them, and all but one
otherwise. Then, we encircle the two strands with an unknotted circle, called a \emph{crossing circle}. The link \(\mathcal{L}\) has $n$ components, that we denote \(K'_1,\dots, K'_n\) corresponding to those of $L$, and additional components \(J_1,\dots, J_m\) given by the crossing circles, one for each twist region of $D$. Each crossing circle bounds a twice-punctured disc in the exterior of \(\mathcal{L}\), that we call \emph{crossing disc}, see Figure~\ref{fig: borr twists}.

By construction, the link \(L\) is the image of \(K'_1 \sqcup \cdots \sqcup K'_n\) after an appropriate Dehn surgery along the crossing circles. More precisely, if the twist region associated to the crossing circle $J_i$ contains $2k_i$ or $2k_i+1$ left-handed crossings, with \(k_i\geq 0\), the surgery coefficient of \(J_i\) is \(\frac{1}{k_i}\). Analogously, if the twist region contains $-2k_i$ or $-(2k_i-1)$ right-handed crossings, with \(k_i\leq 0\), the surgery coefficient is \(\frac{1}{k_i}\). 
\begin{figure}[h]
    \centering
    \includegraphics[width=0.7\textwidth]{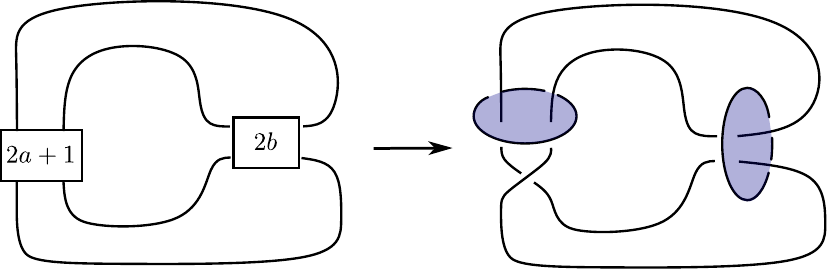}
    \caption{The link $\mathcal{L}$ associated to the diagram on the left, where $a\geq 0$ and $b\ne 0$. The twice-punctured discs bounded by the crossing circles are coloured. }
    \label{fig: borr twists}
\end{figure}
\begin{remark}
When all twist regions in $D$ have even weight, the components \(K'_1,\dots, K'_n\) of \(\mathcal{L}\) are all embedded in the projection sphere. When this does not happen, it is useful to consider an auxiliary link \(\mathcal{L}'\), whose exterior is a mutant of the exterior of \(\mathcal{L}\). The link \(\mathcal{L}'\) is obtained by cutting the exterior of \(\mathcal{L}\) along the crossing discs that are adjacent to crossings, and by regluing with an half-twist so to remove the crossing. The link \(\mathcal{L}'\) has the same crossing circles as \(\mathcal{L}\), and the remaining components (whose number might have changed) are now embedded in the projection sphere. 
\end{remark}

The overall strategy of the proof is the following. We aim to construct a branched surface in the exterior of $\mathcal{L}$, and use it, together with Theorems \ref{boundary train tracks} and \ref{cusps implies persistently foliar}, to construct taut foliations in the exterior of \(\mathcal{L}\) intersecting:
\begin{itemize}
\item the boundary components corresponding to the crossing circles in linear foliations of multislope \((\frac{1}{k_1},\dots, \frac{1}{k_m})\);
\item the boundary components corresponding to the remaining components of $\mathcal{L}$ in linear foliation of any non-trivial multislope $(s_1,\dots, s_n)$.
\end{itemize}
Such a situation would imply that $L$ is persistently foliar, proving the theorem.
\newline

Let $M$ denote the exterior of \(\mathcal{L}\), with \(\partial_i M\) being the boundary component corresponding to \(K'_i\), for \(i=1,\dots n\). 
For each of these boundary components, we fix a boundary parallel torus \(T_i=\partial_i M\times \{1\}\), where 
\[
\nu\partial_i M=\partial_i M \times [0,1],  \text{ with } \partial_i M= \partial_i M\times \{0\}
\]
is a fixed collar of $\partial_i M$.

Our first step is to define a $2$-complex $\Sigma$ in $M$, that we will later smooth into a branched surface.
The $2$-complex $\Sigma$ is defined as the union of:

\begin{itemize}[leftmargin=*]
\item The boundary parallel tori $T_1,\dots, T_n$.
\item The closures of the intersections between the crossing discs $D_1,\dots, D_m$, bounded by the components $J_1,\dots, J_m$ respectively, and \(M\setminus(\nu\partial_1 M\cup\cdots \cup \nu\partial_n M)\). By a slight abuse of notation, we still denote this discs $D_1,\dots, D_m$ and refer to them as the crossing discs.
\item The closures in $M$ of the connected components of 
$$
S^2\setminus (D_1\cup\cdots \cup D_m \cup \nu\partial_1 M\cup\cdots \cup \nu\partial_n M)$$
when the weights of all twist regions in $D$ are even; otherwise, we proceed as follows: consider the image of 
$$
D_1\cup\cdots \cup D_m \cup \nu\partial_1 M\cup\cdots \cup \nu\partial_n M
$$
in the exterior of $\mathcal{L}'$ after the mutation, take the closures of the connected components of its complement in $S^2$, and then take their preimages in $M$ under the mutation. In both  cases, we denote by $\mathcal{S}$ the collection of surfaces obtained in this way.
\end{itemize}
Notice that each element of 
$S$ is an immersed disc in 
$M$ whose interior is embedded. Figure \ref{fig: Sigma} shows an example of $\Sigma$ in the case when all twist regions have an even number of crossings. The exterior of $\Sigma$ has a very simple structure, as shown by the following lemma, proved in \cite[Lemma~5.4]{San24diag}.

\begin{figure}[]
    \centering
    \includegraphics[width=0.4\textwidth]{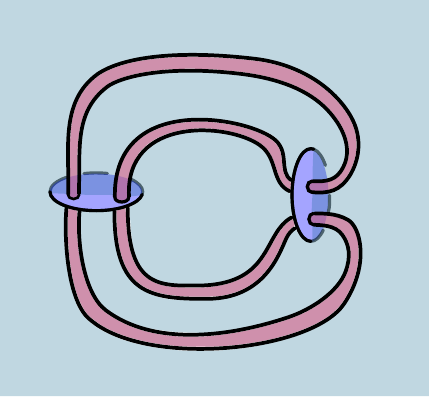}
    \caption{The $2$-complex $\Sigma$. We use different colours to show the different surfaces defined in the construction of $\Sigma$.}
    \label{fig: Sigma}
\end{figure}

\begin{lem}\label{lemma: exterior of Sigma}
The complement of a regular neighbourhood of $\Sigma$ in $M$ is homeomorphic to the disjoint union of two balls, $X_1$ and $X_2$, and products $T_1\times [0,1],\dots, T_n\times [0,1]$. \qed
\end{lem}
We now assign coorientations to the surfaces composing $\Sigma$, which we will use to determine a smoothing of $\Sigma$. We denote by $\delta_1 M,\dots, \delta_m M$ the boundary components of $M$ corresponding to the crossing circles $J_1,\dots J_m$.
Observe that $\delta_1 M,\dots, \delta_m M$ intersect $\Sigma$ transversely, while the remaining boundary components are disjoint from $\Sigma$. We want the branched surface we obtain from $\Sigma$ to have the following two important properties: the boundary train track induced on $\delta_h M$ must realise the slope $\frac{1}{k_h}$ for $h=1,\dots, m$, and each torus $T_i$ must have two \emph{cusps}, for $i=1,\dots, n$.

\begin{defn}\label{def: cusps}
Let $T$ be one of the boundary parallel tori contained in $\Sigma$. A \emph{cusp} on $T \subset \Sigma$ is a local smoothing of $\Sigma$ in a neighbourhood of a simple closed curve $c \subset T$ such that the cusp direction along $c$ points into $\Sigma \setminus T$.
\end{defn}

\begin{figure}[]
\centering
\includegraphics[width=0.5\textwidth]{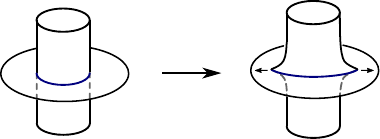}
\caption{A cusp on $T$.}
\label{fig: cusp}
\end{figure}

Figure \ref{fig: cusp} shows an example of a cusp. Note that once a coorientation of the crossing disc $D$ intersecting $T$ along $c$ is chosen, there is a unique choice of coorientation of $T\setminus c$ in a neighbourhood of $c$ such that smoothing according to these coorientations produces a cusp.
\newline

Before starting to assign coorientations, we make explicit a strong relation between the combinatorial properties of the diagram of the link \(\mathcal{L}\) and the graphs associated to the diagram $D$ of $L$ at the beginning of Section \ref{section:foliations}.
From now on we suppose that the diagram $D$ of $L$ satisfies the hypotheses of Theorem \ref{thm: pers fol link}. In particular, all twist regions contain at least two crossings, so that the graphs \(\Gamma\), \(\mathcal{G}_r\), \(\mathcal{G}_g\) and \(\mathcal{G}\) can be defined.
By construction there is a canonical bijection between \(\mathcal{S}\) and the set of connected components of \(S^2\setminus \Gamma\), and hence to the vertices of \(\mathcal{G}\). More precisely, the following properties hold:
\begin{itemize}[leftmargin=*]
    \item each element of \(\mathcal{S}\)  corresponds bijectively to a vertex of \(\mathcal{G}\);
    \item every crossing disc corresponds bijectively to an edge of \(\mathcal{G}\);
    \item an element in \(\mathcal{S}\) intersects the boundary of a crossing disc $D$ if and only if the corresponding vertex in \(\mathcal{G}\) is adjacent to the edge corresponding to $D$.
\end{itemize}
By using this dictionary, we have that a component $K_i$ of $L$ passes through an edge of \(\mathcal{G}\) if and only if the corresponding component $K_i'$ of \(\mathcal{L}\) intersects the crossing disc associated to that edge. Moreover, we can define the weight $w(D)$ of a crossing disc $D$ as the weight of the corresponding twist region (and edge in $\mathcal{G}$). The following lemma is simply a restatement in this setting of some of the hypotheses of Theorem \ref{thm: pers fol link}.
\begin{lem}\label{lem: hyp pers fol restated}
In the notations above, assume that $L$ satisfies the hypotheses of Theorem~\ref{thm: pers fol link}. Then there exists an element \(S\in \mathcal{S}\) and pairwise distinct crossing discs 
\[
D_i^j \text{ for $i=1,\dots, n$ and $j=1,2$,}  
\]
satisfying \(D_i^j\cap \partial \nu K_i'\ne \emptyset\) and $S\cap \partial D_i^j=\emptyset$ for each $i,j$, such that for each \(i\) one of the following holds:
\begin{itemize}
\item \(\max\{|w(D_i^1)|, |w(D_i^2)|\} \geq 3\);
\item every crossing disc intersecting $\partial\nu K'_i$ has even weight and \(w(D_i^1) = -w(D_i^2)\).\qed
\end{itemize}
\end{lem}

\subsubsection{Coorientations of the elements in $\mathcal{S}$.}\label{subsec: coorientations of elements in S}
We fix a colouring of the connected components of $S^2\setminus \Gamma$, and use it, together with the bijection mentioned in the previous section, to assign coorientations to the elements in $\mathcal{S}$. More precisely, recall that $X_1$ and $X_2$ denote the two balls in the exterior of $\Sigma$. For consistency, we adopt the convention that our figures are drawn from the viewpoint of a reader located inside the ball $X_1$. We impose the coorientation to point into $X_1$ (resp. into $X_2$) for sectors corresponding to green (resp. red) complementary regions of $\Gamma$. 
We will also adhere to the convention that the colour green (resp. red) indicates the positive (resp. negative) side of a sector of $\Sigma$. See Figure~\ref{fig: coorientations in S^2} for examples of this coorientation near a crossing disc, both when the disc is adjacent to a crossing and when it is not.

\begin{figure}[]
    \centering
    \includegraphics[width=0.5\textwidth]{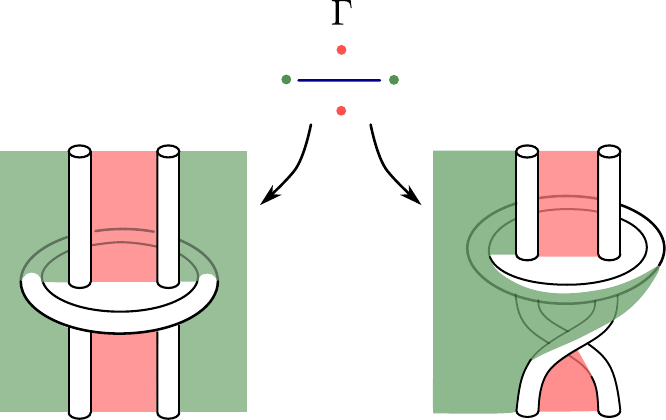}
    \caption{Coorientation of the elements in $\mathcal{S}$ induced by a colouring of the complementary regions of $\Gamma$. To simplify the figure, we have not shown the fourth element in $\mathcal{S}$ in the 2-complex on the right.}
    \label{fig: coorientations in S^2}
\end{figure}

\subsubsection{Coorientations of the crossing discs and the boundary parallel tori.}

Let us fix, for each $i=1,\dots, n$, the crossing discs $D^{1,2}_i$ given by Lemma~\ref{lem: hyp pers fol restated}. Moreover, for each $i$, we fix circles of intersections
$$
c^1_i\subset T_i\cap D_i^1 \text{ and } c^2_i\subset T_i\cap D_i^2.
$$

\begin{defn}
    We say that $D_i^j$ is \emph{$c^j_i$-cooriented} if either \(|w(D^j_i)|\geq 3\) or $w(D^j_i)=2\varepsilon$, where $\epsilon=\pm 1$, and $D^j_i$ is cooriented with the following rule. There are two arcs of intersection between $D^j_i$ and the regions in $\mathcal{S}$ pointing inside $X_1$, and only one of this two arcs intersects $c_i^j$. Denote it by $\gamma$, and orient it so that its starting point $\gamma(0)$ lies in $c^j_i$. Take any vector $V$ tangent to the projection sphere $S^2$ at $\gamma(0)$, oriented so that the cross product $\dot{\gamma}(0)\times V$ points into $X_1$. If $\varepsilon=1$, we coorient $D^j_i$ so that $V$ is positively transverse to $D^j_i$, and negative otherwise.
\end{defn}
Figure \ref{fig: coorientations disc} shows an example in the case the twist region has weight $2$. 

\begin{figure}[]
    \centering
    \includegraphics[width=0.25\textwidth]{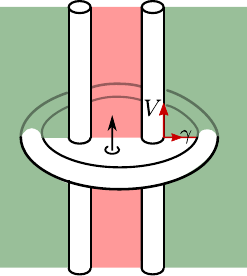}
    \caption{The figure shows the arc $\gamma$ and the vector $V$ used to coorient the crossing disc. In this case, we are assuming that the twist region associated with the disc has weight $2$. If the weight were $-2$, the disc would be cooriented in the opposite way.}
    \label{fig: coorientations disc}
\end{figure}
In the next lemma, we prescribe coorientations for the boundary parallel tori and the discs $D_i^j$.

\begin{lem}\label{lemma: existence of cusps} Assume that \(L\) satisfies the hypotheses of Theorem~\ref{thm: pers fol link}. Then the subcomplex of \(\Sigma\) consisting of the boundary-parallel tori \(T_i\) and the discs \(D_i^j\) can be smoothed into a cooriented branched surface with two cusps on each \(T_i\), for \(i=1,\dots,n\). Moreover, the discs \(D_i^j\) are \(c_i^j\)-cooriented with respect to the induced coorientation.
\end{lem}

\begin{proof}
First, we coorient the discs with weights equal to \(\pm 2\), so that they are \(c_i^j\)-cooriented. Our aim is to create the cusps along the curves $c_i^j$. Each torus $T_i$ is cut by \(c^{1,2}_i\) in two annuli, and in order to create the cusps, we coorient these two annuli in opposite ways. If $D^1_i, D_i^2$ satisfy the first case of the statement Lemma~\ref{lem: hyp pers fol restated}, we assume without loss of generality that $|w(D_i^2)|\geq 3$. If $D^1_i$ is not already cooriented, we coorient it arbitrarily. The coorientation of $D^1_i$ and the constraint of creating a cusp on $c^1_i$ when smoothing force the coorientation on the annuli of $T_i\setminus (c^1_i\cup c^2_i)$. More precisely, the annular component of $T_i\setminus (c^1_i\cup c^2_i)$ intersecting $D_i^1$ from the positive (resp. negative) side must have coorientation pointing out of (resp. into) the product region $T_i\times [0,1]$ in the exterior of $\Sigma$. Next, these coorientations on the annuli and the constraint of creating a cusp along $c^2_i$ impose a unique coorientation on $D_i^2$.

Suppose now that $w(D_i^1) = -w(D_i^2)$, both having absolute value $2$, and that all the crossing discs intersecting $T_i$ are have even weights. As before, the coorientation of $D^1_i$ and the constraint of creating a cusp on $c^1_i$ force the coorientation on the annuli of $T_i\setminus (c^1_i\cup c^2_i)$, and the same happens with $D^2_i$ and $c^2_i$. We claim that these two coorientations coincide, and so that there is a coorientation on $T_i\setminus (c^1_i\cup c^2_i)$ that creates cusps on both discs when smoothing. Without loss of generality, suppose $w(D_i^1)=2$, and let $A$ be the annulus component of $T_i\setminus (c^1_i\cup c^2_i)$ intersecting $D^1_i$ from the positive side. Consider the arc $\delta$ on $A$, going from $c^1_i$ to $c^2_i$, given by the intersection between $A$ and the regions in $\mathcal{S}$ pointing into $X_1$. Parametrise $\delta$ with $[0,1]$, and notice that $\dot{\delta}(0)$ is positive with respect to the coorientation of $D_i^1$, being it $c_i^1$-cooriented.  Fix a metric on $M$, and consider the vector field $x(t)$ along $\delta$ defined by the unit vectors tangent to the regions in $\mathcal{S}$ intersecting $A$. This situation is illustrated in Figure~\ref{fig: compatible cusps}. Since all discs intersecting $T_i$ corresponds to twist regions with an even number of crossings it follows that:

\begin{itemize}

\item the arc $\delta$ lies on the projection sphere;
\item the cross product $x(t)\times \dot{\delta}(t)$ points into $X_1$, for all $t\in [0,1]$. 
\end{itemize}

\begin{figure}[]
    \centering
    \includegraphics[width=0.35\textwidth]{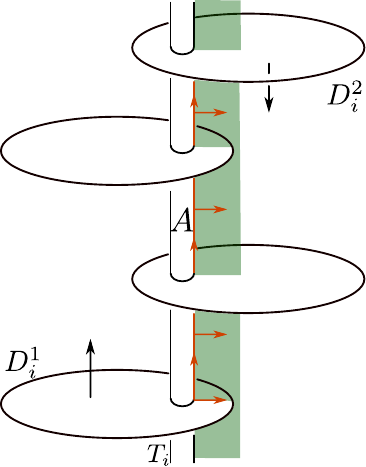}
    \caption{The annulus $A$ intersects both discs $D^1_i$ and $D_i^2$ from the positive side. The arc $\delta$ and the vector field $x(t)$ along it are coloured in light red. In this example, $D_i^1$ has weight $2$ and $D_i^2$ has weight $-2$.}
    \label{fig: compatible cusps}
\end{figure}
In particular,  $x(1)\times \dot{\delta}(1)$, points into $X_1$. Since $w(D_i^2)=-2$, and the disc $D^2_i$ is $c_i^2$-cooriented, the vector $\dot{\delta}(1)$ is negative with respect to the coorientation on $D^2_i$. In other words, the annulus $A$ intersects $D_2$ from the positive side, which is what we wanted.  
\end{proof}

\begin{remark}
The reader can verify that the proof of the previous lemma extends readily to show that the same conclusion holds in the more general setting obtained by replacing the second condition in Lemma~\ref{lem: hyp pers fol restated} with the following:
\begin{itemize}
    \item $|w(D_i^1)| = |w(D_i^2)| = 2$, and 
    $w(D_i^1) = (-1)^{r-1} \, w(D_i^2)$,
    where $r$ denotes the number of discs with odd weight intersected by $T_i$ between the circles $c_i^1$ and $c_i^2$.
\end{itemize}
\end{remark}

From Lemma \ref{lemma: existence of cusps} we know how to obtain the desired two cusps on each boundary parallel torus. The goal now is to show that we can coorient the remaining crossing discs and finally smooth $\Sigma$ so that the boundary train track of the resulting branched surface on $\delta_h M$ realises the slope $\frac{1}{k_h}$, for $h=1,\dots, m$. 

More precisely, we are interested in smoothing a subcomplex of $\Sigma$. Let $S$ be the element of $\mathcal{S}$ given by Lemma~\ref{lem: hyp pers fol restated}, and denote by $\mathcal{D}_S$ the set of crossing discs whose boundary intersects $S$. Notice that, by hypothesis, none of the $D_i^j$ belongs to $\mathcal{D}_S$.
Also notice, that since $S$ is cooriented (and oriented, using the ambient orientation), its intersection with the boundary parallel tori $T_i$,  that we denote by $\sigma$, is naturally oriented.
\begin{defn}
Let \(D\) be a crossing disc in \(\mathcal{D}_S\), fix \(p\in D\cap S\), and let \(W\) be any vector positively tangent to \(\sigma\) at \(p\). We say that \(D\) is \emph{\(S\)-cooriented} if \(W\) is negatively transverse (resp. positively transverse) when \(w(D)>0\) (resp. \(w(D)<0\)).
\end{defn}
For example, the disc in Figure \ref{fig: coorientations disc} is $S$-cooriented if $S$ is the region on the right of the picture.

\begin{lem}\label{lemma: train tracks of $B$}
In the above notations and assumptions, there exists a coorientation of $\Sigma\setminus S$, and a smoothing of it into a cooriented branched surface $B$ such that:
\begin{itemize}
\item each boundary parallel torus $T_i$ has two cusps, for $i=1,\dots n$;
\item the boundary train track $\tau_h=B\cap \delta_h M$ realises the slope $\frac{1}{k_h}$, for $h=1,\dots m$;
\item $B$ has no (half) sink discs. 
\end{itemize}
\end{lem}

\begin{proof}
Notice that, once the coorientations for the surfaces composing $\Sigma\setminus S$ have been fixed, there is a natural way to smooth it into a cooriented branched surface, except for the neighbourhoods of the crossing discs. In fact, each disc intersects the elements in $\mathcal{S}$ along three arcs that do not respect the local model for branched surfaces. This problem is solved by explicitly describing a way to slide the surfaces in $\mathcal{S}$ along the discs in a neighbourhood of each intersection arc, thereby creating two arcs of double points.

In Section~\ref{subsec: coorientations of elements in S} we have already explained how to coorient the regions in $\mathcal{S}\setminus S$. The crossing discs $D_i^j$ and the boundary parallel tori are cooriented as prescribed by Lemma~\ref{lemma: existence of cusps}. We are left with the coorientations of the remaining crossing discs. If \(D\in \mathcal{D}_S\) we coorient it so that it is $S$-cooriented. If $\partial D\cap S= \emptyset$, then we coorient it arbitrarily. We now have to choose a smoothing of $\Sigma \setminus S$ near the crossing discs. This is done by following closely the arguments explained in \cite[Propositions~4.10, 4.12]{San24diag}, to which we refer for details. Roughly speaking, we have to show that we can choose a smoothing whose boundary train tracks realise the desired slopes, without creating half sink discs. It is not difficult to show that -- and here it is important that the discs in \(\mathcal{D}_S\) are \(S\)-cooriented -- it is always possible to desingularise the complex \(\Sigma \setminus S\) near each crossing discs to obtain train tracks realising the desired slopes. If the crossing disc does not contain a cusp (\ie is not one of the \(D^j_i\)) no half sink discs are created during this procedure. Otherwise, there are smoothings that produce the correct train tracks while containing sink discs; but, since the discs are \(c_i^j\)-cooriented, one can show that there is always at least one choice of smoothing that satisfies both requests.
\end{proof}
In the next section we will need a slight variation of the previous lemma. In fact, in the same way, using the arguments of \cite[Propositions~4.10, 4.12]{San24diag}, one proves the following.

 \begin{add}[Addendum to Lemma \ref{lemma: train tracks of $B$}]\label{addendum}
The conclusion of Lemma \ref{lemma: train tracks of $B$} holds also if there exists a  coorientation of the subcomplex of $\Sigma$ comprising the boundary parallel tori and the discs $D_i^j$ as in Lemma~\ref{lemma: existence of cusps} so that one (or more) of the following  conditions occur:
\begin{itemize}
\item The discs \(D_i^j\) are not necessarily pairwise distinct, and whenever \(D_i^j=D_{i'}^{j'}\) for some \((i,j)\ne (i',j')\), the weight \(w(D_i^j)\) is odd.
\item The region \(S\) intersects the boundary of some of the discs \(D_i^j\), and whenever this happens, we have that $D_i^j$ is $S$-cooriented with the induced coorientation\footnote{Notice that if \(|w(D_i^j)| = 2\) this is equivalent to asking that \(S\) intersects \(c_i^j\).}.
\end{itemize}
\end{add}
Examples where this addendum is useful appear in Section \ref{sec: pers fol two-bridge}; see Figures \ref{figure: m=3 case 1}, \ref{figure: m=3 case 2}, \ref{figure: m=3 case 3} and \ref{figure: fibered case}.

The following properties of the branched surface $B$ are proven by applying verbatim the arguments in \cite[Proposition~4.12, 4.14]{San24diag}, to which we refer the reader. We only point out that here is where the hypothesis that the graphs $\mathcal{G}_r$ and $\mathcal{G}_g$ are connected is used, since it plays a crucial role in proving that $X$ is a product sutured ball. 
\begin{prop}\label{prop: properties of B}
The branched surface $B$ is laminar, and the complement $M\setminus {\rm int}(N_{B})$ decomposes as a sutured manifold $X\sqcup Y_1 \sqcup \cdots \sqcup Y_n $, where:

\begin{itemize}
\item $X$ is a product sutured ball. 
\item $(Y_i,\partial_v Y_i)=(\nu \partial_i M, \mathcal{A}^1_i\cup \mathcal{A}^2_i)$, where $\mathcal{A}^1_i\cup \mathcal{A}^2_i\subset  \partial_i M \times \{1\}$ are two annuli whose slope is a meridian of $K'_i$. \qed
\end{itemize}
\end{prop}

\begin{thm}\label{thm: general pers thm}
Let $(s_1,\dots, s_n)$ be any non-trivial multislope on $\partial_1 M\sqcup \cdots \sqcup \partial_n M$ and let $(r_1, \dots, r_m)$ be a rational multislope on $\delta_1 M\sqcup \cdots \sqcup \delta_m M$. Suppose that the multislope $(r_1,\dots, r_m)$ is realised by the boundary train track $\partial B$, and that $r_h$ is not the longitude defined by the disc $D_h$, for $h=1,\dots, m$. 
Then, there exists a coorientable taut foliation in $M$ intersecting the boundary transversely in a linear foliation of multislope $(s_1,\dots, s_n,r_1,\dots, r_m)$. 
\end{thm}

\begin{proof}
The proof follows closely that of \cite[Theorem~4.16]{San24diag}, and we just give a sketch. We start by using Theorem~\ref{boundary train tracks}; more precisely, as in the proof \cite[Theorem~2.5]{Li_boundary}, we modify the branched surface $B$ in a neighbourhood of $\delta_1 M\sqcup \cdots \sqcup \delta_m M$ by splitting so that it intersects these boundary components in parallel closed curves of slope $r_1,\dots, r_m$ respectively. This implies that $B$ can be extended to a branched surface $B(r_1,\dots,r_m)$ in the manifold $M'$ obtained by filling the boundary components $\delta_h M$ along the slopes $r_h$, for $h=1,\dots, m$. This branched surface is defined by gluing to the boundary of $B$ the meridian discs of the glued solid tori. This branched surface is taut by arguing as in the first part of the proof of \cite[Theorem~4.16]{San24diag}. Moreover, since $B$ is laminar, it follows from the proof of \cite[Theorem~2.5]{Li_boundary} that $B(r_1,\dots,r_m)$ is laminar, and hence fully carries a lamination $\mathcal{L}$ in $M'$ by \cite[Theorem~1]{Li_laminar}. 

The manifold $M'$ has $n$ boundary components, naturally identified with $\partial_1 M,\dots, \partial_n M$, and 
$$
M'\setminus {\rm int}(N_{B'})=X'\sqcup Y_1 \sqcup \cdots \sqcup Y_n, 
$$ 
where the $Y_i$ are introduced in the statement of Proposition \ref{prop: properties of B}. Proceeding as in the proof of \cite[Theorem~4.16]{San24diag}, we can modify the lamination $\mathcal{L}$ so that, for each $i=1,\dots,n$, the annuli $\mathcal{A}^1_i, \mathcal{A}^2_i$ satisfy the noncompact extension property relative to ($B(r_1, \dots, r_m), \mathcal{L}$). We conclude by Theorem~\ref{cusps implies persistently foliar} that for every non-trivial multislope $(s_1,\dots, s_n)$ on $\partial_1 M\sqcup \cdots \sqcup \partial_n M$, the lamination $\mathcal{L}$ can be extended to a coorientable taut foliation in $M'$ intersecting the boundary transversely in a foliation of such multislope. 
\end{proof}

We are finally ready to prove Theorem \ref{thm: pers fol link}.
\begin{proof}[Proof of Theorem \ref{thm: pers fol link}]
Let $(s'_1,\dots, s'_n)$ be any non-trivial multislope on the link $L$. Consider the non-zero integers $k_h$, for $h=1,\dots, m$, such that filling the boundary components $\delta_1 M,\dots, \delta_m M$ along the slopes $\frac{1}{k_1},\dots, \frac{1}{k_m}$ yields the exterior of the link $L$. Via this identification, the fixed multislope on $L$ corresponds to a non-trivial multislope $ (s_1,\dots, s_n)$ on $\partial_1 M \sqcup \cdots \sqcup \partial_n M$. By Lemma \ref{lemma: train tracks of $B$}, the boundary train track $\partial B$ realises the multislope $(\frac{1}{k_1},\dots, \frac{1}{k_m})$, and hence we conclude by using Theorem~\ref{thm: general pers thm}.
\end{proof}

\section{Persistently foliar two-bridge links}\label{sec: pers fol two-bridge}
The main goal of this section is to apply Theorem \ref{thm: pers fol link} to the study of two-bridge links, and prove the following result. Recall the links $L_{n,k}$, $L'_{n,k}$ defined in Figure~\ref{fig: L and L'}. 

\begin{thm}\label{thm: persistently fol two-br}
Let $L$ be a two-bridge link, and suppose that $L$ is not isotopic to any the links \(L_{n,k}, L'_{n,k}\), or their mirrors, for \(n,k\geq 0\). Then $L$ is persistently foliar. In particular, every non-trivial surgery on $L$ supports a coorientable taut foliation.
\end{thm}

To prove Theorem \ref{thm: persistently fol two-br}, we first consider the non-fibered two-bridge links satisfying the hypotheses of the statement, and then the fibered ones. We remark that in \cite{San23twobridge} it is shown that all non-trivial surgeries on a fibered two-bridge link satisfying the hypotheses of Theorem \ref{thm: persistently fol two-br} support coorientable taut foliations. In this case, Theorem~\ref{thm: persistently fol two-br}, besides providing a different proof, covers also the case of non-rational multislopes.
\newline

In Section \ref{subsec: fol on Lnk and L'nk}, that is the last part of this section, we construct taut foliations on some surgeries on the links \(L_{n,k}\) and \(L'_{n,k}\). The results proved there will be used both in the classification of their \(L\)-space surgeries in Section \ref{sec: classification} and in the applications to satellite knots in Section \ref{sec: satellite}. Although the proofs in Section \ref{subsec: fol on Lnk and L'nk} also use branched surfaces, they rely on a different construction -- natural for fibered links -- from the one presented in the previous section. Such a construction has been used quite often in the literature; for this reason, we will not give full details, but provide references.
\newline

We fix some notations. Consider a two-bridge link $L=L(a_1,\dots, a_m)$, denote its component by $K_1$ and $K_2$, and let $D$ be the diagram associated to it as in Figure \ref{fig: two bridge ori}. As in the proof of Lemma \ref{lemma: nonfib two-bridge links std diagrams}, we can always suppose that $|a_i|\geq 2$ for each $i=1,\dots, m$. We can hence consider the graphs $\mathcal{G}_r$, $\mathcal{G}_g$ and  $\mathcal{G}=\mathcal{G}_r\cup \mathcal{G}_g$ associated to $D$. Notice that, by definition,  the diagram $D$ has a twist region containing $|a_j|$ crossings, for each $j=1,\dots, m$. We denote this twist region and -- with slight abuse of notation -- the corresponding edge in $\mathcal{G}$ by $R_j$. 

\begin{lem}\label{lemma: graph connect + existence of S}
The graphs $\mathcal{G}_g$ and $\mathcal{G}_r$ are connected. Moreover, if $m\geq 5$, for any choice 
$$
R^1_i, R^2_i\in \mathcal{R}_i \text{ for $i=1,2$ }
$$
of edges of $\mathcal{G}$, there exists a vertex $S$ of $\mathcal{G}$ not adjacent to any of them.     
\end{lem}

\begin{proof}
That the graphs are connected is immediate, see Figure \ref{fig: example theorem} for an example. More precisely, each of the two graphs is a star-shaped tree whose branches have length one\footnote{That is to say, it is a tree with one vertex of valency $v$ and all the others of valency one.}. For this reason, each pair of edges in the same connected component of $\mathcal{G}$ shares a vertex, and hence four edges of $\mathcal{G}$ are adjacent to at most  $6$ vertices. The graph $\mathcal{G}$ has $m+2$ vertices in total, and under the hypothesis $m\geq 5$, we can find the desired vertex $S$.
\end{proof}

\subsection{When \texorpdfstring{$L$}{L} is non-fibered}  Let $L$ be a two-bridge link, and assume that it is not isotopic to \(L_{n,k}, L'_{n,k}\) or their mirrors. By Remark \ref{rem: L_n,k other description}, this is equivalent to saying that $L$ is not isotopic to any link \(L(2,\dots, 2,2z)\), with \(z\in\mathbb{Z}\). If in addition $L$ is non-fibered, then we can write $L=L(a_1,\dots, a_m)$ as described in the statement of Lemma \ref{lemma: nonfib two-bridge links std diagrams}. Since $a_j$ is even for each odd $j$, one of the two components of $L$, say $K_1$, has the property that $K_1$ passes through a twist region $R_j$ if and only if $j$ is odd. In particular, every edge in \(\mathcal{R}_1\) has even weight. We also point out that the component $K_2$ passes through \emph{all} twist regions of $D$.

\begin{prop}\label{prop: nonfibered  m geq 5}
Let $L$ be a non-fibered two-bridge link, and assume that it is not isotopic to $L_{n,k}, L'_{n,k}$ or their mirrors. Write $L=L(a_1,\dots, a_m)$ as in Lemma \ref{lemma: nonfib two-bridge links std diagrams}, and suppose that $|m|\geq 5$. Then $L$ is persistently foliar.
\end{prop}
\begin{proof}
The proof consists in showing that the diagram $D$ coming from the description of $L$ in the statement satisfies the hypotheses of Theorem \ref{thm: pers fol link}. The graphs $\mathcal{G}_g$ and $\mathcal{G}_r$ are connected by Lemma \ref{lemma: graph connect + existence of S}. We now prove the existence of four distinct edges 
\[R_i^j\in \mathcal{R}_i \text{ for $i=1,2$ and $j=1,2$}\]
satisfying the hypotheses of Theorem~\ref{thm: pers fol link}. By Lemma \ref{lemma: nonfib two-bridge links std diagrams}, we have four possibilities:
\begin{itemize}[leftmargin=*]
\item [$1 a$)] In this case, there exist an even index \(i\) such that \(|a_i|>2\), and two odd indices \(i_1,i_2\) such that \(a_{i_1}\) and \(a_{i_2}\) have opposite signs. We consider \(R_{i_1},R_{i_2}\in \mathcal{R}_1\) and \(R_i\in \mathcal{R}_2\) the corresponding twist regions, and fix \(R'\) any edge of \(\mathcal{R}_2\) different from the three already chosen. Then \(R_i\) and \(R'\) satisfy the first condition of Theorem~\ref{thm: pers fol link}, since
\[
|w(R_i)|=|a_i|>2,
\]
and \(R_{i_1}, R_{i_2}\) satisfy either the first or the second condition. If \(|a_{i_1}|=|a_{i_2}|\), we use that all edges in \(\mathcal{R}_1\) have even weights.
\newline

\item[$1b$)] In this case there are indices $i$ even and $i'$ odd such that $|a_i|>2$ and $|a_{i'}|>2$. We consider 
$$R_{i'}, R'\in \mathcal{R}_1 \text{ and } R_i,R\in \mathcal{R}_2,
$$ where all four edges are distinct. Both pairs satisfy the first condition of Theorem \ref{thm: pers fol link}.
\newline    
    
\item[$2a$)] We have an odd index \(i\) such that \(|a_i|>2\), and indices \(i_1\neq i\) and \(i_2\neq i\) such that \(a_{i_1}\) and \(a_{i_2}\) have opposite signs. If \(i_1\) and \(i_2\) are both odd, then the corresponding edges of \(\mathcal{G}\) both lie in \(\mathcal{R}_1\), and we choose \(R_i,R\in \mathcal{R}_2\), with \(R\) distinct from the previously chosen edges. If \(i_1\) and \(i_2\) are both even, and we are not in case $1b$), then all the \(a_j\) are even. In this case, we take
\[
R_i,R\in \mathcal{R}_1 \quad \text{and} \quad R_{i_1},R_{i_2}\in \mathcal{R}_2,
\]
with all edges distinct.
Finally, if one index is odd, say \(i_1\), and the other is even, we first look at the even-indexed coefficients \(a_j\). If at least one of them satisfies \(|a_j|>2\), then we are in case $1b$), and we conclude. Otherwise, they all have absolute value \(2\), and we proceed as above, choosing
\[
R_i,R\in \mathcal{R}_1 \quad \text{and} \quad R_{i_1},R_{i_2}\in \mathcal{R}_2.
\]
In either case, the chosen pairs satisfy the hypotheses of Theorem~\ref{thm: pers fol link}.
\newline

\item[$2b$)] There exist distinct odd indices \(i \neq i'\) such that \(|a_i|>2\) and \(|a_{i'}|>2\). In this case, we consider
\[
R_i,R\in \mathcal{R}_1 \quad \text{and} \quad R_{i'},R'\in \mathcal{R}_2,
\]
chosen so that all four edges are distinct, and we conclude, since both pairs satisfy the first condition in Theorem~\ref{thm: pers fol link}.
\end{itemize}
By Lemma \ref{lemma: graph connect + existence of S} we deduce the existence of a vertex $S$ not adjacent to any of the four edges chosen as above. Therefore, we can apply Theorem \ref{thm: pers fol link} and conclude the proof.
\end{proof}

With the next proposition, we cover the case $m=3$.
\begin{prop}\label{prop: nonfibered $m=3$}
Suppose that $L=L(2a,2b,2c)$ is not fibered and not isotopic to $L_{n,k}, L'_{n,k}$ or their mirrors. Then $L$ is persistently foliar.
\end{prop}
\begin{proof}
By Lemma \ref{lemma: std diagram for m=3} and Proposition \ref{prop: nonfibered  m geq 5}, we have to consider three cases.

If $L=L(-2,2\beta-1,2)$, with $|2\beta|>2$, then, up to mirror, we can assume $2\beta-1\geq 3$. In this case, the diagram $D$ of $L$ associated to this description has only three twist regions, so we cannot apply Theorem \ref{thm: pers fol link}. However, if we consider the link $\mathcal{L}$ obtained from $D$ as explained at the beginning of Section \ref{section:foliations}, we obtain the link in Figure~\ref{figure: m=3 case 1}, and as explained in Addendum~\ref{addendum}, the conclusion of Lemma~\ref{lemma: train tracks of $B$} holds also in this case. Hence, we can conclude that $L$ is persistently foliar.

\begin{figure}[]
   \centering   \includegraphics[width=0.48\textwidth]{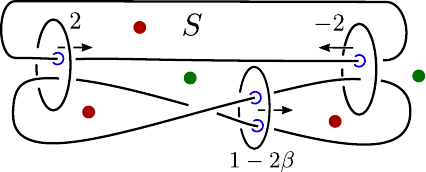}
    \caption{The picture shows the link $\mathcal{L}$ and the weights of the crossing circles. We also describe how to assign coorientations to the surfaces of $\mathcal{S}$, using green/red dots, and the crossing discs. We use blue circles to indicate the curves of intersection $c_i^j$ where we want to create the cusps of the branched surface. Notice that by Addendum \ref{addendum} we can create two cusps on the discs with odd weight, and we can remove the region $S$, although it intersects two of the crossing discs containing cusps.}
    \label{figure: m=3 case 1}
\end{figure}

Assume that $L$ is isotopic, up to mirror, to $L=L(3,2,\dots, 2,3)$, where the sequence has even length. In this case we denote, as in the proof of Proposition \ref{prop: nonfibered  m geq 5}, by $R_1,R_2,\dots$ the edges of the graph $\mathcal{G}$ associated to the diagram of $L$ coming to this description. Then, the edges
$$
R_1,R_2\in \mathcal{R}_1 \text{ and } R_3,R_4\in\mathcal{R}_2
$$
satisfy the hypotheses of Theorem \ref{thm: pers fol link}. If the sequence $(3,2,\dots,2,3)$ has length at least $6$, then we can deduce, by Lemma \ref{lemma: graph connect + existence of S}, that $L$ is persistently foliar. Otherwise, we have $L=L(3,2,2,3)$ and we conclude by using Addendum \ref{addendum}; see Figure \ref{figure: m=3 case 2}.

\begin{figure}[]
   \centering   \includegraphics[width=0.55\textwidth]{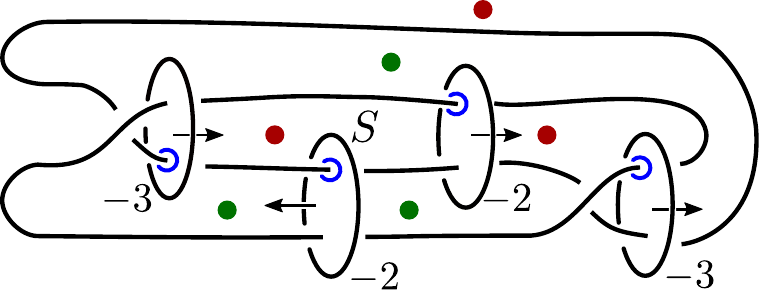}
    \caption{The figure shows the link $\mathcal{L}$ associated to a diagram of $L=L(3,2,2,3)$, with weights of the crossing circles and coorientations. By proceeding as in Lemma \ref{lemma: existence of cusps} we can coorient the boundary parallel tori so to create cusps on the curves $c_i^j$, denoted with blue cirlces. Since the region $S$ satisfies the hypothesis of Addendum \ref{addendum}, we can conclude that $L$ is persistently foliar.}
    \label{figure: m=3 case 2}
\end{figure}

Finally, we have the case when $L$ is isotopic, up to mirror, to $L(2\alpha-1,-2,2\gamma-1)$, with $2a,2\gamma>2$. Also in this case, we can use Addendum \ref{addendum} to deduce that $L$ is persistently foliar; see Figure \ref{figure: m=3 case 3}.
\end{proof}

\begin{figure}[]
   \centering   \includegraphics[width=0.45\textwidth]{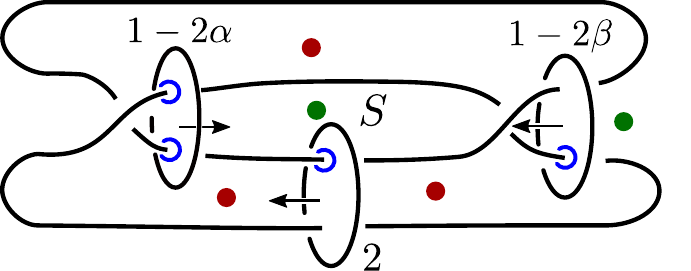}
    \caption{The link $\mathcal{L}$ and the coorientations of the surface in $\mathcal{S}$ and the crossing discs. The coorientations of the discs force, as in the proof of Lemma \ref{lemma: existence of cusps}, local coorientations on the boundary parallel tori in a neighbourhood of the curves $c_i^j$, and the reader can check that these local coorientations extend compatibly. We conclude using Addendum \ref{addendum} to show that $L$ is persistently foliar.}
    \label{figure: m=3 case 3}
\end{figure}

\subsection{When \texorpdfstring{$L$}{L} is fibered} 
We conclude the proof of Theorem~\ref{thm: persistently fol two-br} by studying the case when $L$ is fibered.

\begin{prop}\label{prop: pers fol fibered case}
     Suppose that $L$ is a fibered two-bridge link and assume that $L$ is not isotopic to $L_{n,k}, L'_{n,k}$ or their mirrors. Then $L$ is persistently foliar.
\end{prop}

\begin{proof}
By Lemma~\ref{lemma: fibered case std diagram} we have three cases to consider.
If, up to mirror, $L=L(\alpha+1,4,\beta+1)$, with $\alpha,\beta\geq 2$ even, then we conclude as in Figure~\ref{figure: m=3 case 3}, by changing the coorientations of all the regions in $\mathcal{S}$. 

If $L=L(a_1,\dots, a_m)$, where
$$
(a_1,\dots, a_m)=(\alpha+1,3, a_{\alpha+2}, \dots, a_{m-\beta-1}, \pm 3,\pm(\beta +1))$$
with $\alpha,\beta\geq 2$ even, $m\geq 5$, and all the others indices having absolute value equal to $2$, we consider the edges 
$$
R_1,R_2\in \mathcal{R}_1 \text{ and } R_{m-1}, R_m\in \mathcal{R}_2$$
of $\mathcal{G}$.
These edges satisfy the hypotheses of Theorem~\ref{thm: pers fol link}, and since $m\geq 5$, we can use Lemma~\ref{lemma: graph connect + existence of S} to conclude that $L$ is persistently foliar.

The only case left is when $L=L(\alpha+1,3,2\varepsilon)$ with $\alpha\geq 3$ odd and $\varepsilon=\pm 1$. When $\varepsilon=-1$, the link $\mathcal{L}$ is the one described in Figure \ref{figure: m=3 case 1}, and the same branched surface in its exterior provides the desired foliations. If $\varepsilon=1$, then we conclude by proceeding as described in Figure \ref{figure: fibered case}
\end{proof}

\begin{figure}[]
   \centering   \includegraphics[width=0.48\textwidth]{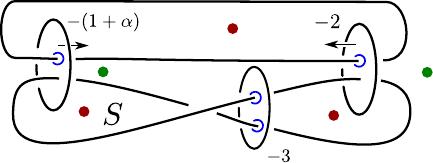}
    \caption{Notice that, since \(|w(D_1)|=\alpha+1>2\) the disc $D_1$ is automatically $c$-cooriented, and also in this case we can use Addendum~\ref{addendum} to conclude that $L$ is persistently foliar.}
    \label{figure: fibered case}
\end{figure}

\begin{proof}[Proof of Theorem \ref{thm: persistently fol two-br}]
When $L$ is not fibered, the result follows from Proposition \ref{prop: nonfibered  m geq 5} and Proposition \ref{prop: nonfibered $m=3$}. When $L$ is fibered, it follows from Proposition \ref{prop: pers fol fibered case}.
\end{proof}

\subsection{Taut foliations and the links 
\texorpdfstring{$L_{n,k}$ and $L'_{n,k}$}{Ln,k and L'n,k}}\label{subsec: fol on Lnk and L'nk}
We conclude the section by construction taut foliations on some of the manifolds obtained by Dehn surgeries along the links $L_{n,k}$ and $L'_{n,k}$ of Figure \ref{fig: L and L'}. We will use these results in Sections~\ref{sec: classification} an \ref{sec: satellite}.

\begin{prop}\label{prop: some ctf Lnk}
Let $n,k\geq 1$, and let $r_1,r_2$ be rationals with \[ \operatorname{min}\{r_1,r_2\}<\operatorname{max}\{n,k\}-\operatorname{min}\{n,k\}+1.
\]
Then the $(r_1,r_2)$-surgery on $L_{n,k}$ supports a coorientable taut foliation. In particular, this happens if $\operatorname{min}\{r_1,r_2\}<1$. 
\end{prop}
\begin{proof}
When \(n=1\) or \(k=1\), the link \(L_{n,k}\) is fibered, and a stronger version of this result is proved in \cite[Proposition~3.20]{San23twobridge}. The same ideas can be adapted to prove the general case. We give a sketch of the argument and refer the reader to \cite{San23twobridge} for details.

From now on, we assume that \(k>1\) and \(n>1\). Moreover, it is clear from the diagram in Figure \ref{fig: L and L'} that \(L_{n,k}\) is isotopic to \(L_{k,n}\); thus, we may assume that \(k\geq n\).
\begin{figure}[]
   \centering   \includegraphics[width=1\textwidth]{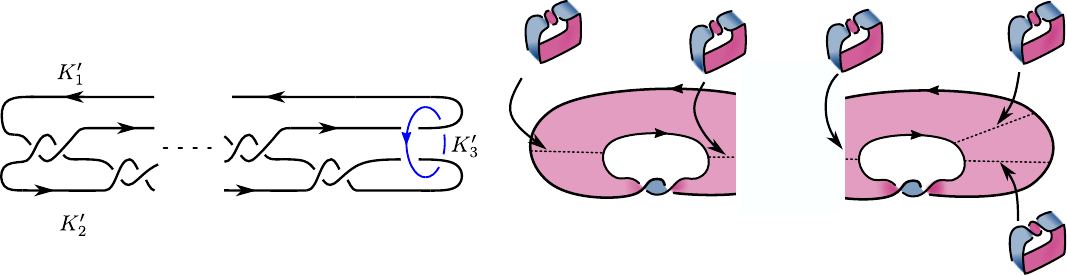}
    \caption{Left: The three-component link $\mathcal{L}_n$. The link $L_{n,k}$ is obtained as $(-\frac{1}{k})$-surgery on the component $K'_3$.
    Right: A fiber surface for the link $\mathcal{L}_n$. The positive (resp. negative) side of the surface is coloured in pink (resp. blue).}
    \label{figure: link Lnk three components}
\end{figure}
By Lemma~\ref{lemma: operation on plumb graphs 1.5}, the link \(L_{n,k}=L(-2n,1,-2k)\) is isotopic to \(L(2,2,\dots,2,-2k)\), where the sequence \((2,2,\dots,2,-2k)\) has length \(2n+1\). From this description, it follows that \(L_{n,k}\) is obtained by surgery on the three-component link \(\mathcal{L}_n\) shown in Figure~\ref{figure: link Lnk three components}.
In the same figure, we also show a Seifert surface \(F\) for \(\mathcal{L}_n\), obtained by starting with the boundary connected sum of \(n\) Hopf bands and then plumbing \(n+2\) additional Hopf bands to this surface. By \cite{Stallings}, \(\mathcal{L}_n\) is a fibered link, and \(F\) is a fiber surface. Moreover, the monodromy can be described explicitly: it is a composition of Dehn twists along the cores of the Hopf bands, see \cite[Corollary~1.4]{GabaiMurasugi} for details. In our case, this implies that the monodromy is given by the diffeomorphism
\begin{equation*}\label{eq: monodromia}
    h=\tau_{c_n}\cdots \tau_{c_1}\tau^{-1}_{d_{n+2}}\tau_{d_{n+1}}\cdots \tau_{d_1}
    \end{equation*}
where $\tau_{c_i}$ and $\tau_{d_i}$ denote the right-handed Dehn twists along the curves $c_i$ and $d_i$ shown in Figure \ref{figure: monodromy}.
\begin{figure}[]
   \centering   \includegraphics[width=0.6\textwidth]{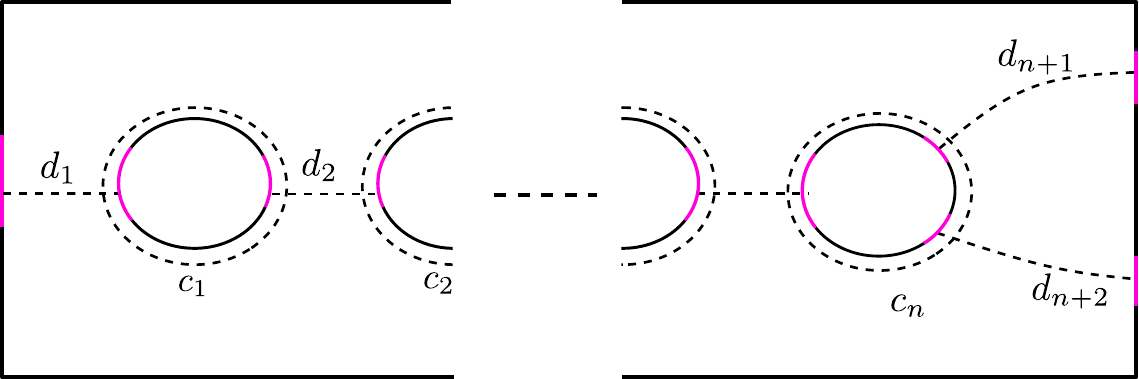}
    \caption{An abstract drawing of the fiber surface $F$, together with the curves $c_i$ and $d_i$. To simplify the picture we do not draw the $1$-handles; we understand that the pink-coloured arcs are pairwise identified in the obvious way.}
    \label{figure: monodromy}
\end{figure}
There is a way to construct a branched surface \(B\) in the exterior of \(\mathcal{L}_n\) by starting with a family of pairwise disjoint, properly embedded, oriented arcs; see, for example, \cite{Roberts1,Roberts2}. Roughly speaking, \(B\) is obtained by smoothing the union of the fiber surface \(F\) and the discs spanned by the arcs under the suspension flow; see \cite[Section~3.2]{San23twobridge} for more details. In particular, the branch locus of \(B\) is contained in \(F\), and is the union of the chosen arcs and a their slight perturbation of their images under the monodromy. In our case, we consider the arcs \(\alpha,\beta,\gamma,\delta\), and the branch locus, with cusp directions as explained in \cite[Section~3.2]{San23twobridge}, is shown in the top part Figure~\ref{figure: arcs nk}.
One proves, arguing as in \cite[Proposition~3.20]{San23twobridge}, that this branched surface implies the existence of coorientable taut foliations on every manifold obtained by surgery on \(\mathcal{L}_n\) with surgery coefficients
\[
(s_1,s_2,s_3)\in (-\infty,2)\times \mathbb{Q}\times (-\infty,0),
\]
where the coefficients are computed with respect to the framing induced by the Seifert surface \(F\), which in general does not agree with the canonical framing. 
More precisely the correspondence between coefficients with respect to the Seifert and canonical framings is given by
\begin{center}
\begin{tikzcd}
{\overbrace{(a, b, c)}^{\text{Seifert framing for $\mathcal{L}_n$}}} \arrow[r] & {\overbrace{(a-l_{12}-l_{13}, b-l_{12}-l_{23}, c-l_{13}-l_{23})}^{\text{Canonical framing for $\mathcal{L}_n$}}}, 
\end{tikzcd}
\end{center}
where $l_{ij}=\operatorname{lk}(K'_i,K'_j)$.
As a consequence, in the canonical framing of \(\mathcal{L}_n\), we obtain the set
\[
(-\infty,-n+1)\times \mathbb{Q}\times (-\infty,0).
\]
Moreover, since \(L_{n,k}\) is obtained by \((-\tfrac{1}{k})\)-surgery along the component \(K'_3\), we conclude that, for each \((r_1,r_2)\in (-\infty,k-n+1)\times \mathbb{Q}\), the corresponding Dehn surgery on \(L_{n,k}\) supports a coorientable taut foliation. Since \(k\geq n\) by hypothesis and two-bridge links are symmetric, the proof is complete.
\end{proof}

The statement of Proposition~\ref{prop: some ctf Lnk} is not optimal, as it does not provide taut foliations on all surgeries along \(L_{n,k}\) that are not
L-spaces.  However, it is enough to obtain interesting applications to satellite knots, see Section~\ref{sec: satellite}. We will improve Proposition~\ref{prop: some ctf Lnk} in Section~\ref{subsec: class for Lnk}, where, by combining it with a recent theorem of Lyu~\cite{Lyu25} and Theorem~\ref{thm: L-space links}, we construct taut foliations on all non-$L$-space surgeries on the links \(L_{n,k}\); see Theorem~\ref{thm: CTF iff Lspace for Lnk}
\begin{figure}[]
   \centering   \includegraphics[width=0.5\textwidth]{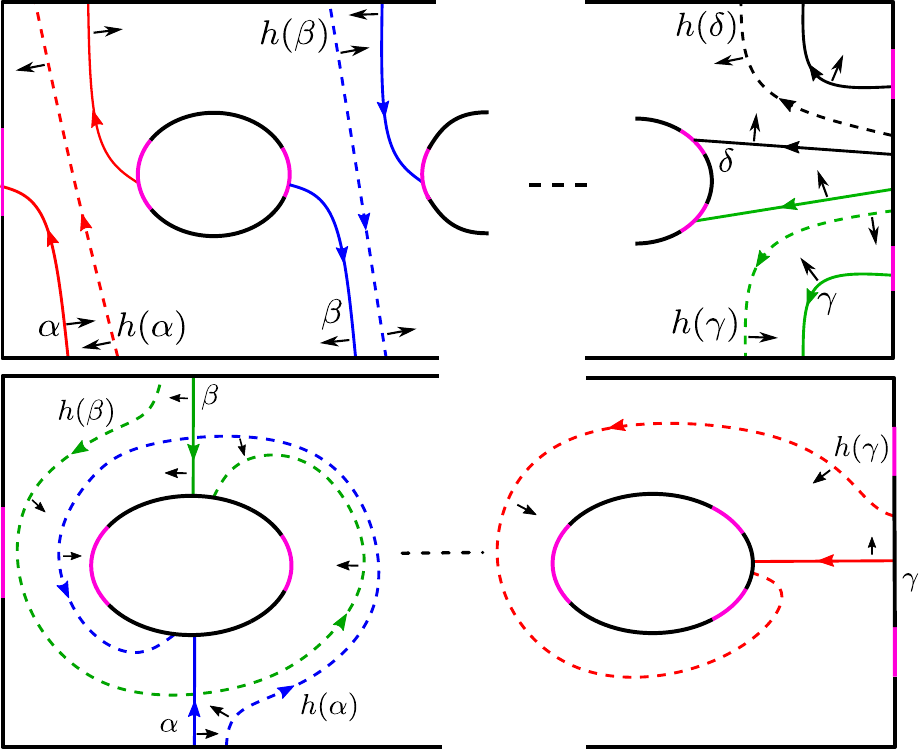}
    \caption{The figure shows the branch loci of the branched surfaces constructed in the proofs of Propositions~\ref{prop: some ctf Lnk} and \ref{prop: some ctf on L'nk}, and the cusp directions along them.}
    \label{figure: arcs nk}
\end{figure}

In the next proposition, we focus on the links \(L'_{n,k}\).

\begin{prop}\label{prop: some ctf on L'nk}
Let \(n\geq 2\) and \(k\geq 1\), and let \(r_1,r_2\in \mathbb{Q}\) satisfy \(\max\{r_1,r_2\}\ll 0\). Then the \((r_1,r_2)\)-surgery on \(L_{n,k}\) supports a coorientable taut foliation. In particular, \(L_{n,k}\) is not the mirror of an \(L\)-space link.
\end{prop}
\begin{proof}
The proof is exactly the same as that of Proposition \ref{prop: some ctf Lnk}. The link \(L'_{n,k}\) is isotopic to \(L(2,2,\dots,2n)\), where the sequence \((2,2,\dots,2n)\) has length \(2k+1\). In particular, \(L'_{n,k}\) is obtained from the link \(\mathcal{L}_k\), introduced in the proof of Proposition \ref{prop: some ctf Lnk}, as \((1/n)\)-surgery on \(K'_3\). The branched surface in the exterior of \(\mathcal{L}_k\) obtained from the arcs \(\alpha,\beta,\gamma\), showed in the bottom part of Figure \ref{figure: arcs nk} can be used to construct taut foliations on all surgeries on \(\mathcal{L}_k\) with surgery coefficients, measured in the Seifert framing,
\[
(s_1,s_2,s_3)\in (-\infty,1)\times (-\infty,1)\times (-\infty,1).
\]
This implies that, for each \((r_1,r_2)\in (-\infty,-n-k)\times (-\infty,2-n-k)\), the corresponding Dehn surgery on \(L'_{n,k}\) supports a coorientable taut foliation.
\end{proof}

\section{Preliminaries on Heegaard Floer homology}\label{section: Heegaard Floer Preliminaries}
In this section, we  briefly review some preliminaries on Heegaard Floer homology needed for Sections \ref{section:rational link surgery} and \ref{section: very good points and non L space surgeries}.
In Section \ref{subsec: Heegaard Floer and link Floer}, we recall the definitions of Heegaard Floer homology and the link Floer complex; then, in Section \ref{subsec: type_A_algebra_and_bimodules}, we review 
the definitions of type-$A$  and type-$D$ modules and bimodules following \cite{LOT, LOT_bimodules} and in Section  \ref{subsec: link_surgery_algebra} and \ref{subsec: solid_tori} we review  the definitions of the link surgery algebra and the type-$A$ module for solid tori from \cite{zemke_bordered}.
% finally, in Section \ref{subsec: homological_perturbation} we recall the homological perturbation lemma for hypercubes proved in \cite{HHSZ_invo_dual_knot}. 
\subsection{Heegaard Floer homology and the link Floer complex} \label{subsec: Heegaard Floer and link Floer}
In \cite{OS3manifold}, Ozsv\'ath-Szab\'o defined a topological invariant for a three manifold $M$, the Heegaard Floer chain complex, which comes in different flavors. The simplest version is the \emph{hat flavor}  \(\widehat{\CF}(M)\), which takes the form of a  chain complex over $\F=\Z/2\Z$. The \emph{minus flavor} $\CF^-(M)$ takes the form of a free chain complex over the ring $\F[U]$, where $U$ is a formal variable. We will also consider the algebraic completion of $\CF^-(M)$ over $\F[U],$ denoted $\bCFm(M)$, which is a chain complex over the power series ring $\F\llbracket U \rrbracket.$  Setting $U=0$ in $\CF^-(M)$ or $\bCFm(M)$ recovers \(\widehat{\CF}(M)\).
Taking the homology yields the Heegaard Floer homologies  \(\widehat{\HF}(M)\), $\HF^-(M)$ and $\bHFm(M)$, which take the form of an $\F$-vector space, and modules over $\F[U]$ and $\F\llbracket U \rrbracket,$ respectively. The Heegaard Floer chain complex and homology admit a homological $\Z$-grading, denoted $gr$, such that $gr(U)=-2$ and $gr(\partial)=-1$. 

The Heegaard Floer chain complex and homology decompose over the set spin$^c$ structures Spin$^c(M)$, for which there is a non-canonical identification Spin$^c(M) \cong H_1(M;\Z)$. We have
\[
\widehat\CF(M) = \bigoplus_{\bt \in \operatorname{Spin}^c(M)} \widehat\CF(M,\bt).
\]
The same decomposition holds for other flavors and for the homology as well.
Ozsv\'ath-Szab\'o proved in \cite[Theorem 10.1]{OS3manifoldProperty} that for a rational homology sphere $M$ it holds 
$\operatorname{dim}\widehat\HF(M,\bt) \geq 1$ for each $\bt \in \operatorname{Spin}^c(M).$ 
\begin{defn}
    A rational homology sphere $M$ is an \emph{$L$-space} if $\operatorname{dim}\widehat\HF(M) = |H_1(M;\Z)|,$ or equivalently,  if $\operatorname{dim}\widehat\HF(M,\bs) = 1$ for each $\bs \in \operatorname{Spin}^c(M).$
\end{defn}
In \cite{OSLinks}, Ozsv\'ath-Szab\'o defined 
 a refinement of the Heegaard Floer chain complex, which is an invariant of a link $L \subset M$.   Later, Zemke provided a reformulation of this invariant.  In the following we recall some of its important properties. For a detailed exposition, the reader is referred to \cite[Section 3.3]{zemkefunctor}. 

Given an $n$-component link $L=K_1\cup \dots \cup K_n$ in a rational homology sphere $M$,  the \emph{link Floer complex} $\cCFL(L)$ is a finitely generated free chain complex over the ring $\F[W_1,Z_1,\dots,W_n,Z_n]$. If we set $U_i=W_iZ_i$, then the homological actions induced by multiplication by $U_i$ coincide for all $i=1,\dots, n$.
% and therefore we impose the extra relation $U=W_iZ_i$. 
Define the \emph{link lattice}  by  \[\bH (L) := \prod_{i=1}^n \left(\Z + \frac{\operatorname{lk}(L,L-K_i)}{2} \right). \]
Each generator of $\cCFL(L)$ has an \emph{Alexander grading} $\bs = (s_1,\dots,s_n) \in \bH$ and  the algebraic elements interact with the Alexander grading as follows: for $i=1,\dots,n,$ multiplication by $Z_i$
increases $s_i$ by $1$ and multiplication by $W_i$ decreases $s_i$ by $1$. In particular, $U_i$ preserves the Alexander grading. The differential preserves Alexander grading, and
there is a decomposition as a chain complex 
\[
 \cCFL(L) = \bigoplus_{\bs \in \bH (L)} \cCFL(L,\bs),
\]
where each $\cCFL(L,\bs)$ is a chain complex over $\F[U_1,\dots,U_n].$ Specifically, suppose $\{\bdy_1,\dots,\bdy_m\}$ is a free basis of $\cCFL(L)$. Then
  each $\cCFL(L,\bs)$ is freely generated over $\F[U_1,\dots,U_n]$ by $\{a_1 \bdy_1,\dots, a_m \bdy_m \}$, where  $a_i \in \F[W_1,Z_1,\dots,W_n,Z_n]$ is a monomial.

%After inverting $Z_1,\dots,Z_n$ in $\cCFL(L)$, the resulting chain complex, viewed as a chain complex over $\F[U_1]$,\ds{check} is chain homotopy equivalent to $\CF^-(M)$ of the ambient manifold by \cite[Theorem 4.4]{OSLinks}.  
There is a homological grading on $\cCFL(L)$, where $gr(W_i)=-2, gr(Z_i)=0$ and $gr(\partial)=-1$.

We recall from the introduction  the definition of an \emph{$L$-space link} by \cite{GorNem16}.
\begin{defn}
An $n$-component link \(L\subset S^3\) is an \emph{$L$-space link} if all of its positive  large integer surgeries are $L$-spaces, \ie if there exists $(p_1,\dots,p_n)\in \Z^n$ such that the surgery \(S^3_{d_1,\dots, d_n}(L)\) is an $L$-space for every integers $d_1,\dots, d_n$ with \(d_i\geq p_i\), \(\forall i=1,\dots, n\).  
\end{defn}
\subsection{Type-$A$ and type-$D$ modules and bimodules} \label{subsec: type_A_algebra_and_bimodules} We recall several definitions regarding bordered modules. 
All definitions in this section are taken from \cite{LOT, LOT_bimodules}. We always assume the ground ring $\br$ has characteristic $2.$
\begin{defn}
Suppose that $A$ is an associative algebra over a ground ring $\br.$
    A \emph{(right) type-$D$ module} $\cN^A$  is a right $\br$-module $N$ together with an $\br$-module map 
    \[
    \delta^1: N \to N \otimes_{\br} A
    \]
    satisfying that 
    \[(\mathbb{I}_{N} \otimes \mu)\circ(\delta_1 \otimes \mathbb{I}_{A}  )\circ \delta^1 = 0,\]
    where $\mu : A \otimes A \to A$ is the multiplication on $A$.
\end{defn}
Once we have a type-$D$ module, we can define maps $\delta^i$ for $i\geq 0$ by setting 
    \[\delta^0 = \mathbb{I}_A, \qquad \delta^i = (\delta^1 \otimes \mathbb{I}_{A^{\otimes i-1}} ) \circ \delta^{i-1} \qquad i > 0.
    \]
In the case of our interest, type-$D$ modules will be \emph{bounded}; see, \cite[Definition~2.22]{LOT}. This means that for each $x\in N$, there exists an here is an integer $m$ so that for all $i\geq m$, $\delta^i(x)=0$. In this case, we set
    \[
    \delta := \sum_{i=0}^\infty \delta^i: N \to N \otimes_{\br} \mathcal{T}^*(A),
    \]
    where \[
\mathcal{T}^*(A) := \bigoplus_{n=0}^{\infty} A^{\otimes n}. 
\]

\begin{defn}
  Suppose that $A$ is an associative algebra over a ground ring $\br.$ A \emph{(left) type-$A$ module}  ${}_{A}\cM$ is a  left $\br$-module $M$, together with $\br$-module maps 
\[
m_{i}:   A ^{\otimes i-1} \otimes_{\br} M   \rightarrow M, \qquad i \in \Z_{>0}
\]
  called the \emph{structure maps}. They satisfy the compatibility relations
 \begin{align*} 
 \sum_{i=0}^{n} m_{n-i+1}( m_{i+1} ( a_n \otimes \cdots  \otimes a_{i}) \otimes \cdots \otimes a_{1}, x )    + 
  \sum_{i=1}^{n-1} m_{n} ( a_n \otimes \cdots  \otimes \mu( a_{i+1} \otimes a_{i}) \otimes \cdots \otimes a_{1}, x )
    = 0.
   \end{align*}
   We define a right type-$A$ module  $\cM_{A}$ analogously.
 The compatibility relation has a graphic description, which we now recall.
   
First, the maps $m_{i}$  can be packed together into a single map \[
m= \sum_{i> 0} m_{i}:  \mathcal{T}^*(A) \otimes_{\br} M  \rightarrow M.\]
We also define the map $D^{A}: \mathcal{T}^*(A) \rightarrow \mathcal{T}^*(A)$  by
\[
   D^{A}(a_n \otimes\cdots\otimes a_1)= \sum_{i=1}^{n-1}  a_n\otimes  \cdots \otimes \mu(a_{i+1}\otimes  a_{i} )\otimes \cdots \otimes a_1
    \]
and the co-multiplication  $   \Delta\colon \mathcal{T}^*(A)\to \mathcal{T}^*(A)\otimes_{\br} \mathcal{T}^*(A)$
 by
\[
\Delta(a_n\otimes \cdots \otimes a_1)=\sum_{i=0}^n (a_n\otimes \cdots \otimes a_{i+1})\otimes (a_{i}\otimes \cdots \otimes a_1).
\]
Then  the compatibility relation can be described graphically  as
\begin{equation*}
    \label{eq:Amodule-relation}
\begin{tikzcd}[column sep=.5cm,row sep=.4cm]
 \ar[d, Rightarrow] &  \ar[dashed,dd] \\
 \Delta \ar[dr, Rightarrow] \ar[ddr, Rightarrow]&\\
& m  \ar[dashed,d]\\
& m  \ar[dashed,d]\\
\ & \ 
\end{tikzcd}
+ 
\begin{tikzcd}[column sep=.5cm,row sep=.4cm]
\ar[d, Rightarrow] &\ar[dashed,dd]   \\
 {D}^{A} \ar[dr, Rightarrow] &   \\
& m  \ar[dashed,d]\\
\ & \ 
\end{tikzcd}
=0
\end{equation*}
where  the dashed lines indicate elements of $M$, and double arrows indicate elements of  $\mathcal{T}^*(A)$.
\end{defn}
If $A$ is commutative, a type-$A$ module ${}_{A}\cM$ with $m_i = 0$ for $i>2$ is a regular chain complex over the ring $A$, with the differential given by $m_1$ and multiplication by elements in $A$ given by $m_2$. 

The notion of a type-$A$ module can be generalized to define a bimodule.
\begin{defn}
  Suppose that $A$ and $B$ are associative algebras over $\br$. A \emph{type-$AA$ bimodule} or an \emph{$A_\infty$-bimodule} $_{A}\cM_{B}$ is an $\br$-module $M$, together with the structure maps, for $i\geq 0, j\geq 0$,
\[
m_{i,1,j}: A ^{\otimes i} \otimes_{\br} M \otimes_{\br} B ^{\otimes j} \rightarrow M
\]
satisfying the compatibility relation
%  \begin{align*} 
%  \sum_{\substack{i\geq s \geq 0\\j\geq t \geq 0}} m_{i-s,1,j-t}(a_i\otimes \cdots \otimes a_{s+1},m_{s,1,t}(a_s\otimes \cdots \otimes a_1,x,b_1\otimes   \cdots \otimes b_t)&b_{t+1}\otimes \cdots \otimes b_{j} )    \\
%   +\sum_{s=1}^i  m_{i-1,1,j}(a_i\otimes \cdots \otimes\mu(a_{s+1}\otimes \cdots \otimes a_s )\otimes \cdots  \otimes a_1,x,b_1\otimes  \cdots \otimes b_j)&\\
%  +\sum_{t=1}^j   m_{i,1,j-1}(a_i\otimes \cdots \otimes a_1,x,b_1\otimes  \cdots \otimes \mu(b_t\otimes \cdots \otimes b_{t+1} )\otimes \cdots \otimes b_j)&    
%     = 0.
%    \end{align*}
% The maps $m_{i,1,j}, i\geq 0, j\geq 0$  can be packed together into a single map \[
% m= \sum_{i,j \geq 0} m_{i,1,j} :\mathcal{T}^*(A)\otimes M\otimes \mathcal{T}^*(B) \rightarrow M\]
% where \[
% \mathcal{T}^*(A) := \bigoplus_{n=0}^{\infty} A^{\otimes n}. 
% \]
%  The compatibility relation has a graphic description as follow.
% Define the map $D^{A}: \mathcal{T}^*(A) \rightarrow \mathcal{T}^*(A)$  by
% \[
%    D^{A}(a_1\otimes a_2\otimes\cdots\otimes a_n)= \sum_{j=1}^{n-1}  a_1\otimes  \cdots \otimes \mu(a_j\otimes  a_{j+1} )\otimes \cdots \otimes a_n
%     \]
% and the co-multiplication  $   \Delta\colon \mathcal{T}^*(A)\to \mathcal{T}^*(A)\otimes \mathcal{T}^*(A)$
%  by
% \[
% \Delta(a_1\otimes \cdots \otimes a_n)=\sum_{j=0}^n (a_1\otimes \cdots \otimes a_j)\otimes (a_{j+1}\otimes \cdots \otimes a_n).
% \]
% Then  the compatibility relation can be described  as
\begin{equation*}
    \label{eq:AAbimodule-relation}
\begin{tikzcd}[column sep=.5cm,row sep=.4cm]
\ar[d, Rightarrow] & \ar[dashed,dd] & \ar[d, Rightarrow]\\
\Delta \ar[dr, Rightarrow] \ar[ddr, Rightarrow] && \Delta \ar[dl, Rightarrow] \ar[ddl, Rightarrow]\\
& m  \ar[dashed,d]&\\
& m  \ar[dashed,d]&\\
&\ &
\end{tikzcd}
+
\begin{tikzcd}[column sep=.5cm,row sep=.4cm]
\ar[d, Rightarrow] & \ar[dashed,dd] & \ar[ddl, Rightarrow]\\
{D}^{A} \ar[dr, Rightarrow]  && \\
& m  \ar[dashed,d]&\\
&\ &
\end{tikzcd}
+ 
\begin{tikzcd}[column sep=.5cm,row sep=.4cm]
\ar[ddr, Rightarrow] & \ar[dashed,dd] &  \ar[d, Rightarrow]\\
 && {D}^{B} \ar[dl, Rightarrow]  \\
& m  \ar[dashed,d]&\\
&\ &
\end{tikzcd}
=0.
\end{equation*}
\end{defn}

 Given a right type-$D$ module $\cN^A$ and a type-$AA$ bimodule ${}_A\cM_B$, the \emph{box tensor product} $\cN^A \boxtimes {}_A\cM_B$ is a right type-$A$ module over $B$, with underlying module  $N\otimes_{\br} M$ and structure maps given
 by the following diagram
\begin{equation*}
\begin{tikzcd}[column sep=.5cm,row sep=.4cm]
\ar[d, dashed] & \ar[dashed,dd] & \ar[ddl, Rightarrow]\\
\delta \ar[dd, dashed] \ar[dr, Rightarrow] &&  \\
& m  \ar[dashed,d]&.\\
\ &\ &
\end{tikzcd}
\end{equation*}

\subsection{Knot and link surgery algebras} \label{subsec: link_surgery_algebra}
In this section we recall the definitions of the knot and link surgery algebras, following \cite{zemke_bordered}. Before doing so, we premise a brief discussion regarding algebras on idempotent rings. 

Assume that $A$ is an associative algebra over the idempotent ring $\bdI=\bdI_0 \oplus \bdI_1$. Here, each $\bdI_i$ is a copy of $\F =\Z/2\Z$, and we denote by $i_0$ and $i_1$ the generators of $\bdI_0$ and $\bdI_1$ respectively. Then, as a module over $\bdI$, we can decompose $A$ as
\[
A=\left( \bdI_0 \cdot A \cdot \bdI_0\right) \oplus \left(\bdI_0 \cdot A \cdot \bdI_1\right) \oplus \left(\bdI_1 \cdot A \cdot \bdI_0\right) \oplus \left(\bdI_1 \cdot A \cdot \bdI_1\right)
\]
In fact, it is not difficult to see that the above subalgebras have pairwise trivial intersections, and they generate since for every $a\in A$ we have:
\[
a=(i_0+i_1)a(i_0+i_1)=(i_0 a i_0) + (i_0 a i_1) + (i_1 a i_0) + (i_1 a i_1).
\]
By reversing this reasoning, we see that in order to define an algebra over $\bdI$ it is enough to describe the four submodules above, and the product rules among their elements. 

\begin{defn} The \emph{knot surgery algebra} $\cK$ is the associative algebra over $\bdI$ defined by setting
\begin{align*}
    \bdI_0 \cdot \cK \cdot \bdI_0 = \F[W,Z], \qquad &\bdI_1 \cdot \cK \cdot \bdI_1 = \F[U,T,T^{-1}] \\
    \bdI_0 \cdot \cK \cdot \bdI_1 = 0, \hspace{0.65 em}\text{ and} \qquad &\bdI_1 \cdot \cK \cdot \bdI_0 = \F[U,T,T^{-1}] \otimes \langle \sigma, \tau \rangle,
\end{align*}
where $U = WZ$, and $T=Z$\footnote{We use this difference in notation since $T$ is invertible in $\bdI_1 \cdot \cK \cdot \bdI_1$ and in $\bdI_1 \cdot \cK \cdot \bdI_0$.}. The products among elements of these submodules are defined as follows. For \(\varepsilon=0,1\), the product in \(\bdI_\varepsilon \cdot \cK \cdot \bdI_\varepsilon\) is given by polynomial multiplication. The only other non-trivial products are those between elements in \(\bdI_1 \cdot \cK \cdot \bdI_0\) and \(\bdI_0 \cdot \cK \cdot \bdI_0\) and they are obtained by imposing 
% We view $\F[U,T,T^{-1}]$ as $\F[U,Z,Z^{-1}] = \F[W,Z,Z^{-1}]$, where the variable T is used to distinguish this from the idempotent $0$, in which $Z$ is not invertible.
the algebra elements $\sigma$ and $\tau$ to satisfy the relation
\[
\sigma \cdot a = g^{\sigma}(a) \cdot \sigma, \qquad \tau \cdot a = g^{\tau}(a) \cdot \tau
\]
for any $a\in \F[W,Z]= \bdI_0 \cdot \cK \cdot \bdI_0.$ Here $g^{\sigma}$ is the natural inclusion map from $\F[W,Z]$ to $\F[U,T,T^{-1}]$ given by
\begin{align*}
g^{\sigma}(W) = UT^{-1} \qquad &g^{\sigma}(Z) = T
\intertext{
and
 $g^{\tau}$ is an algebra homomorphism that satisfies }
 g^{\tau}(W)=T^{-1} \qquad \quad &g^{\tau}(Z)=UT^{2}.\end{align*}
 % For a module over $\cK,$ $Z$ or $T$ increases the Alexander grading by $1,$ $W$ or $T^{-1}$ increases the Alexander grading by $1$ and $U$ preserves the Alexander grading.
\end{defn}

We call $\bdI_0 \cdot \cK \cdot \bdI_0$ and $\bdI_1 \cdot \cK \cdot \bdI_1$ the idempotent $0$ and $1$ of $\cK$, respectively. Observe that they satisfy the property $(\bdI_\varepsilon \cdot \cK \cdot \bdI_\varepsilon) \cdot (\bdI_\varepsilon \cdot \cK \cdot \bdI_\varepsilon) = \bdI_\varepsilon \cdot \cK \cdot \bdI_\varepsilon$ for $\varepsilon =0,1.$
\begin{defn}
The $n$-component \emph{link surgery algebra} is 
\[
\cL =  \cK_1 \otimes_\F \cdots \otimes_\F \cK_n
\]
where each $\cK_i$ is a copy of the knot surgery algebra. 
% We impose an extra relation $W_1 Z_1=\cdots=W_n Z_n=U$ in $\cL$ as this often proves convenient\ds{I don't think this relation holds in the link surgery algebra}.
The link surgery algebra $\cL$ is an algebra over the idempotent ring $\bdE_n: = \underbrace{\bdI \otimes_\F \cdots  \otimes_\F \bdI}_{n}.$
\end{defn}

We give an explicit description of the link surgery algebra as follows. Define $\bE_n := \{ 0,1 \}^n$. For $\varepsilon,\varepsilon' \in \bE_n$, we write  $\ve \leq \ve'$ if the inequality holds in each coordinate. Similarly, we write $\ve < \ve'$ if $\ve \leq \ve'$ and  $\ve \ne \ve'$. We begin by describing the subalgebras $\bdI_{\ve'} \cdot \cL \cdot \bdI_\ve$, for $\ve, \ve'\in \bE_n$, and then the products between their elements.

There are $2^n$ idempotents in $\cL$, namely $\bdI_\ve \cdot \cL \cdot \bdI_\ve$  for each $\ve = (\ve_1, \dots, \ve_n) \in \bE_n$, where  $\bdI_\ve := \otimes_i^n\bdI_{\ve_i}$.
  More concretely, $\bdI_\ve \cdot \cL \cdot \bdI_\ve$  is the tensor of $\F[Z_i,W_i]$ for those indices $i$ such that $\ve_i =0$ and $ \F[U_i, T_i,T^{-1}_i]$ for the indices $i$ such that $\ve_i =1$. Given $\ve,\ve' \in \bE_n,$ if $\ve \leq  \ve'$, then $\bdI_{\ve'} \cdot \cL \cdot \bdI_\ve$ is the tensor of $ \F[U_i, T_i,T^{-1}_i]\otimes  \langle \sigma_i, \tau_i \rangle$ for $i$ such that $\ve'_i-\ve_i=1$ and $\bdI_{\ve_i} \cdot \cK_i \cdot \bdI_{\ve_i}$ for $i$ such that $\ve'_i-\ve_i=0$; otherwise $\bdI_{\ve'} \cdot \cL \cdot \bdI_\ve = 0$. 

Concerning products between elements of the above subalgebras, the only ones that require some discussion are those between elements of $\bdI_{\ve'} \cdot \cL \cdot \bdI_\ve$ and $\bdI_{\ve} \cdot \cL \cdot \bdI_\ve$, when $\ve< \ve'$. We fix the following notation. Given $\ve< \ve'$, suppose that $t\in \bdI_{\ve'} \cdot \cL \cdot \bdI_\ve$ is in the form of  $t = \otimes_i^n t_i$, where  $t_i = \sigma_i$ or $\tau_i$ for the indices $i$ with $\ve_i'-\ve_i=1$ and $t_i = 1$ for those with $\ve_i'-\ve_i=0$.
Define \[g^t(a_1 \otimes\cdots \otimes a_n) =  \otimes_i^n g^{t_i}(a_i) \] where $g^{\sigma_i}$ and $g^{\tau_i}$ are defined as before, and $g^{t_i}$ is the identity map if $t_i=1$. Then the products between elements in $\bdI_{\ve'} \cdot \cL \cdot \bdI_\ve$ and $\bdI_{\ve} \cdot \cL \cdot \bdI_\ve$ are obtained by imposing  
\[
t \cdot a = g^t(a) \cdot t
\]
for any $a\in \bdI_{\ve}\cdot \cL \cdot \bdI_{\ve}.$ 
\newline

 The link surgery algebra is used to define a type-$D$ module $\cX(K)^{\cL}$ related to the link surgery formula, which we will describe in Section \ref{subsec: zemke's reformulation}. 
 \subsection{Type-$A$ modules for solid tori} \label{subsec: solid_tori}
 Another component of the link surgery formula from the bordered module perspective is the bordered modules for a disjoint union of solid tori, which we now discuss.  

We start by describing the type-$A$ module ${}_{\cK}\cD_\lambda$ for a solid torus with integral framing $\lambda$, defined in  \cite[Section 8.2]{zemke_bordered}. Since ${}_{\cK}\cD_\lambda$ is a module over $\bdI$, arguing as in the beginning of Section~\ref{subsec: link_surgery_algebra}, we have that it is enough to describe the submodules $\bdI_0 \cdot \cD_\lambda$ and $\bdI_1 \cdot \cD_\lambda$.  The underlying space $\bdI_0 \cdot \cD_\lambda$ is freely generated over $\F[W,Z]$ by a generator $\bdx^0$ and $\bdI_1 \cdot \cD_\lambda$ is freely generated over  $\F[U,T,T^{-1}]$ by a generator $\bdx^1$.
All structure maps $m_i$ vanish for $i\neq 2$ and the maps $m_2$ are defined as follows. If $a\in \bdI_i \cdot \cK \cdot \bdI_i$ and $x\in \bdI_j \cdot \cD_\lambda$, then we define $m_2(a,x)=a\cdot x$ if $i=j$, and to be $0$ otherwise. We define
 \[
 m_2(\sigma, a\bdx^0) = g^\sigma (a) \bdx^1 \qquad  m_2(\tau, a\bdx^0) = g^\tau (a) \bdx^1
 \]
for $a \in \F[W,Z].$

The module ${}_{\cK}\cD_\lambda$ can be viewed as a type-$AA$ bimodule ${}_{\cK}[\cD_\lambda]_{\F[U]}$, where $m_{0,1,1}$ is defined by the multiplication by power of $U$ and and $m_{1,1,0}$ is the same as $m_2$ on ${}_{\cK}\cD_\lambda$. All other structure maps $m_{i,1,j}$ are set to be zero.

Generalising this to the setting of rationally framed solid torus, the type-$D$ module $\cD^\cK_{p/q}$  is defined in \cite[Section 18.2]{zemke_bordered}. For $\nu \in \{0,1\}$, the underlying space of $\cD^\cK_{p/q}\cdot \bdI_{\nu}$ is the $\F$-span of $\{\bdx^\nu_0,\dots,\bdx^\nu_{q-1} \}$ and the type-$D$ structure map is given by
\[
\delta^1(\bdx^0_k) = \bdx^1_k \otimes \sigma + \bdx^1_{k+p} \otimes T^{\floor{\frac{k+p}{q}}} \tau 
\]
where we choose the representative of $k\in \Z/q\Z$ satisfying $0\leq k<q$. 

We can obtain the type-$A$ module ${}_{\cK}\cD_{p/q}$ by ${}_{\cK}\cD_{p/q} := \cD^\cK_{p/q} \boxtimes {}_{\cK| \cK} [\bI^{\Supset}]$ where ${}_{\cK| \cK} [\bI^{\Supset}]$ is the identity bimodule defined in  \cite[Section 18.4]{zemke_bordered}. We now give an explicit description of  ${}_{\cK}\cD_{p/q}$. 

Suppose that $p\in \Z,q\in \Z_{>0}$ and that $p$ and $q$ are coprime.
The underlying space $ \bdI_0 \cdot \cD_{p/q}$ is freely generated over $\F[W,Z]$ by $q$ generators $\{\bdx^0_0, \dots,\bdx^0_{q-1}\}$ and  $\bdI_1 \cdot \cD_{p/q}$ is freely generated over $\F[U,T,T^{-1}]$ by $q$ generators $\{\bdx^1_0,\dots,\bdx^1_{q-1}\}$. We view the generators  $\bdx^i_k$ as indexed by $i\in \{0,1\}$ and $k\in \Z/q\Z.$
  
As before, all structure maps $m_i$ vanish for $i\neq 2$ and the maps $m_2$ are defined as follows.
If $a\in \bdI_i \cdot \cK \cdot \bdI_i$ and $x\in \bdI_j \cdot \cD_\lambda$, define $m_2(a,x)= a \cdot x$ if $i=j$ and to be zero otherwise. We define
 \[
 m_2(\sigma, a\bdx^0_k) = g^\sigma (a) \cdot \bdx^1_k \qquad  m_2(\tau, a\bdx^0_k) = T^{\floor{\frac{k+p}{q}}}g^\tau (a) \cdot \bdx^1_{k+p}
 \]
for $k\in \Z/q\Z, a \in \F[W,Z].$

The following is a schematic drawing of ${}_{\cK}\cD_{1/3}$, where an arrow from $\bdx$ to $\bdy$ labelled by $u|a$ means that $m_2(u,\bdx)=a\otimes \bdy .$  
\[
\begin{tikzcd} [labels=description, column sep=1.4cm]
\bdx^0_0 \arrow[d, "\sigma|1" description, crossing over]  & \bdx^0_1  & \bdx^0_2 \arrow[d, "\sigma|1" description, crossing over] \arrow[dll, "\tau|T" description, pos=0.15] \\[10 mm]
 \bdx^1_0 & \arrow[from=ul, "\tau|1" description, crossing over] \arrow[from=u, "\sigma|1" description, crossing over] \bdx^1_1 & \arrow[from=ul, "\tau|1" description, crossing over] \bdx^1_2
\end{tikzcd}
\]

We are now ready to define the type-$A$ module of the disjoint union of $n$ (rationally framed) solid tori ${}_{\cL}\cD_\Lambda$.
\begin{defn}
Given a set of rational framings $\Lambda = (p_1/q_1, \dots, p_n/q_n)$  with $p_i\in \Z,q_i\in \Z_{>0}$ and $p_i,q_i$ coprime,  we define 
\[{}_{\cL}\cD_\Lambda :=  {}_{\cK_1}\cD_{p_1/q_1} \otimes_\F \cdots \otimes_\F {}_{\cK_n}\cD_{p_n/q_n}\]
where $\otimes_\F$ is the \emph{external tensor product}. For a detailed definition of the external tensor product of type-$A$ modules, see \cite[Section 3.5]{zemke_bordered} and also \cite[Definition 3.21]{LOTDiagonals}. However,  since each ${}_{\cK_i}\cD_{p_i/q_i}$ has only $m_2$ non-trivial, the definition is greatly simplified. Indeed, it is straightforward to verify that  if type-$A$ modules $M_1$ and $M_2$ over $A_i, i=1,2$ both have $m_j=0$ for $j>2$, then their  external tensor product  has underlying module $M_1 \otimes_\F M_2,$ with structure maps 
\[
m_2(a_1 \otimes a_2,x_1 \otimes x_2) = m_2(a_1,x_1)\otimes m_2(a_2,x_2) \text{ for } x_i \in M_i, a_i\in A_i, i =1,2
\] and $m_j=0$ for $j>2$. 
% Here $|$ denotes the external tensor product. 
In the external tensor product $m_1$ is obtained by the Leibniz rule, therefore it  vanishes for ${}_{\cL}\cD_\Lambda$ as well.
Applying the above formula iteratively, we obtain that the structure map on ${}_{\cL}\cD_\Lambda$ is given by 
\begin{equation}\label{eq: external_tensor}
   m_2(a_1\otimes\cdots\otimes a_n,x_1\otimes\cdots \otimes x_n)=m_2(a_1,x_1)\otimes \cdots \otimes m_2(a_n,x_n)  
\end{equation}
 for $a_i \in \cK_i $ and $ x_i \in \cD_{p_i/q_i}, i =1,\dots,n.$

We now give  an explicit description of ${}_{\cL}\cD_\Lambda$.
For $\ve \in \bE_n,$ $\bdI_\ve \cdot \cD_\Lambda$ is freely generated over $\big(\bdI_\ve \cdot \cL \cdot \bdI_\ve \big)$ by $q_1\cdots q_n$ generators $\{ \bdx^{\ve_1}_{k_1}\bdx^{\ve_2}_{k_2}\cdots \bdx^{\ve_n}_{k_n} : k_i\in \Z/q_i\Z, i=1,\dots,n\}$. We recall that $\bdI_\ve \cdot \cL \cdot \bdI_\ve$  is the tensor product of   $\F[Z_i,W_i]$ for $i$ such that $\ve_i =0$ and $ \F[U_i, T_i,T^{-1}_i]$ for $i$ such that $\ve_i =1$. We denote the generators by $\bdx^\ve_{\bk}:=\bdx^{\ve_1}_{k_1}\bdx^{\ve_2}_{k_2}\cdots \bdx^{\ve_n}_{k_n} $  where $\bk=(k_1,\dots,k_n).$ The structure map $m_j$ is zero if $j\neq 2$ and $m_2$ is given by \eqref{eq: external_tensor}.

More concretely, given $\ve\leq \ve',$ suppose $t\in \bdI_{\ve'}\cdot \cL \cdot \bdI_{\ve} $ is a tensor product of  $\sigma_i$ or $\tau_i$ for $i$ such that $\ve_i'-\ve_i=1$ and $1$ for $i$ such that $\ve_i'-\ve_i=0$. If we let $\frak{d}(t) \subset \{1,\dots,n\}$ denote the set of $i$ such that $\tau_i$ is a tensor factor  in $t$, then for any    $a \in \bdI_{\ve}\cdot \cL \cdot \bdI_{\ve}$ and $b \in \bdI_{\ve'}\cdot \cL \cdot \bdI_{\ve'}$, we have
\[
m_2(b\cdot t, a\cdot \bdx^{\ve}_{\bk}) = b\cdot g^t(a) \cdot (\prod_{i\in \frak{d}(t)}T_i^{\floor{\frac{k_i+p_i}{q_i}}} ) \cdot \bdx^{\ve'}_{\bk'}
\]
where $k'_i=k_i+p_i$ for $i$ such that $i\in \frak{d}(t)$ and $k'_i=k_i$ otherwise. 
\end{defn}

As in the case of one solid torus, ${}_{\cL}\cD_\Lambda$ can be viewed as a type-$AA$ bimodule ${}_{\cL}[\cD_\Lambda]_{\F[U_1,\dots,U_n]}$, where $m_{0,1,1}$ is multiplications by $U_i$, $m_{1,1,0}$ is the same as $m_2$ on ${}_{\cL}\cD_\Lambda$ and $m_{i,1,j}=0$ for all other $i$ and $j$.

\section{Rational link surgery formula}\label{section:rational link surgery}
In this section we deduce a rational link surgery formula as a direct consequence of the works by Zemke \cite{zemke_bordered, zemke_general}. Then we use Liu's work \cite{yajing_lspace} to show that the rational link surgery complex admits a simplified model in the case of two-component $L$-space links. In Section \ref{subsec: integral_link_surgery_formula}-\ref{subsec: zemke's reformulation} we review the statement of the integral link surgery formula in \cite{mo_linksurgery}, the definition of the $H$-function (in the case of $L$-space links) by \cite{GorNem15} and Zemke's reformulation of the link surgery formula using bordered modules \cite{zemke_bordered}. In Section \ref{subsec: rational_link_surgery_formula}, we define the rational link surgery complex and prove that its homology is the same the Heegaard Floer homology of the corresponding rational surgery using the works by Zemke. Finally, in Section \ref{subsec: simplified model},  we obtain a simplified model of the link surgery complex by directly applying Liu's results.
% In this section we deduce a rational link surgery formula, simplifying in the case of $L$-space links with Alexander-formal link Floer homology, as a direct consequence of the works by Zemke \cite{zemke_bordered, zemke_general}. In Section \ref{subsec: integral_link_surgery_formula}-\ref{subsec: zemke's reformulation} we review the statement of the integral link surgery formula in \cite{mo_linksurgery}, the definition of the $H$-function (in the case of $L$-space links) by \cite{GorNem15} and Zemke's reformulation of the link surgery formula using bordered modules \cite{zemke_bordered}. In Section \ref{subsec: rational_link_surgery_formula}, we introduce the rational link surgery complex and prove that it computes the Heegaard Floer chain complex of the corresponding rational surgery. Finally, in Section \ref{subsec: simplified model}, we prove its equivalence to the simplified model by applying the homological perturbation lemma for hypercubes.
\subsection{The integral link surgery formula} \label{subsec: integral_link_surgery_formula}
We recall the construction of the integral link surgery complex by
Manolescu-Ozsv{\'a}th \cite{mo_linksurgery},  following a slight reformulation by Zemke \cite{zemke_bordered}.
Before doing so, we introduce the notion of \emph{hypercube} of chain complexes from \cite{mo_linksurgery}. This formalism appears naturally in this context, and we will make use of it in what follows. Recall the notation \(\bE_n = \{0,1 \}^n\). 
\begin{defn}
An $n$-dimensional hypercube of chain complexes $(C_\ve, D_{\ve,\ve'})_{\ve\in\bE_n}$ consists of:
\begin{itemize}
\item A group $C_{\ve}$ for each $\ve\in \bE_n$.
\item Linear maps \(D_{\ve,\ve'}:C_{\ve} \to C_{\ve'}\) for each pair of indices $\ve,\ve'\in \bE_n$ with $\ve\leq \ve'$.
\end{itemize}
Furthermore, we ask the the following condition holds for each pair $\ve\leq \ve''$
\[
\sum_{\substack{\ve'\in \bE_n\\ \ve\leq\ve'\leq \ve''}}D_{\ve',\ve''}\circ D_{\ve,\ve'}=0.
\]
\end{defn}
 Suppose $L=K_1 \cup \cdots \cup K_n \subset S^3$ is an oriented $n$-component link.  
Let $\Lambda =(\lambda_1, \dots, \lambda_n) \in \Z^n$ be a set of integral surgery framings. We often conflate it with the \emph{surgery matrix} $\Lambda = ( \Lambda_{i,j})_{n\times n}$ defined by
% Define the link lattice  $\bH (L) := \Z^2 + (\ell/2,\ell/2) $ 
\begin{align*}
\Lambda_{i,j} = \begin{cases}
    \lambda_i   &i=j \\
    \operatorname{lk}(K_i,K_j) \qquad & i\neq j.
\end{cases}
\end{align*}

 % The \emph{link lattice} is defined by  \[\bH (L) := \prod_{i=1}^n \left(\Z + \frac{\operatorname{lk}(L,L-K_i)}{2} \right). \]
 
Given $\ve = (\ve_1, \dots, \ve_n) \in \bE_n = \{0,1 \}^n, $ denote by $\Se$ the  multiplicatively closed set in $\F[W_1,Z_1,\dots,W_n,Z_n]$ generated by $Z_i$  for $i$ such that $\ve_i = 1$. Let $\Se^{-1}\cCFL(L)$ be the localisation of $\cCFL(L)$ by $\Se$. There is a decomposition
 \[
 \Se^{-1}\cCFL(L) = \bigoplus_{\bs \in \bH(L)} \Se^{-1}\cCFL(L,\bs)
 \]
 coming from the Spin$^c$ decomposition of $\cCFL(L)$,
 where each $\Se^{-1}\cCFL(L,\bs)$ is a free complex over $\F[ U_1,\dots,U_n ].$
We define  $C_\ve(\bs) $ to be the free complex over the power series ring $\F\llbracket U_1,\dots,U_n \rrbracket$ with the same set of generators as $\Se^{-1}\cCFL(L,\bs)$, and set
 % In $C_\ve$, we write $T_i = Z_i$ whenever $\ve_i =1$ to make a distinction. 
 % For each $\ve,$ the group $C_\ve$  decomposes over the Alexander grading $\bs = (s_1,\dots, s_n) \in \bH(L)$ and we denote each subgroup by 
 \[
 C_\ve  := \prod_{\bs \in  \bH(L)} C_\ve(\bs). 
 \]
 Each complex $C_\ve$ is equipped with a homological grading coming from $\cCFL(L)$\footnote{Observe that, in contrast to the usual case, $C_\ve$ is not the direct sum of its graded part, but we still refer to it as a graded complex.}.
 % The element $W_i Z_i$ for $i=1,\dots,n$ induces the same action on the homology and so we set $U=W_i Z_i$ for $i=1,\dots,n$ on $\cCFL(L)$. The element $W_i$ (resp.~$Z_i$) decreases (resp.~increases) the $i$-th component of the Alexander grading by $1$ and    $U$ preserves the Alexander grading. 
 % Each $C_\ve(\bs)$ is a graded module over $\F\llbracket U_1,\dots,U_n\rrbracket$. 

\begin{remark}\label{rem: Z_i_iso}
 For $i$ such that $\ve_i = 1$, the element $Z_i$ is a unit in $\Se^{-1}\F[W_1,Z_1,\dots,W_n,Z_n]$, and hence the map $Z_i$ is an isomorphism of the chain complex $C_{\ve}(\bs)$ that changes the $i$-th component of the Alexander grading by $1$. In particular,  the isomorphism class of $C_{\ve}(\bs)$ does not depend on $s_i$ for $i$ such that $\ve_i=1.$
\end{remark}

Denote by  $\Lambda_j$ the $j$-th column vector of the surgery matrix  $\Lambda$. Under the identification  $H_1(S^3 \setminus L;\Z) = \Z^n$ induced by the meridians of the components of $L$, the vector  $\Lambda_j$ represents the $\lambda_j$-framed longitude of $K_j$.
Given any arbitrarily (possibly empty) oriented sublink $\vec{M}$ of $L$, define $\Lambda_{L,\vec{M}} $ to be the sum of the vectors $\Lambda_j$
for which the orientation of $K_j$ in $\vec{M}$ is opposite to its orientation in $L$. 
Each $\ve \in \bE_n$ can be viewed as the indicator function of a sublink $M \subset L$, where $K_j \subset M$ if $\ve_j = 1.$ For each $\vec{M}$ and $\ve,$
Manolescu-Ozsv{\'a}th defined a  map 
\begin{equation}\label{eq: phi_vecM}
\phi^{\vec{M}}_{\varepsilon}: C_\ve(\bs) \rightarrow C_{\ve'}(\bs + \Lambda_{L,\vec{M}} )
\end{equation}
 where $\ve' - \ve$ is the indicator function of $M$. 
 % The homological grading shift of  $\phi^{\vec{M}}_{\varepsilon}$  is $|\ve' - \ve|-1$.

It is proved in \cite[Proposition 9.4]{mo_linksurgery} that\footnote{Note the difference in notation: in \cite{mo_linksurgery} $\Phi^{\vec{M}}$ is used to denote the map on the chain level.} for any fixed arbitrarily oriented sublink $\vec{N} \subset L$,
\[
\sum_{\substack{\vec{M_1}\cup \vec{M_2}=\vec{N}\\\vec{M_1}\cap \vec{M_2}= \emptyset}} \phi^{\vec{M_2}}\circ \phi^{\vec{M_1}}  = 0.
\]
 When $\vec{M}$ is the empty sublink, $\phi^{\emptyset}_{\varepsilon}$ is the internal differential on $C_\ve(\bs)$. 
When $|M|=1$,  $\phi^{\vec{M}}_{\varepsilon}$ is a chain map, and we denote by $\Phi^{\vec{M}}_{\varepsilon}$ its induced map on the homology.

% By construction, the map $\phi^{\vec{M}}_{\varepsilon}$, viewed as maps between $C_\ve(\bs)$, changes the $\Z/2\Z$-homological grading by $|M|-1$. 
% By construction, the map $\phi^{\vec{M}}_{\varepsilon}$ depends only on $\vec{M}$ and $\ve$. In particular, it does not depend on the surgery framings\footnote{Observe that the isomorphism class of  $C_{\ve'}(\bs + \Lambda_{L,\vec{M}} )$  does not depend on the surgery framings, since $\ve'_j=1$ whenever $\Lambda_J$ is a term in $ \Lambda_{L,\vec{M}}$.}.
 % The map $\phi^{\vec{M}}_{\varepsilon}$ is defined via Heegaard diagrams and is difficult to compute explicitly. 
     
% We note that, as map between chain complexes, $\phi^{\vec{M}}_{\varepsilon}$
%   is independent of the surgery framing. This is because $\ve'_j=1$ if $\Lambda_J$ is term in $ \Lambda_{L,\vec{M}}$. 
\begin{defn}
The \emph{integral link surgery complex} is defined by
    \[
    C_\Lambda(L) := \bigoplus_{\ve\in \bE_n} \prod_{\bs \in \bH(L)} C_\ve(\bs)
    \]
      equipped with  the differential given by the sum of $\phi^{\vec{M}}_{\varepsilon}$, where $\vec{M}$ ranges over all oriented sublinks of $L.$
      % Shift the  homological grading of each $C_\ve$ by $-|\ve|$ such that $\phi^{\vec{M}}_{\varepsilon}$ always lowers grading by $1$ on $C_\Lambda(L)$.
It is proved in \cite{mo_linksurgery} that 
$C_\Lambda(L)$ is a chain complex and that 
    \[
    H_*(C_\Lambda(L)) \cong \bHFm(S^3_\Lambda(L))
    \]
    as $\F\llbracket U\rrbracket$-modules.
    % , where $\bCFm(S^3_\Lambda(L))$ is the completion of $\CF^{-}(S^3_\Lambda(L))$ over the $U$-action. 
    Moreover,   $C_\Lambda(L)$ naturally splits over Spin$^c$ structures as follows: there is an identification Spin$^c(S^3_\Lambda(L)) \cong \bH(L)/H(\Lambda)$, where $H(\Lambda)$ is the integral lattice spanned by $\Lambda$. Then $C_{\Lambda}$ decomposes as a direct sum of complexes $C_\Lambda(L,\bdve{t} )$, where  for each $\bdve{t} \in \bH(L)/H(\Lambda)$ we set
    \begin{equation}\label{eq: C_Lambda_split_spin_c}
        C_\Lambda(L,\bdve{t} ) :=  \bigoplus_{\ve\in \bE_n} \prod_{\bs \in \bdve{t} + H(\Lambda)} C_\ve(\bs),
    \end{equation}
    with $\bdve{t} + H(\Lambda)$ being the coset of $H(\Lambda)$ in $\bH(L)$ containing $\bdve{t}.$  In \cite{mo_linksurgery}, the authors define a $\Z/2k(\bdve{t})\Z$-grading on each $C_\Lambda(L,\bdve{t} )$, where $k(\bdve{t})\in \Z$ depends on $\bdve{t}$, and they show that the above isomorphism restricts to isomorphisms
    \[
    H_*(C_\Lambda(L, \bdve{t})) \cong \bHFm(S^3_\Lambda(L),\bdve{t}),
    \] 
    as relatively graded $\F\llbracket U\rrbracket$-modules.
    Analogous results hold for the other versions, and we have 
    \[
    H_*(C_\Lambda(L)/U_1) \cong \widehat{HF}(S^3_\Lambda(L)) \quad \text{and} \quad H_*(C_\Lambda(L,\bdve{t})/U_1) \cong \widehat{HF}(S^3_\Lambda(L),\bdve{t}).
    \]
    \end{defn}
\subsection{The $H$-function} \label{subsec: H-function}
By \cite[Theorem 12.1]{mo_linksurgery}, when $L$ is an $L$-space link we have that $H_*(\cCFL(L,\bs)) \cong \F[U]$ for each $\bs \in \bH(L)$. 

The \emph{$H$-function} of $L$ from \cite{GorNem15} is defined by
\[
H_L(\bdve{s}) = -\tfrac{1}{2}  gr(1),
\]
where $1 \in \F[ U] \cong H_*(\cCFL(L,\bs))$.
    % The maps $\Phi^{\vec{M}}_{\varepsilon}$ are determined by the $H$-function.

The $H$-function of $L$-space links can be computed from the Alexander polynomials of all the sublinks by the work of Gorsky-N{\'e}methi.
Recall that given a sublink $L' \subset L$,  Manolescu-Ozsv{\'a}th defined the \emph{reduction map} $\pi_{L,L'}\colon \bH(L)\to \bH(L')$  to be
\[
s_i\mapsto s_i-\frac{\operatorname{lk}(K_i, L - L')}{2}
\]
for all $i$ such that $K_i\subset L'$.
\begin{prop}[{\cite[Theorem~2.2.11]{GorNem15}}] \label{prop:GNH-function} Let $L$ be an oriented $L$-space link in $S^3$. Then
	\[
	H_L(\bdve{s})=\sum_{L'\subset L} (-1)^{|L'|-1} \sum_{\substack{\bdve{s}'\in \bH(L') \\ \bdve{s}'\ge \pi_{L,L'}(\bdve{s}+\bdve{1})}} \chi(\HFL^-(L',\bdve{s}')).
	\]
\end{prop}
In the above, $\HFL^-(L')$ denotes  $H_*(\cCFL(L')/(Z_i: K_i\subset L'))$ and $\bdve{1}$ denotes $(1,\dots, 1)$. 
 Let 
\begin{equation}
	\tilde{\Delta}_L(x_1,\dots,x_n) :=
	\begin{cases}
		(	x_1\cdots x_n)^{\frac{1}{2}}\Delta_L(x_1,\dots,x_n) & \text{if $n>1$}\\
		\Delta_L(x)/(1-x^{-1})& \text{if $n=1$}\\
	\end{cases}       
\label{eq:normalized Alex poly}
\end{equation}
where $\Delta_L(x_1,\dots,x_n)$ is the symmetrised Alexander polynomial. By \cite[Section~1.1]{OSLinks},  the following relation holds
\begin{equation*}
	\tilde{\Delta}_{L}(x_1,\dots,x_n) = \sum_{\bdve{s}\in\mathbb{H}(L)}\chi(\HFL^{-}(L,\bdve{s}))x_1^{s_1}\dots x_n^{s_n},
	\label{eq:euler_char_Alex}
\end{equation*}
where $\bdve{s} = (s_1,\dots,s_n)$. For $n=1$, we expand $1/(1-x^{-1})$ as a power series in $x^{-1}$. There is an overall sign ambiguity of $\tilde{\Delta}_L$ which can be eliminated by the fact that  $H_L(\bdve{s})$
is non-increasing as $\bdve{s}$ increases in each component. 

% For an $L$-space link $L\subset S^3$,
%  the maps $\Phi^{\vec{M}}_{\varepsilon}$ with $|M|=1$ are determined by the $H$-function of the sublinks of $L$.\ds{is this true for any number of components? can we give a reference?} 

Specialising to the case of a two-component $L$-space link $L=K_1 \cup K_2$ where $l=\operatorname{lk}(K_1,K_2)$, when $s_1$ (resp.~$s_2$) is sufficiently large, $H_L(s_1,s_2)$ stabilises to a constant value denoted by $H_L(\infty,s_2)$ (resp.~$H_L(s_1,\infty)$). Moreover, there is a symmetry
\[H_L(s_1,s_2)+s_1+s_2 = H_L(-s_1,-s_2).\]
 The following description of $\Phi^{\vec{M}}_{\varepsilon}$ can be found in \cite[Section 3.2]{BeibeiLiu21}.
\begin{lem} \label{lemma: two_component_$L$-space_H_Lambda_maps}
When $L$ is a two-component $L$-space link, for $\bdve{s}=(s_1,s_2)$ and $\ve \in \bE_2,$ 
we have
\begin{eqnarray}
\label{maps2}
\begin{aligned}
\Phi^{K_1}_{00}=U^{H_L(s_{1}, s_{2})-H_L(\infty, s_{2})}, \quad &\Phi^{-K_1}_{00}=U^{H_L(-s_{1}, -s_{2})-H_L(\infty, -s_{2})}, \\
\Phi^{K_2}_{00}= U^{H_L(s_{1}, s_{2})-H_L(s_{1}, \infty)}, \quad &\Phi^{-K_2}_{00}=U^{H_L(-s_{1}, -s_{2})-H_L(-s_{1}, \infty)}, \\
\Phi^{K_1}_{01}=U^{H_L(s_{1}, \infty)}=U^{H_{K_1}(s_{1}-l/2)}, \quad &\Phi^{-K_1}_{01}=U^{H_{K_1}(l/2-s_{1})}, \\
\Phi^{K_2}_{10}=U^{H_L(\infty, s_{2})}=U^{H_{K_2}(s_{2}-l/2)}, \quad &\Phi^{-K_2}_{10}=U^{H_{K_2}(l/2-s_{2})}. 
\end{aligned}
\end{eqnarray}
\end{lem}

\begin{remark}\label{rmk: even hom degrees}
     Since localisation is an exact functor,  we have $H_*(C_\ve(\bs)) \cong \F\llbracket U \rrbracket$ for any $\ve \in \bE_n $ and $ \bs \in \bH(L)$. Moreover $H_*(C_\ve(\bs))$ is concentrated in even homological gradings. This can be seen as follows. Since $\cCFL(L)$ is finitely generated over $F[W_1,Z_1,\dots, W_n,Z_n]$,  when $s_i\gg 0$ for $i=1,\dots,n$, the chain complex $C_\ve(\bs)$ is generated by $Z_i$ powers of the free generators of $\cCFL(L)$. It follows that $C_\ve(\bs) \simeq C_{(1,\dots,1)}(\bs)$.  Viewed as a chain complex over $\F\llbracket U_1\rrbracket$, the complex $C_{(1,\dots,1)}(\bs)$ is graded chain homotopy equivalent to $\bCFm(S^3)$ by \cite[Theorem 4.4]{OSLinks}, and hence the generator of $H_*(C_\ve(\bs))$ lies in homological grading $0$.   Now, for every $\bs$, multiplication by $(Z_1 \cdots Z_n)^N$ yields an inclusion from $C_\ve(\bs)$ to $C_\ve(\bs')$, which induces an $\F\llbracket U \rrbracket$-equivariant map on homology, where $U_1$ acts as $U$. Furthermore, since $W_i Z_i = U_i,$ the induced map is non-trivial on homology. It follows that  the map induced by $(Z_1 \cdots Z_n)^N$ on homology is the multiplication by a power of $U$. For a  sufficiently large $N,$ since
    the generator of $H_*(C_\ve(\bs'))  \cong \F\llbracket U \rrbracket$ lies in the homological grading $0$,  it follows that the generator of $H_*(C_\ve(\bs)) \cong \F\llbracket U \rrbracket$ lies in homological grading $-2k$ for some $k\in \Z_{\geq 0}$. 
\end{remark}

\subsection{Zemke's reformulation of the link surgery formula} \label{subsec: zemke's reformulation}  In this section we review
 Zemke's reformulation of the link surgery formula using bordered modules \cite{zemke_bordered}, which allows us to generalise the link surgery formula to rational coefficients.
 
We first recall Zemke's definition of the type-$D$ link surgery module from \cite[Section 8.6]{zemke_bordered}.
Given an oriented $n$-component link $L \subset S^3$ and an integral framing $\Lambda,$ Zemke constructed a type-$D$ module $\cX_\Lambda(L)^{\cL}$ which encodes all the maps $\phi^{\vec{M}}_\ve$ in the link surgery complex.

% the link surgery complex.
% can be obtained as
% \[C_{\Lambda_1+\Lambda_2}(L) \simeq \cX_{\Lambda_1}(L)^{\cL} \hboxtimes {}_{\cL}[\cD_{\Lambda_2}]_{\F[U]},
% \]where ${}_{\cL}[\cD_{\Lambda_2}]_{\F[U]}$ is the type-$A$ bimodule of disjoint tori defined in Section \ref{subsec: solid_tori}.

Suppose $\{\bdy_1,\dots,\bdy_m\}$ is a free basis of $\cCFL(L)$.  The underlying space of  $ \cX_\Lambda(L)^{\cL}$  is   
\[ \text{Span}_{\F}\{\bdy_1,\dots,\bdy_m\} \otimes_{\F} \bdE_n.\] In other words,
for each $\ve \in \bE_n,$ and each free generator $\bdy$ of $\cCFL(L)$, we have a generator $\bdy^\ve$.
% Observe that $\cX(L)^{\cL} \cdot \bdI_\ve $ can be identified to $ \Se^{-1}\cCFL(L) = C_\ve.$ (before taking the completion).

We define the map $\delta^1$ as follows. Fix $\ve \in \bE_n,$ and let $\bdw^\ve = \bdy^\ve_{i}$ for some $i \in \{1,\dots,n\}$.
Let $N \subset L$ denote the sublink consisting of $K_i$ for which $\ve_i =0.$ 
For each oriented sublink $\vec{M}\subset N$ and $\ve'\in \bE_n$ such that $\ve' - \ve$ is the indicator function of $M$, we can write
$\phi^{\vec{M}}_{\varepsilon}(\bdw^\ve) = \sum_{j=1}^m b_j \bdy^{\ve'}_j$, where we are considering $\bdw^\ve$ as an element in $C_\ve$, and where $b_j \in S^{-1}_{\ve'}\F[Z_1,W_1,\dots,Z_n,W_n]=\bdI_{\ve'} \cdot \cL \cdot \bdI_{\ve'}$. 

Define an element  $t^{\vec{M}}\in \bdI_{\ve'} \cdot \cL \cdot \bdI_{\ve}$ given by $t^{\vec{M}} = \otimes_{i=1}^n t^{\vec{M}}_{i}$ where
\begin{equation} \label{eq: def_t_M}
    t^{\vec{M}}_{i} := \begin{cases}
        1 &\text{if $K_i$ is not in $M$;}\\
        \sigma_i &\text{if $K_i$ is oriented the same way in $\vec{M}$ as in $L$;}\\
        \tau_i &\text{if  $K_i$ is oriented the opposite ways in $\vec{M}$ as in $L$}.
    \end{cases}
\end{equation}

We declare $\delta^1(\bdw^\ve )$ to have the summand
 \[
  \sum_{j=1}^m \bdy^{\ve'}_j \otimes  b_j \cdot t^{\vec{M}}.
 \]

% Zemke proved several versions of pairing theorems for bordered modules over link surgery algebras in \cite{zemke_bordered, zemke_general}.

In the case of integral framings, the box tensor pairing of bordered modules recovers the Manolescu-Ozsv{\'a}th link surgery formula.
 In the following, $\hboxtimes$ denotes  a completed version of the box tensor product defined in Section \ref{subsec: type_A_algebra_and_bimodules}.  See \cite[Section 7.6]{zemke_bordered} for details. 
 % Roughly speaking, Zemke equips the underlying tensor product with a suitable topology and then takes its completion with respect to that topology.
\begin{prop}
[{\cite[Proposition 8.8]{zemke_bordered}}]\label{prop: integral_pairing}
Given a set of integral surgery framings $ \Lambda =(\lambda_1,\dots,\lambda_n)$, there is a chain isomorphism  
\begin{equation}
    C_{\Lambda}(L)_{\F[U_1,\dots,U_n]}  \cong   \cX_{\Lambda}(L)^\cL \hboxtimes {}_{\cL}[\cD_{(0,\dots,0)}]_{\F[U_1,\dots, U_n]}
    \end{equation}
where ${}_{\cL}[\cD_{(0,\dots,0)}]_{\F[U_1,\dots, U_n]}$ is the type-$AA$ bimodule of disjoint tori.
% \ds{the object on the right is not completed,right? it is only a module over $F[U]$ for now. By looking at Zemke's paper, it seems that the completion is happening in the box tensor (his box tensor has a hat on it, I don't know if it plays a role)}
\end{prop}
% \begin{remark}
%    In the above, we can equally well use $(0,\dots,0)$ framing on $\cX(L)^\cL$ and let ${}_{\cL}[\cD_\Lambda]_{\F[U]}$ supply the surgery framing. One readily checks that the resulting complex from the box tensor product is the same. 
% \end{remark}
Proposition \ref{prop: integral_pairing}, together with \cite{mo_linksurgery}, implies that $\bHFm(S^3_{\Lambda}(L))$ can be recovered as the homology of \(\cX_{\Lambda}(L)^\cL \hboxtimes {}_{\cL}[\cD_{(0,\dots,0)}]_{\F[U_1,\dots, U_n]}\). In fact, by \cite[Theorem 3.8]{zemke_general}, as chain complexes over $\F\llbracket U\rrbracket,$ we have
\[
 \cX_{\Lambda}(L)^\cL \hboxtimes {}_{\cL}[\cD_{(0,\dots,0)}]_{\F[U_1,\dots, U_n]} \simeq \bCFm(S^3_{\Lambda}(L))
\]
where $U$ acts by $U_1$ on the left hand side.

It is pointed out in \cite[Remark 1.8]{zemke_bordered} that this can be generalised to allow
for rational surgeries on links.
We have the following, as a result of  \cite{mo_linksurgery,zemke_bordered,zemke_general}.
\begin{thm}\label{thm: zemke_rational}
Given a set of  rational  surgery coefficients $ \Lambda =(\lambda_1,\dots,\lambda_n)$, there is a homotopy equivalence of chain complexes over $\F\llbracket U\rrbracket$
\begin{equation}\label{eq: rational_link_surgery_box}
      \cX_{(0,\dots,0)}(L)^\cL \hboxtimes {}_{\cL}[\cD_\Lambda]_{\F[U_1,\dots,U_n]} \simeq \bCFm(S^3_{\Lambda}(L))
    \end{equation}
    where $U$ acts by $U_1$ on the left hand side. 
\end{thm}
Since this statement does not appear in the literature, we include a proof at the end of this section for the reader's convenience.
% In the case of rational surgeries on a knot $K$, by construction
%  \[
%   \cD^{\cK}_{p/q} \hboxtimes {}_{\cK} \cX_0(K)  \simeq   \cX_0^{\cK}(K)  \hboxtimes {}_{\cK}\cD_{p/q}
%  \]
% recovers the rational surgery formula on a knot defined in \cite{OSrational11}.  While \cite[Section~18.2]{zemke_bordered} states this only for $K \subset S^3$,  the same construction applies to Morse surgery on knots in arbitrary closed, oriented three-manifolds, as in \cite{OSrational11}. 

We begin by fixing the notation and setup.
As observed in \cite{OSrational11},
a rational surgery on an oriented knot $K \subset S^3$ can be viewed as a Morse surgery on a link in some rational homology sphere.
Here, a Morse surgery means a Dehn surgery in which the meridian of each attached solid torus is identified with a longitude of the corresponding link component. We call the specified longitude the \emph{Morse framing}.
Explicitly,  the \((\frac{p}{q})\)-surgery on  $K\subset S^3$ is realised by a surgery with coefficient $d$ on the knot $K \# O_{q/r} \subset L(q,r).$
Here $d= \floor{\frac{p}{q}}$ and $\frac{p}{q}= d + \frac{r}{q}$. The knot
 $O_{q/r}$ is obtained from the Hopf link by performing $-\frac{q}{r}$ on one component, viewing the other component as a knot in $ L(q,r) = S^3_{-q/r}(U)$. Denote by $\lambda_d$ the Morse framing on $O_{q/r}$ given by the $d$ framed surgery.  See Figure \ref{fig:morseframed} for an example. Furthermore, we may associate to $(O_{q/r}, \lambda_d)$ 
   a linear plumbing\footnote{i.e. a linear chain of unknots pairwise linked as Hopf links.} $\bar J := O \cup J \subset S^3$, whose oriented are oriented arbitrarily and with framing $\Lambda=(d,a_1,\dots,a_m)$, where
 $$
\frac{p}{q}=d-\cfrac{1}{a_1-\cfrac{1}{\ddots-\cfrac{1}{a_m}}},   
$$
such that by performing surgeries on $J$ according to the framing \[\Lambda_J=(a_1,\dots,a_m),\]  the remaining component $O$ becomes  $O_{q/r}$ with the desired Morse framing $\lambda_d$. 

\begin{figure}[ht]
    \centering
    \includegraphics[width=0.35\linewidth]{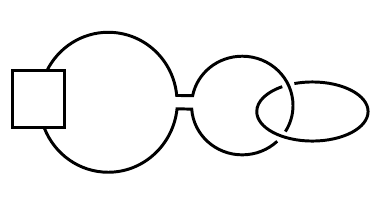}

    % overlay labels
    \begin{picture}(0,0)
    \put(-64,47){{\large $K$}}
        \put(3,22){$d$}
        \put(67,60){$-\frac{q}{r}$}
    \end{picture}
    \caption{The \(\frac{p}{q}\)-surgery on  $K \subset Y$ is realised by a Morse  surgery on the knot $K \# O_{q/r} \subset Y \# L(q,r).$}
    \label{fig:morseframed}
\end{figure}
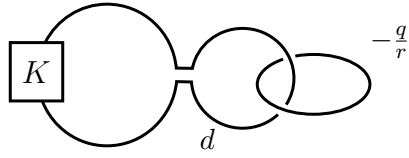
In \cite{zemke_general}, the definition of the type-$D$ module $\cX_\Lambda(L)^{\cL}$ was extended to links in general $3$-manifolds, and more generally to bordered manifolds $Y$ with torus is an invariant of  $Y$. As a consequence, we  $\cX_\Lambda(L)^{\cL}$ is natural with respect to performing Dehn surgeries, in the following sense:
 \begin{thm}[{\cite[Theorem 9.1]{zemke_general}}] \label{thm: dehn_surgery_cX}
     Suppose $(Y',L')$ is a link with Morse framing $\Lambda'$, obtained from a link $(Y,L)$ with Morse framing $\Lambda$ by performing Dehn surgery on one component of $L$ according to its framing. Then
     \[
     \cX_{\Lambda'}(L')^{\cL_{n-1}} \simeq \cX_{\Lambda}(L)^{\cL_{n}} \boxtimes {}_{\cK}\cD_0. 
     \]
 \end{thm}
% Moreover, by \cite[Theorem 10.8]{zemke_general}, both isomorphisms in Theorem \ref{thm: dehn_surgery_cX} and Proposition \ref{prop: integral_pairing} are graded.
Applying the above theorem to our setting, we have 
\begin{equation} \label{eq: linearplumbing_equivalence}
    \cX_{\Lambda}(\bar J)^{\cL_{m+1}} \boxtimes {}_{\cL_m}\cD_{(0,\dots,0)}  \simeq  \cX_{\lambda_d}(O_{q/r})^\cK \simeq \cD_{p/q}^\cK
\end{equation}
where $\cD_{p/q}^\cK$ is the type-$D$ module for rational framed solid torus defined in Section \ref{subsec: solid_tori}.
Denote the type-$D$ module on the left hand side of Equation~\eqref{eq: linearplumbing_equivalence} by $\cY_{\Lambda}(\bar J)^\cK$ and denote
\[
{}_{\cK}\cY_{\Lambda}(\bar J) := \cY_{\Lambda}(\bar J)^\cK \boxtimes {}_{\cK|\cK}[\bI^{\Supset}].
\]

Before proving Theorem~\ref{thm: zemke_rational} at the end of this section,
we briefly discuss the behaviour of the structure maps under the chain homotopy
equivalence in Equation~\eqref{eq: linearplumbing_equivalence} and prove a useful
lemma. This relationship will play an important role in our study of the
Spin$^c$ decomposition of the rational link surgery complex in
Section~\ref{subsec: rational_link_surgery_formula}.

By the proof of \cite[Corollary 9.2]{zemke_general}, we have
 \begin{align}
  \label{eq: cfko_q_r_idempotent0}  \cY_{\Lambda}(\bar J)^\cK \boxtimes {}_{\cK|\cK}[\bdI_0] &\simeq {}_{\F[W,Z]}\cCFK(O_{q/r}) \\
  \label{eq: cfko_q_r_idempotent1} \cY_{\Lambda}(\bar J)^\cK \boxtimes {}_{\cK|\cK}[\bdI_1] &\simeq {}_{\F[U,T,T^{-1}]}S_1^{-1}\cCFK(O_{q/r})
\end{align}
as type-$A$ modules,
where for $\nu \in \{0,1\}$, the bimodule ${}_{\cK|\cK}[\bdI_\nu]$ is obtained by restricting the identity bimodule ${}_{\cK|\cK}[\bI^{\Supset}]$ to idempotent $\bdI_\nu$.
 We denote the complex in the left hand side of Equations~\eqref{eq: cfko_q_r_idempotent0} and  \eqref{eq: cfko_q_r_idempotent1} by $\cY_0(\bar J)$ and $\cY_1(\bar J)$, respectively.  

For every $\bt \in $ Spin$^c(L(q,r))$, $\cCFK(O_{q/r},\bt)$ is freely generated by a single generator with trivial differential. Furthermore, the Alexander grading on $\cCFK(O_{q/r})$ takes values in $\frac{q-1}{2q}+\frac{1}{q}\Z$  and the Spin$^c$ structure of a generator is determined by the Alexander grading via the map
\begin{align*}
    G_{L(q,r),O_{q/r}}: \ \frac{q}{2}+\frac{1}{q}\Z &\to \Z/q\Z\\
   \frac{q-1}{2q}+\frac{k}{q} &\longmapsto k
\end{align*}
defined in \cite[Section 2.2]{OSrational11}\footnote{The Alexander grading determines a relative Spin$^c$ structure
$\underline{\operatorname{Spin}}^c(L(q,r),O_{q/r})$,
which determines a Spin$^c$ structure of $L(q,r)$ by \cite[Lemma 7.1]{OSrational11}. For a computation of the Alexander gradings of $O_{q/r}$ when $r=1$, see \cite[Example 5.15]{HL}. A formula for arbitrary $r$ follows from applying \cite[Equation (1.4)]{HL} recursively.}.
Hence there is an isomorphism of the underlying space
\[
\bdI_{\nu} \cdot  {}_{\cK}\cX_{\lambda_d}(O_{q/r})   \cong \bdI_{\nu} \cdot {}_{\cK}\cD_{p/q}
\]
 sending the unique generator of
$S^{-1}_{\nu}\cCFK(O_{q/r})$ in Alexander grading $\frac{k}{q}$ to
$\bdx_k^\nu$, and extending $S^{-1}_{\nu}\F[W,Z]$-equivariantly.  We equip ${}_{\cK}\cD_{p/q}$ with the Alexander grading via this isomorphism.
Let 
\[
g: {}_{\cK}\cY_{\Lambda}(\bar J) \to  {}_{\cK}\cD_{p/q}
\] denote the map obtained by postcomposing the chain homotopy equivalence in Equation~\eqref{eq: cfko_q_r_idempotent0} and \eqref{eq: cfko_q_r_idempotent1} with the isomorphism of the underlying space above. We have the following. Recall the notation $\bar J := O \cup J$.
\begin{lem} \label{lem: phiM_D_p_q}
    Suppose $\bdy \in C_{\ve}(\bs) \subset C_{\Lambda}(\bar J)$ for $\ve \in \bE_{m+1}$ and $\bs \in \bH(\bar J).$ Let $\vec{M} \subset \bar J$ be a sublink such that $\ve_i=0$ for each component in $M$. If $M \cap O = \emptyset,$ then $g\circ \phi^{\vec{M}}(\bdy) = 0.$ Otherwise $g\circ \phi^{\vec{M}}(\bdy) = m_2(t^{\vec{M} \cap O},g(\bdy)),$ where $t^{\vec{M}}$ is defined in Equation~\eqref{eq: def_t_M}.
\end{lem}
\begin{proof}
  We write $\vec{M} = \epsilon  O \cup \vec{M}(J)$ where $\vec{M}(J) = \vec{M} \cap J$ and $\epsilon \in \{0,1,-1\}$ records whether $O$ belongs to $M$ and, if so, whether the orientation of $O$ induced by $\vec{M}$ agrees with the orientation induced by $L$. We study three cases depending on the value of $\epsilon.$

If $\epsilon=0,$ then $\phi^{\vec{M}}$ induces an internal differential of either  $\cY_0(\bar J)$ or $\cY_1(\bar J)$. Since $\cCFK(O_{q/r})$ has trivial differential and the chain homotopy equivalences in Equation~\eqref{eq: cfko_q_r_idempotent0} and \eqref{eq: cfko_q_r_idempotent1} are chain maps, $g\circ \phi^{\vec{M}}(\bdy) = 0.$

If $\epsilon= \pm 1,$  $\phi^{\vec{M}}$ induces a map $m_2(t^{\vec{M}},-)$ from $\cY_0(\bar J)$ to $\cY_1(\bar J)$. As the elements in $ {}_{\cK}\cD_{p/q}$ are determined by their Alexander gradings, it suffices to determine the Alexander grading shift of $g\circ \phi^{\vec{M}}.$

Let $W_{\Lambda_J}(J)$ be the two-handle cobordism from $S^3$ to $S^3_{\Lambda_J}(J)=L(q,r)$. For $k=1,\dots,m$, let
$\Sigma_k \in H_2(W_{\Lambda_J}(J);\Q)$ denote the homology class represented by the surface obtained as followsby capping a Seifert surface for the $k$-th component of $J$ with the core of the corresponding two-handle. Let $\Sigma_0 \in H_2(W_{\Lambda_J}(J);\Z)$ be the class of a Seifert surface for multiple copies of $O_{q/r}\subset L(q,r)$ capped with disks in $S^3$, and define
\[
\hat\Sigma_0=\frac{\Sigma_0}{|\Sigma_0\cdot\mu|}\in H_2(W_{\Lambda_J}(J);\Q),
\]
where $\mu$ denotes the meridian of $O$. Set
$\Sigma=\hat\Sigma_0+\Sigma_1+\cdots+\Sigma_m$.

By \cite[Theorem 10.8]{zemke_general}, Equation~\eqref{eq: cfko_q_r_idempotent0} and  \eqref{eq: cfko_q_r_idempotent1} are absolutely Alexander graded if we endow $\cY_0(\bar J)$ and $\cY_1(\bar J)$ with the Alexander grading obtained by associating to $\bs\in \bH(O \cup J)$ the quantity
\begin{equation*}
    A(\bs)=s_0 + \frac{\langle c_1(\frak{z}^J_{\bs}), \hat\Sigma_0  \rangle - \Sigma \cdot \hat\Sigma_0}{2},
\end{equation*}
where  $\frak{z}^J_{\bs} \in $ Spin$^c(W_{\Lambda_J}(J))$ is characterised by
\begin{equation*}
    \frac{\langle c_1(\frak{z}^J_{\bs}),  \Sigma_k \rangle
    - \Sigma\cdot\Sigma_k}{2}
    = -2s_k
    \qquad k=1,\dots,m.
\end{equation*}
We have
\[
\hat\Sigma_0 \cdot \Sigma_j = \begin{cases}
    1 \quad &j=1 \\
    0  \quad &j>1
\end{cases} \qquad \Sigma_k \cdot \Sigma_j = \begin{cases}
    a_k \quad & k =j\\
    1 &|k-j|=1\\
    0 &|k-j|>1
\end{cases}.
\]
Write 
 $\hat\Sigma_0 = \sum_{k=1}^m  v_k \Sigma_k $
for $v_k \in \Q$, and note that, since $H_2(\partial W_{\Lambda_J}(J); \Q)$ is trivial, the homology class of $\hat\Sigma_0$ is determined by how it pairs with all the $\Sigma_i$. 
One readily checks that the above system on the left is   by \[v_k = (-1)^{k+1}\frac{D_{k+1}}{D_1}\] where $D_k$ denotes the numerator of the continued fraction $[a_k,\dots,a_m]$ for $k=1,\dots,m$ and $D_{m+1}=1.$ In particular, $D_1=q, D_2=r$ and $D_k$ satisfies the recursive relation
\begin{equation} \label{eq: recursive_D_i}
    D_k = a_{k} D_{k+1} - D_{k+2} \qquad k=1,\dots,m-1.
\end{equation}
We compute
\begin{align*}
    A(\bs)&=s_0 + \frac{\langle c_1(\frak{z}^J_{\bs}), \hat\Sigma_0  \rangle - \Sigma \cdot \hat\Sigma_0}{2}\\
    &=s_0 + \sum_{k=1}^m v_k \Big( \frac{\langle c_1(\frak{z}^J_{\bs}), \Sigma_k  \rangle - \Sigma \cdot \Sigma_k}{2} \Big)\\
    &=s_0 - \sum_{k=1}^m v_k s_k \\
    &= \frac{1}{D_1} \big( D_1 s_0 -D_2 s_1 +\dots + (-1)^m D_{m+1} s_m  \big).
\end{align*}
Since this relation is linear in $\bs$, it suffices to compute the Alexander grading shift associated to $\Lambda_j$ for $j=1,\dots,m+1$, where $\Lambda_j$ denotes the $j$-th row of the surgery matrix of $\bar J$.

We compute 
\begin{align*}
    A(\Lambda_1) = \frac{1}{D_1} \big( D_1 d  - D_2   \big) = \frac{1}{q} \big( q d - r   \big) = \frac{p}{q}
\end{align*}
and
\begin{align*}
    A(\Lambda_j) &= \frac{(-1)^j}{D_1} \big( D_{j-1} - a_{j-1}D_j + D_{j+1}   \big) = 0 \qquad j=2,\dots,m-1\\
    A(\Lambda_{m+1}) &= \frac{(-1)^{m+1}}{D_1} \big( D_{m} - a_{m}D_{m+1}  \big) = 0
\end{align*}
 by Equation~\eqref{eq: recursive_D_i}.

If $\epsilon = 1,$  $\Lambda_{\bar J, \vec{M}}$ is a linear sum of $\Lambda_j$ with $j\neq 1$. Therefore the Alexander grading shift of $g\circ \phi^{\vec{M}}$ is zero. For $\bdy \in C_{\Lambda}(\bar J)$ such that $g(\bdy)=\bdx^0_k,$ we have 
\[
g\circ \phi^{\vec{M}}(\bdy) = m_2(\sigma, g(\bdy)) = \bdx^1_k. 
\]
If $\epsilon = -1,$  the Alexander grading shift of $g\circ \phi^{\vec{M}}$ is $p/q$. For $\bdy \in C_{\Lambda}(\bar J)$ such that $g(\bdy)=\bdx^0_k,$ by comparing the Alexander grading, we have 
\[
g\circ \phi^{\vec{M}}(\bdy) = m_2(\tau, g(\bdy)) = T^{\floor{\frac{p+k}{q}} } \bdx^1_{k+p}
\]
 where the representative of $k\in \Z/q\Z$ is chosen such that $0\leq k < q.$  This completes the proof of the lemma.
\end{proof}

We end this section by proving Theorem \ref{thm: zemke_rational}.
 \begin{proof}[Proof of Theorem \ref{thm: zemke_rational}]
Given a set of rational surgery coefficients \(\Lambda=(\frac{p_1}{q_1},\dots,\frac{p_n}{q_n})\) on a link $L = K_1\cup \cdots \cup K_n \subset S^3$, we can represent 
 $S^3_\Lambda(L)$ as surgery with integral framing $\bar\Lambda$ on a link  $\bar L= L \# \bar J_1 \# \cdots \# \bar J_n \subset S^3$, where $\bar J_i = O_i \cup J_i$ is the linear plumbing with framing $\Lambda^i$, described in the discussion after the statement of Theorem~\ref{thm: zemke_rational}, and the connected sum is understood to be between $K_i \subset L$ and $O_i \subset \bar J_i$.  Thus by Proposition \ref{prop: integral_pairing}, we have
\begin{equation} \label{eq: link_surgery_linear_plumbing}
C_{\bar \Lambda}(\bar L)_{\F[U_1,\dots, U_{N}]} \cong \cX_{\bar \Lambda}(\bar L)^{\cL_N} \hboxtimes {}_{\cL_N}[\cD_{(0,\dots,0)}]_{\F[U_1,\dots, U_{N}]} 
\end{equation}
 and, since $\bar \Lambda$ is an integer vector, \cite{mo_linksurgery} implies that as a chain complex over $\F\llbracket U \rrbracket$ this is chain homotopy equivalent to $\bCFm(S^3_{\Lambda}(L))$.
 
 By combining Theorem~15.1 and Proposition~15.3 of \cite{zemke_bordered}, we have
 \begin{align}   
\cX_{\bar \Lambda}(\bar L)^{\cL_N} \hboxtimes  \big( {}_{\cK_1}[\cD_0]_{\F[U_1]} \otimes_\F \cdots  \otimes_\F {}_{\cK_n}[\cD_0]_{\F[U_n]}  \big) &\simeq   \Big(\cX_{(0,\dots,0)}(L)^{\cL_n} \otimes_\F \cX_{\Lambda^1}(\bar J_1)^{\cL_{m_1}} \otimes_\F \cdots \notag \\
 \cdots \otimes_\F \cX_{\Lambda^n}(\bar J_n)^{\cL_{m_n}}  \Big) &\hboxtimes \big( {}_{\cK_1|\cK_1}[\bI^{\Supset}]_{\F[U_1]}  \otimes_\F \cdots  \otimes_\F {}_{\cK_n|\cK_n}[\bI^{\Supset}]_{\F[U_n]}  \big) \label{eq: link_surgery_pairing}.
 \end{align} 
 Therefore 
  \begin{align}   
C_{\bar \Lambda}(\bar L)_{\F[U_1,\dots, U_{n}]} &\simeq \Big(\cX_{(0,\dots,0)}(L)^{\cL_n} \otimes_\F \cY_{\Lambda^1}(\bar J_1)^{\cK_1} \otimes_\F  \cdots \notag \\
& \qquad \cdots \otimes_\F \cY_{\Lambda^n}(\bar J_n)^{\cK_n}  \Big) \hboxtimes \big( {}_{\cK_1|\cK_1}[\bI^{\Supset}]_{\F[U_1]}  \otimes_\F \cdots  \otimes_\F {}_{\cK_n|\cK_n}[\bI^{\Supset}]_{\F[U_n]}  \big) \notag \\
&\simeq \Big(\cX_{(0,\dots,0)}(L)^{\cL_n} \otimes_\F \cD_{p_1/q_1}^{\cK_1} \otimes_\F  \cdots \otimes_\F \cD_{p_n/q_n}^{\cK_n}  \Big) \notag \\
&  \qquad \hboxtimes \big( {}_{\cK_1|\cK_1}[\bI^{\Supset}]_{\F[U_1]}  \otimes_\F \cdots  \otimes_\F {}_{\cK_n|\cK_n}[\bI^{\Supset}]_{\F[U_n]}  \big) \label{eq: link_surgery_equivalence} \\
% &=\cX_{(0,\dots,0)}(L)^{\cL_n} \hboxtimes \Bigg( \Big( \cY_{\Lambda^1}(\bar J_1)^{\cK_1} \boxtimes {}_{\cK_1|\cK_1}[\bI^{\Supset}]_{\F[U_1]} \Big) \otimes_\F \cdots  \notag \\
% & \qquad \cdots \otimes_\F \Big( \cY_{\Lambda^n}(\bar J_n)^{\cK_n} \boxtimes {}_{\cK_n|\cK_n}[\bI^{\Supset}]_{\F[U_n]} \Big)  \Bigg) \label{eq: link_surgery_equivalence} \\
&=\cX_{(0,\dots,0)}(L)^{\cL_n} \hboxtimes \Bigg( \Big( \cD_{p_1/q_1}^{\cK_1} \boxtimes {}_{\cK_1|\cK_1}[\bI^{\Supset}]_{\F[U_1]} \Big) \otimes_\F \cdots  \notag \\
& \qquad \cdots \otimes_\F \Big( \cD_{p_n/q_n}^{\cK_n} \boxtimes {}_{\cK_n|\cK_n}[\bI^{\Supset}]_{\F[U_n]} \Big)  \Bigg) \notag \\
&=\cX_{(0,\dots,0)}(L)^{\cL_n} \hboxtimes \big( {}_{\cK_1}[\cD_{p_1/q_1}]_{\F[U_1]} \otimes_\F \cdots \otimes_\F {}_{\cK_n}[\cD_{p_n/q_n}]_{\F[U_n]}  \big) \notag \\
&= \cX_{(0,\dots,0)}(L)^{\cL_n} \hboxtimes {}_{\cL_n} [\cD_{\Lambda}]_{\F[U_1,\dots,U_n]}. \notag
 \end{align} 
The first chain homotopy equivalence is by Equation~\eqref{eq: link_surgery_pairing} and the definition of $\cY_{\Lambda^i}(\bar J_i)^{\cK_i}$ given after Theorem \ref{thm: dehn_surgery_cX}. The second chain
homotopy equivalence is
by Equation~\eqref{eq: linearplumbing_equivalence}.
  \end{proof}

\subsection{The rational link surgery formula} \label{subsec: rational_link_surgery_formula} 
In this section we define the rational link surgery complex associated to a link $L\subset S^3$ endowed with a rational framing, and we show with Theorem \ref{thm: rational_surgery_chain} that it can be used to compute the Heegaard Floer homology of the corresponding surgery on $L$.

Let $L=K_1\cup\cdots \cup K_n \subset S^3$ be an $n$-component link, and fix a set of surgery coefficients \(\Lambda = (\frac{p_1}{q_1},\dots,\frac{p_n}{q_n})\) where $p_i \in \Z, q_i \in \Z_{>0} $ and $p_i$ and $q_i$ are coprime. Roughly speaking, the rational surgery complex is defined by assigning, for each lattice point  $\bs \in \bH(L)$ and each $\ve\in\bE_n$,  $q_1\cdots q_n$ copies of the same complex $C_\ve(\bs)$, rather than a single copy as in the integral surgery complex.

We start by defining another  lattice from $\bH(L)$ and a rational link surgery matrix.
\begin{defn}\label{def: tildeH}
We define $\tilde{\bH}(L)$ as
\begin{equation}\label{eq: tildeH}
\tilde{\bH}(L) : = \bH(L)\oplus \Z/q_1 \Z \oplus \cdots \oplus \Z/q_n \Z
\end{equation}
and we denote its elements by $(\bs,\bk):=(s_1,\dots,s_n,k_1,\dots,k_n)$, where $s_i \in \bH(L)$ and $k_i \in \Z/q_i\Z$ for each $i=1,\dots,n$.
We view $\tilde{\bH}(L)$ naturally as an affine $\Z^n$-lattice in $\R^n$. The identification is given by  a  map
\begin{align*}
f: \tilde{\bH}(L) &\to \R^n,
\end{align*}
defined up to an overall constant vector shift in  $\R^n$, as follows. For any $(\bs,\bk) \in \tilde{\bH}(L),$ we define $f(\bs,\bk)=(u_1,\dots,u_n)$, where
\begin{equation}\label{eq: identification_tilde_H}
    u_i = q_is_i + k_i,
\end{equation}
and each representative $k_i$ is chosen such that $0\leq k_i<q_i$.
%     Define the underlying space of the rational surgery complex to be \[C_\ve := \bigoplus_{(\bs,\bk)\in \tilde{\bH}(L)} C_\ve(\bs,\bk)\]
% for each $\ve\in \bE_n$, where
\end{defn}

Define a \emph{(rational) surgery matrix} $\tilde{\Lambda} = (\tilde{\Lambda}_{i,j})_{n\times n}$ by
\begin{align*}
\tilde{\Lambda}_{i,j} = \begin{cases}
    p_i,    &i=j \\
    q_i\operatorname{lk}(K_i,K_j), \qquad & i\neq j
\end{cases}
\end{align*}
whose column vectors we denote $\tilde{\Lambda}_j$ for $j=1,\dots,n.$

 % Recall that  there is a map $\phi^{\vec{M}}_{\varepsilon}: C_\ve(\bs) \rightarrow C_{\ve'}(\bs + \Lambda_{L,\vec{M}} )$ in \eqref{eq: phi_vecM}. 
 % Observe that $\bs'$ agrees with  $\bs + \Lambda_{L,\vec{M}}$ in every $i$-th component for which $\ve_i=0$. Since the isomorphism class of $C_{\ve}(\bs)$ does not depend on $\ve_i$ for which $\ve_i=1$,  it follows that $C_{\ve'}(\bs') \cong C_{\ve'}(\bs + \Lambda_{L,\vec{M}} ).$
 \begin{defn}\label{def: rational_link_surgery_complex}
For each $\ve\in \bE_n$ and $(\bs,\bk)\in \tilde{\bH}(L)$, we define
\[C_\ve(\bs,\bk) := C_\ve(\bs).
\]
Given an oriented sublink $\vec{M}\subset L$,  define  $\tilde{\Lambda}_{L,\vec{M}}$  to be the sum of $\tilde{\Lambda}_j$
for which the orientation of $K_j$   in $\vec{M}$ is opposite to its orientation in $L$. Suppose that $(\bs',\bk')$ satisfies $f(\bs',\bk') = f(\bs,\bk) + \tilde{\Lambda}_{L,\vec{M}}$  and that  $\ve' - \ve$ is the indicator function of $M$.
Set the map
\[
\tilde{\phi}^{\vec{M}}_{\varepsilon}: C_\ve(\bs,\bk) \rightarrow C_{\ve'}(\bs',\bk' )
\] 
to be the same as $\phi^{\vec{M}}_{\varepsilon}: C_\ve(\bs) \rightarrow C_{\ve'}(\bs + \Lambda_{L,\vec{M}} )$\footnote{Observe that $\bs'$ agrees with  $\bs + \Lambda_{L,\vec{M}}$ in every $i$-th component for which $\ve_i=0$. It follows that  $C_{\ve'}(\bs') \cong C_{\ve'}(\bs + \Lambda_{L,\vec{M}} )$.}.  
When $|M|=1,$ define $\tilde{\Phi}^{\vec{M}}_{\varepsilon}$ to be the  map induced by $\tilde{\phi}^{\vec{M}}_{\varepsilon}$ on homology.
The \emph{(rational) link surgery complex} is defined by
    \[C_{\Lambda}(L) : = \bigoplus_{\ve\in \bE_n} \prod_{(\bs,\bk)\in \tilde{\bH}(L)} C_\ve(\bs,\bk)\]
equipped with the differential that is the sum of $\tilde{\phi}^{\vec{M}}_{\varepsilon}$,  where $\vec{M}$ ranges over all oriented sublinks of $L.$
\end{defn}

% When $\tilde{\Lambda}$ is nondegenerate ,
The complex $C_{\Lambda}(L)$ splits over $\tilde{\bH}(L)/H(\tilde{\Lambda})$, where  $H(\tilde{\Lambda})$ denotes the $\Z$-span of the vectors $\tilde{\Lambda}_i.$ More precisely, if for each $\bdve{t} \in \tilde{\bH}(L)/H(\tilde{\Lambda})$ we define 
\[
C_{\Lambda}(L,\bdve{t}) := \bigoplus_{\ve\in \bE_n} \prod_{(\bs,\bk)\in \bdve{t}+H(\tilde{\Lambda})} C_\ve(\bs,\bk),
\]
then $C_{\Lambda}(L,\bdve{t})$ is a summand of $C_{\Lambda}(L)$.  

Figure~\ref{fig: rational_surgery_complex} shows an example of the simplified model of $C_{\Lambda}(L)$, obtained by taking the homology of each $C_{\ve}(\bs,\bk)$ and considering only induced maps with $|M|=1.$

Observe that in the case of  integral surgery coefficients, the above definitions recover the definitions of the integral link surgery complex.

% Shift the  homological grading on each $C_\ve(\bs)$ such that   $\tilde{\phi}^{\vec{M}}_{\varepsilon}$ lowers the  homological grading by $1$ on $C_\Lambda(L)$.
% Define $H_{\Lambda}(L)$ to be obtained from  $C_{\Lambda}(L)$ by replacing each $C_\ve(\bs)$ by $H_*(C_\ve(\bs))$ and $\tilde{\phi}^{\vec{M}}_{\varepsilon}$ by the induced map $\tilde{\Phi}^{\vec{M}}_{\varepsilon}$.

We are now ready to state and prove the rational link surgery formula. 
\begin{thm}\label{thm: rational_surgery_chain}
    Given an oriented $n$-component link \(L=K_1 \cup \cdots \cup K_n \subset S^3\) and a set of rational surgery coefficients $\Lambda$, we have 
     \[
    C_\Lambda(L) \simeq \bCFm(S^3_\Lambda(L))
    \]
    as chain complexes over $\F\llbracket U\rrbracket$. Moreover, for each $\bdve{t}\in$ Spin$^c(S^3_\Lambda(L)) \cong \tilde{\bH}(L)/H(\tilde{\Lambda})$, the above chain homotopy equivalence restricts to
     \[
    C_\Lambda(L,\bdve{t}) \simeq \bCFm(S^3_\Lambda(L),\bdve{t}).
    \]
    
    % Similarly,
    % \[
    % C_\Lambda(L)/U_1 \simeq \widehat{CF}(S^3_\Lambda(L))  \quad \text{and} \quad  C_\Lambda(L, \bdve{t})/U_1 \simeq \widehat{CF}(S^3_\Lambda(L),\bdve{t}).
    % \]
\end{thm}

\begin{proof}
We divide the proof into two parts. In the first part, we establish the chain
homotopy equivalence
\[
C_\Lambda(L) \simeq \bCFm(S^3_\Lambda(L)),
\]
while in the second part we prove the Spin$^c$ decomposition.

The first part of the argument essentially amounts to unpacking the definitions
appearing in \eqref{eq: rational_link_surgery_box}.

By Theorem \ref{thm: zemke_rational}, 
\begin{equation} \label{eq: tensor_pairing_in_proof_rational_surgery}
\cX_{(0,\dots,0)}(L)^\cL \hboxtimes {}_{\cL}[\cD_\Lambda]_{\F[U_1,\dots,U_n]}
\end{equation}
is chain homotopy equivalent to $\bCFm(S^3_\Lambda(L))$ as chain complexes over $\F\llbracket U \rrbracket.$
We need to show that \eqref{eq: tensor_pairing_in_proof_rational_surgery} agrees with our definition of $C_\Lambda (L).$  Once the framing  $(0,\dots,0)$ on $L$ is understood, we drop it from the notation and write $ \cX(L)^\cL$ for simplicity.

Suppose $\{\bdy_1,\dots,\bdy_m\}$ is a free basis of $\cCFL(L)$.
Before completion, the underlying space of 
$\cX_{(0,\dots,0)}(L)^\cL \hboxtimes {}_{\cL}[\cD_\Lambda]_{\F[U_1,\dots,U_n]}$ is the direct sum of
\[  \left( \cX(L)^\cL \cdot \bdI_\ve \right) \otimes_{\bdI_\ve} \left(  \bdI_\ve \cdot {}_{\cL}[\cD_\Lambda]_{\F[U_1,\dots,U_n]} \right).\] 
Each summand is naturally identified with 
\[
\bigoplus_{(\bs,\bk)\in \tilde{\bH}(\Lambda)} C_{\ve}(\bs,\bk)
\]
under the map
\[
\bdy^\ve \otimes   a_1 \bdx^{\ve_1}_{k_1} \otimes \cdots \otimes a_n \bdx^{\ve_n}_{k_n} \longmapsto a_1 \cdots a_n \bdy^\ve \otimes   \bdx^{\ve_1}_{k_1} \otimes \cdots \otimes  \bdx^{\ve_n}_{k_n} 
\]
where $a_1 \cdots a_n \bdy^\ve \in C_\ve(\bs)$ and $\bk = (k_1,\dots,k_n).$ We denote $\bdx^\ve_{\bk}:=\bdx^{\ve_1}_{k_1} \otimes \cdots \otimes  \bdx^{\ve_n}_{k_n}.$
% consists of elements of the form $\sum_{i=1}^m  \bdy^\ve_{i} \otimes   a_i \bdx^{\ve}_{\bk}$ \ds{a bit confused by this. is the sum over $k$ too? and why the coefficients $a_i$ do not depend on $k$?} where $a_i \in \bdI_\ve \cdot \cL \cdot \bdI_\ve$. This space is naturally identified with $q_1\cdots q_n$ copies of $C_{\ve}$ via the map \ds{I am also a bit confused by this map. Where do the $x_k$ go? I think it is just me, maybe I can ask you in person}
% \[
% \sum_{i=1}^m  \bdy^\ve_{i} \otimes   a_i \bdx^{\ve}_{\bk} \mapsto \sum_{i=1}^m  a_i \bdy^\ve_{i}.
% \]
% Under this identification, the subspace 
%  consisting of elements $\sum_{i=1}^m  \bdy^\ve_{i} \otimes   a_i \bdx^{\ve}_{\bk}$ such that $a_i \bdy^\ve_{i} \in  C_{\ve}(\bs)$ for some $\bs \in \bH(L)$ 
% is  identified with $ C_\ve( \bs,\bk)$. 
The $\F[U_1,\dots,U_n]$-module action on $\cX(L)^\cL \hboxtimes {}_{\cL}[\cD_\Lambda]_{\F[U_1,\dots,U_n]}$ is defined by 
\begin{equation*}
    \begin{tikzcd}
        [column sep=0.1 cm, row sep=0.4 cm]
\cX(L) \ar[ddd, dashed] & \otimes & \cD_\Lambda \ar[dd, dashed] & \F[U_1,\dots,U_n] \ar[ddl]
\\
    &  &  &
   \\
  &  & m_{0,1,1}\ar[d, dashed] &
\\
\cX(L) &\otimes & \cD_\Lambda
    \end{tikzcd}.
\end{equation*}
This is simply the multiplication by $U_i$ on the second tensor factor, which coincides with the usual $\F\llbracket U_1,\dots,U_n \rrbracket$-module action on $ C_\ve( \bs,\bk)$. Therefore $\cX(L)^\cL \hboxtimes {}_{\cL}[\cD_\Lambda]_{\F[U_1,\dots,U_n]}$ agrees with $C_\Lambda (L)$ as a $\F\llbracket U_1,\dots,U_n \rrbracket$-module.

 It remains to  show that the differential defined by \eqref{eq: rational_link_surgery_box} coincides with that of  $C_\Lambda(L)$.
It suffices to consider elements of the form $\bdw^\ve \otimes   a \bdx^{\ve}_{\bk}$ where $a\in \bdI_\ve \cdot \cL \cdot \bdI_\ve$ is a homogeneous element and $\bdw^\ve$ is  $\bdy^\ve_{i}$ for some $i \in \{1,\dots,n\}.$  Suppose $a \bdw^\ve  $ lies in $ C_{\ve}(\bs)$ for some $\bs\in \bH(L).$ By an abuse of notation, we write $\bdw^\ve \otimes   a \bdx^{\ve}_{\bk}$ also for the corresponding element in $ C_{\ve}(\bs,\bk).$

Let $N$ be the sublink of $L$ consisting of $K_i$ for which $\ve_i = 0.$ 
We will show that $\partial \big(\bdw^\ve \otimes   a \bdx^{\ve}_{\bk}\big)$ is equal to the sum of $\tilde{\phi}^{\vec{M}}_{\varepsilon}(\bdw^\ve  \otimes a \bdx^{\ve}_{\bk})$ where $\vec{M}$ ranges over all oriented sublinks of $N.$

The only non-trivial $m_{i,1,j}$ in ${}_{\cL}[\cD_\Lambda]_{\F[U_1,\dots,U_n]}$ with $i>0$ is  $m_{1,1,0}.$ Hence the only terms contributing to the differential have the form 
\begin{equation} \label{eq: rationa_hypercube_structure_map}
    \begin{tikzcd}
        [column sep=0.1 cm, row sep=0.4 cm]
\cX(L) \ar[d, dashed] & \otimes & \cD_\Lambda \ar[dd, dashed]
\\
  \delta^1 \ar[dd, dashed] \ar[drr] &  & 
   \\
  &  & m_{1,1,0}\ar[d, dashed]
\\
\cX(L) &\otimes & \cD_\Lambda
    \end{tikzcd}.
\end{equation}
We  compute $ \partial \big(\bdw^\ve \otimes   a \bdx^{\ve}_{\bk}\big)$ by plugging the algebraic output of $\delta^1(\bdw^\ve)$ into $m_{1,1,0} (-, a \bdx^{\ve}_{\bk})$.

Given $\vec{M}\subset N,$ let   $\ve' - \ve$ be the indicator function of $M$. We can write $\phi^{\vec{M}}_{\varepsilon}(\bdw^\ve) = \sum_{j=1}^m b_j \bdy^{\ve'}_j$.
By definition $\delta^1(\bdw^\ve)$ has a summand $\sum_{j=1}^m \bdy^{\ve'}_j \otimes b_j \cdot t^{\vec{M}}$
where   $b_j \in \bdI_{\ve'} \cdot \cL \cdot \bdI_{\ve'}$ and $t^{\vec{M}}$ is defined by \eqref{eq: def_t_M}. 
Let $\frak{d}(t^{\vec{M}}) \subset \{1,\dots,n\}$ denote the set of indices $i$ such that $\tau_i$ is a tensor factor  in $t^{\vec{M}}$.

 We  compute
\[m_{1,1,0}(b_j \cdot t^{\vec{M}},a \bdx^{\ve}_{\bk}) = b_j \cdot g^{t^{\vec{M}}}(a)  \cdot \Big(\prod_{i\in \frak{d}(t^{\vec{M}})}T_i^{\floor{\frac{k_i+p_i}{q_i}}} \Big) \cdot \bdx^{\ve'}_{\bk'} \]
where $g^{t^{\vec{M}}}$ is the algebra homomorphism defined in Section \ref{subsec: link_surgery_algebra}, $k'_i=k_i+p_i$ for $i\in \frak{d}(t^{\vec{M}})$ and $k'_i=k_i$ otherwise. We conclude that for each $\vec{M}\subset N$,  with indicator function $\ve' - \ve$, $\partial \big(\bdw^\ve \otimes  a \bdx^{\ve}_{\bk}\big)$ has a summand 
\begin{align}\label{eq: rational_surgery_proof_summand}
\sum_{j=1}^m \bdy^{\ve'}_j\otimes \left(\prod_{i\in \frak{d}(t^{\vec{M}})}T_i^{\floor{\frac{k_i+p_i}{q_i}}} \right) \cdot g^{t^{\vec{M}}}(a) \cdot b_j   \cdot   \bdx^{\ve'}_{\bk'}. 
\end{align}
This expression can be further simplified.
By \cite[Lemma 6.7]{zemke_bordered},  
\begin{equation*}
    \begin{aligned}
        \phi^{\vec{M}}_{\varepsilon}(a\bdw^\ve) &= g^{t^{\vec{M}}}(a) \cdot  \phi^{\vec{M}}_{\varepsilon}(\bdw^\ve) \\
        &= g^{t^{\vec{M}}}(a) \cdot \sum_{j=1}^m b_j \bdy^{\ve'}_j.
    \end{aligned}
\end{equation*}
If we write $k_i + p_i = r_i + h_i q_i$ where $h_i = \floor{\frac{k_i+p_i}{q_i}}$,  and $0\leq k_i, r_i< q_i$, then the summand  \eqref{eq: rational_surgery_proof_summand}  corresponds to the element 
\begin{equation} \label{eq: rational_surgery_proof_summand2}
    \Big(\prod_{i\in \frak{d}(t^{\vec{M}})}T_i^{h_i} \Big) \phi^{\vec{M}}_{\varepsilon}(a\bdw^\ve)  \in C_{\ve'}(\bs',\bk')
\end{equation} 
 for some $\bs' \in \bH(L)$.  

 On the other hand, by definition $\tilde{\phi}^{\vec{M}}_{\varepsilon}$ is identified with $\phi^{\vec{M}}_{\varepsilon}$ as maps between complexes and the action of $T_i$ only shifts the Alexander grading. Therefore,  to show that the expression in \eqref{eq: rational_surgery_proof_summand2} equals to $\tilde{\phi}^{\vec{M}}_{\varepsilon}(\bdw^\ve \otimes a \bdx^{\ve}_{\bk})$, we need only verify that $\tilde{\phi}^{\vec{M}}_{\varepsilon}(\bdw^\ve \otimes a \bdx^{\ve}_{\bk})$ also lies in $ C_{\ve'}(\bs',\bk').$
 In other words, we need to show that  $f(\bs',\bk')=f(\bs,\bk) + \tilde{\Lambda}_{L,\vec{M}}$.

Since the   framing on $L$ is $(0,\dots,0)$, we have that
 $\phi^{\vec{M}}_{\varepsilon}(a\bdw^\ve) $ lies in $ C_{\ve}(\bs + \sum_{i\in \frak{d}(t^{\vec{M}})}\Lambda^0_i)$ where $\Lambda^0_i=(\Lambda^0_{i,j})_{n\times 1}$ is given by 
\begin{align*}
\Lambda^0_{i,j} = \begin{cases}
    0,    &i=j \\
    \operatorname{lk}(K_i,K_j), \qquad & i\neq j.
\end{cases}
\end{align*}

% , where $\bs' = \bs + \sum_{i\in \frak{d}(t^{\vec{M}})}\Lambda'_i$ where $\Lambda'_i=(\Lambda'_{i,j})_{n\times 1}$ is given by 
% \begin{align*}
% \Lambda'_{i,j} = \begin{cases}
%     h_i,    &i=j \\
%     \operatorname{lk}(K_i,K_j), \qquad & i\neq j.
% \end{cases}
% \end{align*}
The map $T_i^{h_i}$ changes the $i$-th component of the Alexander grading by $h_i$.
Consider the $i$-th component of $f(\bs',\bk')-f(\bs,\bk)$. If $i\in \frak{d}(t^{\vec{M}})$, this is
\begin{align*}
q_i(s'_i - s_i) + (k'_i - k_i) &= q_i \Big( h_i + \sum_{\substack{j\in \frak{d}(t^{\vec{M}}) \\ j\neq i}}  \operatorname{lk}(K_i,K_j)\Big) + r_i -k_i \\
&= q_i h_i + r_i - k_i + q_i \sum_{\substack{j\in \frak{d}(t^{\vec{M}}) \\ j\neq i}}  \operatorname{lk}(K_i,K_j)   \\
&=p_i + q_i \sum_{\substack{j\in \frak{d}(t^{\vec{M}}) \\ j\neq i}}  \operatorname{lk}(K_i,K_j),
\end{align*}
whereas if $i\notin \frak{d}(t^{\vec{M}})$  this is 
\begin{align*}
q_i(s'_i - s_i) + (k'_i - k_i) &= q_i \sum_{\substack{j\in \frak{d}(t^{\vec{M}})}}  \operatorname{lk}(K_i,K_j).  
\end{align*}
Therefore $f(\bs',\bk')=f(\bs,\bk) + \tilde{\Lambda}_{L,\vec{M}}$. We have shown that the summand  \eqref{eq: rational_surgery_proof_summand} is equal to $\tilde{\phi}^{\vec{M}}_{\varepsilon}(\bdw^\ve \otimes a \bdx^{\ve}_{\bk})$. Repeating this argument for every oriented sublink $\vec{M} \subset N$ proves the desired statement. 
% \ds{this sentence on gradings has to be removed}The grading statement is  by \cite[Theorem 10.8]{zemke_general} together with Theorem  \ref{thm: zemke_rational}.
% Finally, observe that $C_{\Lambda}(L)$ naturally splits over $\tilde{\bH}(L)/H(\tilde{\Lambda})$,  into $|\tilde{\Lambda}|$ summands.
% Each summand is in one Spin$^c$ structure since elements inside are related by differential and $U_i$. Similarly, the set Spin$^c(S^3_\Lambda(L))$ has has cardinality  $|\tilde{\Lambda}|$, since it is in bijection with  $H_1(S^3_\Lambda(L))$, and the transpose of $\tilde{\Lambda}$ is a presentation matrix $H_1(S^3_\Lambda(L))$.
% Since for a rational homology sphere $M$, by \cite[Theorem 10.1]{OS3manifoldProperty}, $\operatorname{rank}\HF^-(M,\bt) = 1$ for each $\bt \in$ Spin$^c(M)$,
% it follows that there is a bijection between Spin$^c(S^3_\Lambda(L)) \cong H_1(S^3_\Lambda(L))$ and $\tilde{\bH}(L)/H(\tilde{\Lambda})$. 
\end{proof}
 In the following, we prove the part of the statement of Theorem \ref{thm: rational_surgery_chain} regarding the splitting over the Spin$^c$ structures.
\begin{proof}[Proof of Theorem \ref{thm: rational_surgery_chain}, continued]
Recall from the proof of Theorem \ref{thm: zemke_rational} that we represent $S^3_{\Lambda}(L)$ as a surgery with integral framing $\bar \Lambda$ on the link \[\bar L = L \# \bar J_1 \# \cdots \bar J_n \] where $\bar J_i = O_i \cup J_i$ is a linear plumbing with framing $\Lambda^i$ and the connected sum is between $K_i \subset L$ and $O_i \subset \bar J_i.$ We equip $L$  with the framing  $\Lambda^0=(0,\dots,0).$

Given any $\bs \in \bH(\bar L),$
by \cite{mo_linksurgery}, we know that
\[
\bCFm(S^3_{\Lambda}(L),[\bs]) \simeq C_{\bar \Lambda}(\bar L, [\bs]) 
\]
where $[\bs]$ denotes the equivalence class of $\bs$ in $\bH(\bar L)/H(\bar \Lambda) \cong$ Spin$^c(S^3_{\Lambda}(L)).$ We will define a map
\begin{align*}
    \rho\colon \bH(\bar L) &\to \tilde{\bH}(L)\\
    \bs &\mapsto (\bt,\bk)
\end{align*}
which induces a bijection
\[
\rho_*: \bH(\bar L)/H(\bar\Lambda)
\to
\tilde{\bH}(L)/H(\tilde{\Lambda}),
\]
such that the chain homotopy equivalence between
$C_{\bar\Lambda}(\bar L)$ and $C_{\Lambda}(L)$ in
Equation~\eqref{eq: link_surgery_equivalence} restricts to a chain homotopy equivalence
\begin{equation}\label{eq: spinc_decom_claim}
    C_{\bar\Lambda}(\bar L,[\bs])
    \simeq
    C_{\Lambda}(L,\rho_*[\bs]).
\end{equation}
We recall that by Equation~\eqref{eq: link_surgery_equivalence}, both $C_{\bar \Lambda}(\bar L) $  and $  C_{\Lambda}( L)$ are chain homotopy equivalent to an intermediate object
\begin{align}
Y_{\bar \Lambda}(\bar L) &:=\cX_{(0,\dots,0)}(L)^{\cL_n} \hboxtimes \Big( {}_{\cK_1}[\cY_{\Lambda^1}(\bar J_1)] \otimes_\F  \cdots  
 \otimes_\F {}_{\cK_n}[\cY_{\Lambda^n}(\bar J_n)]  \Big)
 %\\  = &\cX_{(0,\dots,0)}(L)^{\cL_n} \hboxtimes \Bigg( \Big( \cY_{\Lambda^1}(\bar J_1)^{\cK_1} \boxtimes {}_{\cK_1|\cK_1}[\bI^{\Supset}] \Big) \otimes_\F \cdots  
 % \cdots \otimes_\F \Big( \cY_{\Lambda^n}(\bar J_n)^{\cK_n} \boxtimes {}_{\cK_n|\cK_n}[\bI^{\Supset}] \Big) \Bigg)   \notag 
\end{align}
Denote the chain homotopy equivalences in  Equation~\eqref{eq: link_surgery_equivalence} by 
\[h: C_{\bar \Lambda}(\bar L) \to Y_{\bar \Lambda}(\bar L) \quad \text{and} \quad g: Y_{\bar \Lambda}(\bar L) \to C_{\Lambda}( L) \]
and define the map $h'$ and $g'$ to be the homotopy inverses of $h$ and $g$, respectively.
We have $g=\id_{\cX(L)} \otimes g_1 \otimes \cdots \otimes g_n$ where, 
for $i=1,\dots,n$, we denote by $g_i$ the chain homotopy equivalence given by
\[
{}_{\cK_i}[\cY_{\Lambda^i}(\bar J_i)] \simeq  {}_{\cK_i}[\cX_{\lambda_d^i}(O_{q_i/r_i})] \cong {}_{\cK_i}[\cD_{p_i/q_i}].
\]
Similarly, $g'=\id_{\cX(L)} \otimes g'_1 \otimes \cdots \otimes g'_n$ where each $g'_i$ is a homotopy inverse of $g_i$.

A key feature of the maps $h$ and $h'$ is that they preserve the Alexander grading, in the following sense.

As explained in \cite[Section~12.2]{zemke_bordered}, $h$ can be interpreted
as decomposing the hypercube $C_{\bar\Lambda}(\bar L)$ into subcubes according
to the values of the coordinates of $\bar\varepsilon\in\bE_N$ corresponding to
the components $K_i\subset L$. More precisely, the underlying module of
$C_{\bar\Lambda}(\bar L)$ is the tensor product of the modules underlying
$n+1$ hypercubes, namely $C_{\Lambda^0}(L)$ and
$C_{\Lambda^i}(\bar J_i)$ for $i=1,\dots,n$. Since the Alexander grading is additive under the tensor product, $\phi^{\vec{M}}$ has the same Alexander grading shift as
\begin{equation} \label{eq: map_decomp}
 \psi_0^{\vec{M}} \otimes \psi_1^{\vec{M}} \otimes \psi_1^{\vec{M}} \otimes  \cdots \otimes \psi_n^{\vec{M}}.
\end{equation}
Here, we set
\[
\psi_0^{\vec{M}} = \begin{cases}
   \phi^{L, \vec{M}\cap L} \quad &M \cap L \neq \emptyset\\
    \id_{C_{\Lambda^0}(L)}  &M \cap L = \emptyset 
\end{cases} 
\]
where
$\phi^{L, \vec{M}\cap L}$ denotes the hypercube map in $C_{\Lambda^0}(L)$  corresponding to the sublink $\vec{M}\cap L$. Similarly, for $i=1,\dots,n$, we set
\[
\psi_i^{\vec{M}} = \begin{cases}
    \phi^{\bar J_i, \vec{M}\cap \bar J_i} \quad &M \cap \bar J_i \neq \emptyset\\
    \id_{C_{\Lambda^i}(\bar J_i)}  &M \cap \bar J_i = \emptyset 
\end{cases} 
\]
where
$\phi^{\bar J_i, \vec{M}\cap \bar J_i}$ denotes the hypercube map in $C_{\Lambda^i}(\bar J_i)$ corresponding to the sublink $\vec{M}\cap \bar J_i.$ 

By the proof of \cite[Theorem~12.1]{zemke_bordered}, we have that 
the maps $h$ and $h'$ arise
from homotopies between $\phi^{\vec{M}}$ and the map in Equation~\eqref{eq: map_decomp}. 

The maps $g$ and $g'$ also interact with the Alexander grading in a predictable way.
 Suppose $\bdy_i\in C_{\Lambda^i}(\bar J_i)$ has Alexander grading
$(s(O_i),\bs(J_i)) \in \bH(\bar J_i)$.
By the proof of Lemma \ref{lem: phiM_D_p_q}, 
each term in $g_i(\bdy_{i})$, thought as as an element of  $S^{-1}_{\nu}\cCFK(O_{q_i/r_i})$ with $\nu\in\{0,1\}$, has Alexander grading 
 \begin{equation} \label{eq: spinc_decomp_A_O_q_r}
     s(O_i)-\bdve{v}^i \cdot \bs(J_i)
 \end{equation}
 for some vector $\bdve{v}^i$ determined by $p_i/q_i.$

We now define the map $\rho$.
 Take $\bdy \in C_{\bar \Lambda}(\bar L)$ with Alexander grading $\bs$ such that $h(\bdy)\ne 0$. 
Each term in $h(\bdy)$ is of the form  
\[
\bdw \otimes \bdy_{1} \otimes \cdots \otimes \bdy_{n},
\]
where $\bdw $  is a free generator of $ S_{\ve}^{-1}\cCFL(L)$ for some $\ve \in \bE_n$  with Alexander grading $\bs_{\bdw} \in \bH(L)$ and   $\bdy_{i}\in {}_{\cK_i}\cY_{\Lambda^i}(\bar J_i)$  with Alexander grading $\bs_{\bdy_i} = (s(O_i),\bs(J_i)) \in \bH(\bar J_i)$ for $i=1,\dots,n$. We have
\begin{align}
\label{eq: spinc_decomp_Alex_K_i}    \pi_{\bar L, K_i}(\bs) &= \pi_{L, K_i}(\bs_{\bdw}) + \pi_{\bar J_i, O_i}(\bs_{\bdy_i}) \\
\notag    \pi_{\bar L, J_i}(\bs) &=  \pi_{\bar J_i, J_i}(\bs_{\bdy_i})
\end{align}
 for $i=1,\dots,n$, where $\pi_{L, L'}: \bH(L) \to  \bH(L')$ is the reduction map defined in Section \ref{subsec: H-function}.

Next, apply the map $g$, and note that the Spin$^c$ structure $k_i \in \Z/q\Z  $ of $g_i(\bdy_i)$ depends only on $\bs$. In fact, $\bs(J_i)$ is determined by $\bs$, and Equation~\eqref{eq: spinc_decomp_A_O_q_r} implies that $s(O_i)$ only changes the Alexander grading by an integer, which has no effect under the map
\[
G_{L(q,r),O_{q/r}}: \frac{q_i-1}{2q_i}+ \frac{k_i}{q_i} \longmapsto k_i,
\]
from the set of all Alexander gradings of $\cCFK(O_{q_i/r_i})$ to Spin$^c(L(q_i,r_i)) \cong \Z/q_i\Z$.
Moreover, by Equation~\eqref{eq: spinc_decomp_A_O_q_r} and \eqref{eq: spinc_decomp_Alex_K_i}, each term in $g\circ h(\bdy)$ viewed as an element in 
\[
S^{-1}_\ve \cCFL(L \# O_{q_1/r_1} \# \cdots \# O_{q_n/r_n}) \cong S^{-1}_\ve \Big( \cCFL(L) \otimes \cCFK(O_{q_1/r_1}) \otimes \cdots \otimes \cCFK(O_{q_n/r_n}) \Big),
\]
has the same Alexander grading. Therefore, we may define the map $\rho$ by 
\[
\rho: \bs \mapsto (\pi_{\bar L,L}(\bs),\bk)
\]
where $\bk =(k_1,\dots,k_n)$.

To show that $\rho$ induces a map from $ \bH(\bar L)/H(\bar\Lambda)$ to $\tilde{\bH}(L)/H(\tilde{\Lambda})$, we show that for any  arbitrarily oriented non-trivial sublink $\vec{M}\subset \bar L$, we have 
\begin{equation}
    \label{eq: spinc_decomp_coset}\rho(\bs+ \Lambda_{\bar L, \vec{M}}) = \rho(\bs) + \tilde{\Lambda}_{L,\vec{M}\cap L}.
\end{equation}
If $M \cap L = \emptyset$, then by the proof of Lemma \ref{lem: phiM_D_p_q} the map $\psi_i^{\vec{M}}$ induces an internal differential of ${}_{\cK_i}[\cY_{\Lambda^i}(\bar J_i)]$, which preserves the Spin$^c$ structure $k_i$; moreover, the Alexander grading shift of $\psi_i^{\vec{M}}$ is zero. Therefore $\rho(\bs+ \Lambda_{\bar L, \vec{M}}) = \rho(\bs).$
 We now assume that
$M\cap L\neq \emptyset$.

As before, let $\bdy \in C_{\bar \Lambda}(\bar L)$ be an element with Alexander grading $\bs$ and let
$\bdw \otimes \bdy_{1} \otimes \cdots \otimes \bdy_{n}$
be a term in $h(\bdy)$.

We compute $\tilde{\phi}^{\vec{M}\cap L}\big(g(\bdw \otimes \bdy_{1} \otimes \cdots \otimes \bdy_{n})\big)$ using Equation~\eqref{eq: rationa_hypercube_structure_map}. In $\cX_{\Lambda^0}(L)^{\cL}$, the terms in $\delta^1(\bdw)$ whose coefficients contain $t^{\vec{M}\cap L}$  are defined using  $\phi^{L, \vec{M}\cap L}(\bdw)$. 

By Lemma \ref{lem: phiM_D_p_q}, for  $i \in \{1,\dots,n\}$ such that $M \cap K_i \neq \emptyset,$ 
\[ g_i \circ \phi^{\bar J_i, \vec{M}\cap \bar J_i}(\bdy_i) = m_{1,1,0}(t^{\vec{M}\cap K_i}, g_i(\bdy_i)). \] 
We compute
\begin{align*}
   g \Big( \psi_0^{\vec{M}}(\bdw) \otimes \psi_1^{\vec{M}}(\bdy_1) \otimes  \cdots \otimes \psi_n^{\vec{M}}(\bdy_n) \Big) 
  & = \psi_0^{\vec{M}}(\bdw) \otimes g_1\big(\psi_1^{\vec{M}}(\bdy_1)\big) \otimes  \cdots \otimes g_n\big(\psi_n^{\vec{M}}(\bdy_n)\big)\\
    &= \psi_0^{\vec{M}}(\bdw) \otimes m_{1,1,0}\big( t^{\vec{M}\cap L}, g_1(\bdy_1) \otimes  \cdots \otimes  g_n(\bdy_1) \big) \\
    &=\tilde{\phi}^{\vec{M}\cap L}\big(\bdw \otimes g_1(\bdy_{1}) \otimes \cdots \otimes g_n(\bdy_{n})\big)\\
  &=  \tilde{\phi}^{\vec{M}\cap L}(g( \bdw \otimes \bdy_{1} \otimes \cdots \otimes \bdy_{n})).
\end{align*}
The first and the last equality follow from the definition of $g$. For the second equality, we partition the indices $i\in \{1,\dots,n\}$ into two sets according to whether $M \cap K_i = \emptyset,$ and use the fact that $m_{2}$ on ${}_{\cL}[\cD_\Lambda]$ is the tensor product of $m_{2}$ of each ${}_{\cK_i}[\cD_{p_i/q_i}]$. The third equality follows from Equation~\eqref{eq: rationa_hypercube_structure_map}. 

 The previous computation shows that $g\circ h (\phi^{\vec{M}}(\bdy))=\tilde{\phi}^{\vec{M}\cap L}(g\circ h(\bdy))$, and as a consequence that $\rho(\bs+ \Lambda_{\bar L, \vec{M}}) = \rho(\bs) + \tilde{\Lambda}_{L,\vec{M}\cap L}$.
Therefore $\rho$ induces a map $\rho_*$ from $ \bH(\bar L)/H(\bar\Lambda)$ to $\tilde{\bH}(L)/H(\tilde{\Lambda})$.

Since $\rho$ is surjective, $\rho_*$ is surjective. To prove that $\rho_*$ is injective, we would like to show that 
  $\rho_*([\bs]) = \rho_*([\bs'])$ implies $[\bs]=[\bs'].$ By Equation~\eqref{eq: spinc_decomp_coset}, it suffices to show that 
  if $\rho(\bs) = \rho(\bs')=(\bt,\bk)$ for some $(\bt,\bk) \in \tilde{\bH}(L)$, then $[\bs]=[\bs'].$

 As discussed earlier, $(\bt,\bk)$ represents an Alexander grading in
 \[\cCFL(L) \otimes \cCFK(O_{q_1/r_1}) \otimes \cdots \otimes \cCFK(O_{q_n/r_n}).\]
 Since both $g$ and $h$ preserve the Alexander grading of $\cCFL(L)$, it follows that  $\pi_{\bar L, L}(\bs) = \pi_{\bar L, L}(\bs')$.
On the other hand, for each $J_i$,  $\bH(J_i)/H(\Lambda_{J_i}) \cong \Z/q_i\Z$ where $H(\Lambda_{J_i})$ is spanned by the column vectors of  the surgery matrix $\Lambda_{J_i}$ of $J_i$.
Moreover, by Equation~\eqref{eq: spinc_decomp_A_O_q_r}, if $\vec{M}\subset \bar J_i$ contains the first component of $J_i$ in the opposite orientation,  $\phi^{\bar J_i,\vec{M}}$ increases the $i$-th component Alexander grading of $\cCFL(L)$ by $1$.  Since $\pi_{\bar L, L}(\bs) = \pi_{\bar L, L}(\bs')$, 
it follows that $\pi_{\bar L, J_i}(\bs) - \pi_{\bar L, J_i}(\bs') \in H^\circ(\Lambda_{J_i}),$  where $H^\circ(\Lambda_{J_i})\subset H(\Lambda_{J_i})$ is spanned by the second through $m$-th column vectors of $\Lambda_{J_i}$.
Therefore $[\bs]=[\bs'],$ as desired.
This completes the proof. 
\end{proof}

\subsection{A simplified model of $C_{\Lambda}(L)$} \label{subsec: simplified model}
Following \cite{yajing_lspace}, 
we define a simplified model  $H_{\Lambda}(L)$ -- referred to in \cite{yajing_lspace} as the perturbed complex -- of $C_{\Lambda}(L)$ as follow. 
Recall from Definition \ref{def: rational_link_surgery_complex} that when $|M|=1,$
\[
\tilde{\Phi}^{\vec{M}}_{\varepsilon}: H_*(C_\ve(\bs,\bk)) \rightarrow H_*(C_{\ve'}(\bs',\bk' ))
\]
 are the maps induced by $\tilde{\phi}^{\vec{M}}_{\varepsilon}$ on homology,
where  $f(\bs',\bk') = f(\bs,\bk) + \tilde{\Lambda}_{L,\vec{M}}$.
% Here   $f$ is the map from Definition \ref{def: tildeH} that identifies $\tilde{\bH}(L)$ with $\bH(L).$ 
\begin{figure}\captionsetup{width=\textwidth}
 \begin{tikzpicture}[scale=1.2]
 \begin{scope}[on background layer]
   \draw[fill=green!10!white, draw=none]
    (-4.2,-3.2) rectangle (4.2,3.2);
 \foreach \i in {0,1,2,3}
 {\draw[thin, black!20!white] (-4.2,{\i*2-3}) -- (4.2,{\i*2-3}); }
 \foreach \j in {-1,0,1}
 {\draw[thin, black!20!white] ({\j*4},-3.2) -- ({\j*4},3.2);}
 \end{scope}
 \begin{scope}
      \foreach \i in {-1,0,1}
    \foreach \j in {0,1}
    \foreach \k in {0,1}
     { \draw[very thick, gray!30!white, fill=red!10!white] ({-3+4*\j+2*\k},{2*\i}) circle (0.8);
     \node (\i\j\k00) at ({-3+4*\j+2*\k+0.4},{2*\i+0.4}) {{\tiny $00$}};
     \node (\i\j\k01) at ({-3+4*\j+2*\k-0.2},{2*\i+0.2}) {{\tiny $01$}};
     \node (\i\j\k10) at ({-3+4*\j+2*\k+0.2},{2*\i-0.2}) {{\tiny $10$}};
     \node (\i\j\k11) at ({-3+4*\j+2*\k-0.4},{2*\i-0.4}) {{\tiny $11$}};
       \draw[-stealth, shorten <=-2pt, shorten >=-2pt] (\i\j\k00)  to node[] {} (\i\j\k01);
        \draw[-stealth, shorten <=-2pt, shorten >=-2pt] (\i\j\k00)  to node[] {} (\i\j\k10);
         \draw[-stealth, shorten <=-2pt, shorten >=-2pt] (\i\j\k01)  to node[] {} (\i\j\k11);
          \draw[-stealth, shorten <=-2pt, shorten >=-2pt] (\i\j\k10)  to node[] {} (\i\j\k11);
     }
      % \foreach \j in {-6,6}
      % {\fill[gray!40] (\j+-2,1.5) rectangle (\j+1.5,0.5);
      % \draw (\j+0.5,0.5) to (\j+1.5,0.5);
      % } 
      \end{scope}      
       \draw[-stealth, shorten <=-4pt, shorten >=-4pt] (-10000)  to node[pos=0.5, fill=red!10!white, inner sep=1.5pt] {{\ }} node[pos=0.53, fill=red!10!white, inner sep=1.5pt] {{\ }} node[pos=0.56, fill=red!10!white, inner sep=1.5pt] {{\ }} node[pos=0.6, fill=red!10!white, inner sep=2.8pt] {{\ }} node[pos=0.64, fill=red!10!white, inner sep=1.8pt] {{\ }} (01001);
       \node () at (-0.7,-0.58) {{\tiny $\Phi^{-K_2}$}};
       \draw[-stealth, shorten <=-4pt, shorten >=-4pt] (-10010)  to  node[pos=0.7, fill=green!10!white, inner sep=1pt, xshift=2pt, yshift=-2pt] {{\tiny $\Phi^{-K_2}$}} (01011);
       \draw[-stealth, shorten <=-4pt, shorten >=-4pt] (-10001)  to node[pos=0.1, fill=red!10!white, inner sep=2.8pt] {{\ }}  node[pos=0.13, fill=red!10!white, inner sep=2.8pt] {{\ }}  node[pos=0.15, fill=red!10!white, inner sep=2.5pt] {{\ }} node[pos=0.78, fill=green!10!white, inner sep=5pt] {{\ }} (00111);
      \node () at (-2.1,-0.75) {{\tiny $\Phi^{-K_1}$}};
       \draw[-stealth, shorten <=-4pt, shorten >=-4pt] (-10000)  to  node[midway, fill=green!10!white, inner sep=1pt, xshift=1pt, yshift=-1pt] {{\tiny $\Phi^{-K_1}$}} (00110);  
         \draw[-stealth, shorten <=-4pt, shorten >=-4pt] (01000)  to node[pos=0.24,fill=green!10!white, inner sep=3pt] {{\ }} (11110);
         \node () at (2.2,0.8) {{\tiny $\Phi^{-K_1}$}};
       \draw[-stealth, shorten <=-4pt, shorten >=5pt] (01001)  to node[midway, fill=green!10!white, inner sep=1pt] {{\tiny $\Phi^{-K_1}$}} ($(11111)+(2 pt,0)$);
         \draw[-stealth, shorten <=-4pt, shorten >=-4pt] (00100)  to node[pos=0.3, fill=green!10!white, inner sep=1pt] {{\tiny $\Phi^{-K_2}$}} (11101);
       \draw[-stealth, shorten <=-4pt, shorten >=5pt] (00110)  to node[pos=0.4, fill=red!10!white, inner sep=2.5pt] {{\ }} node[pos=0.36, fill=red!10!white, inner sep=2.4pt] {{\ }} ($(11111)+(0,1pt)$); 
       \node () at (0.82,0.52) {{\tiny $\Phi^{-K_2}$}};
    
       \node () at (-2.4,-1.81) {{\tiny $\Phi^{K_1}$}};
       \node () at (-2.9, -1.5) {{\tiny $\Phi^{K_2}$}};
       \node () at (-3.55,-2.05) {{\tiny $\Phi^{K_1}$}};
       \node () at (-2.9, -2.5) {{\tiny $\Phi^{K_2}$}};
       \node () at (-3.6, -2.9) {{\tiny $(\begin{pmatrix} s_1 \\ s_2  \end{pmatrix},0)$}};
       \node () at (-1.6, -2.9) {{\tiny $(\begin{pmatrix} s_1 \\ s_2  \end{pmatrix},1)$}};
        \node () at (0.6, -2.9) {{\tiny $(\begin{pmatrix} s_1 +1 \\ s_2  \end{pmatrix},0)$}};
       \node () at (2.6, -2.9) {{\tiny $(\begin{pmatrix} s_1 + 1\\ s_2  \end{pmatrix},1)$}};

\node () at (-3.2, 0.95) {{\tiny $(\begin{pmatrix} s_1 \\ s_2 + 1  \end{pmatrix},0)$}};
       \node () at (-1.2, 0.95) {{\tiny $(\begin{pmatrix} s_1 \\ s_2 + 1  \end{pmatrix},1)$}};
        \node () at (1.7, -0.8) {{\tiny $(\begin{pmatrix} s_1 +1 \\ s_2 +1  \end{pmatrix},0)$}};
       \node () at (3.7, -0.8) {{\tiny $(\begin{pmatrix} s_1 + 1\\ s_2 + 1 \end{pmatrix},1)$}};

       \node () at (-3, 3.1) {{\tiny $(\begin{pmatrix} s_1  \\ s_2 +2  \end{pmatrix},0)$}};
       \node () at (-1, 3.1) {{\tiny $(\begin{pmatrix} s_1 \\ s_2 + 2 \end{pmatrix},1)$}};
        \node () at (1, 3.1) {{\tiny $(\begin{pmatrix} s_1 +1 \\ s_2 +2 \end{pmatrix},0)$}};
       \node () at (3, 3.1) {{\tiny $(\begin{pmatrix} s_1 + 1\\ s_2+2  \end{pmatrix},1)$}};

 \node () at (-5.2, 3) {{\tiny $\operatorname{lk}(K_1,K_2)=1$}};
         \node () at (-5.3, 2.5) {{\tiny $\tilde{\Lambda} = \begin{pmatrix} 1 & 2\\ 1 & 1  \end{pmatrix}$}};         
\end{tikzpicture}
\caption{The simplified model  $H_{\Lambda}(L)$ of the link surgery complex for the $(\tfrac{1}{2},1)$-surgery on a link $L = K_1 \cup K_2$ with $\operatorname{lk}(K_1,K_2)=1$. Inside each disk, there are four vertices corresponding to the summands $H_{*}(C_{\ve}(\bs,\bk))$, where the numbers indicate the values of $\ve \in \bE_2$. The pair $(\bs,\bk)$ is indicated nearby, where we suppress $k_2$ and record only $k_1 \in \Z/2\Z$. We also suppress the subindex of the map $\Phi_{\ve}^{\vec{M}}$.}
\label{fig: rational_surgery_complex}
\end{figure}
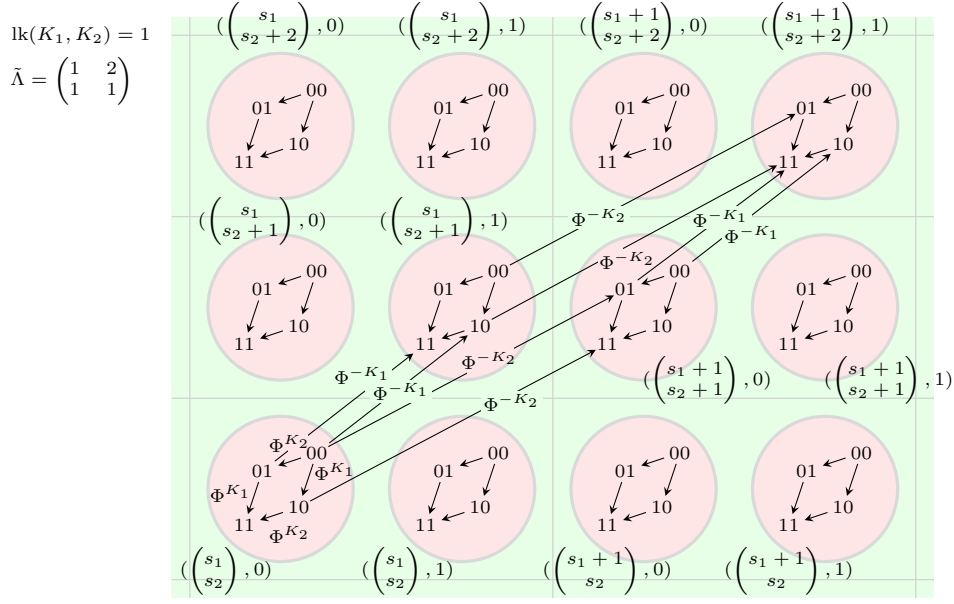
\begin{defn}
Define $H_{\Lambda}(L)$ to be the complex whose underlying space is
\[
H_{\Lambda}(L) = \bigoplus_{\ve\in \bE_n} \prod_{(\bs,\bk)\in\tilde{\bH}(L)} H_*(C_{\ve}(\bs,\bk))
\]
and with differential given by the sum of $\tilde{\Phi}^{\vec{M}}_{\varepsilon}$, ranging over oriented sublinks with $|M|=1$.
See Figure \ref{fig: rational_surgery_complex} for an example of $H_{\Lambda}(L)$ where we consider the $(\tfrac{1}{2},1)$-surgery on a link $L = K_1 \cup K_2$ with $\operatorname{lk}(K_1,K_2)=1$.
\end{defn}
We now show $C_{\Lambda}(L)\simeq H_{\Lambda}(L)$ for rational surgeries on two-component $L$-space links by applying  Liu's work \cite{Yajing_twobridge}.

Liu proved that one can replace the  chain complex in each vertex of the hypercube by a chain homotopy equivalent one.
\begin{prop}[{\cite[Proposition 5.19]{Yajing_twobridge}}]
    Given a $n$-dimensional hypercube $(C_\ve,D_{\ve,\ve'})$, suppose that $C'_\ve$ is chain homotopy equivalent to $ C_\ve$ for $\ve \in \bE_n.$ Then there exists a hypercube $(C'_\ve,D'_{\ve,\ve'})$ which is homotopy equivalent to $(C_\ve,D_{\ve,\ve'}).$
\end{prop}
In particular, when $n=2$, there is a simple description of the new hypercube. 
\begin{prop}[{\cite[Example 5.20]{Yajing_twobridge}}] \label{prop: yajing_two_component}
   Let $(C_\ve,D_{\ve,\ve'})$ be a $2$-dimensional hypercube, and let $(C'_\ve)$ be a  collection of chain complexes for $\ve \in \bE_2$. Assume that for each $\ve \in \bE_2$ we have homotopy equivalences  $\pi_\ve: C_\ve \to C'_\ve$ and $i_\ve: C'_\ve \to C_\ve$ such that
$i_\ve \circ \pi_\ve$ is homotopic to $\id_{C_\ve}$ via $h_\ve$, and $\pi_\ve \circ i_\ve$ is homotopic to $\id_{C'_\ve}$ via $h'_\ve.$
 Define
   \begin{align*}
   D'_{\ve,\ve'} &= \pi_\ve \circ D_{\ve,\ve'} \circ i_\ve,     \qquad \text{for} \  |\ve-\ve'|=1\\
   D'_{00,11} &= \pi_{11} \circ \big( D_{00,11}  +    D_{01,11} \circ h_{01} \circ D_{00,01}  \\ &\hspace{6.5em}+   D_{10,11} \circ h_{10} \circ D_{00,10} \big) \circ i_{00}.
   \end{align*}
   Then $(C'_\ve, D'_{\ve,\ve'})$ is a hypercube and is homotopy equivalent to $(C_\ve,D_{\ve,\ve'})$.
\end{prop}
We now prove the following theorem from the introduction. The argument is taken from \cite{Yajing_twobridge}.
\begin{namedtheorem}[\ref{thm: intro_formal_link_H_Lambda}]
    If $L \subset S^3$ is a two-component $L$-space link, then for any rational surgery coefficients $\Lambda$, we have $H_\Lambda(L) \simeq C_\Lambda(L)$.  
\end{namedtheorem}
\begin{proof}
We view  $C_{\Lambda}(L)$ as a hypercube.  Here 
 \[
 C_\ve = \prod_{(\bs,\bk)\in \tilde{\bH}(L)} C_\ve(\bs,\bk)   
 \]
 and $D_{\ve,\ve'}$ is the sum of $\tilde{\phi}^{\vec{M}}_\ve$, ranging over all possible orientations of $M$, where $M \subset L$ is the sublink with indicator function $\ve'-\ve$. The hypercube relation follows from the fact that $\partial^2=0$ in $C_{\Lambda}(L)$. 
 
 For each $(\bs,\bk) \in \tilde{\bH}(L),$ we consider the $\Z/2\Z$-gradings induced on the $\F\llbracket U_1 \rrbracket$-chain complexes  $C_\ve$ and $H_*(C_\ve)$ by their homological gradings\footnote{We view $H_*(C_\ve)$ as a chain complex with vanishing differential and identify  $U$ with $U_1$.}. By \cite[Corollary 5.6]{Yajing_twobridge} we can find homotopy equivalences, preserving the homological gradings,
 \[\pi_\ve^{(\bs, \bk)}: C_\ve(\bs, \bk) \to H_*(C_\ve(\bs, \bk)) \quad \text{and} \quad  i_\ve^{(\bs, \bk)}: H_*(C_\ve(\bs, \bk)) \to C_\ve(\bs, \bk)\] such that
$i_\ve^{(\bs, \bk)} \circ \pi_\ve^{(\bs, \bk)}$ is homotopic to the identity on $C_\ve(\bs, \bk)$ via $h_\ve^{(\bs, \bk)}$ and $\pi_\ve^{(\bs, \bk)} \circ i_\ve^{(\bs, \bk)}$ is equal to the identity on $H_*(C_\ve(\bs, \bk))$.  Define
% and $\pi_\ve^{(\bs, \bk)} \circ i_\ve^{(\bs, \bk)}$ is homotopic to  via $h_\ve^{'(\bs, \bk)}.$
\[
\pi_\ve = \prod_{(\bs,\bk)\in \tilde{\bH}(L)} \pi_\ve^{(\bs, \bk)}, \qquad i_\ve = \prod_{(\bs,\bk)\in \tilde{\bH}(L)} i_\ve^{(\bs, \bk)}, \qquad h_\ve = \prod_{(\bs,\bk)\in \tilde{\bH}(L)} h_\ve^{(\bs, \bk)}. 
\]
The maps $\pi_\ve, i_\ve$ and $h_\ve$ satisfy the hypotheses of Proposition~\ref{prop: yajing_two_component}, and hence  $C_{\Lambda}(L)$ is homotopy equivalent to a hypercube $(H_*(C_\ve(\bs, \bk)), D'_{\ve,\ve'}).$ It remains to verify that  $D'_{\ve,\ve'}$ agrees with the differential defined on $H_\Lambda(L)$.
 
When $|\ve-\ve'|=1$, denote by $M$ the component of $L$ with indicator function $\ve'-\ve$. Then we have $D'_{\ve,\ve'} =\pi_{\ve'} \circ D_{\ve,\ve'}\circ i_{\ve} = \pi_{\ve'} \circ (\tilde{\phi}^{{M}}_{\varepsilon} + \tilde{\phi}^{-{M}}_{\varepsilon}) \circ i_{\ve} = \tilde{\Phi}^{{M}}_{\varepsilon} + \tilde{\Phi}^{-{M}}_{\varepsilon}$.

We claim that $D'_{00,11}$ changes the $\Z/2\Z$ homological grading by $1$.
Indeed,
the maps $i_\ve$ and $\pi_\ve$  both  preserve the homological grading, and $h_\ve$ changes it by $1$.  The maps $\tilde{\phi}^{\vec{M}}_{\varepsilon}$ are defined to be the same as ${\phi}^{\vec{M}}_{\varepsilon}$, which have homological gradings $|M|-1$, see for example \cite[Lemma 8.12]{mo_linksurgery}. %\ds{I don't think we need the next sentence. We only need the gradings of the $C_e$ right?}By \cite[Section 9.3]{mo_linksurgery} the integral link surgery hypercube admits a $\Z/2\Z$ grading which agrees with the $\Z/2\Z$ homological grading on $C_\ve$. 
According to the description of $D'_{00,11}$ in Proposition \ref{prop: yajing_two_component},  inside the parentheses,  the first term $D_{00,11}$ is a sum of $\tilde{\phi}^{\vec{L}}_{\varepsilon}$ which
 changes $\Z/2\Z$ homological grading by $1$ since $L$ has two components; each of the other two terms is a composition of   ${\phi}^{\vec{M}}_{\varepsilon}$ with $|M|=1$, which preserves the grading and $h_\ve,$ which changes it by $1.$ Therefore $D'_{00,11}$ changes the $\Z/2\Z$ homological grading by $1$.

Since $H_*(C_\ve)$ is supported in even homological grading by Remark~\ref{rmk: even hom degrees}, it follows that $D'_{00,11} = 0.$ This completes the proof.
\end{proof}

\section{Very good points and non $L$-space surgeries on $L_{n,k}$}\label{section: very good points and non L space surgeries}
In this section, we combine  Gorsky-N{\'e}methi's notion of ``very good point''   in \cite{GorNem18}, arguments in \cite{BeibeiLiu21} and the rational link surgery formula to obtain Lemma \ref{lemma: very_good_point}, which provides an obstruction to $L$-space surgeries on two-component $L$-space links in terms of the $H$-function.  Applying it to $L_{n,k},$ we obtain the following result. 
\begin{prop}
 \label{prop: L_n_k_non_L_space}
     For any $n,k \geq 1$,  $S^3_{d_1,d_2}(L_{n,k})$ is not an $L$-space for any $(d_1,d_2)\in (n+k-3,n+k-1)\times \Q.$
\end{prop}
This result is used later in the proof of Theorem \ref{thm: L-space links}, where in order to classify $L$-space surgeries on $L_{n,k}$, we need to  obstruct certain rational surgeries from being $L$-space.

We first recall the rational link surgery formula using the simplified model described in Section \ref{subsec: simplified model}, restricted to the case of a two-component $L$-space link $L$. 
Fix any surgery coefficients $\Lambda = (p_1/q_1,p_2/q_2)$ where $p_i\in\Z, q_i\in \Z_{>0}$ and with $p_i, q_i$ coprime for $i=1,2$. Then, by Theorems \ref{thm: rational_surgery_chain} and \ref{thm: intro_formal_link_H_Lambda},  we have $\widehat{\HF} (S^3_{\Lambda}(L)) \cong H_*(H_{\Lambda}(L)/U)$. Here $H_{\Lambda}(L)$ is the simplified surgery complex, with the underlying space given by 
\[
H_{\Lambda}(L) = \bigoplus_{ \ve \in \bE_2} \prod_{(\bs,\bk)\in \tilde{\bH}(L)} H_*(C_{\ve}(\bs,\bk)),
\]
where  
 $H_*(C_{\ve}(\bs,\bk)) \cong \F\llbracket U\rrbracket$ for any $(\bs,\bk)\in \tilde{\bH}(L)$ and $ \ve \in \bE_2$ since $L$ is an $L$-space link.  Recall that the differential is the sum of 
\[
\tilde{\Phi}^{\vec{M}}_{\varepsilon}: C_\ve(\bs,\bk) \rightarrow C_{\ve'}(\bs',\bk' )
\]
over all oriented sublinks $\vec{M}\subset L$ with $|M|=1$, and where
where  $f(\bs',\bk') = f(\bs,\bk) + \tilde{\Lambda}_{L,\vec{M}}$ under the identification \eqref{eq: identification_tilde_H}. By definition, we have $\tilde{\Phi}^{\vec{M}}_{\varepsilon} = \Phi^{\vec{M}}_{\varepsilon}$ as endomorphism of $\F\llbracket U\rrbracket$. All the non-trivial $\Phi^{\vec{M}}_{\varepsilon}$ are described by Lemma \ref{lemma: two_component_$L$-space_H_Lambda_maps}.
\subsection{Very good points}
We recall  Gorsky-N{\'e}methi's definition of a very good point.
\begin{defn}[\cite{GorNem18}]
Let $L$ be a two-component with $H$-function $H_L(\bdve{s})$, where $ \bdve{s}\in\bH(L).$
    A point $(s_1,s_2) \in \bH(L)$ is called \emph{good} if $H_L(s_1,s_2)>H_L(s_1,\infty)$ and $H_L(s_1,s_2)>H_L(\infty, s_2)$. A point $(s_1,s_2)$ is called \emph{very good} if both $(s_1,s_2)$ and $(-s_1,-s_2)$ are good.
\end{defn}
The following is our main technical lemma. The proof uses  arguments similar to those in \cite{BeibeiLiu21}; see, for example, the proof of \cite[Proposition 4.12]{BeibeiLiu21}.
\begin{lem} \label{lemma: very_good_point}
Let $L=K_1\cup K_2$ be a two-component $L$-space link with $l=\operatorname{lk}(K_1,K_2)$.
    If for some $r\in \Z$ the following are satisfied:
    \begin{enumerate}
    \item \label{it: very_good_v_1} there is a very good point $\bs_0 \in \bH(L)$;
    \item \label{it: very_good_v_2}  $2\bs_0 =(r, l)$,
    \end{enumerate}
    then the $(d_1,d_2)$-surgery on $L$ is not an $L$-space for any $(d_1,d_2)
     \in (r-1,r+1) \times \Q$.
\end{lem}
\begin{proof}
This is proved by considering the simplified complex $H_{\Lambda}(L)/U$.
For $i=1,2$, write $d_i = p_i/q_i$, where $p_i \in \Z, q_i\in \Z_{>0}$ and $p_i,q_i$ coprime, and assume $\frac{p_1}{q_1}\in (r-1,r+1)$.

By the definition of very good point and Lemma \ref{lemma: two_component_$L$-space_H_Lambda_maps}, the maps $\Phi^{\vec{M}}_{\varepsilon}$ at $H_{(0,0)}(\bs_0,\bk)$ and $H_{(0,0)}(-\bs_0,\bk)$ both have positive $U$-powers for any  $\vec{M}\in \{\pm K_1,\pm K_2\}$ and any $\bk = (k_1,k_2) \in \Z/q_1\Z \oplus \Z/q_2\Z$\footnote{Here with a slight abuse of notations, we are identifying integers in $\{0,\dots, q_i-1\}$ with their classes in \(\Z/q_1\Z\).}. Therefore the corresponding maps in the hat version, which are obtained by quotienting by $U$, are all zero. Thus  $H_{(0,0)}(\bs_0,\bk)$ and $H_{(0,0)}(-\bs_0,\bk)$ each splits as an $\F$ summand in $H_{\Lambda}(L)/U$.

Choose $k_1$ and $k'_1$ such that $k_1 - k'_1 = p_1 - q_1 r$. Observe that this is possible because $p_1/q_1 \in (r-1,r+1)$ whereas  $k_1-k'_1 $ attains every value in $  \{-q_1 + 1, \dots,q_1-1 \}$. Fix $k_2 = k'_2$ arbitrarily. If we set $\bk=(k_1,k_2)$ and $\bk'=(k'_1,k'_2)$, then by \eqref{eq: identification_tilde_H} we have
\begin{align*}
    f(\bs_0,\bk) - f(-\bs_0,\bk') = (q_1 r + k_1 - k'_1,q_2 l) = (p_1, q_2 l).
\end{align*}
 In particular, $(\bs_0,\bk)$ and  $(-\bs_0,\bk')$
are in the same coset of $\tilde{\bH}(L)/H(\tilde{\Lambda})$. This implies that if we denote the corresponding spin$^c$ structure by $\bdve{t}\in $ Spin$^c(S^3_{d_1,d_2}(L))$, then $\widehat{\HF}(S^3_{d_1,d_2}(L),\bdve{t})$ is at least two dimensional. Therefore $S^3_{d_1,d_2}(L)$ is not an $L$-space.
\end{proof}
\subsection{The $H$-function of $L_{n,k}$} \label{subsec: H_function_L_n_k}
To apply Lemma \ref{lemma: very_good_point}, we need to compute the $H$-function of $L_{n,k}.$ 
We recall from
Proposition \ref{prop: Alex poly of L_{n,k}} that
for $n,k\geq 1$, the multivariable Alexander polynomial $\Delta_{n,k}$ of $L_{n,k}$ is
\[
\Delta_{n,k}\doteq\big{(}\sum_{i=0}^{k}x^{i}y^{k-i}\big{)}r_{n-1} -\big{(}\sum_{i=0}^{k-1}x^{i}y^{k-i-1}\big{)}r_n, 
\]
where $r_n=1+xy+\cdots+ x^ny^n$. The highest power for both $x$ and $y$ is $n+k-1$ and the lowest power for both $x$ and $y$ is $0$. Therefore the symmetrised Alexander polynomial is 
\[
\Delta_{n,k}=(xy)^{-\frac{n+k-1}{2}}\Big(\big{(}\sum_{i=0}^{k}x^{i}y^{k-i}\big{)}r_{n-1} -\big{(}\sum_{i=0}^{k-1}x^{i}y^{k-i-1}\big{)}r_n\Big), 
\]
up to a $-$ sign. We may sketch $\Delta_{n,k}$ in the $(i,j)$-plane, where each term $x^i y^j$ corresponds to a dot in the $(i,j)$-coordinates,  as depicted in Figure \ref{fig: H_function_L_n_k}.
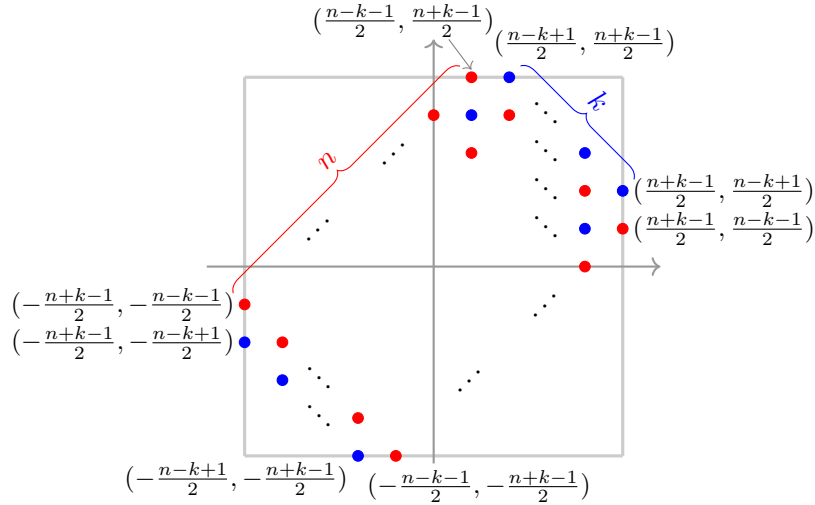
\begin{figure}[hbtp!]
    \centering
  \begin{tikzpicture}[scale=1]
  \begin{scope}
      \draw [very thick,black!20!white] (-2.5,-2.5)--(-2.5,2.5);
      \draw [very thick,black!20!white] (-2.5,-2.5)--(2.5,-2.5);
      \draw [very thick,black!20!white] (2.5,-2.5)--(2.5,2.5);
      \draw [very thick,black!20!white] (-2.5,2.5)--(2.5,2.5);
       \draw [->,thick,black!40!white] (-3,0)--(3,0);
       \draw [->,thick,black!40!white] (0,-2.6)--(0,3);
  \end{scope}
  \filldraw [blue] (2.5,1) circle (2pt) node[]  {};
   \node [right] at (2.5,1) {{\small $(\frac{n+k-1}{2},\frac{n-k+1}{2})$}};
     \filldraw [red] (2.5,0.5) circle (2pt)  node[]  {};
       \node [right] at (2.5,0.5) {{\small $(\frac{n+k-1}{2},\frac{n-k-1}{2})$}};
     \filldraw [blue] (1,2.5) circle (2pt) node[]  {};
     \node [above] at (2,2.6) {{\small $(\frac{n-k+1}{2},\frac{n+k-1}{2})$}};
     \node [above] at (-0.4,2.9) {{\small $(\frac{n-k-1}{2},\frac{n+k-1}{2})$}};
     \filldraw [red] (0.5,2.5) circle (2pt)  node[]  {};
     \filldraw [blue] (2,1.5) circle (2pt) node[]  {};
      \filldraw [red] (2,1) circle (2pt)  node[]  {};
     \filldraw [red] (1,2) circle (2pt)  node[]  {};
 \draw [<-,black!50!white] (0.5,2.6)--(0.2,3);
 
     \node[rotate=-45] at (1.5,1.5) {$\cdots$};
\node[rotate=-45] at (1.5,2) {$\cdots$};
\draw[color=blue,decorate,decoration={brace,amplitude=5pt,raise=6pt}]
  (1,2.5) -- (2.5,1)  
  node[midway,sloped,above= 10 pt,color=blue] {$k$};
\draw[color=red,decorate,decoration={brace,amplitude=5pt,raise=6pt}]
  (-2.5,-0.5) -- (0.5,2.5)  
  node[midway,sloped,above= 10 pt,color=red] {$n$};
  
\filldraw [blue] (2,0.5) circle (2pt) node[]  {};
     \filldraw [red] (2,0) circle (2pt)  node[]  {};
        \filldraw [red] (0.5,1.5) circle (2pt)  node[]  {};
     \filldraw [blue] (0.5,2) circle (2pt) node[]  {};
     \filldraw [red] (0,2) circle (2pt)  node[]  {};
  \node[rotate=-45] at (1.5,0.5) {$\cdots$};
\node[rotate=-45] at (1.5,1) {$\cdots$};
 \node[rotate=45] at (-0.5,1.5) {$\cdots$};
  \node[rotate=45] at (-1.5,0.5) {$\cdots$};
         \node [left] at (-2.5,-0.5) {{\small $(-\frac{n+k-1}{2},-\frac{n-k-1}{2})$}};
     \node [left] at (-2.5,-1) {{\small $(-\frac{n+k-1}{2},-\frac{n-k+1}{2})$}};
       % \draw [<-,black!50!white] (-0.5,-2.6)--(-0.5,-2.9);
 \node  at (0.6,-2.9) {{\small $(-\frac{n-k-1}{2},-\frac{n+k-1}{2})$}};
 \node [left] at (-1,-2.8) {{\small $(-\frac{n-k+1}{2},-\frac{n+k-1}{2})$}};

     \filldraw [blue] (-2.5,-1) circle (2pt) node[]  {};
     \filldraw [red] (-2.5,-0.5) circle (2pt)  node[]  {};
     \filldraw [blue] (-1,-2.5) circle (2pt) node[]  {};
     \filldraw [red] (-0.5,-2.5) circle (2pt)  node[]  {};
     \filldraw [blue] (-2,-1.5) circle (2pt) node[]  {};
      \filldraw [red] (-2,-1) circle (2pt)  node[]  {};
     \filldraw [red] (-1,-2) circle (2pt)  node[]  {};

    \node[rotate=-45] at (-1.5,-1.5) {$\cdots$};
    \node[rotate=-45] at (-1.5,-2) {$\cdots$};
 \node[rotate=45] at (0.5,-1.5) {$\cdots$};
  \node[rotate=45] at (1.5,-0.5) {$\cdots$};

  \end{tikzpicture}
    \caption{The symmetrised Alexander polynomial $\Delta_{n,k}$. Each blue dot (resp.~red dot) in coordinates $(i,j)$ represents a term $-x^iy^j$ (resp.~$x^iy^j$) in $\Delta_{n,k}$. The diagram depicts the case when  $n-k=3$. }
    \label{fig: H_function_L_n_k}
\end{figure}
Indeed, the terms 
\[
-(xy)^{-\frac{n+k-1}{2}}\left(\sum_{i=0}^{k-1}x^{i}y^{k-i-1}\right)
\]
have coordinates $(-\frac{n+k-1}{2},-\frac{n-k+1}{2}),\cdots,(-\frac{n-k+1}{2}, -\frac{n+k-1}{2})$, where from one term to the next the coordinates change by $(1,-1)$. They correspond to the $k$ blue dots in the bottom left of the diagram. Multiplying by $r_n$ changes the coordinates by $(i,i)$ for $i=0,\dots,n$, which corresponds $n+1$ copies of the diagonal sequence of $k$ blue dots in total, each shifted by $(1,1)$ in coordinates.   

Similarly, the terms 
\[
(xy)^{-\frac{n+k-1}{2}}\left(\sum_{i=0}^{k}x^{i}y^{k-i}\right)
\]
have coordinates $(-\frac{n+k-1}{2},-\frac{n-k-1}{2}),\cdots,(-\frac{n-k-1}{2}, -\frac{n+k-1}{2})$ with coordinates change $(1,-1)$ between adjacent terms, which correspond to the $k+1$ red dots in the bottom left of the diagram. Multiplying by $r_{n-1}$ changes the coordinates by $(i,i)$ for $i=0,\cdots,n-1$, corresponding to $n$ copies of the diagonal sequence of $k+1$ red dots, each shifted by $(1,1)$ in coordinates. 

The full $H$-function $H_{L_{n,k}}$ may be computed using the above description. However, for our application, we only need the following lemma. For simplicity, we suppress the index from $H_{L_{n,k}}$.  
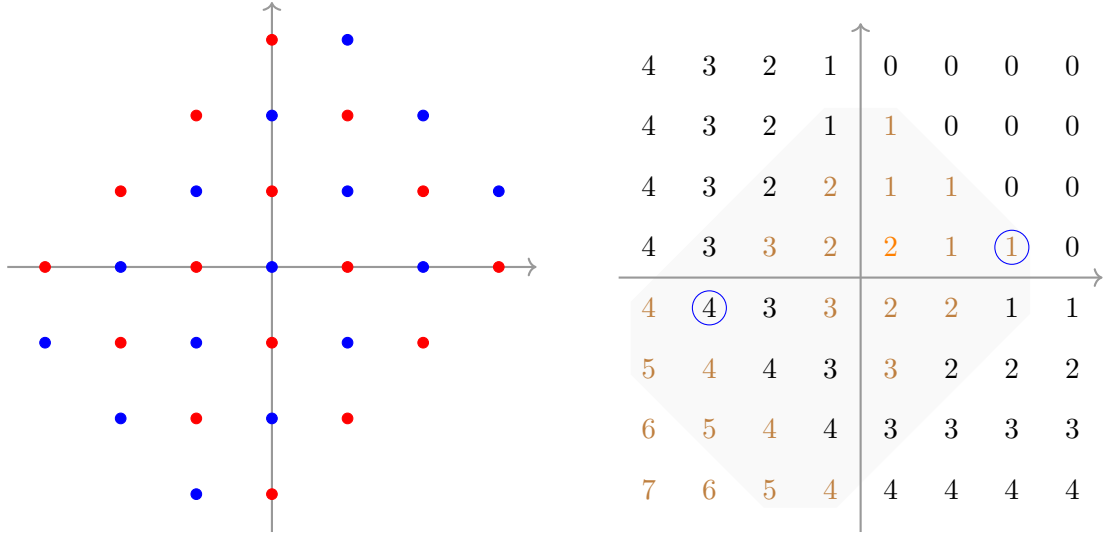
\begin{figure}\captionsetup{width=\textwidth}
\begin{subfigure}{0.45\textwidth}
  \centering
  \begin{tikzpicture}[scale=1]
  \begin{scope}
          \draw [black!0!white] (0,-3.3)--(0,3.3);
       \draw [->,thick,black!40!white] (-3.5,0)--(3.5,0);
       \draw [->,thick,black!40!white] (0,-3.5)--(0,3.5);
  \end{scope}
   \foreach \i in {0,...,4}
   \foreach \j in {0,1,2}
{ 
\filldraw [blue] ({-3+\j+\i},{-1-\j+\i}) circle (2pt) node[]  {};
}
  \foreach \i in {0,...,3}
   \foreach \j in {0,...,3}
{ 
\filldraw [red] ({-3+\j+\i},{-\j+\i}) circle (2pt) node[]  {};
}
  \end{tikzpicture}
  \end{subfigure}
\hfill
\begin{subfigure}{0.45\textwidth}
 \begin{tikzpicture}[scale=0.8]
 \fill[gray!5!white] 
    (2.8,0.6) -- (2.8,-0.6) -- (-0.4,-3.8) -- (-1.6,-3.8) -- (-3.8,-1.6) -- (-3.8,-0.4) -- (-0.6,2.8) -- (0.6,2.8) -- cycle;
 \begin{scope}
     \draw[->,thick, black!40!white]  (-4, 0) -- (4, 0);
     \draw[->,thick,black!40!white]  (0, -4.2) -- (0, 4.2);
 \end{scope}
 \node () at (-3.5,3.5) {{$4$}};
 \node () at (-2.5,3.5) {{$3$}};
 \node () at (-1.5,3.5) {{$2$}};
    \node () at (-0.5,3.5) {{$1$}};
 \node () at (0.5,3.5) {{$0$}};
 \node () at (1.5,3.5) {{$0$}};
 \node () at (2.5,3.5) {{$0$}};
    \node () at (3.5,3.5) {{$0$}};

 \node () at (-3.5,2.5) {{$4$}};
 \node () at (-2.5,2.5) {{$3$}};
 \node () at (-1.5,2.5) {{$2$}};
    \node () at (-0.5,2.5) {{$1$}};
     \node () [brown] at (0.5,2.5) {{$1$}};
 \node () at (1.5,2.5) {{$0$}};
 \node () at (2.5,2.5) {{$0$}};
    \node () at (3.5,2.5) {{$0$}};

     \node () at (-3.5,1.5) {{$4$}};
 \node () at (-2.5,1.5) {{$3$}};
 \node () at (-1.5,1.5) {{$2$}};
    \node ()[brown] at (-0.5,1.5) {{$2$}};
     \node () [brown] at (0.5,1.5) {{$1$}};
 \node () [brown] at (1.5,1.5) {{$1$}};
 \node () at (2.5,1.5) {{$0$}};
    \node () at (3.5,1.5) {{$0$}};

         \node () at (-3.5,0.5) {{$4$}};
 \node () at (-2.5,0.5) {{$3$}};
 \node ()[brown] at (-1.5,0.5) {{$3$}};
    \node ()[brown] at (-0.5,0.5) {{$2$}};
     \node ()[orange] at (0.5,0.5) {{$2$}};
 \node ()[brown] at (1.5,0.5) {{$1$}};
 \node ()[brown] at (2.5,0.5) {{$1$}};
    \node () at (3.5,0.5) {{$0$}};

 \node ()[brown] at (-3.5,-0.5) {{$4$}};
 \node () at (-2.5,-0.5) {{$4$}};
 \node () at (-1.5,-0.5) {{$3$}};
    \node ()[brown] at (-0.5,-0.5) {{$3$}};
     \node ()[brown] at (0.5,-0.5) {{$2$}};
 \node ()[brown] at (1.5,-0.5) {{$2$}};
 \node () at (2.5,-0.5) {{$1$}};
    \node () at (3.5,-0.5) {{$1$}};

     \node ()[brown] at (-3.5,-1.5) {{$5$}};
 \node ()[brown] at (-2.5,-1.5) {{$4$}};
 \node () at (-1.5,-1.5) {{$4$}};
    \node () at (-0.5,-1.5) {{$3$}};
     \node ()[brown] at (0.5,-1.5) {{$3$}};
 \node () at (1.5,-1.5) {{$2$}};
 \node () at (2.5,-1.5) {{$2$}};
    \node () at (3.5,-1.5) {{$2$}};

       \node ()[brown] at (-3.5,-2.5) {{$6$}};
 \node ()[brown] at (-2.5,-2.5) {{$5$}};
 \node ()[brown] at (-1.5,-2.5) {{$4$}};
    \node () at (-0.5,-2.5) {{$4$}};
     \node () at (0.5,-2.5) {{$3$}};
 \node () at (1.5,-2.5) {{$3$}};
 \node () at (2.5,-2.5) {{$3$}};
    \node () at (3.5,-2.5) {{$3$}};

 \node ()[brown] at (-3.5,-3.5) {{$7$}};
 \node ()[brown] at (-2.5,-3.5) {{$6$}};
 \node ()[brown] at (-1.5,-3.5) {{$5$}};
    \node ()[brown] at (-0.5,-3.5) {{$4$}};
     \node () at (0.5,-3.5) {{$4$}};
 \node () at (1.5,-3.5) {{$4$}};
 \node () at (2.5,-3.5) {{$4$}};
    \node () at (3.5,-3.5) {{$4$}};
    
    \draw [blue] (2.5,0.5) circle (8pt);
    \draw [blue] (-2.5,-0.5) circle (8pt);
 \end{tikzpicture}
 \end{subfigure}
    \caption{On the left, the symmetrised Alexander polynomial $\Delta_{4,3}$. On the right, the $H$-function of $L_{4,3}$. Values of $H(\bs)$ that differ from those of the unlink (after shifting) by $1$ are indicated in brown, while the one that differs by $2$ is  indicated in orange. The region of $\mathbf{s}$ for which $\mathbf{s} + \mathbf{1}$ corresponds to a lattice point in $\tilde{\Delta}_{4,3} = (xy)^{1/2}\Delta_{4,3}$ is shaded. A pair of very good points $\pm\bigl(\tfrac{5}{2}, \tfrac{1}{2}\bigr)$ is circled.}
    \label{fig: H_fun_L_4_3}
\end{figure}
\begin{lem}\label{lem: L_nk_very_good_point}
For any $n,k \geq 1$, we have that
    $\left(\frac{n+k}{2}-1,\frac{n-k}{2}\right) \in \bH(L_{n,k})$ is a very good point.
\end{lem}
\begin{proof}
    Both components of $L_{n,k}$ are unknots and the linking number  is $l=n-k$. Therefore, by Proposition \ref{prop:GNH-function}, the $H$-function is obtained from that of the unlink by first shifting by $\bigl(\tfrac{n-k}{2}, \tfrac{n-k}{2}\bigr)$ and then adding the contribution of $\tilde{\Delta}_{n,k} = (xy)^{1/2}\Delta_{n,k}$. See Figure \ref{fig: H_fun_L_4_3} for an example when $n=4$ and $k=3.$ 
    In particular, 
     we have that $H(s,\infty) = H(\infty,s)=H_U(s-\frac{n-k}{2})$ where $ H_U$ is the $H$-function of the unknot given by
    \begin{equation*}
      H_U(s)=  \begin{cases}
          0 &s \geq 0\\
          -s &s < 0
       \end{cases}
    \end{equation*}
    Therefore $H(\frac{n+k}{2}-1,\infty) = H(\infty,\frac{n-k}{2})=0$. We compute
    \[
    H\left(\frac{n+k}{2}-1,\frac{n-k}{2}\right) = - \sum_{\substack{\bdve{s}'\in \bH(L) \\ \bdve{s}'\ge (\frac{n+k}{2},\frac{n-k}{2}+1)}} \chi(\HFL^-(L_{n,k},\bdve{s}').
    \]
 From the previous discussion, we know that the only  term  \(c(s_1,s_2)x^{s_1}y^{s_2}\) in \(\tilde{\Delta}_{n,k}=(xy)^{\frac{1}{2}}\Delta_{n,k}\) with \((s_1,s_2) \geq (\frac{n+k}{2},\frac{n-k}{2}+1)\) 
 is \(-x^{\frac{n+k}{2}}y^{\frac{n-k}{2}+1}\). Therefore 
 \[\chi(\HFL^-(L_{n,k},\bdve{s})) = -1 \text{ when } \bdve{s}= \left(\frac{n+k}{2},\frac{n-k}{2}+1\right)\]
 and is zero for any other \(\bdve{s}\ge (\frac{n+k}{2},\frac{n-k}{2}+1)\). It follows that \[
 H\left(\frac{n+k}{2}-1,\frac{n-k}{2}\right) =1 > H\left(\frac{n+k}{2}-1,\infty\right) = H\left(\infty,\frac{n-k}{2}\right).\]    
On the other hand, by the symmetry of the $H$-function, we have 
\begin{align*}
H\left(-\frac{n+k}{2}+1,-\frac{n-k}{2}\right)&=H\left(\frac{n+k}{2}-1,\frac{n-k}{2}\right) + \frac{n+k}{2}-1 + \frac{n-k}{2}\\
&=1 + \frac{n+k}{2}-1 + \frac{n-k}{2}=n,
\end{align*}
whereas 
\[H\left(-\frac{n+k}{2}+1,\infty\right)=\frac{n-k}{2}-\left(-\frac{n+k}{2}+1\right)=n-1\]
and 
\[H\left(\infty,-\frac{n-k}{2}\right)=\frac{n-k}{2} - \left(-\frac{n-k}{2}\right) = n-k.
\]
It follows that $H(-\frac{n+k}{2}+1,-\frac{n-k}{2})>H(-\frac{n+k}{2}+1,\infty) \geq H(\infty,-\frac{n-k}{2})$, and hence $(\frac{n+k}{2}-1,\frac{n-k}{2})$ is a very good point.
\end{proof}
\begin{proof}[Proof of Proposition \ref{prop: L_n_k_non_L_space}]
    Let $r=n+k-2$ and we have $l=n-k$ for $L_{n,k}$.  Set $\bs_0 = (\frac{n+k}{2}-1,\frac{n-k}{2}) \in \bH(L_{n,k})$. By Lemma \ref{lem: L_nk_very_good_point}, $\bs_0 $ is a very good point. Both conditions in Lemma \ref{lemma: very_good_point} are satisfied, and hence the result follows.
\end{proof}

\section{Turaev torsion and $L$-spaces}\label{section: turaev torsion}
In this section, we recall how the Turaev torsion can be used to determine the set of $L$-space surgery slopes on rational homology solid tori, following Rasmussen-Rasmussen and their main theorem in \cite{RR}.
We start by recalling some properties of the Turaev torsion, and we refer the reader to \cite{Tur} for a comprehensive discussion on the topic.
\newline

Let $N$ be a compact orientable $3$-manifold whose boundary is a possibly empty union of tori, let $H$ denote its first integer homology group and let $Q(H)$ be the ring of fraction of $\Q[H]$, i.e. the localization of $\Q[H]$ by the multiplicative system of all non-zerodivisors. 

The \emph{Turaev torsion} (or maximal abelian torsion) $\tau(N)$ of $N$ is a topological invariant of $N$ and it is an element of $Q(H)$, well-defined up to multiplication by $\pm H$, where $H$ is identified with its image in $Q(H)$ via the natural inclusion. We do not need the precise definition of this invariant, but only some of its properties.

First of all, it is possible to remove the ambiguity in the sign of the Turaev torsion by endowing $N$ with a \emph{homological orientation}, i.e. an orientation of the vector space $H_{*}(N; \R)$. 

In some cases there is a natural choice of a homology orientation:
\begin{itemize}
\item if $N$ is closed and oriented we define the \emph{natural homology orientation} $\omega_N$ of $N$ as the one defined by a basis $(h_0, h_1, h_2, h_3)$ of $H_*(N; \R)$ such that $h_i$ is a basis of $H_i(N; \R)$ and $h_i$ and $h_{3-i}$ are dual with respect to the non-degenerate pairing $H_i(N;\R)\times H_{3-i}(N ;\R) \rightarrow \R$ given by the orientation of $N$. One can check that $\omega_N$ does not depend on the choice of $(h_0, h_1, h_2, h_3)$;
\item if $N$ is the exterior of an oriented link $L=K_1 \cup \cdots \cup K_m$ in a rational homology sphere $M$, we define the \emph{natural homology orientation} $\omega_N$ of $N$ to be the one determined by $([\textit{pt}], [\mu_1],\dots, [\mu_m], [T_1], \dots, [T_{m-1}])$, where $[pt]$ is the class of a point in $H_0(N; \R)$ and $[\mu_i]$ and $[T_i]$ are the classes the oriented meridian of $L_i$ and the oriented boundary of a tubular neighbourhood of $K_i$ respectively. Here, as usual, the meridian is oriented so to have linking number $1$ with $K_i$ and $T_i$ is oriented with the ``outward vector first'' convention. One can show that $\omega_N$ does not depend on the numeration of the components of $L$. 
\end{itemize}

It is possible to eliminate also the ambiguity given by the multiplication by elements in $H$, by fixing a so-called \emph{Euler structure} on $N$. We will not need do to this and we refer the reader to \cite{Tur} for more details. 

\subsection{Turaev torsion and Alexander polynomial} 
We now suppose that the boundary of $N$ is not empty. Denote by $G=H/\text{Tors} H$ and let $\text{pr}:Q(H)\rightarrow Q(G)$ denote the map induced by the projection $H\rightarrow G$. Recall that the Alexander polynomial $\Delta(M)$ is an element in $\Z[G]$, defined up to sign and multiplication by $G$. 

When $b_1(N)=1$, $N$ is a so-called a \emph{rational homology solid torus}, since its boundary is a torus and it has the same rational homology groups of $S^1\times D^2$. In this case 
we also fix an identification $H=\Z \oplus \text{Tors}H$, and denote by $t$ the element $(1,0)$, and by $\Tilde{t}$ its image in $G$. 
The Turaev torsion and the Alexander polynomial are strongly related:

\begin{thm}[{\cite[Section~II.5.2]{Tur}}]\label{Thm: torsion vs alexander}
Suppose that $b_1(N)\geq 2$. Then
$$
\emph{pr}(\tau(N))=\Delta(N)
$$
as elements in $\Z[G]/\pm G$.
If $b_1(N)=1$ then 
$$
\emph{pr}(\tau(N))(\Tilde{t}-1)=\Delta(N)
$$
as elements in $\Z[G]/\pm G$.\qed
\end{thm}
In particular, when $H$ is torsion free and has rank at least $2$ the Alexander polynomial and the Turaev torsion coincide. We will use this in the case of the exterior of links in $S^3$. 
\begin{corollary}
    If $N$ is the exterior of a link in $S^3$ with at least two components then
    $$ 
    \tau(N)=\Delta(N)
    $$
    as elements in $\Z[H]/\pm H$. \qed
\end{corollary}
In this case we notice that the Turaev torsion of $N$ can be represented by an element in $\Z[H]$. This holds more generally:

\begin{prop}[{\cite[Section~II.4.3]{Tur}}]
If $b_1(N)\geq 2$, then $\tau(N)\in\Z[H]/\pm H$. \qed
\end{prop}

Notice that when $N$ is a rational homology solid torus the identification $H=\Z \oplus \text{Tors}H$ defines a homology orientation of $H$ obtained by choosing $t$ as a generator of $H_1(N; \R)$. Therefore we can consider the sign-refined torsion associated to this homology orientation. It follows from the proof of \cite[Lemma~II.4.5.2]{Tur} that, if we expand $(1-t)^{-1}$ as an infinite sum in positive powers of $t$, then we can normalise $\tau(N)$ as a formal sum
$$
\tau(N)=\sum_{\substack{h\in H \\ \text{pr}(h)\geq 0}}a_h h
$$
where $a_h$ is an integer for all $h\in H$, $a_0\ne 0$ and $a_h=1$ for all but finitely many $h\in H$.

For example, when $H$ has no torsion, it follows from Theorem \ref{Thm: torsion vs alexander} that the normalisation is
$$
\tau(N)=\frac{\Delta(N)}{1-t}\in \Z[[t]]
$$
where $(1-t)^{-1}$ is expanded as an infinite sum in positive powers of $t$ and the Alexander polynomial $\Delta(N)$ is normalised so that $\Delta(N)\in \Z[t]$, $\Delta(N)(0)\ne 0$ and $\Delta(N)(1)=1$. 

The last property of Turaev torsion we need is its behaviour under Dehn filling.
\begin{thm}[{\cite[Theorem~VII.1.4]{Tur}}]\label{thm: gluing formula turaev}
Let $N$ be a compact orientable $3$-manifold with $b_1(N)\geq 2$ whose boundary is union of tori and let $N'$ be obtained by gluing a solid torus to one boundary component of $N$. Let $f$ denote the inclusion-induced map $f:H_1(N; \Z)\rightarrow H_1(N'; \Z)$, and let $f_{\#}$ be the corresponding map $f_{\#}: \Z[H_1(N;\Z)]\rightarrow \Z[H_1(N'; \Z)]$. Orient the core of the solid torus arbitrarily, and let $h$ denote its class in $H_1(N'; \Z)$. 
Then we have
$$
f_{\#}(\tau(N))=(h-1)\tau(N')
$$
as elements in $\Z[H_1(N'; \Z)]/\pm H_1(N'; \Z)$.\qed
\end{thm}

\subsection{Turaev torsion and \texorpdfstring{$L$}{L}-space filling slopes}

The Turaev torsion will be useful for us because it contains all the information needed to determine the set of $L$-space surgeries on knots in rational homology spheres. We now briefly review how, following \cite{RR}.

Let $Y$ be a rational homology solid torus and let $\Sl(Y)$ denote the set of (rational) slopes in $\partial Y$
\[
\Sl(Y)=\{ \alpha\in H_1(\partial Y;\Z)\,| \,\alpha\text{ is primitive}\}/ \pm 1=\mathbb{P}(H_1(\partial Y;\Q)).
\]
Each slope $[\alpha]\in \Sl(Y)$ determines a Dehn filling of $Y$, that we denote with $Y(\alpha)$.
Since $Y$ is a rational homology solid torus, there exists a distinguished slope called \emph{homological longitude}, that is defined as $[l]$, where $l\in H_1(\partial Y;\Z)$ is a generator of the kernel of the inclusion-induced map
$$
\iota:H_1(\partial Y;\Z)\rightarrow H_1(Y;\Z).
$$
The element $l$ is unique up to sign, and hence the class $[l]\in \Sl(Y)$ is well-defined. Equivalently, $[l]$ can be defined as the unique slope defining a filling that is not a rational homology sphere.
For example, when $Y$ is the exterior of a knot $K$ in $S^3$, the homological longitude $[l]$ coincides with the slope defined by the (canonical) longitude of $K$. We denote $\Sl^*(Y)=\Sl(Y)\setminus [l]$.

We say that $[\alpha]$ is an \emph{$L$-space filling slope} of $Y$ if $Y(\alpha)$ is an $L$-space, and we denote by $L(Y)$ the set of $L$-space filling slopes of $Y$.
We say that $Y$ is \emph{Floer simple} if there exist at least two Dehn fillings of $Y$ that are $L$-spaces, i.e. if $|L(Y)|>1$.

It turns out that if $Y$ is Floer simple, then it is possible to describe explicitly the set of $L$-spaces filling slopes of $Y$ in terms of its Turaev torsion.
Fix an identification $H=\Z \oplus \text{Tors}H$ and write 
$$
\tau(Y)=\sum_{\substack{h\in H_1(Y;\Z) \\ \text{pr}(h)\geq 0}}a_h h.
$$
We define the following subset of $H_1(Y; \Z)$
$$
D^{\tau}_{>0}=\{x-y\,| x\notin S[\tau(Y)],\, y\in S[\tau(Y)] \text{ and } \text{pr}(x)>\text{pr}(y)\}\cap \iota(H_1(\partial Y; \Z))
$$
where $S[\tau(Y)]=\{ h\in H_1(Y;\Z)\,|\, a_h\ne 0 \}$ is the \emph{support} of $\tau(Y)$, and where \[\iota:H_1(\partial Y;\Z)\rightarrow H_1(Y;\Z)\] is induced by the inclusion $\partial Y\hookrightarrow Y$.

We are now ready to state the main result of \cite{RR}. 
\begin{thm}[{\cite[Theorem 1.6]{RR}}] \label{theorem RR}
If $Y$ is Floer simple, then either
\begin{itemize}
    \item $D_{>0}^\tau(Y)=\emptyset$ and $L(Y)=\Sl^*(Y)$, or
    \item $D_{>0}^\tau(Y)\ne\emptyset$ and $L(Y)$ is a closed interval whose endpoints are consecutive elements in $\iota^{-1}(D_{>0}^\tau(Y))$.\qed
\end{itemize}
\end{thm}
Recall that $[l]\in \Sl(Y)$ denotes the homological longitude of $Y$, and since the corresponding Dehn filling of $Y$ is not a rational homology sphere, by the definition of $L$-space we have $[l]\notin L(Y)$. The second part of the statement of the theorem can be made more concrete as follows. Once we fix a basis $(\mu,\lambda)$ for $H_1(\partial Y;\Z)$, we can define a map
\[
\begin{array}{ccc}
  H_1(\partial Y;\Z)&\longrightarrow& \Q\cup \{\infty\}\\
p\mu+q\lambda&\longmapsto&\frac{p}{q} 
\end{array}
\]
and this association induces an identification between $\Sl(Y)$ and $\Q \cup \{\infty\}$. If the set $D^{\tau}_{>0}$ is not empty, we can apply this map to the set $\iota^{-1}(D^{\tau}_{>0})$ and Theorem \ref{theorem RR} states that if $Y$ is Floer simple then $L(Y)$ is a closed interval in $\Sl(Y)=\Q\cup \{\infty\}$ whose endpoints are consecutive elements in the image of $\iota^{-1}(D^{\tau}_{>0})$ in $\Q \cup \{\infty\}$.
\subsection{Multiple components and box-convexity}\label{subsec: box convexity}
In this section, we observe that Theorem~\ref{theorem RR} implies a useful structural property of the set of \(L\)-space fillings slopes of a manifold with multiple boundary components.

Let $N$ be the exterior of link in a rational homology sphere, with boundary components $T_1,\dots, T_n$, and let
\[\Sl(N)=\mathbb{P}(H_1(T_1;\Q))\times \dots \times \mathbb{P}(H_1(T_n;\Q))\]
be the set of (rational) multislopes in $\partial N$. We denote $L(N)$ the set of $L$-space fillings slopes of $N$, and by $\Sl^*(N)$ the set of slopes that produce rational homology spheres by filling.

\begin{defn}
A \emph{rectangle} in $\Sl(N)$ is a subset of the type $I_1\times\dots\times I_n$, where each $I_i$ is either a closed interval or a point in \(\mathbb{P}(H_1(T_i;\Q))\). We say that $R$ is \emph{degenerate} if some of the $I_i$ is a point, and \emph{non-degenerate} otherwise. 
\end{defn}

\begin{defn}
A subset $S\subset \Sl^*(N)$ is \emph{weakly box-convex} if whenever the vertices of a rectangle $R\subset \Sl^*(N)$ are contained in $S$, then $R\subset S$. 
\end{defn}
\begin{remark}
The notion of box-convexity was inspired by Massoni and Zung, who discuss -- in a forthcoming paper -- a stronger version of this property for multislopes realised by taut foliations transverse to pseudo-Anosov flows.
\end{remark}
We will usually omit the prefix \emph{weakly}; however ``box-convex" should be understood to mean ``weakly box-convex''.
We define the \emph{box-convex hull} of a subset $S\subset \Sl^*(N)$ as the smallest box-convex set containing $S$. 

\begin{prop}\label{prop: box-convex}
Let $N$ be the exterior of a link in a rational homology sphere. Then the set $L(N)$ is box-convex.
\end{prop}
\begin{proof}
We proceed by induction on the number of boundary components of $N$. When $\partial N$ is connected, this follows directly from Theorem \ref{theorem RR}. Assume $N$ has boundary components $T_1,\dots, T_n$, and let
\[R=I_1\times\cdots \times I_n\subset \Sl^*(N)
\]
be a rectangle whose vertices are contained in $L(N)$. We can assume that $R$ is non-degenerate, since otherwise we could easily conclude by using the inductive hypothesis. 

Let  \((x^{\varepsilon_1}_1,\dots, x^{\varepsilon_n}_n)\), where \(\varepsilon_i \in\{0,1\}\) for  \(i=1,\dots, n \),  denote the vertices of $R$, and consider the manifolds $N^0$ and $N^1$ obtained by filling $N$ along the slope $x_1^0$ and $x_1^1$, respectively. These manifolds are exterior of links in rational homology spheres, and have $n-1$ boundary components each. By inductive hypothesis, the rectangles
\begin{align*}
R^0=I_2\times\cdots \times I_n\subset \Sl^*(N^0)\\
R^1=I_2\times\cdots \times I_n\subset \Sl^*(N^1)
\end{align*}
are contained in \(L(N^0)\) and \(L(N^1)\) respectively. In other words, the degenerate rectangles 
\[\{x_1^0\}\times I_2\times\cdots \times I_n \text{ and }\{x_1^1\}\times I_2\times\cdots \times I_n
\] are contained in $L(N)$. This implies the desired conclusion, by using again the inductive hypothesis on the rectangles $I_1\times \{p_1\}\times \dots \times \{p_n\}$, where the $p_i$ range in $I_i$ for all $i=2,\dots, n$.
\end{proof}

\begin{remark}\label{rem: longitudes 2-comp}
Observe that a direct computation using Mayer-Vietoris shows that if $L\subset S^3$ is a two-component link, with exterior $E_L$, then  
\[
\Sl^*(E_L)=\Sl(E_L)\setminus \left\{x=\frac{\operatorname{lk}(L)^2}{y}\right\},
\]
where we parametrise $H_1(\partial E_L; \Z)$ with the canonical meridian-longitude bases of the link components, and we set $\frac{0}{0}=\infty$.
\end{remark}
We conclude the section with the following consequence of Proposition~\ref{prop: box-convex}.
\begin{lem}\label{lem: small surgery implies gen L space link}
Let \(L\subset S^3\) be an oriented two-component link with unknotted components \(K_1,K_2\). Suppose that \(S^3_{r,s}(L)\) is an \(L\)-space, for some non-zero rational numbers \(r,s\) of the same sign satisfying \(rs<\operatorname{lk}(L)^2\). Then \(L\) is a generalised \((+,-)\) and \((-,+)\,L\)-space link. Moreover, if \(r\) and \(s\) are both positive, then for each integer \(m<0\) the manifolds \(S^3_{\frac{1}{m},\bullet}(L)\) and \(S^3_{\bullet,\frac{1}{m}}(L)\) are exteriors of \(L\)-space knots in \(S^3\). Analogously, if \(r\) and \(s\) are both negative, then for each integer \(m>0\) the manifolds \(S^3_{\frac{1}{m},\bullet}(L)\) and \(S^3_{\bullet,\frac{1}{m}}(L)\) are exteriors of mirrors of \(L\)-space knots in \(S^3\).
\end{lem}
\begin{figure}[]
   \centering   \includegraphics[width=0.45\textwidth]{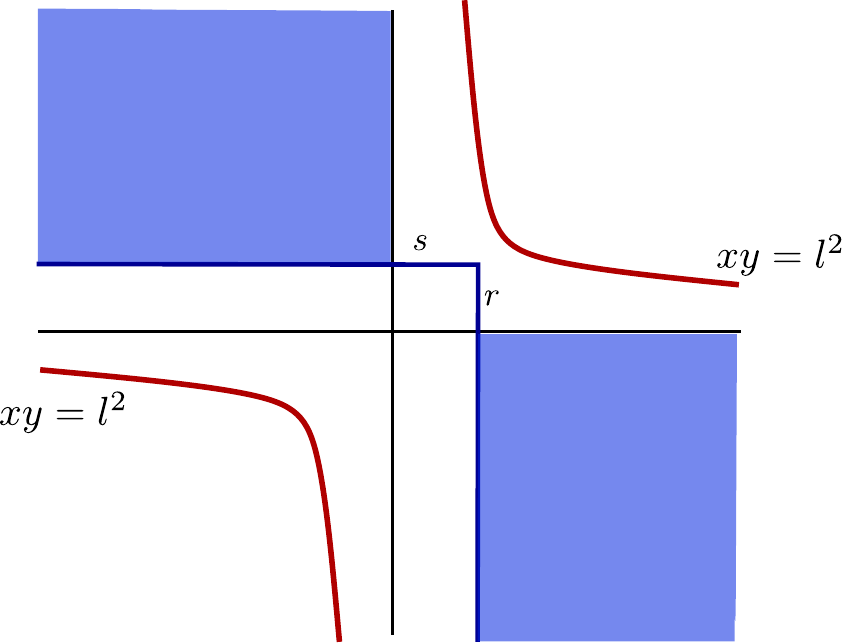}
    \caption{The figure shows part of the $L$-space surgeries slopes of a link with two unknotted components in $S^3$, with linking number $l$, and whose $(r,s)$-surgery is an $L$-space.}
    \label{figure: pictorial sketch 1}
\end{figure}
\begin{proof}
We assume $r,s>0$, the other case being analogous, and we denote $l=\operatorname{lk}(L)$ and $E_L$ the exterior of $L$. Since the components of $L$ are unknots, we have \[
(r',\infty), (\infty, s')\in L(E_L) \text{ for every } r'\ne 0, s'\ne 0.\]
By hypothesis, $(r,s)\in L(E_L)$, and hence by box-convexity and Remark~\ref{rem: longitudes 2-comp}, the rectangle $[-\infty, r]\times \{s\}$ is contained in $L(E_L)$, i.e. $S^3_{r',s}(L)$ is an $L$-space for all $r'\leq r$. By box-convexity again, we have that for each pair of rationals $(r',s')$ with $r'<0, s'>s$
the surgery $S^3_{r',s'}(L)$ is an $L$-space. Notice that this implies that $L$ is a generalised $(-,+)\,L$-space link, and that, for $m<0$, large enough surgeries on the knot $K\subset S^3$ whose exterior is $S^3_{\frac{1}{m},\bullet}$ are $L$-spaces, \ie $K$ is an $L$-space knot. 

Repeating the same argument starting with the rectangle \(\{r\}\times [-\infty,s]\) one shows that $L$ is a generalised $(+,-)\,L$-space link, and that, for $m<0$, the knot $K'\subset S^3$ whose exterior is $S^3_{\bullet,\frac{1}{m}}$ is an $L$-space knot. Figure \ref{figure: pictorial sketch 1} shows a pictorial sketch of the proof.
\end{proof}

\section{Classification of the $L$-space surgeries}\label{sec: classification}
In this section we classify the set of $L$-space surgeries on the links \(L_{n,k}\) and \(L'_{n,k}\). The classification for the links $L_{n,k}$ is simpler, and is presented in Section \ref{subsec: class for Lnk}. The majority of the section is devoted to the case of the links \(L'_{n,k}\), and is presented in Section \ref{subsec: class for L'nk}.
\newline

We start by observing the following.
\begin{lem}\label{lemma: mirror of generalised L-space links}
The mirror of the link $L'_{n,k}$ is the link $L'_{k+1,n-1}$.
\end{lem}

\begin{proof}
The link $L'_{n,k}$ is equal to $b(pq+1,-q)$, where $p=2k+1$ and $q=2n-1$. So its mirror is $L=b(pq+1,q)$, and by Theorem~\ref{teo: classification two-bridge links} this is isotopic, as unoriented link, to $b(pq+1,-p)=L'_{k+1,n-1}$
\end{proof}

In particular, the previous lemma implies that the links $L'_{n, n-1}$ are amphichiral. 

\begin{prop}\label{prop: finding some L-space surg on nk/nk'}
For $n,k\geq 1$ the \((n+k-1,n+k-1)\)-surgery on \(L_{n,k}\) is an $L$-space. Moreover, there exists a rational number \(r\ne 0\) satisfying \(r^2<(n+k)^2\) such that \(S^3_{r,r}(L'_{n,k})\) is an $L$-space. In particular \(L'_{n,k}\) is a generalised $(+,-)$ and $(-,+)\,L$-space link.
\end{prop}

\begin{proof}
Consider the three-component link $\mathcal{L}$ shown in Figure \ref{fig: L-space surgeries}. By construction, for each $n\geq 1$, we have
\begin{align*}
S^3_{r_1,r_2,-\frac{1}{n-1}}(\mathcal{L})=S^3_{r_1+n-1,r_2+n-1}(L_{n,k})
\\
S^3_{r_1,r_2,\frac{1}{n+1}}(\mathcal{L})=S^3_{r_1-n-1,r_2-n-1}(L'_{n,k}).
\end{align*}
Denote by $Y$ the rational homology solid torus \(S^3_{k,k,\bullet}(\mathcal{L})\). The slopes $1$ and $\infty$ are both $L$-space filling slopes for $Y$. In fact \(S^3_{k,k,1}(\mathcal{L})= S^3_{k-1,k-1}(T(2,2k))\) is an $L$-space by \cite[Corollary~3.6]{yajing_lspace}, and \(S^3_{k,k,\infty}(\mathcal{L})= S^3_{k,k}(L_{1,k})\) is an $L$-space by \cite[Proposition~1.5]{San23twobridge}.
\begin{figure}[]
    \centering
    \includegraphics[width=0.23\textwidth]{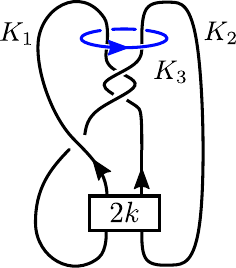}
    \caption{The link \(\mathcal{L}\).}
    \label{fig: L-space surgeries}
\end{figure}

We show that the homological longitude of $Y$ is the slope $2$. To do so, we fix orientations of the components of \(\mathcal{L}\) as in Figure~\ref{fig: L-space surgeries}. By computing the linking numbers among the components, we deduce that for a rational number \(\frac{p}{q}\), a presentation matrix for the first homology group of \(S^3_{k,k,\frac{p}{q}}(\mathcal{L})\) is given by
    $$
    A=\begin{pmatrix}
    k &k-1 &-q \\
    k-1 &k &q \\
    -1 &1 &p
    \end{pmatrix}
    $$
    and in particular \(S^3_{k,1,\frac{p}{q}}(\mathcal{L})\) is not a rational homology sphere if and only if the determinant of $A$ is zero. By direct computation, one sees that this happens if and only if $p=2q$, and therefore $2$ is the homological longitude of the manifold \(S^3_{k,k,\bullet}(\mathcal{L})\).

This implies, by Proposition \ref{prop: box-convex}, that \(S^3_{k,k,s}(\mathcal{L})\) is an $L$-space for all $s\leq 1$, proving the desired conclusion regarding \(L_{n,k}\). For what concerns \(L'_{n,k}\) we deduce that \((k-n-1,k-n-1)\) surgery on \(L'_{n,k}\) is an $L$-space. This provides the desired result unless $k=n+1$. In such a case, we can just consider the mirror of \(L'_{n,k}\), using Lemma \ref{lemma: mirror of generalised L-space links}. The last part of the statement follows from Lemma \ref{lem: small surgery implies gen L space link} and the fact that there exists an orientation of \(L'_{n,k}\) such that $l=n+k$.
\end{proof}

\subsection{Classification for the links \texorpdfstring{$L_{n,k}$}{Lnk}}\label{subsec: class for Lnk}
We are now ready to prove Theorem~\ref{thm: L-space links}, which we recall here.

\begin{namedtheorem}[\ref{thm: L-space links}]
A hyperbolic two-bridge link $L$ is an $L$-space link if and only if it is isotopic to $L_{n,k}$ for some $n\geq 1, k\geq 1$. In this case, for rationals $r,s$, the $(r,s)$-surgery on $L_{n,k}$ is an $L$-space if and only if $\operatorname{min}\{r,s\}\geq n+k-1$.
\end{namedtheorem}
\begin{proof}
First of all, we observe that if $n,k\geq 1$, then $L_{n,k}$ is a hyperbolic $L$-space link. Since two-bridge links are alternating, by Menasco \cite{Menasco} we have that a two-bridge link is either hyperbolic or a torus link. If $n,k\geq 1$, then Proposition~\ref{prop: some ctf Lnk} implies that $L_{n,k}$ is not a generalised $(+,-)\,L$-space link, and hence is not a torus link. This shows that $L_{n,k}$ is hyperbolic.  

Let us fix an orientation on $L_{n,k}$ such that $\operatorname{lk}(L_{n,k})=n-k$.
We know from Proposition~\ref{prop: finding some L-space surg on nk/nk'} that $(n+k-1,n+k-1)$-surgery on $L_{n,k}$ is an $L$-space. Since $(n+k-1)^2>(n-k)^2$, we deduce from Proposition~\ref{prop: box-convex} and Remark~\ref{rem: longitudes 2-comp} that $S^3_{r,s}(L_{n,k})$ is an $L$-space when $\operatorname{min}\{r,s\}\geq n+k-1$. In particular, we have that $L_{n,k}$ is a hyperbolic $L$-space link.

On the other hand, assume that $L$ is a hyperbolic $L$-space two-bridge link. Then, by Theorem~\ref{thm: persistently fol two-br}, $L$ must be isotopic, up to mirror, to one of the links \(L_{n,k}\) or \(L'_{n,k}\). However, by Proposition~\ref{prop: some ctf on L'nk} and Lemma~\ref{lemma: mirror of generalised L-space links}, the links $L'_{n,k}$ are not $L$-space links, unless they are torus links. This, and the assumption that $L$ is hyperbolic, implies that $L$ is isotopic to $L_{n,k}$, with $n,k\geq 1$, or their mirrors.  However, by Proposition~\ref{prop: some ctf Lnk} the mirror of $L_{n,k}$ is not an $L$-space link. Hence $L$ is isotopic to $L_{n,k}$, for some $n.k\geq 1$.

We conclude the proof by showing that if   \(\operatorname{min}\{r,s\}< n+k-1\), then the $(r,s)$-surgery on \(L_{n,k}\) is not an $L$-space.
Assume, by contradiction, that $S^3_{r,s}(L_{n,k})$ is an $L$-space. 
Since two-bridge links are symmetric, we can assume that $r<n+k-1$. Moreover, by Proposition \ref{prop: some ctf Lnk} and Proposition \ref{prop: L_n_k_non_L_space}, we must have \[
1\leq r<n+k-3 \text{ and } 1\leq s.\]
Denote by $K_1$ and $K_2$ the components of $L_{n,k}$, and consider the manifold $Y=S^3_{\bullet,s}(L_{n,k})$ obtained by drilling $K_1$ and by performing $s$-surgery on $K_2$. By hypothesis there are two fillings of $Y$ that are $L$-spaces, corresponding to $S^3_{r,s}(L_{n,k)}$ and $S^3_{\infty, s}(L_{n,k})$, the latter being an $L$-space since the components of $L_{n,k}$ are unknotted. If $r<\frac{l^2}{s}$, by box-convexity and Remark \ref{rem: longitudes 2-comp} we deduce that $S^3_{r',s}(L_{n,k})$ is an $L$-space for all $r'\leq r$, contradicting Proposition \ref{prop: some ctf Lnk}. Analogously, if $r>\frac{l^2}{s}$, then the same conclusion holds for $S^3_{r',s}(L_{n,k})$, for $r'\geq r$, and since $r<n+k-3$, this contradicts Proposition \ref{prop: L_n_k_non_L_space}.

\begin{figure}[]
   \centering   \includegraphics[width=0.4\textwidth]{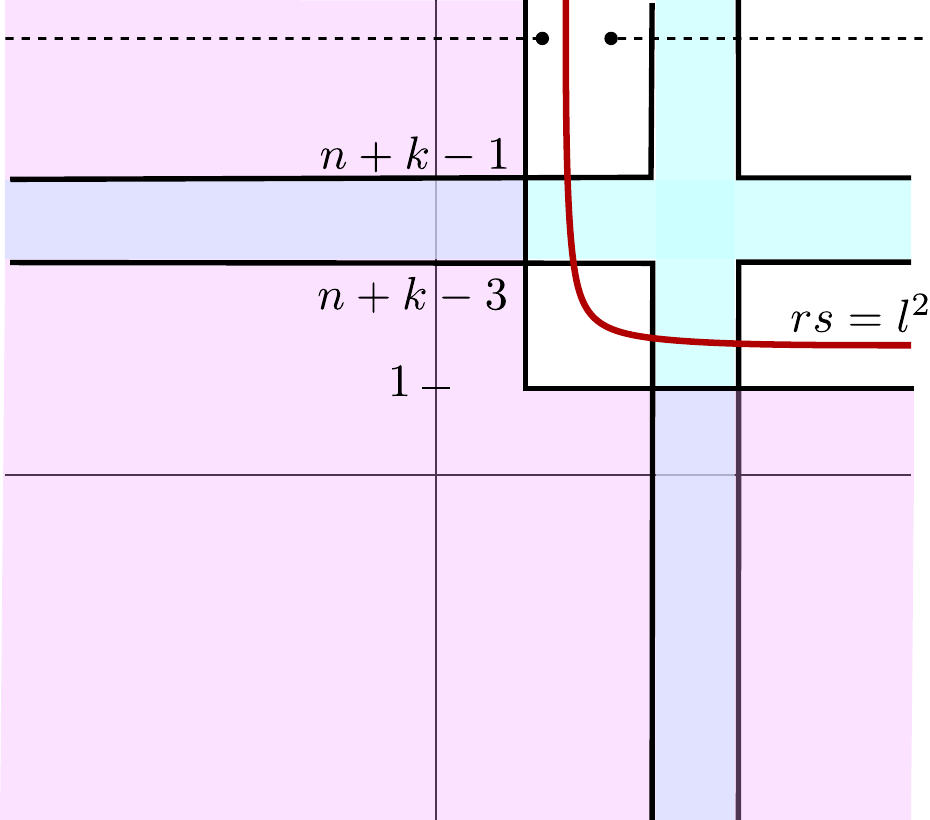}
    \caption{The rational points in the plane correspond to surgery coefficients for the link $L_{n,k}$. The coloured regions corresponds to the non-$L$-space surgeries provided by Proposition \ref{prop: some ctf Lnk} and Proposition \ref{prop: L_n_k_non_L_space}, and the red curve is one of the two connected components of the curve $rs=l^2$. If the dot represents the coefficients $(r,s)$ of an $L$-space surgery, then the dotted half-line would correspond to additional $L$-space surgeries provided by Theorem \ref{theorem RR}. In both cases, this half-line would intersect a region of non-$L$-space surgeries. }
    \label{figure: pictorial sketch}
\end{figure}

We conclude by observing that $r\ne \frac{l^2}{s}$, since otherwise $S^3_{r,s}(L_{n,k})$ would not be a rational homology sphere, and in particular not an $L$-space. A pictorial sketch of the proof is shown in Figure \ref{figure: pictorial sketch}.
\end{proof}

We observe that, if $n\ne 1$ and $k\ne 1$, then there are surgeries on $L_{n,k}$ that are non-$L$-space but are not covered by the result of Proposition~\ref{prop: some ctf Lnk}. We can combine Theorem~\ref{thm: L-space links} with the main result of \cite{Lyu25} to provide taut foliations on the remaining surgeries.

\begin{namedtheorem}[\ref{thm: CTF iff Lspace for Lnk}]
Let $n,k\geq 1$, and let $M$ be a Dehn surgery on the link $L_{n,k}$. Then $M$ is not an $L$-space if and only if it supports a coorientable taut foliation.    
\end{namedtheorem}

\begin{proof}
Assume  $M=S^3_{r_1,r_2}(L_{n,k})$ and denote by \(K_1, K_2\) the components of $L_{n,k}$. If $r_1=\infty$ or $r_2=\infty$, then $M$ is either an $L$-space or $S^2\times S^1$ and there is nothing to prove. If $r_1=0$ or $r_2=0$, then the result follows from Proposition~\ref{prop: some ctf Lnk}. In the remaining cases, the manifold $S^3_{r_1}(K_1)$ is a lens space and the component $K_2$ corresponds to a $(1,1)$-knot \(K'_2\subset S^3_{r_1}(K_1)\), see for example \cite{MorSak91,Saito04}. We see from Theorem~\ref{thm: L-space links} that if $r_1<n+k-1$, then the knot $K'_2$ has no $L$-space surgeries, and hence is persistently foliar by \cite{Lyu25}. This implies that if
\[
r_1<n+k-1 \text{ and } r_2\in \Q
\]
then $M$ supports a coorientable taut foliation. By using the symmetry of two-bridge links and Theorem\ref{thm: L-space links}, we deduce that this happens for every non-$L$-space surgery on $L_{n,k}$.
\end{proof}
\subsection{Classification for the links \texorpdfstring{$L'_{n,k}$}{L'nk}}\label{subsec: class for L'nk}

In this section we prove Theorem ~\ref{thm: gen L-space links}, which classifies $L$-space surgeries on the (hyperbolic) links in the family $L'_{n,k}$. We can reformulate the statement of Theorem~\ref{thm: gen L-space links} by using the language of box-convex sets introduced in the previous section.
We denote $l=n+k$.

\begin{thm}[Reformulation of Theorem \ref{thm: gen L-space links}]\label{thm: reformulation}
A hyperbolic two-bridge link $L$ is a generalised $(+,-)\,L$-space link if and only if it is isotopic to \(L'_{n,k}\), with $n\geq 2, k\geq 1$. In this case the set of $L$-space surgeries \(L(L'_{n,k})\) is the box-convex hull of 
\[
\mathcal{B}\cup H_+\cup H_-\cup L_+\cup L_-,
\]
where
\[
\begin{array}{ll}
\mathcal{B}=\left(\{\infty\}\times \overline{\mathbb{Q}}\setminus\{0\}\right)\cup \left(\overline{\mathbb{Q}}\setminus\{0\}\times \{\infty\}\right)\\
   H_+=\left\{\left(\frac{1}{m},h_+(\frac{1}{m})\right), \left(h_+(\frac{1}{m}), \frac{1}{m})\right)\Big{|}\, m\in \Z_{>0}\right\} &\text{with } h_+(x)=2k+1+\frac{l}{x}\\
    H_-=\left\{\left(-\frac{1}{m},h_-(-\frac{1}{m})\right), \left(h_-(-\frac{1}{m}), -\frac{1}{m}\right)\Big{|}\, m\in \Z_{>0}\right\} &\text{with } h_-(x)=1-2n+\frac{l}{x} \\
    L_+=\Big{\{}\left(m,l_+(m)\right)\Big{|}\, m\in [2,2k+1]\cap\Z\Big{\}}&\text{with } l_+(x)=2k+3-x \\
    L_-=\Big{\{}(-m,l_-(-m))\Big{|}\, -m\in[1-2n,-2]\cap \Z_{>0}\Big{\}} &\text{with } l_-(x)=-1-2n-x
    \end{array}
\]
\end{thm}
We observe that the sets $H_+/H_-$ lie on hyperbolas, while the sets $L_+/L_-$ lie on lines.

The section is divided in two parts. In Section \ref{subsubsec: finding L-space surgeries} we prove that \(L(L'_{n,k})\) contains the above mentioned box-convex hull, and then in Section~\ref{subsubsec: constraints L-space surgeries} we show that it coincides with it.

We begin by identifying the hyperbolic links among the $L'_{n,k}$. It is clear from the definition that $L'_{n,k}$ is a torus link when $k=0$, and by Lemma \ref{lemma: mirror of generalised L-space links}, this also happens when $n=1$. On the other hand we have:

\begin{lem}\label{lemma: gener. L-space links are not L-space links}
If $k\geq 1$ and $n\geq 2$, then the link $L'_{n,k}$ is hyperbolic.
\end{lem}

\begin{proof}
From Proposition \ref{prop: some ctf on L'nk} and Lemma \ref{lemma: mirror of generalised L-space links} we have that neither $L'_{n,k}$ nor its mirror are $L$-space links when $n\geq 2$ and $k\geq 1$. Since torus links are $L$-space links we deduce that $L'_{n,k}$ is not a torus link, and non-torus alternating links are hyperbolic \cite{Menasco}, and two-bridge links are alternating, we conclude.
\end{proof}

From now on, unless otherwise stated, we will suppose $n\geq 2$ and $k\geq 1$. 
\subsubsection{Finding $L$-space surgeries.}\label{subsubsec: finding L-space surgeries}
In this section we show that the set \(\mathcal{B}\cup H_+\cup H_-\cup L_+\cup L_-\), and hence its box-convex hull by Proposition \ref{prop: box-convex}, is contained in $L(L'_{n,k})$. Since the components of $L'_{n,k}$ are unknotted, we have \(\mathcal{B}\subset L(L'_{n,k})\), and  the next proposition implies that the same holds for $H_+\cup H_-$.
\begin{prop}\label{prop: L-space knots as surgeries on L'}
For each positive integer $m$, the manifolds
  \renewcommand*{\arraystretch}{1.5}
\[
\begin{array}{ll}
    S^3_{-\frac{1}{m},s}(L'_{n,k})=S^3_{s,-\frac{1}{m}}(L'_{n,k}) &\textit{ for }s\geq 1-2n-ml\\
    S^3_{\frac{1}{m},s}(L'_{n,k})=S^3_{s,\frac{1}{m}}(L'_{n,k})   &\textit{ for }s\leq 2k+1+ml
\end{array}
\]
are $L$-spaces.
\end{prop}
\begin{proof}
We know from the proof of Proposition \ref{prop: finding some L-space surg on nk/nk'} that $(k-n-1,k-n-1)$-surgery on $L'_{n,k}$ is an $L$-space. We assume for now that $k>n+1$, so that Lemma \ref{lem: small surgery implies gen L space link} implies that the knot $K_{-m}\subset S^3$, obtained by $(-\frac{1}{m})$-surgery on one component of $L'_{n,k}$, is an $L$-space knot. The knot $K_{-m}$ can be explicitly described as follows. Observe that $L'_{n,k}$ is isotopic to $\hat{\beta}\cup \mathfrak{a}$, where $\hat{\beta}$ is the closure of the braid
$$
\beta=\sigma_k\cdots\sigma_1\sigma^{-1}_{k+1}\cdots\sigma^{-1}_{n+k-1}\in B_{n+k}
$$
and $\mathfrak{a}$ is the braid axis, see Figure \ref{figure: braid}.
\begin{figure}[]
   \centering   \includegraphics[width=0.38\textwidth]{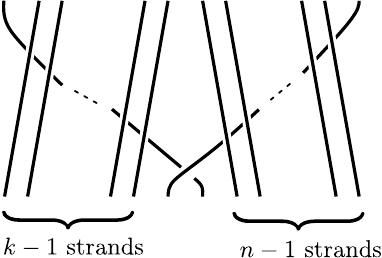}
    \caption{The braid $\beta$. The link $L'_{n,k}$ is isotopic to its closure union the braid axis.}
    \label{figure: braid}
\end{figure} Here, we denote by $B_{n+k}$ the group of braids on $n+k=l$ strands, and by $\sigma_i$ the Artin generators. We also denote 
$$
\Delta=(\sigma_{1}\cdots\sigma_{l-1})^{l}=(\sigma_{l-1}\cdots\sigma_1)^{l}.
$$
Since $L'_{n,k}=\hat{\beta}\cup \mathfrak{a}$, the knot $K_{-m}$ is the closure of the braid
\begin{align*}
\beta_{-m}&=\sigma_k\cdots\sigma_1\sigma^{-1}_{k+1}\cdots\sigma^{-1}_{l-1}\Delta^{m}=\sigma_k\cdots\sigma_1\sigma^{-1}_{k+1}\cdots\sigma^{-1}_{l-1}(\sigma_{l-1}\cdots\sigma_{1})^{ml} \\
&=\sigma_k\cdots\sigma_1\sigma_{k}\cdots\sigma_1(\sigma_{l-1}\cdots\sigma_{1})^{ml-1}.
\end{align*}
The braid $\beta_{-m}$ is positive, and hence fibered \cite{Stallings}. Moreover a fiber (and hence minimal genus) surface can be explicitely constructed starting from the braid word, and its Euler characteristic is given by the difference between the number of strands of $\beta_{-m}$ and its length as a word in the positive powers of the Artin generators. Therefore, the genus $g(K_{-m})$ satisfies
$$
2g(K_{-m})-1=2k+1-2l+ml^2-ml=1-2n+ml^2-ml.
$$
Since $K_{-m}$ is an $L$-space knot, the manifold $S^3_r(K_{-m})$ is an $L$-space if and only if $r\geq 2g(K_{-m})-1=1-2n+ml^2-ml$ by \cite[Lemma 2.13]{HeddenCabling2}. So we have that
$$
S^3_{-\frac{1}{m},s}(L'_{n,k})=S^3_{s,-\frac{1}{m}}(L'_{n,k})= S^3_{s+ml^2}(K_{-m})
$$
is an $L$-space if and only if $s\geq 1-2n-ml$. This proves the first half of the statement, and moreover implies that $L'_{n,k}$ also has a negative surgery with coefficient $(r,r)$ satisfying $r^2<l^2$ that is an $L$-space (for example, one can take $r=-1$). Therefore, if we consider the mirror $L'_{n',k'}$ of $L'_{n,k}$ --- where $n'=k+1$ and $k'=n-1$ by Lemma \ref{lemma: mirror of generalised L-space links} --- we can apply the same reasoning as above and deduce that
$$
S^3_{\frac{1}{m},s}(L'_{n,k})=S^3_{-\frac{1}{m},-s}(L'_{n',k'})
$$
is an $L$-space if and only if 
$$
-s\geq 1-2n'-ml=-1-2k-ml \Leftrightarrow s\leq2k+1+ml,
$$
which concludes the proof of the proposition when $k>n+1$. To prove it in general, we need two observations. First, note that in order to make the proof work, we only need the existence of a positive rational $r$ satisfying $r^2< l^2$ and such that $S^3_{r,r}(L'_{n,k})$ is an $L$-space.
Secondly, note that proving the proposition for $L'_{n,k}$ is equivalent to proving it for its mirror $L'_{n',k'}$. 
If $k<n+1$, then $r=n+1-k$ is positive and satisfies $r^2<l^2$. Moreover, the $(r,r)$-surgery on the mirror $L'_{n',k'}$ of $L'_{n,k}$ is an $L$-space. Hence, the proposition holds for the mirror of $L'_{n,k}$, and for $L'_{n,k}$ as a consequence. Finally, if $k=n+1$, we consider the mirror of $L'_{n',k'}$ of $L'_{n,k}$, for which the proposition holds, since $k'<n'+1$.
\end{proof}

\begin{corollary}\label{cor: quadrant 2 and 4 are L-spaces}
Fix $n,k\geq 0$ such that $(n,k)\ne(0,0)$. If $(r,s)$ are rational numbers satisfying $rs\leq 0$, then $S^3_{r,s}(L'_{n,k})$ is an $L$-space.
\end{corollary}
\begin{proof}
When $n\geq 2,k\geq 1$ this follows from the previous proposition and the fact that the set of $L$-space surgeries is box-convex. Otherwise, $L'_{n,k}$ is (up to mirror) a torus link $T(2,2l)$, where $l=n+k$. In this case $-\frac{1}{m}$-surgery on one component gives the torus knot $T(ml+1,l)$ and $\frac{1}{m}$-surgery gives the mirror of $T(ml-1,l)$. This implies that a statement similar to that of Proposition \ref{prop: L-space knots as surgeries on L'} holds, and the desired result follows.
\end{proof}
%\ds{check that the torus links are correct}

Our goal now is to show that also the set $L_+\cup L_-$ is contained in \(L(L'_{n,k})\). This is the content of Proposition \ref{prop: L space surgeries on L'nk small box}. We need some preliminary lemmas. We consider the link $\mathcal{L}$ of Figure \ref{figure: L-space surgeries2}, and fix an integer number $x$ satisfying $0\leq x <n-1$. We denote by $Y$ the rational homology solid torus $S^3_{x,-1-x,\bullet}(\mathcal{L})$. As usual, we parametrise slopes on $\partial Y$ by using the canonical meridian-longitude basis of the component $K_3$ of $\mathcal{L}$.
\begin{figure}[]
   \centering   \includegraphics[width=0.25\textwidth]{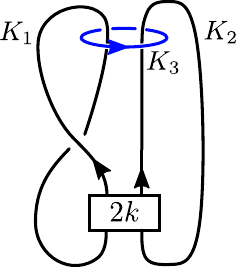}
    \caption{The link $\mathcal{L}$.}
    \label{figure: L-space surgeries2}
\end{figure} 

\begin{lem}\label{lemma: longitude is neg}
The homological longitude of $Y$ is a negative rational number.
\end{lem}
\begin{proof}
We compute the homological longitude explicitly. We orient the components of $\mathcal{L}$ as in Figure~\ref{figure: L-space surgeries2}, and see that the linking numbers among them are
$$
\operatorname{lk}(K_1,K_2)=k \quad \operatorname{lk}(K_1,K_3)=-1 \quad \operatorname{lk}(K_1,K_2)=1
$$
and hence, for $r=\frac{p}{q}\in \mathbb{Q}\cup \{\infty\}$, a presentation matrix for the first homology group of $S^3_{x,-1-x,r}(\mathcal{L})$ is given by
 $$
    A=\begin{pmatrix}
    x &k &-q \\
    k &-1-x &q \\
    -1 &1 &p
    \end{pmatrix}.
    $$
By definition, the slope $r$ is the homological longitude of $Y$ if and only if $p$ and $q$ solve the equation 
$$
\operatorname{det}A= -p(x(1+x)+k^2)+q(1-2k)=0.
$$
It is easy to see that $r=\infty$ does not solve the above equation, and so we can divide by $q$ and obtain
$$
\operatorname{det}A=0 \Leftrightarrow r=\frac{1-2k}{x(1+x)+k^2}
$$
and since $x\geq 0$ by hypothesis, we conclude.
\end{proof}

\begin{lem}\label{lemma: Y0 is L-space}
In the same setting as above, the manifolds $Y_{\frac{1}{x}}$ and $Y_0$ are $L$-spaces.
\end{lem}
 
\begin{proof}
The manifold $Y_{\frac{1}{x}}$ is an $L$-space since
\[
Y_{\frac{1}{x}}=S^3_{x,-1-x,\frac{1}{x}}(\mathcal{L})=S^3_{0,-1-2x}(L'_{x,k}),
\]
and the last manifold in the chain of equality is an $L$-space by Corollary \ref{cor: quadrant 2 and 4 are L-spaces}.
In order to show that $Y_0$ is an $L$-space, we first show that it is a Seifert manifold by exhibiting it as a double branched cover on $S^3$ along a Montesinos link, and then use the classification of Seifert manifolds $L$-space by Lisca-Stipsicz \cite{LiscaStipsicz}. Figure \ref{figure: Montesinos} shows how a strongly invertible surgery description of the link $\mathcal{L}$ and its  $(x,-x-1,0)$-surgery can be obtained, and demonstrates how one may take the quotient and perform appropriate rational tangle
replacements to produce a link whose double branched cover is $Y_0$. This link is the Montesinos link \(M(\frac{k}{2k+1},-\frac{1}{2}, \frac{1}{3+2k})\), and hence its double branched cover is the Seifert fibered space \(M(0; \frac{k}{2k+1},-\frac{1}{2}, \frac{1}{3+2k})\)\footnote{Here we are using in the notation of Lisca–Stipsicz \cite{LiscaStipsicz}, where the
Seifert fibered space \(M(e_0; r_1, r_2,\dots,r_k)\) is obtained by $e_0$–surgery on an unknot with $k$ meridians having
\(-\frac{1}{r_i}\) surgery coefficient on the $i$-th one.}.
\begin{figure}[]
   \centering   \includegraphics[width=0.9\textwidth]{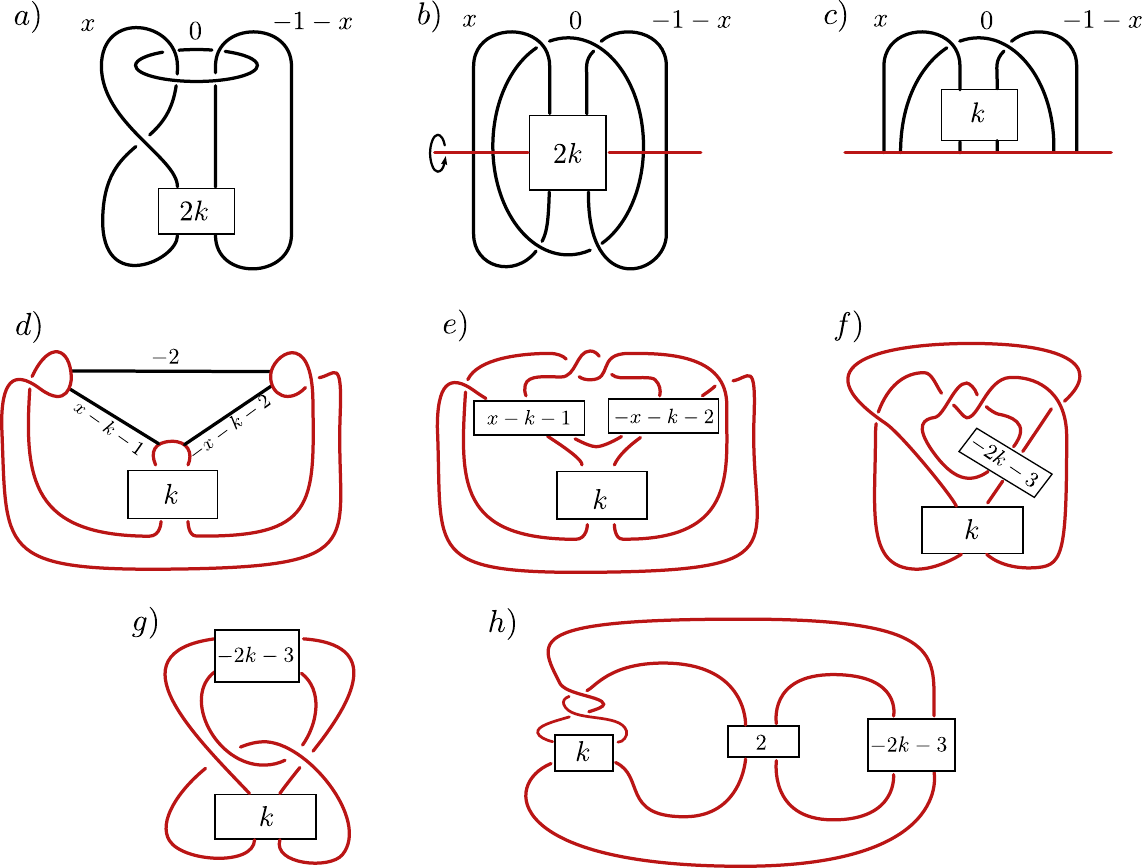}
    \caption{In $b)$, we shows a strongly invertible description of the surgery presentation given in $a)$. In $c)$ we draw the quotient of the surgery description by the involution, and $d)$ shows the effect of an isotopy that straightens the black arcs. By rational tangle replacements along these arcs, we obtain the link in $e)$, whose double branched cover is the manifold $Y_0$. In $f)-h )$ we perform isotopies to show that this link is the Montesinos link \(M(\frac{k}{2k+1},-\frac{1}{2}, \frac{1}{3+2k})\). }
    \label{figure: Montesinos}
\end{figure} 

We now show that these manifolds are $L$-spaces. By \cite[Theorem~1]{LiscaStipsicz} the Seifert fibered space \(M=M(e_0; r_1, r_2, r_3)\), with \(1 \geq r_1 \geq r_2 \geq r_3 \geq 0\),
is an $L$-space if and only if either $M$ or $-M$ does not carry a positive transverse contact structure. By \cite{LiscaMatic}, such a Seifert fibered space $M$ carries no positive transverse contact structure if and
only if either \(e_0 \geq 0\) or \(e_0=-1\) and there exist no coprime integers $a$ and $m$
such that \(mr_1 < a < m(1-r_2)\) and \(mr_3 < 1\). If we rewrite our manifolds to apply \cite{LiscaStipsicz}, we have \(Y_0=M(-1;\frac{1}{2},\frac{k}{2k+1},\frac{1}{3+2k})\). Assume that there exist coprime integers $a$ and $m$
such that 
\begin{equation}
\frac{1}{2}m< a < \left(\frac{k+1}{2k+1}
\right)m \quad \text{and} \quad\frac{m}{3+2k}< 1.
\label{eq:LS eq}
\end{equation}
Since $\frac{1}{2}< \frac{k+1}{2k+1}$, the first inequality implies $m>0$. Moreover, we also deduce 
$$
2a\in (m,m+\frac{m}{2k+1})
$$
and in order for this interval to contain an even number, it must be either that 

\begin{enumerate}
\item $m$ satisfies \(\frac{m}{2k+1}>2\) or
\item $m$ is odd and satisfies \(1<\frac{m}{2k+1} \leq 2\).
\end{enumerate}
We show that in both cases we have an absurd. In the first case, using the second inequality in \eqref{eq:LS eq} we deduce
$$
4k+2<m<2k+3,
$$
which cannot hold. In the second case, we have
$$
2k+1<m<2k+3 \Rightarrow m=2k+2,
$$
contradicting the assumption of $m$ being odd. Hence such integers $a$ and $m$ cannot exist, and $Y_0$ is an $L$-space.
\end{proof}

We conclude this section by finding more $L$-space surgeries on $L'_{n,k}$.

\begin{prop}\label{prop: L space surgeries on L'nk small box}
For each integers $r\in [1-2n,-2]$ and $s$ satisfying $r+s=-1-2n$, the manifold \(S^3_{r,s}(L'_{n,k})\) is an $L$-space.
\end{prop}

\begin{proof}
First of all, we can assume $-n\leq r\leq -2$. In fact, if $r<-n$, then 
$$
s=-1-2n-r>-1-2n+n>-1-n \Rightarrow s\geq -n 
$$
and \(S^3_{r,s}(L'_{n,k})=S^3_{s,r}(L'_{n,k})\) since two-bridge links are symmetric. We have 
\[
S^3_{r,s}(L'_{n,k})=S^3_{r+n,s+n,\frac{1}{n}}(\mathcal{L})= S^3_{x,-1-x,\frac{1}{n}}(\mathcal{L}),
\]
where $\mathcal{L}$ is the link depicted in Figure \ref{figure: L-space surgeries2}, we have set $x=n+r$, and we have used 
$$
s+n=-1-2n-r+n=-1-n-r=-1-x.
$$
Notice our assumptions imply $0\leq x<n-1$. If we denote $Y=S^3_{x,-1-x,\bullet}(\mathcal{L})$, then we know from Lemma \ref{lemma: Y0 is L-space} that $Y_0$ and $Y_{\frac{1}{x}}$ are $L$-spaces, and by Theorem~\ref{theorem RR} the same holds for all surgeries whose coefficient is contained in the interval between $0$ and $\frac{1}{x}$ not containing the homological longitude of $Y$. By Lemma~\ref{lemma: longitude is neg}, the homological longitude is negative, and hence the for all surgery coefficients \(r'\in [0,\frac{1}{x}]\) the manifold \(Y_{r'}\) is an $L$-space. Since $x<n$, the rational \(\frac{1}{n}\) belongs to this interval, and we obtain the desired result.
\end{proof}

\begin{corollary}\label{cor: contains convex hull}
Fix $n\geq 2, k\geq 1$. The set \(L(L'_{n,k})\) contains 
\[
\mathcal{B}\cup H_+\cup H_-\cup L_+\cup L_-,
\]
and hence its box-convex hull. \qed
\end{corollary}

\subsubsection{Constraints on the set of $L$-space surgeries.}\label{subsubsec: constraints L-space surgeries}
In this section we complete the proof of the classification Theorem~\ref{thm: gen L-space links}. Recall that $L'_{n,k}$ is hyperbolic if and only if $n\geq 2$ and $k\geq 1$, and this will be the assumption for the course of the section. 

\begin{comment}\begin{thm}
Let $r,s$ be rational numbers. Then the manifold $S^3_{r,s}(L'_{n,k})$ is an $L$-space if and only, up to switching $r$ and $s$, we have:   
\begin{itemize}
\itemsep 0.7em 

\item 
$r> 2k+1$ and $s\leq \frac{1}{\ceil{\frac{r}{l}-\frac{2k+1}{l}}}$
\item 
$r< -(2n-1)$ and $s\geq \frac{1}{\left\lfloor{\frac{r}{l}+\frac{2n-1}{l}}\right\rfloor}$
\item $r\in [-a,-a+1)$ with $-a\in \Z\cap [1-2n,-2]$ and $s\geq -1-2n+a$
\item $r\in (a-1,a]$ with $a\in \Z\cap [2,2k+1]$ and $s\geq 2k+3-a$.
\end{itemize}
\end{thm}
The symmetry in the surgery coefficients in the statement of the theorem is due to the fact that two-bridge links are symmetric.
\end{comment}
We denote $L=L'_{n,k}$, and $l=n+k$. Observe that, by virtue of Lemma \ref{lemma: mirror of generalised L-space links}, we can limit ourselves to study surgeries \(S^3_{r,s}(L)\) with $r\leq 0$. Also notice that Theorem~\ref{thm: gen L-space links} holds when $r=0$, since the manifold $S^3_{0,s}(L)$ is an $L$-space for all $s\ne 0$. For this reason, from now on we assume $r<0$.
\newline

The proof is roughly divided in three parts. We write $r=-\frac{p}{q}$ for coprime $p,q>0$. We begin by assuming $p$ coprime with $l$, in which case the first homology group of $S^3_{r,\bullet}(L)$ is isomorphic to $\mathbb{Z}$. We prove Theorem~\ref{thm: gen L-space links} by computing the Turaev torsion of $S^3_{r,\bullet}(L)$ and by using Theorem \ref{theorem RR}. We do this by studying the cases
\begin{itemize}
\item $r<-(2n-1)$
\item $r\in [-a,-a+1]$ with $-a\in \mathbb{Z}\cap[1-2n,-2]$
\end{itemize}
separately. Finally, we prove the theorem for general $r<0$, by approximating $r$ with rationals whose numerator is coprime with $l$.
\newline

As premised, we fix $r=\frac{p}{-q}$ with $p,q$ coprime and $p>0, q>0$, and suppose $p$ coprime with $l$. We fix canonical meridian-longitude bases $\mu_i,\lambda_i$ for the components $K_1$ and $K_2$ of $L$. In this case we have 
$$
H_1(S^3_{r,\bullet}(L); \Z)=\dfrac{<\mu_1, \mu_2>}{p\mu_1-ql\mu_2}\cong \Z
$$
and we fix the class of the element $a\mu_1+ b\mu_2$ as positive generator of $\Z$, where $a,b\ \in \Z$ satisfy
\[
\det\begin{pmatrix}
p & a  \\
-ql & b 
\end{pmatrix}=1.
\]

The quotient map $\pi:H_1(E_L; \Z)\rightarrow H_1(S^3_{r,\bullet}(L); \Z)$ in these coordinates is:
\[
\begin{array}{ccc}
\Z^2&\overset{\pi}{\longrightarrow}&\Z\\
\mu_1=(1,0)&\longmapsto&ql\\
\mu_2=(0,1)&\longmapsto& p
\end{array}
\]
and the induced map $\pi_{\#}$ to the group rings is:
\[
\begin{array}{ccc}
\Z[H_1(E_L;\Z)]=\Z[x,y]&\overset{\pi_{\#}}{\longrightarrow}& \Z[H_1(S^3_{r,\bullet}(L); \Z)]=\Z[t]\\
u(x,y)&\longmapsto&u(t^{ql}, t^p)
\end{array}
\]
\begin{prop}\label{prop: Turaev torsion general}
The Turaev torsion of $S^3_{r,\bullet}(L)$ is represented by:

$$
\tau(S^3_{r,\bullet}(L))=\sum_{i=1}^{n-1}\sum_{j=0}^{k-1}\sum_{a=0}^{q-1}t^{(i+j)p}t^{(j-i)ql}t^{al} + \left(\sum_{j=0}^{k}t^{j(p+ql)}+ \sum_{i=1}^{n-1}t^{(k+i)p}t^{(k-i)ql}\right)(1-t^l)^{-1}
$$

where $(1-t^l)^{-1}=1+t^l+t^{2l}+\cdots$ is expanded as a formal sum.
\end{prop}

\begin{proof}
We use the gluing formula of Theorem \ref{thm: gluing formula turaev}. 
This states that, up to sign and multiplication by $t^{\pm m}$, we have:
\[
\pi_{\#}(\tau(E_L))=(h-1)\tau(S^3_{r,\bullet}(L))
\]
where $h$ is the homology class of the oriented core of the solid torus.
We know by Theorem \ref{Thm: torsion vs alexander} that the Turaev torsion of $E_L$ is equivalent to the Alexander polynomial $\Delta'_{n,k}$, which was computed in Proposition \ref{prop: Alex poly of L'_{n,k}}, and a direct computation shows that 
\begin{align*}
\pi_{\#}(\Delta'_{n,k})&=
\sum_{i=0}^{n-1}\sum_{j=0}^k t^{i(p-ql)}t^{j(p+ql)}-t^p\sum_{i=0}^{n-2}\sum_{j=0}^{k-1}t^{i(p-ql)}t^{j(p+ql)}\\
&=\sum_{i=0}^{n-1}\sum_{j=0}^k t^{(i+j)p}t^{(j-i)ql}-t^p\sum_{i=0}^{n-2}\sum_{j=0}^{k-1}t^{(i+j)p}t^{(j-i)ql}.
\end{align*}
We determine now the homology class $h$ of the core of the solid torus. We fix the canonical meridian and longitude $(\mu_1, \lambda_1)$  for the boundary component of $E_L$ affected by the Dehn filling, and a basis $(\mu, \lambda)$ for the boundary of the glued solid torus, where $\mu$ is the meridian and $\lambda$ is any longitude, oriented in such a way that their intersection is positive when the boundary torus is oriented with the orientation induced by the orientation of the solid torus. 
The gluing is then parametrised by a matrix
\[
\begin{pmatrix}
p & \alpha  \\
-q & \beta 
\end{pmatrix}
\]
with determinant $1$. We take $\lambda$ as representative of the core of the solid torus and its class in $H_1(S^3_{r,\bullet}(L); \Z)=\Z$ is
$$
\pi(\alpha \mu_1 + \beta \lambda_1)=\pi(\alpha \mu_1 + l\beta \mu_2)=l(q\alpha + \beta p)=l.
$$
Hence we have
\begin{align*}
\tau(S^3_{r,\bullet}(L))&=\left(\sum_{i=0}^{n-1}\sum_{j=0}^k t^{(i+j)p}t^{(j-i)ql}-t^p\sum_{i=0}^{n-2}\sum_{j=0}^{k-1}t^{(i+j)p}t^{(j-i)ql}\right)(1-t^l)^{-1},
\end{align*}
where the sign is chosen so that in the end we will get a formal sum with only finitely many coefficients different from $1$.

We can simplify the above expression by noting that for every $0\leq i \leq n-2$ and $0\leq j \leq k-1$ the term
$$
-t^pt^{(i+j)p}t^{(j-i)ql}(1+t^l+t^{2l}+\cdots)
$$
cancels with 
$$
t^{(i'+j)p}t^{(j-i')ql}(t^{ql}+t^{(q+1)l}+\cdots)
$$
where $i'=i+1$, and so we get
\begin{align*}
\tau(S^3_{r,\bullet}(L))&=\sum_{i=0}^{n-2}\sum_{j=0}^{k-1}\sum_{a=0}^{q-1}t^{(i+1+j)p}t^{(j-i-1)ql}t^{al} + \left(\sum_{j=0}^{k}t^{j(p+ql)}+ \sum_{i=1}^{n-1}t^{(k+i)p}t^{(k-i)ql}\right)(1-t^l)^{-1}\\
&=\sum_{i=1}^{n-1}\sum_{j=0}^{k-1}\sum_{a=0}^{q-1}t^{(i+j)p}t^{(j-i)ql}t^{al} + \left(\sum_{j=0}^{k}t^{j(p+ql)}+ \sum_{i=1}^{n-1}t^{(k+i)p}t^{(k-i)ql}\right)(1-t^l)^{-1}
\end{align*}
which is the desired result.
\end{proof}

\begin{remark}\label{rem: 0 is in support}
Note that the coefficient of the trivial element in $H_1(S^3_{r,\bullet}(L);\Z)$ is non-zero. Moreover when $p-ql\geq 0$ the support of $\tau$ is contained in $\Z_{\geq 0}$.  
\end{remark}

We now assume $r<-(2n-1)$. We recap the current assumptions and our goal.
\begin{tcolorbox}
\textbf{Standing assumptions:} $r<-(2n-1)$, $r=\frac{p}{-q}$, with $p>0, q>0$ coprime integers. Equivalently, we suppose $\frac{p}{q}> 2n-1$. Moreover, $p$ is coprime with $l$.
\newline

\textbf{Goal:} Prove that $S^3_{r,s}(L)$ is an $L$-space if and only if $
s\geq \dfrac{1}{\left\lfloor{\frac{r}{l}+\frac{2n-1}{l}}\right\rfloor}$
\end{tcolorbox}
We wanto to understand $\iota^{-1}(D^{\tau}_{>0})\subset H_1(\partial S^3_{r,\bullet}(L);\Z)$.
We start by proving the following lemma.

\begin{lem}\label{lem: bound on element in support}
Let $m\geq 0$ be an integer and suppose that $m\equiv (l-1)p \mod{l}$. Then $m\in S[\tau]$ if and only if $m\geq (l-1)p + (k-n+1)ql$.
\end{lem}

\begin{proof}
The Turaev torsion of $S^3_{r,\bullet}(L)$ is sum of three terms $A, B, C$ where:

\begin{align*}
&A=\sum_{i=1}^{n-1}\sum_{j=0}^{k-1}\sum_{a=0}^{q-1}t^{(i+j)p}t^{(j-i)ql}t^{al}  \\
&B=\left(\sum_{j=0}^{k}t^{j(p+ql)}\right)(1+t^l+t^{2l}+\cdots) \\
&C=\left(\sum_{i=1}^{n-1}t^{(k+i)p}t^{(k-i)ql}\right)(1+t^l+t^{2l}+\cdots).
\end{align*}
The elements in the support of $A$ are congruent to $(i+j)p \mod{l}$, for some $1\leq i\leq n-1$ and $0\leq j\leq k-1$, and those in the support of $B$ are congruent to $jp \mod{l}$, for some $0\leq j\leq k$. This implies that $m$ is in the support of $\tau$ if and only if it is in the support of $C$. The elements in the support of $C$ that are congruent to $(l-1)p=(n+k-1)p \mod{l}$ are of the form $(l-1)p + (k-n+1)ql +cl$, with $c\in \Z_{\geq 0}$ and this proves the lemma.
\end{proof}

We fix the canonical meridian and longitude basis $(\mu_2, \lambda_2)$ for the first homology group of the boundary of \(S^3_{r,\bullet}(L)\), and we use it to consider the image of $\iota^{-1}(D^{\tau}_{>0})$ in $\Q \cup \{\infty\}$ via the map $\varphi:x\mu_2+ y\lambda_2\mapsto \frac{x}{y}$.

\begin{remark}\label{rem: infinity is an endpoint}
By Corollary~\ref{cor: contains convex hull}, we know that $S^3_{r,\bullet}(L)$ is a Floer simple manifold and that $\infty$ is an $L$-space filling slope. Moreover $\infty$ must be one of the endpoints of the interval of $L$-space fillings on $S^3_{r,\bullet}(L)$. In fact, if this was not the case, we would deduce that for $r_1, r_2<<0$ the manifold $S^3_{r_1,r_2}(L)$ is an $L$-space, contradicting Proposition  \ref{prop: some ctf on L'nk}. 
\end{remark}

The previous remark and Theorem \ref{theorem RR} imply that in order to completely determine the set of $L$-space filling slopes of $S^3_{r, \bullet}(L)$ it is enough to locate the maximum point of the image of $\iota^{-1}(D^{\tau}_>0)$ in $\Q$. We denote this set by $\mathcal{P}$.
Thus our goal is to show that the element
\[
\rho=\dfrac{1}{\left\lfloor{\frac{r}{l}+\frac{2n-1}{l}}\right\rfloor}
\]
is the maximum element of $\mathcal{P}$.

We first prove that $\rho$ belongs $\mathcal{P}$, and then that it is its maximum. 
Recall that we denote by $\iota:H_1(\partial S^3_{r,\bullet}(L))\rightarrow H_1(S^3_{r,\bullet}(L))=\Z$ the map induced by the inclusion.

\begin{lem}\label{lemma: v_- belongs to D^t}
The element $v=-p + ql^2\left\lceil{\frac{p}{ql}-\frac{2n-1}{l}}\right\rceil$ belongs to $D^{\tau}_{>0}$.
\end{lem}

\begin{proof}
Recall that by definition $D^{\tau}_{>0}$ is the subset of $H_1(S^3_{r,\bullet}(L);\Z)=\Z$ defined as
$$
D^{\tau}_{>0}=\{x-y\,| x\notin S[\tau],\, y\in S[\tau] \text{ and } x>y\}\cap \iota(H_1(\partial S^3_{r,\bullet}(L);\Z).
$$
It follows from Lemma \ref{lem: v is in the image} that $v$ is in the image of $\iota$ (which is indeed a surjection). We know from Remark \ref{rem: 0 is in support} that $0$ is in the support of $\tau$, and hence to prove the lemma it is enough to show that $v$ is a positive number that is not in $S[\tau]$.
\begin{itemize}[leftmargin=*]
\item \emph{$v$ is positive:} denote $R=\left\lceil{\frac{p}{ql}-\frac{2n-1}{l}}\right\rceil$ and we observe that our assumptions imply that $R\geq 1$. We have:
$$
v=-p + ql^2R>0 \Leftrightarrow ql^2R>p \Leftrightarrow lR> \frac{p}{ql}
$$
and by the definition of $R$ we know that 
\[
\frac{p}{ql}\leq R + \frac{2n-1}{l}< R +2.
\]
Therefore, it is enough to prove that $lR\geq R+2$, or equivalently that $(l-1)R\geq 2$. Since we are working under the assumptions that $n\geq 2$ and $k\geq 1$, we have that $l=n+k\geq 3$ and conclude.
\item \emph{$v$ is not in the support of $\tau$:} since $v\equiv (l-1)p \mod{l}$, by virtue of Lemma \ref{lem: bound on element in support}, it is sufficient to prove that $v<(l-1)p + (k-n+1)ql$ and this holds if and only if
$$
ql^2R<lp + (k-n+1)ql. 
$$
By the definition of $R$ we have 
\begin{align*}
\frac{p}{ql}-\frac{2n-1}{l}> R-1 \Rightarrow p>(2n-1)q+(R-1)ql \Rightarrow \\
\Rightarrow pl>(2n-1)ql+(R-1)ql^2=Rql^2 +(n-k-1)ql
\end{align*}
and therefore 
$$
lp + (k-n+1)ql>Rql^2 +(n-k-1)ql + (k-n+1)ql= Rql^2
$$
which is what we wanted to prove.
\end{itemize}
This concludes the proof.
\end{proof}
We now consider the class in $H_1(\partial S^3_{r,\bullet}(L);\Z)$ defined as
\[
u=-\mu_2 + \left\lceil{\frac{p}{ql}-\frac{2n-1}{l}}\right\rceil\lambda_2
\]
\begin{lem}\label{lem: v is in the image}
The image of $u$  via the map $\iota$ is $v$. 
\end{lem}

\begin{proof}
We have the following commutative diagrams of inclusions:
\[\begin{tikzcd}
\partial S^3_{r,\bullet}(L) \arrow{d}{i_1}\arrow{r}{i}   &  S^3_{r,\bullet}(L)\\
E_L\arrow{ur}{i_2}
\end{tikzcd}
\]
which induces a corresponding commutative diagrams on the homologies. The map $i_1$ sends the homology classes of $\mu_2,\lambda_2\in H_1(S^3_{r,\bullet}(L);\Z)$ to $\mu_2$ and $l\mu_1$ respectively. We already showed at the beginning of Section~\ref{subsubsec: constraints L-space surgeries} that the map $i_2$ sends $\mu_2$ to $p\in \Z=H_1(S^3_{r,\bullet}(L))$ and $\mu_1$ to $ql$.  By commutativity we get $\iota(\mu_2)=p$ and $\iota(\lambda_2)=ql^2$.
\end{proof}
\begin{comment}
We start by noting that the preimage of this element in $H_1(\partial S^3_{r,\bullet}(L))$ are the classes

\[
u_{\pm}=\pm(\mu_2 + \left\lfloor{\frac{r}{l}+\frac{2n-1}{l}}\right\rfloor\lambda_2)=\pm(\mu_2 - \left\lceil{\frac{-r}{l}+\frac{2n-1}{l}}\right\rceil\lambda_2)=\pm(\mu_2 - \left\lceil{\frac{p}{ql}-\frac{2n-1}{l}}\right\rceil\lambda_2)
\]
Recall that we denote by $\iota:H_1(\partial S^3_{r,\bullet}(L))\rightarrow H_1(S^3_{r,\bullet}(L))=\Z$ the map induced by the inclusion.
\begin{lem}\label{lemma: two candidates}
The image of $u_{\pm}$ in $H_1(S^3_{r,\bullet}(L))$ via the map $\iota$ is the point $v_{\pm}=\pm(p - ql^2\left\lceil{\frac{p}{ql}-\frac{2n-1}{l}}\right\rceil)$. 
\end{lem}

\begin{proof}
We have the following commutative diagrams of inclusions:
\[\begin{tikzcd}
\partial S^3_{r,\bullet}(L) \arrow{d}{i_1}\arrow{r}{i}   &  S^3_{r,\bullet}(L)\\
E_L\arrow{ur}{i_2}
\end{tikzcd}
\]
which induces the corresponding commutative diagrams on the homologies. The map $i_1$ sends the homology classes of $\mu_2,\lambda_2\in H_1(S^3_{r,\bullet}(L))$ to $\mu_2$ and $l\mu_1$ respectively. We already showed that the map $i_2$ sends $\mu_2$ to $p\in \Z=H_1(S^3_{r,\bullet}(L))$ and $\mu_1$ to $ql$.  By commutativity we get $\iota(\mu_2)=p$ and $\iota(\lambda_2)=ql^2$.
\end{proof}
\end{comment}

\begin{corollary}
The element $\rho$ belongs to \(\mathcal{P}\).
\end{corollary}

\begin{proof}
It follows from Lemmas~\ref{lemma: v_- belongs to D^t} and \ref{lem: v is in the image} that \(u\) belongs to \(\iota^{-1}(D^\tau_{>0})\). We can rewrite \(u\) as
\[
u=-\mu_2+\left\lceil \frac{p}{ql}-\frac{2n-1}{l}\right\rceil \lambda_2
=-\mu_2-\left\lfloor \frac{p}{-ql}+\frac{2n-1}{l}\right\rfloor \lambda_2
=-\mu_2-\left\lfloor \frac{r}{l}+\frac{2n-1}{l}\right\rfloor \lambda_2.
\]
Its image in \(\mathbb{Q}\cup\{\infty\}\) under the map \(\varphi\) is therefore $\rho$, as claimed.
\end{proof}

\begin{prop}\label{prop: the desired element is the maximum}
The element \(\rho\) is the maximum of \(\mathcal{P}\).
\end{prop}

\begin{proof}
If $\rho$ is not the maximum of $\mathcal{P}$, then it follows from Theorem \ref{theorem RR}, that there exist $\rho<s'\in \mathbb{Q}$ such that
$$
S^3_{r,s}(L) \text{ is an $L$-space } \Leftrightarrow
s\geq s'.
$$
We show that this produces an absurd by proving that $S^3_{r,\rho}(L)$ is an $L$-space. It follows from our assumptions that $\rho=-\frac{1}{m}$ for some integer $m>0$. By Proposition~\ref{prop: L-space knots as surgeries on L'}, we know that the manifold $S^3_{r',-\frac{1}{m}}(L)$ is an $L$-space for all $r'\geq 1-2n-ml$. In particular, since
\[
1-2n-ml=1-2n+\left\lfloor{\frac{r}{l}+\frac{2n-1}{l}}\right\rfloor l\leq 1-2n +(r+2n-1)=r,
\]
we deduce that $S^3_{r,-\frac{1}{m}}(L)=S^3_{r,\rho}(L)$ is an $L$-space, which is what we wanted.
\end{proof}

\begin{corollary}\label{cor: L-space surgeries out of box}
Let \(r=-\frac{p}{q}<-(2n-1)\) and suppose that $p$ is coprime with $l$. Then \(S^3_{r,s}(L)\) is an $L$-space if and only if \(
s\geq \dfrac{1}{\left\lfloor{\frac{r}{l}+\frac{2n-1}{l}}\right\rfloor}\)
\end{corollary}

\begin{proof}
The result follows by combining Remark~\ref{rem: infinity is an endpoint} and Proposition \ref{prop: the desired element is the maximum} with Theorem \ref{theorem RR}.
\end{proof}

We now focus our attention to the case when $r\in [-a,-a+1)$ with $-a\in \Z\cap [1-2n,-2]$. We summarise here our setting.

\begin{tcolorbox}
\textbf{Standing assumptions:} $r=-\frac{p}{q}\in [-a,-a+1)$ with $-a\in \Z\cap [1-2n,-2]$ and $p$ is coprime with $l$.
\newline

\textbf{Goal:} Prove that $S^3_{r,s}(L)$ is an $L$-space if and only if $s\geq -1-2n+a$.
\end{tcolorbox}

First of all, we show that we can assume $a\in [1+n,2n-1]$.

\begin{lem} \label{lemma: assumption on a}
If \emph{\textbf{Goal}} holds for $a\in [1+n, 2n-1]$, then it holds for $a\in [2,2n-1]$.
\end{lem}

\begin{proof}
The lemma is a consequence of the fact that two-bridge links are symmetric. Assume $r\in [-a,-a+1)$ with $a\in [2,n]$. First of all, observe that we already know by Proposition~\ref{prop: L space surgeries on L'nk small box} and Corollary~\ref{cor: contains convex hull} that $S^3_{r,s}(L)$ is an $L$-space when $s\geq -1-2n+a$, so we only need to prove the inverse implication. By Propositions~\ref{prop: some ctf on L'nk} and \ref{prop: box-convex}, it is enough to show that for arbitrarily small $\varepsilon > 0$, the manifold $S^3_{r,s'}(L)$ is not an $L$-space, where $s'=-1-2n+a-\varepsilon$. Notice that 
$$
-1-2n+a-1 \leq s' < -1-2n+a,
$$
and if we set $-a'=-1-2n+a-1$, then the assumption on $a$ implies $a'\geq n+1$. Since $s'\in [-a', -a'+1)$, and \textbf{Goal} holds for $a'$, we have that $S^3_{r,s'}(L)=S^3_{s',r}(L)$
is an $L$-space if and only if $r\geq -1-2n+a'=1-a$. Since $r< 1-a$, the manifold $S^3_{r,s'}(L)$ is not an $L$-space, and this implies the desired result.
\end{proof}

From now on we assume $a\in[1+n,2n-1]$. Analogously as we did before, we want to show that $-1-2n+a$ is the maximum element of $\mathcal{P}$. Recall that $\mathcal{P}$ denotes the image of $\iota^{-1}(D^{\tau}_>0)$ in $\Q$.

\begin{lem}\label{lemma: v belongs to Dt 2}
The element $v=(a-1-2n)p + ql^2$ belongs to $D^{\tau}_{>0}$.
\end{lem}

\begin{proof}
The proof is similar to the one of Lemma \ref{lemma: v_- belongs to D^t}. Since $0$ is in the support of $\tau$ by Remark \ref{rem: 0 is in support}, it is enough to show that $v$ is positive and that it does not belong to the support of $\tau$. We start by proving that $v$ is positive, or equivalently that 
$$
(1+2n-a)\frac{p}{q}< l^2.
$$
Recall that by assumption we know that $a-1< \frac{p}{q} \leq a$ and that $a\in [1+n,2n-1]$. This implies that $(1+2n-a)\frac{p}{q}\leq (1+2n-a)a$, and hence it is enough to show that 
$$
(1+2n-a)a < (n+1)^2 \leq l^2.
$$
We have 
$$
(1+2n-a)a < (n+1)^2 \Longleftrightarrow a-1-2n < (a-n)^2,
$$
and the last inequality holds since by our assumptions the left-hand side is negative.  To show that $v$ is not in the support of $\tau$, we write $\tau=A+B+C$ as in the proof of Lemma 
\ref{lem: bound on element in support} and prove that $v$ is not in the support of any of them.

\begin{itemize}[leftmargin=*]
\item \emph{$v$ is not in the support of $A$:} recall that 
$$
A=\sum_{i=1}^{n-1}\sum_{j=0}^{k-1}\sum_{\alpha=0}^{q-1}t^{(i+j)p}t^{(j-i)ql}t^{\alpha l},
$$
and so $v$ is in the support of $A$ if and only if 
$$
v=(a-1-2n)p+ql^2=(i+j)p+(j-1)ql+\alpha l
$$
for some $1\leq i \leq n-1, 0\leq j\leq k-1, 0\leq\alpha\leq q-1$.
Since $p$ and $l$ are coprime, by considering this equality mod $l$ we deduce 
$$
i+j\equiv a-1-2n \mod l,
$$ and since $1\leq i\leq n-1, 0\leq j\leq k-1$ and $a\in [n+1,2n-1]$, we must have 
$i+j=a-1-n+k$. We consider the following set
$$
\mathcal{A}=\{w=(i+j)p+(j-i)ql+\alpha l| \text{ $w$ is in the support $A, i+j=a-1-n+k$}\},
$$
and we claim that for every $w\in \mathcal{A}$ we have $w<v$, and hence $v$ is not in the support of $A$. Of course, it is enough to prove this for the maximum of $\mathcal{A}$, which we denote by $w_{\operatorname{max}}$. This element is obtained by fixing $\alpha=q-1$, and by taking $i,j$ that maximise
$j-i$ in the set 
\[
\{(i,j)\, |\, 1\leq i \leq n-1, 0\leq j\leq k-1 \text{ and } i+j=a-1-n+k\}.
\]
Therefore, $w_{\operatorname{max}}$ is obtained by taking 
$$
i=a-n,\, j=k-1,\, \alpha=q-1,
$$
and to prove the claim we need to show that 
$$
(a-1-n+k)p+(k-1-a+n)ql+(q-1)l < (a-1-2n)p+ql^2=v,
$$
which happens if and only if
\begin{align*}
&lp< ql^2-ql+l-(k+n-1-a)ql \Longleftrightarrow p< ql-q+1-(l-1-a)q=qa+1 \\
&\Longleftrightarrow\frac{p}{q}< a+ \frac{1}{q}
\end{align*}
and the last inequality holds since $\frac{p}{q}\leq a$ by hypothesis.
\item \emph{$v$ is not in the support of $B$:} recall that 
$$
B=\left(\sum_{j=0}^{k}t^{j(p+ql)}\right)(1+t^l+t^{2l}+\cdots).
$$
Hence $v$ is in the support of $B$ if and only if 
$$v=jp+jql+\alpha l \text{ for some $0\leq j\leq k$ and $0 \leq \alpha$}.$$
Arguing as above, if this happen then we must have $j=a-1-n+k$. We consider the set
$$
\mathcal{B}=\{w=jp+jql+\alpha l| \text{ $w$ is in the support $B$ and $j=a-1-n+k$}\}
$$
and its minimum $w_{\operatorname{min}}=(a-1-n+k)(p+ql)$.
We show that $v<w_{\operatorname{min}}$ and, as a consequence, that $v$ is not in the support of $B$. First, we observe that since $a-1<\frac{p}{q}$ and $a-1-2n<0$ we have
\begin{align*}
&v=(a-1-2n)p+ql^2<(a-1)(a-1-2n)q+ql^2 \text{ and}\\
&w_{\operatorname{min}}\geq (a-1-n+k)\big{[}(a-1)q+ql\big{]},
\end{align*}
and so it is enough to show that
$$
(a-1-n+k)\big{[}(a-1)q+ql\big{]}\geq(a-1)(a-1-2n)q+ql^2.
$$
By explicit computation, one sees that the above inequality is equivalent to
$$
l(a-1) \geq l(2n-a+1) \Longleftrightarrow a\geq n+1,
$$
which holds since $a\in[n+1,2n-1]$ by assumption.
\item \emph{$v$ is not in the support of $C$:} the summand $C$ is equal to
$$
C=\big{(}\sum_{i=1}^{n-1}t^{(k+i)p}t^{(k-i)ql}\big{)}(1+t^l+t^{2l}+\cdots)
$$
and, proceeding as above one shows that if $v$ is in its support then we must have
$$
v=(a-1-n+k)p+(k-a+1+n)ql+\alpha l \text{ for some $\alpha\geq 0$}.
$$
We show that 
$$
(a-1-n+k)p+(k-a+1+n)ql\geq \big{[}(a-1)(a-1-2n)+l^2\big{]}q
$$
and conclude as in the previous case. We denote the term on left-hand side by $w'_{\operatorname{min}}$.  Since $\frac{p}{q}\geq a-1$, we have
$$
w'_{\operatorname{min}}\geq q\big{[}(a-1-n+k)(a-1)+(k-a+1+n)l\big{]}
$$
and this implies the desired result, because of the following chain of equivalences
\begin{align*}
&q\big{[}(a-1-n+k)(a-1)+(k-a+1+n)l\big{]}\geq \big{[}(a-1)(a-1-2n)+l^2\big{]}q \Longleftrightarrow\\
&(a-1)(n+k)\geq l^2-l(k-a+1+n) \Longleftrightarrow l(a-1)\geq l(a-1).
\end{align*}
\end{itemize}
In conclusion, since $v$ is not in the support of any of the summands $A, B$ and $C$, we deduce $v\notin S[\tau]$, and hence $v\in D^{\tau}_{>0}$.
\end{proof}

\begin{prop}\label{prop: v is the max 2}
The element $\rho=-1-2n+a$ is the maximum of $\mathcal{P}$.
\end{prop}

\begin{proof}
Recall that $\iota$ denotes the map induced in homology by the inclusion \(\partial S^3_{r,\bullet}(L)\subset S^3_{r,\bullet}\). Arguing as in the proof of Lemma \ref{lem: v is in the image}, one checks that the image of the element 
$$
u=-(1+2n-a)\mu_2 +\lambda_2
$$
via $\iota$ is $v$. This shows that $u\in \iota^{-1}(D^{\tau}_{>0})$. Since we are using the basis $(\mu_2, \lambda_2)$ to identify the slopes on $S^3_{r,\bullet}(L)$ with $\mathbb{Q}\cup \{\infty\}$, this implies that $\rho\in \mathcal{P}$. We now show that it is the maximum of $\mathcal{P}$. By arguing as in the proof of Proposition \ref{prop: the desired element is the maximum} it is enough prove that that $S^3_{r,\rho}(L)$ is an $L$-space. By Proposition \ref{prop: L space surgeries on L'nk small box} we know that $S^3_{-a,\rho}(L)$ is an $L$-space, and since $r\geq -a$, by box-convexity we obtain the desired conclusion. 
\end{proof}

\begin{corollary}\label{cor: class L' small box}
Suppose $r=-\frac{p}{q}\in [-a,-a+1)$ with $-a\in \Z\cap [1-2n,-2]$ and $p$ is coprime with $l$. Then $S^3_{r,s}(L)$ is an $L$-space if and only if $s\geq -1-2n+a$.
\end{corollary}
\begin{proof}
When $a\in[n+1,2n-1]$, this is a consequence of Remark \ref{rem: infinity is an endpoint}, Proposition \ref{prop: v is the max 2} and Theorem \ref{theorem RR}. In the general case, it follows from Lemma \ref{lemma: assumption on a}.
\end{proof}

In order to deal with the case when $p$ and $l$ are not coprime, and conclude the proof in the general case, we need the following elementary lemma.

\begin{lem}\label{lemma: approx}
For any $\varepsilon,\delta>0$ and any rational $r=\frac{p}{q}$, where $p$ and $q$ are coprime integers, there exist rationals $r_1$ and $r_2$ such that 
$$
r-\delta <r_1\leq r\leq r_2<r+\varepsilon,
$$ where $r_i=\frac{p_i}{q_i}$ and $p_i$ and $l$ are coprime for each $i=1,2$.
\end{lem}
\begin{proof}
If $p$ is coprime with $l$ there is nothing to prove. Otherwise, let $a_1,\dots, a_k$ be the distinct primes appearing in the factorisation of $q$, and $b_1,\dots, b_h$ those appearing in the factorisation of $l$. We consider the integers $m$ satisfying 
\[
\begin{cases}
m\not\equiv 0 \mod a_i \quad \forall i=1,\dots, k\\
pm+q\not\equiv 0 \mod b_j \quad \forall j=1,\dots, h
\end{cases}.
\]
There are infinitely many such integers, both negative and positive, and for any such $m$ we consider the rational
$$
\frac{p_m}{q_m}=\frac{p}{q}+\frac{1}{m}=\frac{pm+q}{qm}.
$$
Since $m$ and $q$ are coprime, we have that $\operatorname{gcd}(p_m,q_m)=1$. Moreover, by construction we also have \(\operatorname{gcd}(p_m,l)=1\). By taking $m$ with large enough absolute value, we obtain the desired result. 
\end{proof}

We can now complete the proof of Theorem~\ref{thm: gen L-space links}.
\begin{proof}[Proof of Theorem~\ref{thm: gen L-space links}]
Assume $L$ is a hyperbolic two-bridge link that is a generalised $(+,-)\,L$-space link. Then by Theorem~\ref{thm: persistently fol two-br} and Proposition~\ref{prop: some ctf Lnk}, \(L\) must be isotopic to $L'_{n,k}$ with $n\geq 2,k\geq 1$. 

Thus we fix any such \(L'_{n,k}\). By the reformulation given in Theorem~\ref{thm: reformulation} and Corollary~\ref{cor: contains convex hull}, we know that the $(r,s)$-surgery on $L$ is an $L$-space whenever the pair $(r,s)$ satisfies the hypotheses of the theorem, and we have to prove the converse. By virtue of Lemma \ref{lemma: mirror of generalised L-space links}, it is enough to prove the theorem for $r\leq0$. Moreover,  when $r=0$, the manifold $S^3_{0,s}(L)$ is an $L$-space for all $s\ne 0$, so we can suppose $r<0$. We write $r=-\frac{p}{q}$ for $p,q >0$ coprime integers, and analyse the possible cases:
\begin{figure}[]
   \centering   \includegraphics[width=0.45\textwidth]{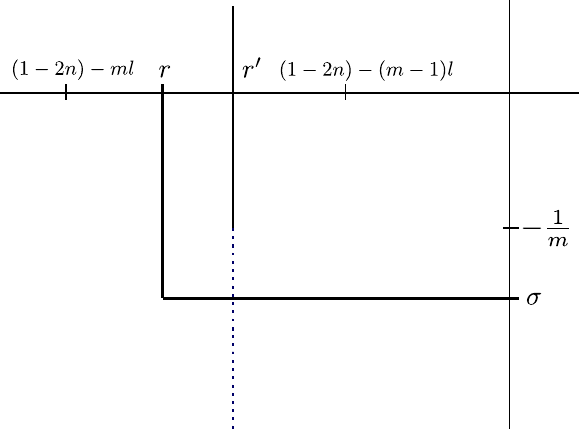}
    \caption{The solid vertical line passing through $(r',0)$ denotes the $L$-spaces surgeries, while the dashed part denotes the non-$L$-space ones, as proved in Corollary \ref{cor: L-space surgeries out of box}. If $S^3_{r,\sigma}(L)$ is an $L$-space, then all the surgeries on the horiziontal solid arc starting from $(r,\sigma)$ are $L$-space and this produces a contradiction.}
    \label{figure: picture sketch}
\end{figure} 
\begin{enumerate}[leftmargin=*]
    \item $r<1-2n$: if $p$ is coprime with $l=n+k$, the result follows from Corollary \ref{cor: L-space surgeries out of box}. In general, we denote
    $$
    -\frac{1}{m}=\frac{1}{\left\lfloor{\frac{r}{l}+\frac{2n-1}{l}}\right\rfloor} 
    $$
    for some $m>0$, so that $1-2n-ml\leq r<1-2n-(m-1)l$. We know by box-convexity that $S^3_{r,s}(L)$ is an $L$-space for $s\geq -\frac{1}{m}$. If the converse implication does not hold, then for all rationals $\sigma< -\frac{1}{m}$ arbitrarily close to $-\frac{1}{m}$ we would have that $S^3_{r,\sigma}(L)$ is an $L$-space. This, by Proposition \ref{prop: box-convex}, implies that the same holds for $S^3_{r',\sigma}(L)$ for all $r'>r$. We take any such $\sigma$ and $r'=-\frac{p'}{q'}$ so that
    $$
    r<r'<(1-2n)-(m-1)l  \, \text{ and } \operatorname{gcd}(p',l)=1,
    $$
    and obtain a contradiction, since $S^3_{r',\sigma}(L)$ cannot be an $L$-space by Corollary~\ref{cor: L-space surgeries out of box}. See Figure \ref{figure: picture sketch} for a pictorial sketch of the proof.
    \item $1-2n\leq r <1$: assume that $r\in [-a,-a+1)$ with $-a\in \Z\cap [1-2n,-2]$. The proof is similar to the one above. If $p$ and $l$ are coprime, we conclude by Corollary \ref{cor: class L' small box}. Otherwise, we conclude by approximating $r$ with $r'=-\frac{p'}{q'}$, where $p'$ coprime with $l$ and $r<r'<-a+1$.
    \item $-1\leq r<0$: this case follows from the ones above. In fact consider $S^3_{r,s}(L)$ for $s\in \mathbb{Q}$. If $s\geq -1$, the manifold is an $L$-space by box-convexity. If $s<-1$ then $S^3_{r,s}(L)=S^3_{s,r}(L)$, and hence this case falls in one of the two cases above.
\end{enumerate}
This concludes the proof.
\end{proof}

Analogously to what happened in the case of the link $L_{n,k}$, we have many non-$L$-space surgeries on \(L'_{n,k}\) that are not covered by the result Proposition~\ref{prop: some ctf on L'nk}. Even though we believe that the proof of Proposition~\ref{prop: some ctf on L'nk} can be refined to produce a larger family of surgeries admitting taut foliations, we do not expect this approach to cover the entire set of non-$L$-space surgeries on $L'_{n,k}$.

\begin{prob}
Construct coorientable taut foliations on the non-$L$-space surgeries on $L'_{n,k}$ that are not covered by Proposition~\ref{prop: some ctf on L'nk} (and its proof). 
\end{prob}

\section{Applications to satellite knots}\label{sec: satellite}
In this final section, we prove Theorem \ref{thm: satellite}, which classifies $L$-space satellite knots whose pattern link is a two-bridge link. 

Recall from the introduction that we denote by $P(K,m)$ the $m$-twisted satellite knot of $K$ with pattern $P$. In this setting, the pattern link $L_P=P \cup K_0$ is the two-component link in $S^3$ obtained by viewing the solid torus containing $P$ as a tubular neighbourhood of an unknot in $S^3$, whose meridian we denote by $K_0$. Since the case when $L_P$ is a torus two-bridge link follows from the results of \cite{HeddenCabling2,Homcable}, we focus on the hyperbolic case. 

\begin{namedtheorem}[\ref{thm: satellite}]
Let $P(K,m)$ be the satellite knot of a non-trivial knot $K$, and assume that the pattern link $L_P$ is a hyperbolic two-bridge link. Then $P(K,m)$ is an $L$-space knot if and only if 
\begin{itemize}
\item$K$ is an $L$-space knot, 
\item $L_P$ is isotopic to $L'_{n,k}$ with $n\geq 2, k\geq 1$, and
\item $m> 2g(K)-1$.
\end{itemize}
Moreover, if $L_P$ is not isotopic to any of the links $L'_{n,k}$, then every non-trivial surgery on $P(K,m)$ supports a coorientable taut foliation.
\end{namedtheorem}

We start by fixing some notation. We denote by $\mu_{K},\lambda_{K}$ the meridian and longitude of $K$, and by $\mu_0, \lambda_{0}$ and $\mu_{P},\lambda_{P}$ the meridian and longitude of $K_0$ and $P$ respectively. We use these meridian-longitude bases to identify the slopes on the various boundary components with $\overline{\mathbb{Q}}$. By construction, the exterior $E_{P(K,m)}$ of $P(K,m)$ is diffeomorphic to $E_{L_P}\cup_{\varphi}E_{K}$, where the gluing map
$$
\varphi:  \partial \nu K_0\rightarrow \partial \nu K 
$$
satisfies $\varphi(\mu_0)=m\mu_K+\lambda_K$ and $\varphi(\lambda_0)=\mu_K$.

\begin{lem}\label{lemma: satellite}
Suppose that $K\subset S^3$ is a non-trivial $L$-space knot, and that $L_P$ is isotopic to $L'_{n,k}$ for some $n\geq2, k\geq1$. Then $P(K,m)$ is an $L$-space knot if and only if $m > 2g(K)-1$.    
\end{lem}

\begin{proof}
By definition, $P(K,m)$ is an $L$-space knot if and only if large surgeries on it are $L$-spaces, i.e. if and only if large fillings of the manifold $E_{P_L}\cup_{\varphi}E_{K}$ are $L$-spaces\footnote{Note that in general the longitude $\lambda_P$ that we use to identify slopes on the boundary with $\overline{\mathbb{Q}}$ does not coincide with the canonical longitude of the knot $P(K,m)$. Nonetheless, the property of having large surgeries that are $L$-spaces does not depend on the choice of the longitude.}. Since $L_P$ is hyperbolic by assumption, we have that $P(K,m)$ is an $L$-space knot if and only for large enough rationals $r$ the $r$-filling $M_r$ of $E_{L_P}\cup_{\varphi}E_{K}$ is an $L$-space and $S^3_{\bullet,r}(L_P)$ is hyperbolic. Note that by construction we have 
$$
M_r\cong S^3_{\bullet,r}(L_P)\cup_{\varphi} E_K,
$$
and both $S^3_{\bullet,r}(L_P)$ and $E_K$ have incompressible boundary, the first manifold being hyperbolic, and the second being the exterior of a non-trivial knot in $S^3$. Then, by \cite[Theorem~1.14]{HRW}, $M_r$ is an $L$-space if and only if 
$$
\varphi(L^{\circ}(S^3_{\bullet,r}))\cup L^{\circ}(E_K)=\Sl(E_K),
$$
where we use the superscript $^{\circ}$ to denote the interior of a set, and where, with a slight abuse of notation, we denote by $\varphi$ the bijection induced on the set of slopes by the gluing map $\varphi$. Of course this is equivalent to 
$$
L^{\circ}(S^3_{\bullet,r})\cup \varphi^{-1}(L^{\circ}(E_K))=\Sl(S^3_{\bullet,r}).
$$
It is straightforward from the definition that the map $\varphi^{-1}$ identifies the slope $s$ with $(s-m)^{-1}$. Since $K$ is an $L$-space knot, we have that $L(E_K)=[2g(K)-1,\infty]$ and hence
$$
\varphi^{-1}({L^{\circ}(E_K)})=\left(0,\frac{1}{2g(K)-1-m}\right),
$$
where the right-hand side has to be intended as the oriented interval in $\overline{\mathbb{Q}}$ from $0$ to $\frac{1}{2g(K)-1-m}$; for example, when \(\frac{1}{2g(K)-1-m}\) is negative, it contains $\infty$. The result follows from the classification of $L$-space surgeries on $L'_{n,k}$ given by Theorem \ref{thm: gen L-space links}. 
\end{proof}

\begin{proof}[Proof of Theorem \ref{thm: satellite}]
By repeating verbatim the proof of \cite[Theorem~1.11]{San23twobridge} and using Theorem~\ref{thm: pers fol link} and Proposition~\ref{prop: some ctf Lnk}, one deduces that if the pattern link \(L_P\) is not isotopic to \(L'_{n,k}\), for any \(n,k\geq 0\), then every non-trivial surgery on \(P(K,m)\) supports a coorientable taut foliation. In particular, \(P(K,m)\) is not an \(L\)-space link.

It remains to consider the case $L_P=L'_{n,k}$. Since $L_P$ is hyperbolic by assumption, we must have $n\geq2, k\geq 1$. If $K$ is an $L$-space knot, then the conclusion follows from Lemma \ref{lemma: satellite}. Conversely, if $P(K,m)$ is an $L$-space knot, then by using \cite[Theorem~1.14]{HRW} and the classification given by Theorem \ref{thm: gen L-space links} as in the proof of Lemma \ref{lemma: satellite}, one deduces that $K$ must be an $L$-space knot, and that the desired result holds\footnote{We point out that the statement of \cite[Theorem~1.15]{HRW}, which implies that $K$ is an $L$-space knot cannot be directly applied here, since the authors in \cite{HRW} use a slightly different definition of $L$-space knot (they do not require the $L$-space surgery to be positive).}.
\end{proof}

\bibliographystyle{amsalpha}
\bibliography{bibliography}

\end{document}